\documentclass[12pt]{amsart}
\usepackage{amsrefs,amsmath,amssymb}
\usepackage{mathrsfs,MnSymbol,stmaryrd}
\usepackage{url,color}
\usepackage{upgreek}
\usepackage{fancyhdr}
\usepackage{hyperref}
\usepackage{caption}
\usepackage[percent]{overpic}
\usepackage{pict2e}
\usepackage{scalefnt}
\usepackage{url}
\usepackage{tensor}
\usepackage{tikz}

\newtheorem{theorem}{Theorem}[section]
\newtheorem{proposition}[theorem]{Proposition}
\newtheorem{lemma}[theorem]{Lemma}
\newtheorem{corollary}[theorem]{Corollary}

\theoremstyle{definition}
\newtheorem{definition}[theorem]{Definition}
\newtheorem{remark}[theorem]{Remark}
\newtheorem{notation}[theorem]{Notation}

\numberwithin{equation}{section}

\newcommand{\leftexp}[2]{{\vphantom{#2}}^{#1}{#2}}

\newcommand{\Pol}{\mathscr{P}}

\newcommand{\TBang}{T_{Bang}}
\newcommand{\TCrunch}{T_{Crunch}}

\newcommand{\Mfour}{\mathbf{M}}

\newcommand{\gfour}{\mathbf{g}}
\newcommand{\StMet}{\gamma}
\newcommand{\newg}{G}
\newcommand{\EnergyEstimatesMetric}{\mathbf{G}}

\newcommand{\Rot}[2]{\Omega_{(#1 #2)}}

\newcommand{\ScalarCur}{\mbox{\upshape R}}
\newcommand{\ScalarCurArg}[1]{\mbox{\upshape R}[#1]}
\newcommand{\Ric}{\mbox{\upshape Ric}}

\newcommand{\Riem}{\mbox{\upshape Riem}}

\newcommand{\scale}{\mathcal{A}}
\newcommand{\antiscale}{\Upsilon}
\newcommand{\inverseantiscale}{\Upsilon^{-1}}

\newcommand{\Hubble}{\mathcal{H}}

\newcommand{\mytr}{\mbox{\upshape tr}_{\newg}}

\newcommand{\newlapse}{N}

\newcommand{\SecondFund}{k}
\newcommand{\FreeSecondFund}{\hat{k}}
\newcommand{\FreeNewSec}{\hat{K}}

\newcommand{\newtimescalar}{\Psi}
\newcommand{\newspacescalar}{\pi}

\newcommand{\Lie}{\mathcal{L}}
\newcommand{\SigmatLie}{\underline{\mathcal{L}}}
\newcommand{\SigmatProject}{\underline{\Pi}}

\newcommand{\Dfour}{\mathbf{D}}

\newcommand{\GLap}{\Delta_{\newg}}

\newcommand{\Gdiv}{\mbox{\upshape div}_{\newg}}
\newcommand{\Divfour}{\mbox{\upshape \textbf{Div}}}

\newcommand{\dif}{\mathcal{D}}
\newcommand{\difarg}[1]{\mathcal{D}^{#1}}

\newcommand{\CommutedMomBorderInhomUp}[1]{\leftexp{(#1);(Border)}{\widetilde{\mathfrak{M}}}}

\newcommand{\CommutedMomBorderInhomDown}[1]{\leftexp{(#1);(Border)}{\mathfrak{M}}}

\newcommand{\CommutedMetBorderInhom}[1]{\leftexp{(#1);(Border)}{\mathfrak{G}}}
\newcommand{\CommutedInvMetBorderInhom}[1]{\leftexp{(#1);(Border)}{\widetilde{\mathfrak{G}}}}
\newcommand{\CommutedMetJunkInhom}[1]{\leftexp{(#1);(Junk)}{\mathfrak{G}}}
\newcommand{\CommutedInvMetJunkInhom}[1]{\leftexp{(#1);(Junk)}{\widetilde{\mathfrak{G}}}}

\newcommand{\CommutedGradMetBorderInhom}[1]{\leftexp{(#1);(Border)}{\mathfrak{H}}}
\newcommand{\TopCommutedGradMetBorderInhom}[1]{\leftexp{(#1);(Border-Top)}{\mathfrak{H}}}

\newcommand{\CommutedGradMetJunkInhom}[1]{\leftexp{(#1);(Junk)}{\mathfrak{H}}}

\newcommand{\CommutedSecFunBorderInhom}[1]{\leftexp{(#1);(Border)}{\mathfrak{K}}}
\newcommand{\CommutedSecFunJunkInhom}[1]{\leftexp{(#1);(Junk)}{\mathfrak{K}}}

\newcommand{\CommutedTimeSfBorderInhom}[1]{\leftexp{(#1);(Border)}{\mathfrak{P}}}
\newcommand{\CommutedTimeSfJunkInhom}[1]{\leftexp{(#1);(Junk)}{\mathfrak{P}}}
\newcommand{\CommutedSpaceSfBorderInhom}[1]{\leftexp{(#1);(Border)}{\mathfrak{Q}}}
\newcommand{\CommutedSpaceSfJunkInhom}[1]{\leftexp{(#1);(Junk)}{\mathfrak{Q}}}

\newcommand{\CommutedLapseHighBorderInhom}[1]{\leftexp{(#1);(Border)}{\mathfrak{N}}}

\newcommand{\CommutedLapseLowBorderInhom}[1]{\leftexp{(#1);(Border)}{\widetilde{\mathfrak{N}}}}

\newcommand{\CommutedLapseHighJunkInhom}[1]{\leftexp{(#1);(Junk)}{\mathfrak{N}}}
\newcommand{\CommutedLapseLowJunkInhom}[1]{\leftexp{(#1);(Junk)}{\widetilde{\mathfrak{N}}}}

\newcommand{\MetricCurrentBorder}[1]{\leftexp{(#1);(Border)}{\mathfrak{J}_{(Metric)}}}
\newcommand{\MetricCurrentJunk}[1]{\leftexp{(#1);(Junk)}{\mathfrak{J}_{(Metric)}}}
\newcommand{\SfCurrentBorder}[1]{\leftexp{(#1);(Border)}{\mathfrak{J}_{(Sf)}}}
\newcommand{\SfCurrentJunk}[1]{\leftexp{(#1);(Junk)}{\mathfrak{J}_{(Sf)}}}

\newcommand{\RicErrorInhom}[1]{\leftexp{(#1);(Low)}{\mathfrak{R}}}

\newcommand{\Jfour}{\mathbf{J}}

\newcommand{\Sttvol}{\varpi_{\StMet}}
\newcommand{\tvol}{\varpi_{\newg}}

\newcommand{\Metricenergy}[1]{\mathscr{E}_{(Metric);#1}}
\newcommand{\Sfenergy}[1]{\mathscr{E}_{(Sf);#1}}
\newcommand{\Totalenergy}[2]{\mathscr{E}_{(Total);#1;#2}}
\newcommand{\SupTotalenergy}[2]{\underline{\mathscr{E}}_{(Total);#1;#2}}

\newcommand{\Energyborder}[1]{\leftexp{(#1);(Border)}{\mathfrak{E}}}
\newcommand{\Energyjunk}[1]{\leftexp{(#1);(Junk)}{\mathfrak{E}}}

\newcommand{\kfour}{\mathbf{k}}
\newcommand{\Ricfour}{\mathbf{Ric}}
\newcommand{\Riemfour}{\mathbf{Riem}}
\newcommand{\Rfour}{\mathbf{R}}
\newcommand{\Tfour}{\mathbf{T}}

\newcommand{\Nml}{\mathbf{N}}

\newcommand{\highnorm}[1]{\mathscr{H}_{(Total);#1}}

\newcommand{\smallparameter}{\uptheta}

\newcommand{\Tboot}{T_{(Boot)}}

\newcommand{\ID}{\mathbb{I}}

\topmargin - .23 in

\begin{document}


\title{The maximal development of near-FLRW data
for the Einstein-scalar field system 
with spatial topology $\mathbb{S}^3$
}
\author{Jared Speck$^{*}$}

\thanks{$^{*}$Massachusetts Institute of Technology, Department of Mathematics, 77 Massachusetts Ave, Room 2-265, Cambridge, MA 02139-4307, USA. \texttt{jspeck@math.mit.edu}}

\thanks{$^{*}$JS gratefully acknowledges support from NSF grant \# DMS-1162211,
from NSF CAREER grant \# DMS-1454419,
from a Sloan Research Fellowship provided by the Alfred P. Sloan foundation,
and from a Solomon Buchsbaum grant administered by the Massachusetts Institute of Technology.
}

\begin{abstract}
The Friedmann--Lema\^{\i}tre--Robertson--Walker (FLRW) solution
to the Einstein-scalar field system with spatial topology $\mathbb{S}^3$
models a universe that emanates from a singular spacelike hypersurface (the Big Bang),
along which various spacetime curvature invariants blow up, only to re-collapse
in a symmetric fashion in the future (the Big Crunch). In this article,
we give a complete description of the maximal developments of perturbations
of the FLRW data at the chronological midpoint of its evolution. 
We show that the perturbed solutions also exhibit curvature blowup along 
a pair of spacelike hypersurfaces, signifying the stability of the Big Bang and the Big Crunch. 
Moreover, we provide a sharp description of the asymptotic behavior of the solution up to the singularities, 
showing in particular that various time-rescaled solution variables converge to regular tensorfields on 
the singular hypersurfaces that are close to the corresponding FLRW tensorfields. 
Our proof crucially relies on $L^2$-type approximate monotonicity identities
in the spirit of the ones we used in our joint works with Rodnianski, in which we proved
similar results for nearly spatially flat solutions with spatial topology $\mathbb{T}^3$.
In the present article, we rely on new ingredients to handle nearly round spatial metrics
on $\mathbb{S}^3$, whose curvatures are order-unity near the initial data hypersurface.
In particular, our proof relies on 
\textbf{i)} the construction of a globally defined spatial vectorfield frame
adapted to the symmetries of a round metric on $\mathbb{S}^3$;
\textbf{ii)} estimates for the Lie derivatives of
various geometric quantities with respect to the elements of the frame;
and \textbf{iii)} sharp estimates for the
asymptotic behavior of the FLRW solution's scale factor near the singular hypersurfaces.
\bigskip

\bigskip


\end{abstract}

\maketitle

\centerline{\today}

\setcounter{tocdepth}{1}
\tableofcontents 
\newpage 

\section{Introduction} \label{S:INTRO}
The Einstein-scalar field system of general relativity
models the evolution of a dynamic spacetime
$(\Mfour,\gfour)$ that is coupled to a scalar field
$\phi$, where $\Mfour$ is the (four-dimensional)
spacetime manifold, $\gfour$ is 
a Lorentzian ``spacetime'' metric of signature $(-,+,+,+)$ on $\Mfour$,  
and $\phi$ is a function on $\Mfour$.
The scalar field is a simple but important matter model in mathematical general relativity,
and the Cauchy problem for the system has been 
well-studied in the regime of asymptotically flat initial data; 
see, for example, \cites{dC1991,dC1999b,dC1993,dC1987,dC1986b,dC1986c}.
Relative to arbitrary coordinates, the 
Einstein-scalar field equations take the following form:\footnote{Throughout we use Einstein's summation convention.}
\begin{subequations}
\begin{align}
	\Ricfour_{\mu \nu} - \frac{1}{2} \Rfour \gfour_{\mu \nu} & = \Tfour_{\mu \nu},
		\label{E:EINSTEINSF} \\
	(\gfour^{-1})^{\alpha \beta} \Dfour_{\alpha} \Dfour_{\beta} \phi & = 0,  \label{E:WAVEMODEL}
\end{align}
\end{subequations}
where 
$\Ricfour$ denotes the Ricci tensor of $\gfour$, 
$\Rfour = (\gfour^{-1})^{\alpha \beta}\Ricfour_{\alpha \beta}$ 
denotes the scalar curvature of $\gfour$, 
$\Dfour$ denotes the Levi--Civita connection of $\gfour$, 
and $\Tfour$ denotes the energy-momentum tensor of the scalar-field:
\begin{align} \label{E:EMTSCALARFIELD}
	\Tfour_{\mu \nu} 
		& := \Dfour_{\mu}\phi \Dfour_{\nu} \phi 
			- 
			\frac{1}{2} \gfour_{\mu \nu} (\gfour^{-1})^{\alpha \beta} \Dfour_{\alpha} \phi \Dfour_{\beta} \phi.
\end{align}
Our main result in this article is a 
proof of stable blowup for an open set of solutions.
More precisely, our main theorem provides a detailed description of the maximal developments,
relative to constant mean curvature (CMC from now on)-transported spatial coordinates gauge,
of an open set of initial data that are close to the data of the well-known 
Friedmann--Lema\^{\i}tre--Robertson-Walker (FLRW from now on)
solution \underline{with spatial topology $\mathbb{S}^3$}.

In Subsect.\ \ref{SS:CONTEXTFORFLRW}, we will provide some background information
on the family of FLRW solutions. Here we only note that
for the scalar field matter model with spatial topology $\mathbb{S}^3$,
the FLRW solution models a cosmological spacetime
that emanates from a singular hypersurface
(the ``Big Bang''), along which various spacetime curvature invariants blow up,
only to later re-collapse in the reverse fashion 
(the ``Big Crunch''). The re-collapse, which is shown in Fig.~\ref{F:SCALEFACTOR} on pg.~\pageref{F:SCALEFACTOR}, 
is initiated by the positive scalar curvature\footnote{In the case
of the FLRW solution, the positive spatial scalar curvature influences the form of Friedmann's ODEs
(see Lemma~\ref{L:FRIEDMANN}). In particular, the form of equation
\eqref{E:FRIEDMANNSECONDORDER} is influenced by the positive scalar curvature of the FLRW spatial metric; 
note that equation \eqref{E:FRIEDMANNSECONDORDER} implies that 
the scale factor has negative second derivative at its maximum value $\scale(0)=1$,
which causes the scale factor to begin re-collapsing.}
of the Riemannian metric induced on the constant-time hypersurfaces (which are diffeomorphic to $\mathbb{S}^3$).
In studying perturbed solutions, we consider data that are near the state of the FLRW 
solution at the chronological midpoint\footnote{Relative to the coordinates of Subsect.\ \ref{SS:FLRWANDSCALEFACTOR}, the FLRW solution's chronological midpoint corresponds to $t = 0$.}
of its evolution.
We show that like the FLRW solution, the perturbed solutions exhibit
\emph{stable uniform curvature blowup}
along a pair of spacelike hypersurfaces,
signifying the stability of the Big Bang and the Big Crunch.
Our result complements our joint works
\cites{iRjS2014a,iRjS2014b} with Rodnianski,
in which we proved similar linear and nonlinear stability results
for solutions to the Einstein-scalar field and Einstein-stiff fluid\footnote{A stiff fluid is such that the speed of sound is equal to the speed of light.} 
systems
with spatial topology $\mathbb{T}^3$
that are close to members of a family of spatially flat background ``Kasner'' solutions
(see below for further discussion).
More precisely, the Kasner-like solutions in \cites{iRjS2014a,iRjS2014b}
did not undergo re-collapse, and we studied the problem only near the Big Bang singularity.
That is, we gave a sharp description of ``half'' of the maximal development of the data.
These works provided the first proofs of stable curvature blowup 
without symmetry assumptions
for solutions to Einstein's equations along a spacelike hypersurface, 
and it is of interest to determine to what extent they can be extended to other 
classes of initial data, spatial topologies, matter models, etc.
Below we will describe the results \cites{iRjS2014a,iRjS2014b} in more detail 
and highlight the new ideas that are needed to handle the case of spatial
topology $\mathbb{S}^3$ and the corresponding metrics with 
positive spatial curvature. 

We believe that our main results could be shown to hold in any number of spatial dimensions, 
more precisely in the case of spatial topology $\mathbb{S}^n$ for any positive integer $n$.
However, for most values of $n$,
some aspects of our analytical framework would need to be modified,
since here we rely on the parallelizability of $\mathbb{S}^3$ to construct the vectorfield frame
that we use to analyze solutions (see Lemma~\ref{L:BASISOFKILLINGFIELDS}).

\subsection{Rough statement of the main results}
\label{SS:EQNSANDMAINRESULTSROUGH}
The FLRW solution to \eqref{E:EINSTEINSF}-\eqref{E:WAVEMODEL}
(see Lemma~\ref{L:FRIEDMANN} for a proof that it is indeed a solution) is\footnote{Note that $\phi$ itself does
not appear in the Einstein-scalar field equations, but rather only its derivatives. For this reason, we typically do not bother
to refer to ``$\phi_{FLRW}$.''}
$\gfour_{FLRW} := - dt^2 + \scale^{2/3}(t) \StMet$,
$\partial_t \phi_{FLRW} := \sqrt{\frac{2}{3}} \scale^{-1}(t)$,
$\nabla \phi_{FLRW} := 0$,
where $\StMet$ is the round metric on $\mathbb{S}^3$ with 
scalar curvature equal to\footnote{Our choice that $\StMet$ has scalar curvature equal to $2/3$ 
is merely a convenient normalization condition.} 
$\frac{2}{3}$
and $\nabla$ denotes the connection of the Riemannian metric induced on $\Sigma_t$,
where here and throughout, $\Sigma_t$ denotes a hypersurface
of constant time $t$.
The scale factor\footnote{Note that our denoting of the 
FLRW spatial metric by
$\scale^{2/3}(t) \StMet$ breaks the usual convention,
in which the power exponent is $2$ rather than $2/3$.
That is, our definition of the scale factor $\scale(t)$ is different than the usual one
found in the literature.
This is mainly for mathematical convenience; by our 
convention, $\scale(t)$ will vanish \emph{linearly}
at the Big Bang and Big Crunch, which is convenient for tracking blowup-rates near the singularities; 
see Subsect.\ \ref{SS:SCALEFACTORANDHUBBLE}.\label{FN:SCALEFACTOR}} 
$\scale(t)$ is an even function
that vanishes at some time $t = \TBang < 0$, increases until $t=0$ with $\scale(0) = 1$, and then shrinks
again until re-collapsing at $t = \TCrunch = - \TBang > 0$;
see Fig.~\ref{F:SCALEFACTOR} for its graph and Sect.\ \ref{S:FLRW} for a rigorous analysis of its properties.
The Ricci invariant
$\Ricfour^{\alpha \beta} \Ricfour_{\alpha \beta}$
of the FLRW solution blows up along $\Sigma_{\TBang}$ 
and
$\Sigma_{\TCrunch}$
like $\scale^{-4}(t)$.

\begin{center}

\begin{overpic}[scale=.60]{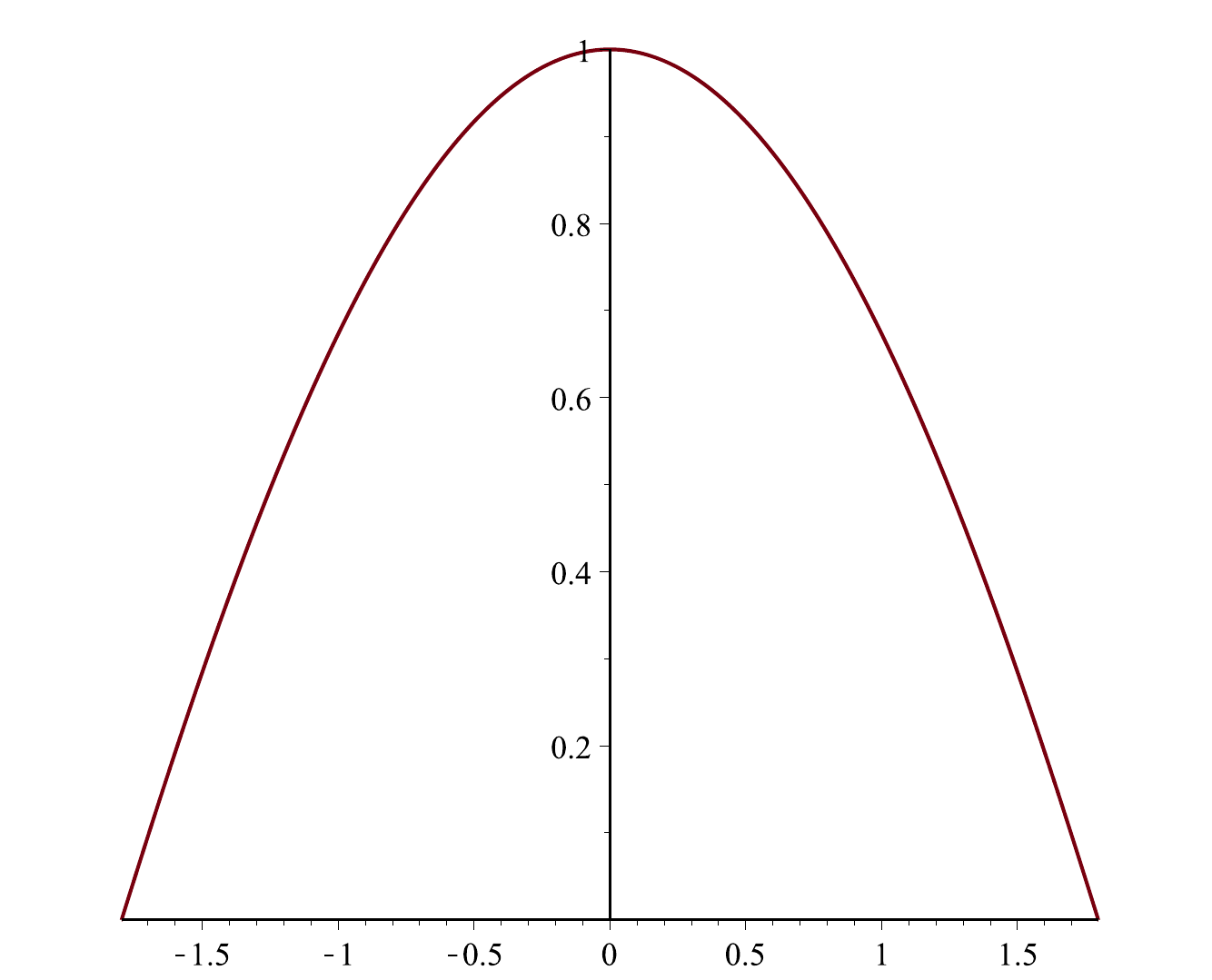} 
\put (46,80) {$\displaystyle \scale(t)$}
\put (46.6,-1.2) {$\displaystyle \longleftrightarrow$}
\put (49.3,-4) {$\displaystyle t$}
\end{overpic}
\captionof{figure}{Graph of the FLRW scale factor}
\label{F:SCALEFACTOR}

\end{center}

Our main result is the following theorem, which we state here in a rough form; see Theorem~\ref{T:MAINTHM} for precise statements.

\begin{theorem}[\textbf{Global nonlinear stability of the FLRW Big Bang-Big Crunch solution} (rough version)]
\label{T:VERYROUGH}
The FLRW solution is globally nonlinearly stable under perturbations of its
data at time $0$.
%
More precisely, the maximal developments corresponding to 
an open (in a suitable Sobolev topology)
set of near-FLRW data on $\Sigma_0 = \lbrace 0 \rbrace \times \mathbb{S}^3$
are geodesically incomplete to the future and to the past.
Moreover, there exists a time function $t \in (\TBang,\TCrunch)$ whose level sets $\Sigma_t$ 
have constant mean curvature $-(1/3) \scale'(t)/\scale(t)$ and foliate the spacetime manifold,
such that timelike geodesics terminate at $\Sigma_{\TBang}$ and $\Sigma_{\TCrunch}$, 
along which the spacetime Ricci curvature invariant
$\Ricfour^{\alpha \beta} \Ricfour_{\alpha \beta}$ blows up
like $\scale^{-4}(t)$
(recall that $\scale(\TBang) = \scale(\TCrunch) = 0$).
Finally, the solution exhibits asymptotically velocity term dominated (AVTD) behavior
in the limit as these singular hypersurfaces are approached. 
In particular, various time-rescaled solution variables converge
to regular tensorfields along 
$\Sigma_{\TBang}$ and $\Sigma_{\TCrunch}$.
Moreover, the solution converges to a solution of the 
velocity term dominated (VTD)
equations, which by definition are truncated versions of
the Einstein-scalar field equations
in CMC-transported spatial coordinates gauge
in which all spatial derivative terms have been discarded;
see Remark~\ref{R:AVTD} for further discussion of this point.
\end{theorem}

\begin{remark}[\textbf{On the number of derivatives}]
	\label{R:NUMBEROFDERIVATIVES}
	The norm $\highnorm{16}(t)$ 
	that we use to control the solution
	(see \eqref{E:HIGHNORM})
	corresponds to commuting the equations up to $16$ times with appropriate differential operators.
	This is a somewhat wasteful number of derivatives, 
	but it allows for a simplified approach to some aspects of our analysis.
\end{remark}

Our proof of Theorem~\ref{T:VERYROUGH} crucially relies 
on $L^2$-type approximate monotonicity identities\footnote{By ``approximate monotonicity identity,''
we roughly mean that that the identities involve ``main terms,'' which have a sign (this is the monotonicity),
and unsigned error terms that have to be controlled. \label{FN:APPROXIMATEMONOTONICITY}}
and the corresponding
energy estimates that they afford,
whose availability relies on \underline{special cancellations}
that occur in well-chosen combinations of divergence identities; see especially Lemma~\ref{L:SFCURRENTDIV}.
These cancellations and other special structures are visible relative to 
the gauge we use: \emph{constant mean curvature-transported spatial coordinates gauge}. 
The energy estimates (which hold up to top order), 
though mildly singular near the Big Bang and 
Big Crunch, allow us to prove that the solution exists long enough to form curvature singularities.
By themselves, the mildly singular energy estimates are not sufficient to close the
proof. They must be complemented (in the context of a bootstrap argument) 
with less singular estimates at the lower derivative levels,
which yield sharper control over the solution
and in particular show that its spatial derivatives are less singular than its
time derivatives.
We derive these sharper estimates using arguments that exploit the 
special structure of the equations and the mildly singular nature of the
high-order energy estimates. In particular, our arguments are based on treating
the evolution equations like transport equations\footnote{To derive sharp estimates for the lapse,
we also derive maximum principle estimates for the elliptic PDE that it solves.} 
with principal part $\partial_t$ 
and with source terms that depend on the solution's higher
spatial derivatives. The source terms can be singular near the Big Bang and Big Crunch
in a manner that depends on the FLRW scale factor $\scale(t)$,
which appears in our formulation of the equations due to our gauge choices.
The key point is that by integrating in time 
and using sharp information about the behavior of $\scale(t)$ near the singularities, 
\emph{we can reduce the strength of the singularities}, which leads
to less singular estimates at the lower derivative levels. 
This can be caricatured by estimates such as
$\int_{s=0}^t (\TCrunch - s)^{-1/2}  \, ds \lesssim 1$ and 
$\int_{s=0}^t (\TCrunch - s)^{-(1+q)} \, ds \lesssim_q (\TCrunch - t)^{-q}$
for constants $q > 0$ and $0 < t < \TCrunch$; see Cor.~\ref{C:SCALEFACTORTIMEINTEGRALS} for precise statements.
Although this approach results in a loss of derivatives 
stemming from the source terms in the transport equations,
such a loss is permissible below top order.

The approach described above has its origin in our joint works
\cites{iRjS2014a,iRjS2014b},
in which we proved similar stable curvature blowup results
for solutions with spatial topology $\mathbb{T}^3$
such the background Kasner solutions were spatially flat.
Like our work here, the proofs in \cites{iRjS2014a,iRjS2014b} relied on some special properties 
verified by the scalar field and stiff fluid matter models;
see Subsect.\ \ref{SS:STABLESINGULARITYFORMATION} for further discussion on
this point and for further description of the results of \cites{iRjS2014a,iRjS2014b}. 
Our proof of Theorem~\ref{T:VERYROUGH} requires new ideas
to handle the positive spatial scalar curvature of solutions near the background FLRW
solution with $\mathbb{S}^3$ spatial topology.
In particular, different from the works \cites{iRjS2014a,iRjS2014b}, 
the coordinate partial derivative vectorfields are not suitable for differentiating the equations
and obtaining estimates for the perturbed solution's derivatives, the reason being that they
are not adapted to the approximate symmetries of nearly round metrics on $\mathbb{S}^3$. 
That is, differentiating the spatial metric with typical spatial coordinate partial derivatives 
would not lead to a small quantity,
and without smallness, there is no hope of closing our perturbative nonlinear analysis.
For this reason and other related ones, 
our proof relies on the following new ingredients:
\begin{itemize}
	\item A well-constructed globally defined spatial vectorfield frame tied to the symmetries
		of the FLRW solution's round spatial metric on $\mathbb{S}^3$; see Sect.\ \ref{S:LIETRANSPORTEDFRAME}.
	\item Estimates showing that the perturbed solution enjoys approximate symmetries all the way up to the singularities;
		see Prop.~\ref{P:STRONGSUPNORMESTIMATES}.
		To obtain these estimates and most of the other estimates in this article, 
		we take \emph{Lie derivatives} of various geometric quantities with respect to the vectorfield frame mentioned above.
	\item Sharp estimates for the FLRW scale factor, especially near the singular times; see Sect.\ \ref{S:FLRW}.
\end{itemize}
In particular, the FLRW scale factor
leads to the presence of favorable spacetime integrals in the energy identities.
This is precisely the ``$L^2$-type approximate monotonicity identity'' mentioned above
(see also Footnote~\ref{FN:APPROXIMATEMONOTONICITY}).
To exploit these favorable integrals,
we derive \emph{precise quantitative information about the scale factor $\scale(t)$,
which allows us to show that the good spacetime integrals are strong enough to 
completely absorb other ``dangerous'' spacetime error integrals};
if not for this, then the dangerous error integrals would have led 
to very singular energy estimates near the Big Bang and 
the Big Crunch, which in turn would have obstructed our proof of nonlinear stability.
We outline the proof of our main result in more detail in Subsect.\ \ref{SS:PROOFIDEAS}.
We first provide some further context for the problem under study and some background material.

\subsection{Prior work on perturbations of FLRW solutions}
\label{SS:CONTEXTFORFLRW}
FLRW solutions are a family of spatially homogeneous solutions to various Einstein-matter systems,
and their behavior forms the basis of many of the predictions of modern cosmology \cite{sW2008}.
The local and global properties of the FLRW solutions depend on several factors,
including the matter model, the value of the cosmological constant, 
and the topology/geometry of the initial Cauchy hypersurface.
In view of their distinguished role in cosmology, it is of fundamental 
mathematical and physical interest to determine whether or not the FLRW solutions
are globally nonlinearly stable under perturbations of their initial data. 
In all cases, the FLRW metric $\gfour_{FLRW}$ 
is a Lorentzian metric on a product manifold
$I \times \Sigma$
that is often written in the form $\gfour_{FLRW} = - dt^2 + a^2(t) d \Sigma$,
where $I$ is an interval of time, $\Sigma$ is the ``spatial manifold,
$d \Sigma$ is a Riemannian metric on $\Sigma$
and, by virtue of the Einstein equations, 
the scale factor $a(t) > 0$ solves ODEs (Friedmann's equations) that depends on the matter model
and the cosmological constant; see Lemma~\ref{L:FRIEDMANN}.
Note that here we are using the terminology
``scale factor'' in a slightly different manner compared to the
rest of the article (see Footnote~\ref{FN:SCALEFACTOR}).
The most prominent feature of the FLRW solutions
studied in cosmology is that it is possible for the scale factor to tend to either
$\infty$ or to $0$ as $t$ varies. These scenarios represent, respectively, 
the expansion and collapse of the universe and often coincide, respectively,
with geodesic completeness and curvature blowup.
For this reason, prior studies of the stability of FLRW solutions 
can be separated into the following two categories:
\begin{itemize}
	\item \textbf{\textbf{(Expansion)}} The study of the stability of FLRW solutions such that 
		$a(t) \uparrow \infty$ as $t \uparrow t_0$ for some value of $t_0$ (possibly infinite).
	\item \textbf{\textbf{(Collapse)}} The study of the stability of FLRW solutions such that 
		$a(t) \downarrow 0$ as $t \downarrow t_0$ for some value of $t_0$.
\end{itemize}
Aside from our joint works \cites{iRjS2014a,iRjS2014b}, 
all previous results in three spatial dimensions
concerning the nonlinear stability of an FLRW solution
(without symmetry assumptions) 
have been in the expanding case.
	In an expanding direction, FLRW solutions are typically geodesically complete and thus 
	the stability problem is essentially tantamount to the study of whether or not the associated Einstein-matter PDEs 
	(in an appropriate gauge)
	admit global, geodesically complete (in the expanding direction) 
	solutions for near-FLRW data.
	In contrast, collapsing FLRW solutions typically exhibit curvature blowup at times
	for which the scale factor vanishes and thus the stability problem essentially corresponds to the study
	of whether or not perturbed solutions to the
	corresponding Einstein-matter PDEs 
	also exhibit curvature blowup.
	That is, the question is one of whether or not
	the singularity formation is dynamically stable under perturbations of the FLRW data.
	
	The original result on global stability in the expanding case
	is due to Friedrich \cite{hF1986a},
	who used the conformal method
	to prove the nonlinear stability of
	the de Sitter solution to the Einstein-vacuum equations with a positive cosmological constant
	in three spatial dimensions. His result was extended to all odd spatial
	dimensions by Anderson \cite{mA2005}.
	In \cites{hR2008,hR2009},
	Ringstr\"{o}m developed an alternate approach, based on generalized wave coordinates,
	that allowed him to prove a global stability result in the expanding
	direction for (asymmetric) perturbations of a large class of spatially homogeneous solutions
	with various spatial topologies and with scalar field matter whose potential $V(\phi)$ has suitable properties 
	(where \cite{hR2008} effectively includes the Einstein-vacuum equations with a positive cosmological constant as a special case).
	In our joint work \cite{iRjS2013} with Rodnianski on the irrotational Euler--Einstein equations 
	with a positive cosmological constant
	under the equations of state
	$p = c_s^2 \rho$ for $0 < c_s^2 < 1/3$,
	we extended Ringstr\"{o}m's framework and proved the stability of FLRW solutions
	with spatial topology $\mathbb{T}^3$
	in the expanding direction.
	Here $p$ is the fluid pressure, $\rho$ is the proper energy density,
	and the constant $c_s > 0$ is the speed of sound.
	See also \cite{jS2012} for the same result without the irrotationality assumption.
	See also the work \cite{cLjK2013} for a proof
	in the conformally invariant case $c_s^2 = 1/3$
	and \cite{mHjS2015} for a proof in the dust case $c_s = 0$.
	Readers can also consult \cite{tO2016} for an alternate proof in the
	cases $0 < c_s^2 \leq 1/3$.
	See also \cites{hR2013,hAhR2016} for the proofs of similar results
	for the Einstein--Vlasov system with a positive cosmological constant.
	The above Euler--Einstein stability results show, in particular,
	that rapid exponential expansion,
	in those cases induced by the presence of a positive cosmological constant,
	suppresses the formation of fluid shocks, which typically occurs in Minkowski spacetime;
	see, for example, Riemann's famous proof \cite{bR1860} of shock formation
	in solutions to the non-relativistic Euler equations in one spatial dimension,
	Christodoulou's remarkably sharp description of shock formation \cite{dC2007}
	for solutions to the relativistic Euler equations in irrotational regions,
	or our recent joint extension \cites{jLjS2016a,jLjS2016b} of Christodoulou's result
	to handle the case in which the vorticity is non-zero at the location of the shock.
	
	In the collapsing case, 
	the only prior proofs
	of the nonlinear stability of an FLRW solution
(without symmetry assumptions)
are our aforementioned joint works \cites{iRjS2014a,iRjS2014b}, where we respectively considered
FLRW solutions to the Einstein-scalar field and Einstein-stiff fluid\footnote{A stiff fluid has speed of sound equal to the speed of light, that is, its equation of state is $p = \rho$.} 
systems with a vanishing cosmological constant and with spatial topology $\mathbb{T}^3$. 
We will describe these results in more detail in Subsect.\ \ref{SS:STABLESINGULARITYFORMATION}
in the context of singularity formation results for Einstein's equations.
We also note that Ringstr\"{o}m \cite{hR2017} has 
recently derived estimates for solutions to a large family of linear
wave equations whose corresponding metrics 
model the behavior that can occur in solutions
to Einstein's equations near cosmological singularities.
His work serves as a natural starting point for trying to prove
stable cosmological singularity formation results for Einstein's equations
in more general regimes than the ones treated in \cites{iRjS2014a,iRjS2014b}
and in the present work (for example, for far-from-spatially isotropic initial conditions).

\subsection{The initial value problem for the Einstein-scalar field equations}
\label{SS:IVP}
In this subsection, we recall some basic facts about the initial value problem for Einstein's equations.
The fundamental results \cite{CB1952} and \cite{cBgR1969},
which are due to Choquet--Bruhat and Choquet--Bruhat + Geroch respectively,
show that the system \eqref{E:EINSTEINSF}-\eqref{E:WAVEMODEL}
has an initial value problem formulation in which sufficiently regular data
give rise to a unique 
(up to diffeomorphism)
maximal globally hyperbolic development
$(\Mfour_{Max},\gfour_{Max},\phi_{Max})$.
Roughly, this is the largest possible solution to the Einstein-scalar field 
equations that is uniquely determined by the data.
Although, it is conceptually important to know that the
maximal globally hyperbolic development exists,
the results \cites{CB1952,cBgR1969} do not provide any information about
its nature. Our main result provides, for an open set of data, 
quantitative and qualitative information about its nature.
The ``geometric data'' consist of the following fields on the initial data hypersurface $\Sigma_0$,
which in this article is diffeomorphic to $\mathbb{S}^3$:
$(\mathring{g},\mathring{k},\mathring{\phi},\mathring{\phi}_0)$.
Here, $\mathring{g}$ is a Riemannian metric,
$\mathring{k}$ is a symmetric two-tensor,
and $\mathring{\phi}$ and $\mathring{\phi}_0$ are a pair of functions.
A solution launched by the data consists of a time-oriented spacetime
$(\mathbf{M},\gfour)$, 
a scalar field $\phi$ on $\mathbf{M}$, and  
an embedding $\Sigma_0 \overset{\iota}{\hookrightarrow}\mathbf{M}$ such that 
$\iota(\Sigma_0)$ is a Cauchy hypersurface in $(\mathbf{M},\gfour)$. 
The spacetime fields must verify the equations 
\eqref{E:EINSTEINSF}-\eqref{E:WAVEMODEL}
and be such that $\iota^* \gfour = \mathring{g}$, 
$\iota^* \kfour = \mathring{k}$, 
$\iota^* \phi = \mathring{\phi}$, 
$\iota^* \Nml \phi = \mathring{\phi}_0$, 
where $\kfour$ is the second fundamental form of $\iota(\Sigma_0)$
(our sign convention is given in \eqref{E:SECONDFUNDDEF}), 
$\Nml \phi$ is the derivative of $\phi$ in the direction of the future-directed normal $\Nml$ to $\iota(\Sigma_0)$,
and $\iota^*$ denotes pullback by $\iota$. 
Throughout the article, we will often suppress the embedding and identify $\Sigma_0$ with $\iota(\Sigma_0)$. 

It is well-known (see also Prop.~\ref{P:EINSTEININCMC}) 
that the data are constrained by the \emph{Gauss} and \emph{Codazzi} equations, 
which take the following form for the Einstein-scalar field system:
\begin{subequations}
\begin{align}
	\mathring{\mbox{\upshape R}} 
	- \mathring{k}_{\ b}^a \mathring{k}_{\ a}^b
	+ (\mathring{k}_{\ a}^a)^2 
	& = 2 \Tfour(\Nml,\Nml)|_{\Sigma_0} 
	= \mathring{\phi}_0^2 + \nabla^a \mathring{\phi} \nabla_a \mathring{\phi}, 
	\label{E:GAUSSINTRO} \\
	\nabla_a \mathring{k}_{\ j}^a 
		- \nabla_j \mathring{k}_{\ a}^a  
		& = - \Tfour(\Nml,\frac{\partial}{\partial x^j})|_{\Sigma_0}
		= - \mathring{\phi}_0 \nabla_j \mathring{\phi}. 
		\label{E:CODAZZIINTRO}
\end{align}
\end{subequations}
Above, $\Tfour(\Nml,\Nml) := \Tfour_{\alpha \beta} \Nml^{\alpha} \Nml^{\beta}$,
$\nabla$ denotes the Levi--Civita connection of $\mathring{g}$, $\mathring{\mbox{\upshape R}}$ denotes 
the scalar curvature of $\mathring{g}$, and indices are lowered and raised with
$\mathring{g}$ and its inverse. Equations \eqref{E:GAUSSINTRO}-\eqref{E:CODAZZIINTRO} 
are known as the \emph{Hamiltonian} and \emph{momentum} constraints.
In this article, we consider only initial data that verify the constant mean curvature condition
\begin{align} \label{E:INITIALCMC}
	\mathring{k}_{\ a}^a
	& \equiv 
	- 1,
\end{align}	
which is compatible with the CMC-transported spatial coordinates
gauge that we use in our analysis. 

\begin{remark}[\textbf{The CMC assumption is not a true restriction}]
\label{R:NONEEDFORCMC}
For the near-FLRW solutions under consideration,
the assumption \eqref{E:INITIALCMC} is not truly restrictive:
in \cite{iRjS2014b}, we showed that for near-FLRW data, 
not necessarily with constant mean curvature, 
the solution has a CMC hypersurface lying near the initial Cauchy hypersurface. 
That is, we could use the results of \cite{iRjS2014b} to generate a CMC hypersurface,
and then apply the methods of this paper starting from the 
state of the solution along it.
More precisely, in \cite{iRjS2014b}, we proved the existence of a CMC
hypersurface for solutions near the FLRW solution 
with Cauchy hypersurfaces diffeomorphic to $\mathbb{T}^3$.
However, it is straightforward to extend the result to the
$\mathbb{S}^3$ case considered here.
\end{remark}

\subsection{Prior results on singularity formation in solutions to Einstein's equations}
\label{SS:STABLESINGULARITYFORMATION}
Our main results are connected not only to the FLRW stability results
described in Subsect.\ \ref{SS:CONTEXTFORFLRW}, 
but also to a large body of work on the study of singularity formation
in solutions to various Einstein-matter systems.
The study of singularity formation in general relativity 
was jump-started by the famous ``singularity'' formation results
of Hawking and Penrose \cites{sH1967,rP1965}, which
showed that for matter models satisfying the \emph{strong energy condition},\footnote{This condition
is that $(\Tfour_{\mu \nu} - \frac{1}{2} (\gfour^{-1})^{\alpha \beta} \Tfour_{\alpha \beta} \gfour_{\mu \nu}) \bf{X}^{\mu} \bf{X}^{\nu} \geq 0$
for all timelike vectors $\bf{X}$. The scalar field matter model satisfies the condition.}
a large set of initial data lead to maximal globally hyperbolic developments that are geodesically
incomplete. Though compelling in their broad applicability,\footnote{In particular, ``Hawking's theorem'' \cite{rW1984}*{Theorem 9.5.1}
shows that the initial data that we consider in our main theorem lead to a spacetime
with incomplete timelike geodesics.}
the Hawking--Penrose theorems are soft in that they do not provide any information 
about the nature of the incompleteness.
In particular, the theorems leave open 
the possibilities that the incompleteness is tied to the blowup of an invariant quantity,
such as curvature, or alternatively that it is tied to the development of a Cauchy horizon,
beyond which the solution cannot be uniquely continued due to lack of information for
how to continue. Since both scenarios can be realized\footnote{For example, 
members of the famous Kerr black hole family of solutions to the Einstein-vacuum equations develop Cauchy horizons.} 
by explicit solutions, it is of interest
to prove theorems that reveal the nature of the incompleteness
based on knowledge of the initial data.

We now describe some prior breakdown results for solutions to Einstein's equations 
that provide information beyond that of the Hawking--Penrose theorems; our main result falls into this category.
There are many constructive results that provide a detailed description of 
stable singularity formation in solutions to various Einstein-matter systems 
under the assumption of spatial homogeneity, in which case the
equations reduce to a system of ODEs; see, for example,
\cites{aR2005b,jWgE1997} for overviews of these results.
There are also constructive stable singularity formation results 
in various symmetry reduced cases such that
the equations reduce to a system of PDEs in $1+1$-dimensions; see
\cites{pCjIvM1990,jIvM1990,hR2009b,hR2010}.
We now highlight the foundational works \cites{dC1991,dC1999b} of Christodoulou
on the Einstein-scalar field system in \emph{spherical symmetry} for $1$- or $2$-ended
asymptotically flat data,
in which he showed that the maximal globally hyperbolic
future developments of generic data are future-inextendible
as time-oriented Lorentzian manifolds with a $C^0$ metric.
This can be viewed as a ``severe'' form of breakdown, at the level of the metric.
See also the recent works of \cites{jLsO2017a,jLsO2017b}
on the spherically symmetric Einstein-Maxwell-(real) scalar field system
with asymptotically flat $2$-ended initial data, in which the authors proved that
the maximal globally hyperbolic
future developments of generic data are future-inextendible
as time-oriented Lorentzian manifolds with a $C^2$ metric, which is especially interesting
since Dafermos--Rodnianski \cites{mD2005,mDiR2005a} showed that the statement
is \emph{false} for this system if one replaces $C^2$ with $C^0$. That is,  
Dafermos--Rodnianski showed that for appropriate spherically symmetric data with non-zero charge,
the metric is \emph{extendible} as $C^0$ Lorentzian metric past the boundary of the maximal development.
There are also results that yield the construction of (but not the stability of)
singularity containing solutions under symmetry assumptions
\cites{sKjI1999,gR1996,sKaR1998,kA2000a,yCBjIvM2004,fS2002,fBpL2010c,eAfBjIpL2013b} 
and/or that rely on spatial analyticity assumptions \cites{lAaR2001,tDmHrAmW2002}.
There is yet another body of work in which the authors
constructed singular solutions by using a formulation of
Einstein's equations that allows one to solve a Cauchy problem
with initial data given on the singular hypersurface itself. This
is essentially equivalent to prescribing the singular data; 
see, for example \cites{kApT1999a,cCpN1998,rN1993a,rN1993b,pT1990,pT1991,pT2002}.
Many of the above results have been described and compared in detail by Rendall in \cite{aR2000b}, 
a work in which his main result was a proof of the existence of singular solutions to 
the Einstein-vacuum equations with Gowdy symmetry.

In three spatial dimensions without any symmetry assumptions, 
the only stable singularity formation
results for Einstein-matter systems
are our aforementioned works \cites{iRjS2014a,iRjS2014b}
and Luk's work \cite{jL2013},
in which he exhibited a class of solutions to the Einstein-vacuum equations
without symmetry assumptions arising from characteristic initial data 
such that the boundary of the maximal development
contains a null portion along which the metric
remains $C^0$ but its Christoffel symbols blow-up in $L^2$.

We now describe the results of \cites{iRjS2014a,iRjS2014b} in more detail.
The results proved there are similar to the ones we have obtained here,
but with different spatial topology and geometry. Specifically,  
in \cites{iRjS2014a,iRjS2014b}, we studied nearly spatially flat solutions with spatial topology $\mathbb{T}^3$.
These two aspects of the solutions allowed for a simplified approach to the analysis
compared to the present article, due in part to the fact that 
we were able to treat all spatial curvature terms as small error terms.
Specifically, in \cite{iRjS2014a}, we primarily studied 
\emph{linearized} versions of the Einstein-scalar field equations, 
where we linearized the equations about generalized Kasner solutions 
(which we simply refer to as ``Kasner solutions'' from now on, even though
traditionally only vacuum solutions are referred to as such).
These are spatially flat solutions to the 
Einstein-scalar field system that can be expressed as
\begin{align} \label{E:KASNER}
	\gfour & = - dt^2 + g,
	\qquad g = \sum_{i=1}^3 t^{2q_i} (dx^i)^2,
	\qquad \phi = A \ln t,
		\qquad (t,x) \in (0,\infty) \times \mathbb{T}^3,
\end{align}
where, the constants $q_i$ are known as the \emph{Kasner exponents}
and $A \geq 0$ is a constant.
The exponents $q_i$ and $A$ are constrained by the equations
 \begin{subequations} 
\begin{align} 
	\sum_{i = 1}^3 q_i & = 1, 
	\label{E:KASNERTRACECONDITION} 
	\\
	\sum_{i=1}^3 q_i^2 & = 1 - A^2,
	\label{E:KASNERHAMILTONIANCONSTRAINT}
\end{align}
\end{subequations}
where \eqref{E:KASNERTRACECONDITION} is a
normalization condition and \eqref{E:KASNERHAMILTONIANCONSTRAINT} is a consequence of 
the Hamiltonian constraint. 
The FLRW solution is the unique spatially
isotropic member of the Kasner family and corresponds to the case in which
all Kasner exponents are equal to $1/3$. 
In most cases, the Kasner solution has a Big Bang
singularity at $t=0$ where its Kretschmann scalar
$\Riemfour^{\alpha \beta \gamma \delta} \Riemfour_{\alpha \beta \gamma \delta}$
blows up like $t^{-4}$.
Our main result in \cite{iRjS2014a}
was a proof of linear stability for near spatially isotropic Kasner backgrounds.
More precisely, we linearized the Einstein-scalar field equations 
in CMC-transported coordinates gauge\footnote{In \cite{iRjS2014a}, we also exhibited a new family of parabolic
lapse gauges in which the stability results hold.}
around a Kasner solution and studied the asymptotic behavior of the linear solution as $t \downarrow 0$, that is,
as the Kasner solution's Big Bang singularity is approached. We
showed that if all Kasner exponents are near $1/3$, then the linearized solution
converges towards a \emph{Kasner footprint state} as $t \downarrow 0$.
Roughly, a Kasner footprint state is a solution to a VTD version of the linearized equations
in which all spatial derivative terms are discarded. If $x \in \mathbb{T}^3$ denotes
a ``spatial point'', then Kasner footprint states
can be thought of as an $x$-dependent family of Kasner solutions
whose Kasner exponents are allowed to vary with $x$ and whose diagonal form
has possibly been destroyed due to the influence of an $x$-dependent 
change of spatial basis. The results of \cite{iRjS2014a} 
can be viewed as linear analogs of Theorem~\ref{T:VERYROUGH} in the simpler case of 
$\mathbb{T}^3$ spatial topology, but for a larger class of background solutions. In \cite{iRjS2014a}, we also sketched a proof of the nonlinear
stability of the FLRW solution to the Einstein-scalar field system with spatial topology $\mathbb{T}^3$
in a neighborhood of its Big Bang singularity. In \cite{iRjS2014b}, 
we gave the complete proof of nonlinear stability of the FLRW solution near its Big Bang singularity, 
not for the Einstein-scalar field system, but instead for the Einstein-stiff fluid system.
The stiff fluid is a generalization of the scalar field in that it reduces to the scalar field
matter model when the fluid's vorticity is $0$.

Like the results of \cites{iRjS2014a,iRjS2014b}, the results of Theorem~\ref{T:VERYROUGH}
show that the singularity formation occurs in a monotonic, controlled fashion.
This kind of monotonic singularity formation had previously been predicted via heuristic arguments
for ``general'' singular solutions
in the scalar field
model case \cite{vBiK1972} and in the stiff fluid model case \cite{jB1978}.
Although these works are provocative, we
emphasize that our main results and those of \cites{iRjS2014a,iRjS2014b}
confirm the heuristic picture of
\cites{vBiK1972,jB1978} only for a small set of initial data.
We next emphasize that 
it might be that the scalar field and stiff fluid matter models
are special in that similar monotonic, controlled-type blowup 
results do not hold for typical matter models.
In any case, our proof of Theorem~\ref{T:VERYROUGH} 
certainly exploits various special properties of the scalar
field model. For example, the evolution equation \eqref{E:PARTIALTKCMC} for the
second fundamental form does not depend on the time derivatives of the
scalar field, which is a matter model-dependent property in part tied to the fact that the characteristics
of the scalar field wave equation are the same as those of the Einstein field equations.
The absence of $\partial_t \phi$-dependent terms in equation \eqref{E:PARTIALTKCMC}  
is important because our approach is fundamentally based on showing that
spatial derivatives are less singular than time derivatives; i.e.,
the absence of $\partial_t \phi$-dependent terms
signifies the absence of the most singular terms in the evolution equation
for the second fundamental form.
Similar structural results hold for the stiff fluid; see \cite{iRjS2014a}.
As a second example, we note that the special cancellations
in the divergence identity of Lemma~\ref{L:SFCURRENTDIV},
which are crucial for the energy estimates,
might not generalize to typical matter models.

The works \cites{vBiK1972,jB1978} painted
a very different heuristic picture of singularity formation 
than the picture painted in the work \cite{vBiKeL1970} on the Einstein-vacuum equations.
The latter work suggested that in three spatial dimensions,
vacuum solutions typically exhibit highly oscillatory behavior
near singularities, which are
 ``generically'' spacelike. 
The picture painted in \cite{vBiKeL1970},
though vague, is often referred to as the ``BKL conjecture.''
Although \cite{vBiKeL1970} stimulated a great deal of research on singularities in general relativity,
being not rigorous, it also generated a lot of controversy.
For example, Luk's aforementioned work \cite{jL2013} yields 
a non-trivial set of characteristic Einstein-vacuum initial data 
such that the maximal development's boundary has a null portion along which the Christoffel symbols blow up,
contradicting the picture of spacelike singularities.
We now note that our approach in \cites{iRjS2014a,iRjS2014b} and the present article to proving
the existence of stable, monotonic spacelike singularities
does not seem to directly extend to the vacuum case in three spatial dimensions;
this is at least compatible with the oscillatory picture suggested by \cite{vBiKeL1970}.
The obstruction to implementing our approach seems to be tied, at least in part, 
to the fact that in three spatial dimensions, the Hamiltonian constraint \eqref{E:GAUSSINTRO}
precludes the existence of singularity-forming
spatially homogeneous Einstein-vacuum solutions with second fundamental forms
having a small trace-free part.
In contrast, for the Einstein-scalar field system,
the FLRW solution itself has a second fundamental form with \emph{vanishing trace-free part}.
Hence, perturbed solutions have (at least initially) a second fundamental form with a small trace-free part.
In \cites{iRjS2014a,iRjS2014b} and in the present work, 
we crucially exploit this smallness in our perturbative nonlinear analysis.
We now highlight that the situation might be different in very high spatial dimensions.
Specifically, as one increases the number of spatial dimensions, it
is possible to write down explicit singularity-forming
spatially homogeneous Einstein-vacuum solutions 
with appropriately time-rescaled second fundamental forms that have 
small eigenvalues. If one were to study perturbations of these solutions,
then this smallness would allow one to show that many error terms
have small amplitudes (at least initially), 
which would restrain the effect of these terms on the dynamics (at least for short times). 
For this reason, we speculate that the smallness of the eigenvalues might be sufficient for proving
the nonlinear stability of these singular solutions using our approach.
Some evidence in favor of this was provided in \cite{tDmHrAmW2002} and
in \cite{jDmHpS1985}; in \cite{jDmHpS1985}, the authors provided
heuristic arguments for the existence of non-oscillatory spatially dependent solutions 
to the Einstein-vacuum equations in $10$ or more spatial dimensions
while in \cite{tDmHrAmW2002}, the authors rigorously constructed
a family of spatially analytic non-oscillatory singular solutions to
various Einstein-matter systems, including the Einstein-vacuum equations in $10$ or more spatial dimensions.
It is of interest to understand whether or not these singular solutions are dynamically stable.

We close this subsection by noting that in three spatial dimensions, 
the oscillatory picture\footnote{The work \cite{hR2001} also treated the stiff fluid case $c_s=1$
in Bianchi IX symmetry and yielded monotonic-type singularity formation results
similar to the ones we obtained in \cite{iRjS2014b}.} 
of solutions near singularities was
in fact confirmed by
Ringstr\"{o}m \cite{hR2001}
for solutions with Bianchi IX symmetry (a symmetry class in which the solutions are spatially homogeneous)
to the Euler-Einstein equations
under the equations of state
$p = c_s^2 \rho$, with $0 < c_s < 1$
and for the Einstein-vacuum equations.
However, outside of the class of spatially homogeneous solutions,
there are currently no examples of Einstein-vacuum solutions 
that are rigorously known to exhibit the kind of oscillatory behavior
near a singularity conjectured in \cite{vBiKeL1970}.
In total, as of the present, it is not clear to what extent 
the heuristic picture painted in \cite{vBiKeL1970} holds true.

\subsection{Ideas behind the proof}
\label{SS:PROOFIDEAS}
In this subsection, we summarize the main ideas behind the proof of our main result, 
namely Theorem~\ref{T:MAINTHM} (which we roughly summarized as Theorem~\ref{T:VERYROUGH}).

\begin{enumerate}
		\item  (\textbf{Analysis of the scale factor})
		It suffices to prove the blowup result as $t \uparrow \TCrunch$ since the same
		approach can be used to prove blowup as $t \downarrow \TBang$.
		In Sect.\ \ref{S:FLRW}, we exhibit some qualitative and quantitative properties of the FLRW scale factor $\scale(t)$.
		Obtaining a sharp picture of the asymptotic behavior of $\scale(t)$ as $t \uparrow \TCrunch$ 
		is of critical importance for our analysis
		since the strength of the singularities is tied to its behavior.
	\item (\textbf{Gauge choices and rescaled variables})
		We introduce time-rescaled variables (see Def.\ \ref{D:RESCALEDVARIABLES})
		and derive the Einstein-scalar field equations relative to constant-mean-curvature
		transported spatial coordinates gauge for the rescaled variables, where the mean curvature of
		$\Sigma_t$ is tied to the FLRW scale factor (see Prop.~\ref{P:RESCALEDVARIABLES}).
		This gauge features an elliptic PDE for the lapse, which introduces an infinite propagation speed into the PDE system.
		The infinite speed is in fact essential for synchronizing the singularities.
		The advantage of the rescaled variables is the following:
		we will show that at the low derivative levels, the rescaled variables either remain bounded\footnote{In fact, as we stated in 
		Theorem~\ref{T:VERYROUGH}, some variables converge; see Step (9).}
		or blow up at most at a very mild rate as the singularities are approached.
		For this reason, it is easy to roughly assess the strength of nonlinear products
		that are expressed in terms of the rescaled variables.
	\item (\textbf{A good vectorfield frame})
		To derive estimates, we construct 
		(see Sect.\ \ref{S:LIETRANSPORTEDFRAME})
		a globally defined vectorfield frame $\mathscr{Z} := \lbrace Z_{(1)}, Z_{(2)}, Z_{(3)}  \rbrace$
		on $\Sigma_t = \lbrace t \rbrace \times \mathbb{S}^3$ that is orthonormal with respect to the round background FLRW spatial metric $\StMet$.
		Clearly this step relies on the parallelizability of $\mathbb{S}^3$.
		We use this frame for differentiating the equations and also for contracting against various tensorfields
		to generate tensorfield frame components that we estimate.
	\item (\textbf{Solution norm and bootstrap assumptions})
		We introduce a total solution norm $\highnorm{16}(t)$
		(see Def.\ \ref{D:HIGHNORM})
		that measures the deviation of the perturbed time-rescaled solution variables
		from the analogous FLRW solution variables.
		The subscript $16$ indicates that the norm controls the derivatives of the solution
		variables in a manner that corresponds to commuting the equations up to $16$ times with
		the elements of $\mathscr{Z}$. We also note that
		the norm controls the \emph{components} of the solution with respect to the frame $\mathscr{Z}$ and the 
		corresponding co-frame.
		On any slab $[0,\Tboot) \times \mathbb{S}^3$ of classical existence for the perturbed solution,
		with $0 < \Tboot < \TCrunch$,
		we make the bootstrap assumption 
		$\highnorm{16}(t) \leq \varepsilon \scale^{-\upsigma}(t)$,
		where $\varepsilon$ and $\upsigma$
		are two small bootstrap parameters verifying $0 < \varepsilon^{1/2} \leq \upsigma$.
		We adjust their smallness throughout the course of the analysis.
		Note that $\highnorm{16}(t)$ is allowed to blow up as $t \uparrow \TCrunch$ since $\scale(\TCrunch)=0$.
		The main step in the proof is to derive improvements of the bootstrap assumption
		via a priori estimates.
	\item (\textbf{Improved estimates at the lower derivative levels})
		Using the bootstrap assumptions and a small-data assumption, 
		we derive improved sup-norm estimates at the lower derivative levels,
		where here and throughout, ``derivatives'' means the Lie derivatives of the time-rescaled solution variables
		with respect to the elements of $\mathscr{Z}$.
		The improved estimates show that at the lower derivative levels,
		the blowup-rate of the solution is less severe than
		the rate that is directly implied by the bootstrap assumptions for $\highnorm{16}(t)$.
		In fact, our estimates show that some of the time-rescaled solution variables remain
		$\mathcal{O}(\varepsilon)$, which turns out to be of crucial importance for the energy estimates.
		The proofs of the improved estimates are based on treating the 
		evolution equations as transport equations that are allowed to lose
		derivatives. The sharp information for the FLRW scale factor that we derived
		in Step (1) is important for these estimates.
	\item  (\textbf{Approximate monotonicity identities and the energy integral inequality})
			The starting point for our energy estimates is a family of divergence identities
			whose proofs rely on the observation of special cancellations. This is perhaps the most
			important aspect of the proof. Specifically,
			in Lemma~\ref{L:SFCURRENTDIV}, we combine divergence identities for the scalar field and the lapse
			in a manner that leads to the cancellation of some singular error terms and the emergence of
			some favorable ones. In Lemma~\ref{L:METRICCURRENTDIV}, we provide a similar, but less less delicate,
			divergence identity for the metric.
			Upon integrating these two divergence identities over $[0,t] \times \mathbb{S}^3$
			and combining them in a suitable proportion,
			we obtain energy identities
			showing that the perturbation of the solution away FLRW will not grow
			towards the singularities, modulo error integrals that have to be controlled.
			This is what we mean by ``approximate monotonicity identity.''
			The $M^{th}$-order energies control the derivatives of the time-rescaled solution variables
			from order $1$ up to order $M$,
			and the error integrals depend on the error terms that arise when we
			commute the equations with Lie derivatives
			with respect to the elements of $\mathscr{Z}$
			(see Sect.\ \ref{S:COMMUTEDEQUATIONS} for the structure of the error terms).
			Putting all of the error integrals in absolute values
			and using the crucial properties of the scale factor from Step (1),
			we obtain integral inequalities for a family of
			energies; see Prop.~\ref{P:FUNDAMENTALENERGYINTEGRALINEQUALITY}.
			After estimating the error integrals, we will be able to obtain a priori
			estimates for the energies, which is the main step in improving the bootstrap assumption.
			To control the non-differentiated solution,
			we use a separate argument that loses derivatives; 
			see Lemma~\ref{L:L2ESTIMATESFORNONDIFFERNTIATED}.
			The reason that we use a separate argument is that the non-differentiated
			equations involve a large source term, namely the  
			Ricci curvature of the time-rescaled metric (see the term $\Ric^{\#}$ on RHS~\eqref{E:EVOLUTIONSECONDFUNDRENORMALIZED}),
			which is somewhat inconvenient to treat.
			In particular, this aspect of the proof is more difficult compared to our work
			\cite{iRjS2014b} concerning nearly spatially flat metrics on $\mathbb{T}^3$.
	\item  (\textbf{Pointwise and $L^2$ estimates for the error terms})
		Using the improved estimates of Step (5), we derive pointwise
		estimates for the error terms in the $\mathscr{Z}$-commuted equations. Based on these
		pointwise estimates, we bound the $L^2$ norms of the error terms 
		in terms of the energies. Some of the error terms are borderline in
		a sense explained in the next point.
	\item (\textbf{A priori energy estimates and improvement of the bootstrap assumptions})
		Using the energy integral inequalities from Step (6), the $L^2$ estimates for error
		terms from the previous step, 
		and Gronwall's inequality, we inductively 
		derive a priori estimates for the energies,
		which, under a near-FLRW assumption, can easily be used to derive an improvement
		of the norm bootstrap assumptions; see Cor.~\ref{C:MAINAPRIORIENERGYESTIMATES}. 
		The FLRW scale factor appears in the Gronwall estimates and thus the sharp scale factor estimates
		that we derived in Step (1) are also important for this step. We stress that some of the terms
		appearing in the Gronwall estimates for the energies are borderline in the sense
		that they allow for the possibility of mild energy blowup. These borderline terms
		are sensitive in that to control their effect on the Gronwall estimates, 
		we crucially rely on the
		improved estimates of Step (5); without the improved estimates, 
		we would have obtained much more singular Gronwall estimates, which in turn would have prevented
		us from deriving an improvement of the norm bootstrap assumption.
	\item (\textbf{Proof of stable blowup and convergence}) Thanks to the previous steps, 
		it is straightforward to prove the main theorem (Theorem~\ref{S:MAINTHM}).
		More precisely, based on the a priori estimates from the previous step,
		it is a standard result that the perturbed solution exists on $[0,\TCrunch) \times \mathbb{S}^3$.
		Moreover, it is straightforward to prove, using arguments from \cite{iRjS2014b},
		that various time-rescaled solution variables converge 
		as $t \uparrow \TCrunch$ 
		and that
		$\Ricfour^{\alpha \beta} \Ricfour_{\alpha \beta}$
		blows up like as $\scale^{-4}(t)$ as $t \uparrow \TCrunch$; the proof of these facts
		essentially relies only on the improved estimates of Step (5).
\end{enumerate}

\subsection{Paper outline}
\label{SS:PAPEROUTLINE}
The remainder of the paper is organized as follows.

\begin{itemize}
	\item In Subsect.\ \ref{SS:NOTATIONANDCONVENTIONS}, we summarize some of our notation and conventions.
	\item In Sect.\ \ref{S:GEOMETRYBACKGROUND}, we construct some basic geometric objects on $\mathbb{S}^3$
		that we use throughout our analysis.
	\item In Sect.\ \ref{S:FLRW}, we formally introduce the FLRW solution
		and derive some properties of its scale factor,
		including information about its asymptotic behavior
		near the Big Bang and Big Crunch.
	\item In Sect.\ \ref{S:EQUATIONSINCMCTRANSPORTEDGAUGE}, we provide the Einstein-scalar field equations
		relative to CMC-transported spatial coordinates gauge.
	\item In Sect.\ \ref{S:RESCALEDVARSANDEQNS}, we introduce
		time-rescaled solution variables and derive the constraint/evolution/elliptic
		equations that they satisfy (relative to CMC-transported spatial coordinates gauge).
	\item In Sect.\ \ref{S:LIETRANSPORTEDFRAME}, we define $\Sigma_t$-projected Lie derivatives
		and explain how we extend various tensorfields defined on $\Sigma_0 = \lbrace 0 \rbrace \times \mathbb{S}^3$ 
		to the whole spacetime.
	\item In Sect.\ \ref{S:FIRSTVARANDCOMMUTATION}, we provide some geometric 
		variation and commutation identities.
	\item In Sect.\ \ref{S:COMMUTEDEQUATIONS}, we commute the evolution and constraint equations
		verified by the time-rescaled solution variables with 
		$\Sigma_t$-projected Lie derivatives and characterize the error terms.
	\item In Sect.\ \ref{S:ENERGYCURRENTSANDDIVIDENTITIES}, 
		we construct energy currents and use them to derive the fundamental
		divergence identities that hold for solutions.
		These can be viewed as ``approximate monotonicity identities'' in divergence form,
		and they form the starting point for our $L^2$ analysis.
	\item In Sect.\ \ref{S:INTEGRALSNORMSANDENERGIES}, we define the norms and energies
		that we use in our analysis.
	\item In Sect.\ \ref{S:DATAASSUMPTIONSANDBOOTSTRAPASSUMPTIONS},
		we state our smallness assumptions on the initial data and
		introduce norm bootstrap assumptions on a slab 
		of the form $[0,\Tboot) \times \mathbb{S}^3$.
		The norm is allowed to blow up near the Big Crunch.
	\item In Sect.\ \ref{S:PRELIMINEQUALITIES}, we use the 
		bootstrap assumptions to derive some preliminary comparison and sup-norm estimates.
	\item In Sect.\ \ref{S:STRONGC0ESTIMATES}, we
		prove ``strong'' sup-norm estimates, 
		which are less singular than the estimates directly implied by
		the bootstrap assumptions. These estimates are essential
		for closing the energy estimates and for proving convergence
		results near the singularities.
	\item In Sect.\ \ref{S:FUNDAMENTALENERGY},
		we integrate the divergence identities
		of Sect.\ \ref{S:ENERGYCURRENTSANDDIVIDENTITIES} 
		to derive integral inequalities for
		the energies. The integral inequalities
		feature the inhomogeneous terms from the commuted
		equations, which we control in the next three sections.
	\item In Sect.\ \ref{S:POINTWISE}, we derive pointwise estimates
		for the error terms in the commuted equations.
	\item In Sect.\ \ref{S:ELLIPTICLAPSEESTIMATES}, 
		we use the pointwise estimates from Sect.\ \ref{S:POINTWISE}
		to derive some preliminary $L^2$ estimates for the below-top-order derivatives of various solution variables
		as well as elliptic estimates that yield control of the lapse
		in terms of the energies.
	\item In Sect.\ \ref{S:L2BOUNDSFORTHEERRORTERMS}, we use the
		pointwise estimates from Sect.\ \ref{S:POINTWISE} 
		and the elliptic estimates of Sect.\ \ref{S:ELLIPTICLAPSEESTIMATES}
		to control the $L^2$ norms of all error terms in the commuted equations
		in terms of the energies.
	\item In Sect.\ \ref{S:ENERGYESTIMATES}, we derive our main priori energy estimates
		and derive improvements of the bootstrap assumptions.
	\item In Sect.\ \ref{S:MAINTHM}, we use the results from the previous
		sections to prove our main stable blowup theorem.
\end{itemize}

\subsection{Notation and conventions}
\label{SS:NOTATIONANDCONVENTIONS}
For the reader's convenience, we now summarize some notation and conventions that we use throughout the article.
Some of the concepts and objects referred to here are not formally defined until later in the article.

\subsubsection{Foliations}
The spacetime manifolds $\mathbf{M}$ (with boundary) that arise in our analysis are equipped with a time function $t$
that partitions certain regions $\mathbf{V} \subset \mathbf{M}$ 
into spacelike hypersurfaces of constant time: 
$\mathbf{V} = [0,T) \times \mathbb{T}^3 = \cup_{t \in [0,T)} \Sigma_t$. The
$\Sigma_t$ are CMC hypersurfaces. The level sets of $t$ 
are denoted by $\Sigma_t$:
\begin{align}
	\Sigma_t:= \lbrace (s,x) \in \mathbf{V} \ | \ s = t \rbrace.
\end{align}

\subsubsection{Metrics and connections}
\label{SSS:METRICSANDCONNECTIONS}
The spacetime metrics $\gfour$ under study are of the form $\gfour = - n^2 dt^2 + g_{ab} dx^a dx^b$.
$n(t,x)$ is the lapse function, and $g_{ij}(t,x)$ is a Riemannian metric on $\Sigma_t$.

Throughout, $\Dfour$ denotes the Levi--Civita connection of the spacetime metric $\gfour$ 
and $\nabla$ the Levi--Civita of the Riemannian metric $g$. $\nabla$ agrees with the Levi--Civita connection of 
the time-rescaled metric $\newg$ defined in Def.\ \ref{D:RESCALEDVARIABLES}.

\subsubsection{Indices} \label{SSS:INDICESANDFRAME}
Greek ``spacetime'' indices $\alpha, \beta, \cdots$ take on the values $0,1,2,3$ 
and are used to denote components of spacetime tensorfields relative to the 
spacetime coordinates $\lbrace x^{\alpha} \rbrace_{\alpha=0,1,2,3}$, where $x^0 = t$.
Lowercase Latin ``spatial'' indices $a,b,\cdots$ take on the values $1,2,3$
and are used to denote components of tensorfields
relative to the transported spatial coordinates $\lbrace x^a \rbrace_{a=1,2,3}$. 
We use capital Latin indices $A,B,\cdots$ to denote the components of tensorfields 
relative to the $\Sigma_t$-tangent frame 
$\mathscr{Z}
	:= 
\left\lbrace
			Z_{(1)},
			Z_{(2)},
			Z_{(3)}
		\right\rbrace
$
and co-frame
$
\Theta
:= 
		\left\lbrace
			\theta^{(1)},
			\theta^{(2)},
			\theta^{(3)}
		\right\rbrace$,
	which we construct in Subsect.\ \ref{SS:EXTENDINGSIGMA0TANGENTTENSORFIELDSTOSIGMATTANGENTTENSORFIELDS}.
Primed indices as $\alpha'$ are used in the same way as their non-primed counterparts,
and the same remarks hold for tilded indices such as $\widetilde{\alpha}$.
Repeated indices are summed over (from $0$ to $3$ if they are Greek, and from $1$ to $3$ if they are Latin). 
Lowercase spatial indices are lowered and raised with the Riemannian $3-$metric $g_{ij}$ and its inverse $g^{ij}$. 
We never implicitly lower and raise indices with the time-rescaled metric $\newg$ defined in Def.\ \ref{D:RESCALEDVARIABLES}; 
we always explicitly indicate the factors of $\newg$ and $\newg^{-1}$ whenever the rescaled metric is involved in lowering or raising,
or we use the sharp notation ``$\#$'' from Def.\ \ref{D:MUSICAL}.

\subsubsection{Natural contractions}
\label{SSS:NATURALCONTRACTIONS}
We use the notation ``$\cdot$'' to denote a natural contraction (without the need to raise or lower indices) of two tensorfields.
For example, if $\xi_{\alpha}$ is a one-form and $X^{\alpha}$ is a vectorfield,
then $\xi \cdot X := \xi_{\alpha} X^{\alpha}$.

\subsubsection{Spacetime tensorfields and $\Sigma_t-$tangent tensorfields}
\label{SSS:SPACETIMEANDSIGMATTANGENT}
We denote spacetime tensorfields $\Tfour_{\nu_1 \cdots \nu_n}^{\ \ \ \ \ \ \mu_1 \cdots \mu_m}$ in bold font. 
We denote the $\gfour$-orthogonal projection of $\Tfour_{\nu_1 \cdots \nu_n}^{\ \ \ \ \ \ \mu_1 \cdots \mu_m}$ 
onto the constant-time hypersurfaces $\Sigma_t$ in non-bold font: 
$T_{b_1 \cdots b_n}^{\ \ \ \ \ a_1 \cdots a_m}$,
or by using the $\Sigma_t$-projection notation defined in
Subsect.\ \ref{SS:PROJECTIONTENSORFIELD}.
We also denote general $\Sigma_t-$tangent tensorfields in non-bold font.

\subsubsection{Frame components and differential operators} \label{SSS:COORDINATES}
Many of our estimates are for the components of tensorfields relative to the  
frame $\mathscr{Z}$ and co-frame $\Theta$ described in Subsubsect.\ \ref{SSS:INDICESANDFRAME}.

If $\vec{I} = (i_1,i_2,\cdots,i_N)$ is an array with $i_a \in \lbrace 1,2,3 \rbrace$ for $1 \leq a \leq N$, then
$\SigmatLie_{\mathscr{Z}}^{\vec{I}} := \SigmatLie_{Z_{(i_1)}} \SigmatLie_{Z_{(i_2)}} \cdots \SigmatLie_{Z_{(i_N)}}$
denotes the corresponding $N^{th}$-order $\Sigma_t$-projected Lie derivative operator (see definition \eqref{E:SIGMATPROJECT}),
where the $Z_{(A)}$ are elements of the frame
$
\mathscr{Z}
	:= 
\left\lbrace
			Z_{(1)},
			Z_{(2)},
			Z_{(3)}
		\right\rbrace
$.
$|\vec{I}| := N$ denotes the order of $\vec{I}$.
If $\SigmatLie_{\mathscr{Z}}^{\vec{I}}$ acts on a scalar function $f$, then we sometimes
write $\mathscr{Z} ^{\vec{I}} f$ instead of $\SigmatLie_{\mathscr{Z}}^{\vec{I}} f$.
If $\vec{I} = (\iota_1, \iota_2, \cdots, \iota_N)$,
		then 
		$\vec{I}_1 + \vec{I}_2 = \vec{I}$ 
		means that
		$\vec{I}_1 = (\iota_{k_1}, \iota_{k_2}, \cdots, \iota_{k_m})$
		and
		$\vec{I}_2 = (\iota_{k_{m+1}}, \iota_{k_{m+2}}, \cdots, \iota_{k_N})$,
		where $1 \leq m \leq N$ and
		$k_1, k_2, \cdots, k_N$ is a permutation of 
		$1,2,\cdots,N$. 
		Sums such as $\vec{I}_1 + \vec{I}_2 + \cdots + \vec{I}_M = \vec{I}$
		have an analogous meaning.

\subsubsection{Commutators and Lie brackets} \label{SSS:COMMUTATORS}
Given two operators $A$ and $B$,
\begin{align} 
	[A,B]
\end{align}
denotes the operator commutator $A B - B A$.

If $\mathbf{X}$ and $\mathbf{Y}$ are two vectorfields, then
\begin{align}
	\mathcal{L}_{\mathbf{X}}\mathbf{Y} = [\mathbf{X}, \mathbf{Y}] 
\end{align}
denotes the Lie derivative of $\mathbf{Y}$ with respect to $\mathbf{X}$.
Relative to an arbitrary coordinate system,
\begin{align} \label{E:LIEBRACKETCOORDINATES}
	[\mathbf{X}, \mathbf{Y}]^{\mu} = \mathbf{X}^{\alpha} \partial_{\alpha} \mathbf{Y}^{\mu} 
		- \mathbf{Y}^{\alpha} \partial_{\alpha} \mathbf{X}^{\mu}.
\end{align}

\subsubsection{Constants}
\label{SSS:CONSTANTS}
We use the symbols $C$ and $c$ to denote positive constants that are free to 
vary from line to line. These constants can be chosen to be independent of the
bootstrap parameters $\varepsilon$ and $\upsigma$ from Subsect.\ \ref{SS:BOOTSTRAPASSUMPTIONS},
as long as $\varepsilon$ and $\upsigma$ are sufficiently small.
If we want to emphasize that a constant depends on a quantity ``$q$,'' then we use the notation $C_q$.
We write $A \lesssim B$ to mean that there exists a constant $C > 0$ such that
$A \leq C B$. We write $A = \mathcal{O}(B)$ to mean that $|A| \lesssim B$.

\section{Geometry of $\mathbb{S}^3$ and the round metric}
\label{S:GEOMETRYBACKGROUND}
In this section, we construct some basic geometric objects on $\mathbb{S}^3$.
Although these objects are adapted to the background
FLRW spatial geometry,
some of our key analysis for perturbed solutions
relies on this geometry.

\begin{definition}[\textbf{A round metric on} $\mathbb{S}^3$]
	\label{D:ROUNDMETRIC}
	We let $\StMet$ denote the round metric on $\mathbb{S}^3$
	with scalar curvature $\ScalarCurArg{\StMet} = \frac{2}{3}$.
\end{definition}

\begin{remark}
	\label{R:BACKGROUNDMETRICINRECTANGULARCOORDINATES}
	If we view $\mathbb{S}^3 \subset \mathbb{R}^4$ as the submanifold 
	$\lbrace (y^1,y^2,y^3,y^4) \in \mathbb{R}^4 \ | \ \sum_{i=1}^4 (y^i)^2 = 9 \rbrace$,
	then $\StMet = E|_{\mathbb{S}^3}$, where $E$ is the standard Euclidean metric on $\mathbb{R}^4$, that is,
	$E_{ij} := \mbox{diag}(1,1,1,1)$ relative to standard ``rectangular'' coordinates $\lbrace y^i \rbrace_{i=1,\cdots,4}$
	on $\mathbb{R}^4$.
\end{remark}

\begin{definition}[\textbf{Rotations on} $\mathbb{S}^3$]
	\label{D:KILLINGOFROUND}
	For $1 \leq k < l \leq 4$,
	we define the (six) rotation vectorfields $\Rot{k}{l}$
	on $\mathbb{S}^3 := \lbrace (y^1,y^2,y^3,y^4) \in \mathbb{R}^4 \ | \ \sum_{i=1}^4 (y^i)^2 = 9 \rbrace$
	as follows:
	\begin{align} \label{E:ROTATIONS}
	\Rot{k}{l} 
	& :=
	\frac{1}{3}
	\left\lbrace
		y^k \frac{\partial}{\partial y^l} - y^l \frac{\partial}{\partial y^k}
	\right\rbrace|_{\mathbb{S}^3}
	\end{align}
	relative to the rectangular coordinates on $\mathbb{R}^4$ mentioned 
	in Remark~\ref{R:BACKGROUNDMETRICINRECTANGULARCOORDINATES}.
\end{definition}

It is a standard fact that the 
$\Rot{k}{l}$ form a basis for the six-dimensional Lie algebra of Killing fields\footnote{Recall
that $Z$ is a $\StMet-$Killing field if and only if $\Lie_Z \StMet = 0$.} of $\StMet$.

In the next lemma, we use the rotation vectorfields
to construct a global frame of $\StMet-$Killing vectorfields
and a corresponding global co-frame. This construction relies on
the parallelizability of $\mathbb{S}^3$.

\begin{lemma}[\textbf{An} $\mathbb{S}^3$-\textbf{basis of Killing vectorfields for} $\StMet$ \textbf{and the corresponding co-frame}]
	\label{L:BASISOFKILLINGFIELDS}
		Consider the following $\StMet-$Killing vectorfields on $\mathbb{S}^3$:
		\begin{align} \label{E:ROUNDMETRICKILLING}
			Z_{(1)} 
			& := \Rot{1}{2} + \Rot{3}{4},
			\qquad
			Z_{(2)}
			:= \Rot{2}{3} + \Rot{1}{4},
			\qquad
			Z_{(3)}
			:= \Rot{1}{3} - \Rot{2}{4}.
		\end{align}
		
		Then
		$
		\left\lbrace
			Z_{(1)},
			Z_{(2)},
			Z_{(3)}
		\right\rbrace
		$
		are linearly independent.
		Thus, we can define corresponding basis-dual one-forms $\theta^{(A)}$ 
		by the following formula:\footnote{In this article, we do not need precise expressions for 
		the $\theta^{(A)}$ relative to coordinates.}
		\begin{align} \label{E:DUALTOROUNDMETRICKILLING}
			\theta^{(A)}(Z_{(B)}) 
			& = \delta_B^A,
			&&
			(A,B = 1,2,3),
		\end{align}
		where $\delta_B^A$ is the standard Kronecker delta function
		and $\theta^{(A)}(Z_{(B)}) := \theta_a^{(A)} Z_{(B)}^a$
		relative to arbitrary local coordinates on $\mathbb{S}^3$.
\end{lemma}

\begin{proof}
	The linear independence of the
	$
		\left\lbrace
			Z_{(1)},
			Z_{(2)},
			Z_{(3)}
		\right\rbrace
		$
	is easy to check by direct calculation.
	Alternatively, the linear independence follows from 
	Lemma~\ref{L:FRAMEVECTORFIELDLIEBRACKET} below,
	where it is shown that the $\lbrace Z_{(A)} \rbrace_{A=1,2,3}$
	form a global $\StMet-$orthonormal frame on $\mathbb{S}^3$.
\end{proof}

\begin{definition}[\textbf{The $\StMet-$orthonormal frame and co-frame}]
	\label{D:ZANDTHETASETS}
	We define the vectorfield frame $\mathscr{Z}$ on $\mathbb{S}^3$ and the co-frame $\Theta$ as follows:
	\begin{subequations}
	\begin{align}
		\mathscr{Z}
		& := 
		\left\lbrace
			Z_{(1)},
			Z_{(2)},
			Z_{(3)}
		\right\rbrace,
			\label{E:ZSET} \\
		\Theta
		& := 
		\left\lbrace
			\theta^{(1)},
			\theta^{(2)},
			\theta^{(3)}
		\right\rbrace.
		\label{E:THETASET}
	\end{align}
	\end{subequations}
\end{definition}

In the next lemma, we exhibit some basic properties of the frame $\mathscr{Z}$.

\begin{lemma}[\textbf{Basic properties of the frame} $\mathscr{Z}$]
\label{L:FRAMEVECTORFIELDLIEBRACKET}
	The elements of the set $\mathscr{Z}$ from Def.\ \ref{D:ZANDTHETASETS}
	form a global $\StMet-$orthonormal frame on $\mathbb{S}^3$.
	Moreover, the following vectorfield commutation relations hold:
	\begin{align} \label{E:FRAMEVECTORFIELDLIEBRACKET}
		[Z_{(A)},Z_{(B)}] = \frac{2}{3} \epsilon_{ABC} Z_{(C)},
	\end{align}
	where $\epsilon_{ABC}$ is the fully antisymmetric symbol normalized by $\epsilon_{123} = 1$.
\end{lemma}

\begin{proof}
	By construction, the elements of $\lbrace Z_{(A)} \rbrace_{A=1,2,3}$ are tangent to 
	$\mathbb{S}^3 = \lbrace (y^1,y^2,y^3,y^4) \in \mathbb{R}^4 \ | \ \sum_{i=1}^4 (y^i)^2 = 9 \rbrace \subset \mathbb{R}^4$.
	Moreover, it is straightforward to compute that the $\lbrace Z_{(A)} \rbrace_{A=1,2,3}$,
	viewed as vectorfields on $\mathbb{R}^4$ defined by \eqref{E:ROUNDMETRICKILLING},
	are orthonormal with respect to the Euclidean metric $E$ on $\mathbb{R}^4$.
	Since $\StMet = E|_{\mathbb{S}^3}$, it follows that
	the $\lbrace Z_{(A)} \rbrace_{A=1,2,3}$
	form a global $\StMet-$orthonormal frame on $\mathbb{S}^3$.
	\eqref{E:FRAMEVECTORFIELDLIEBRACKET} follows from 
	definition \eqref{E:ROUNDMETRICKILLING}
	and straightforward computations.
\end{proof}

\begin{corollary}[\textbf{Frame decomposition of $\StMet^{-1}$ and $\StMet$}]
	\label{C:STMETRICFRAMEEXPANDED}
	With $\delta^{AB}$ and $\delta_{AB}$ denoting standard Kronecker deltas, 
	we have
	\begin{align} \label{E:STINVERSEMETRICFRAMEEXPANDED}
		\StMet^{-1}
		& = \delta^{AB} Z_{(A)} \otimes Z_{(B)},
	&&
	\StMet
		= \delta_{AB} \theta^{(A)} \otimes \theta^{(B)}.
	\end{align}
\end{corollary}

\begin{proof}
	The identity
	$\StMet = \delta_{AB} \theta^{(A)} \otimes \theta^{(B)}$
	is easy to verify using that
	the vectorfields $\lbrace Z_{(A)} \rbrace_{A=1,2,3}$
	form a global $\StMet-$orthonormal frame on $\mathbb{S}^3$
	and using the defining property \eqref{E:DUALTOROUNDMETRICKILLING}.
	The identity 
	$\StMet^{-1} = \delta^{AB} Z_{(A)} \otimes Z_{(B)}$
	then follows from the identity 
	$\StMet = \delta_{AB} \theta^{(A)} \otimes \theta^{(B)}$
	and \eqref{E:DUALTOROUNDMETRICKILLING}.
\end{proof}

\section{FLRW solution and properties of its scale factor}
\label{S:FLRW}
In this section, we exhibit some basic properties of the FLRW solution
to the Einstein-scalar field system. In particular, we derive some 
quantitative properties of its scale factor, which are important
for all of the ensuing analysis.

\subsection{The FLRW solution and its scale factor}
\label{SS:FLRWANDSCALEFACTOR}
As we mentioned at the beginning, the well-known FLRW solution 
to \eqref{E:EINSTEINSF}-\eqref{E:WAVEMODEL}
with spatial topology $\mathbb{S}^3$
can be expressed as follows:
\begin{align} \label{E:FLRW}
	\gfour_{FLRW}
	& := - dt^2 + \scale^{2/3}(t) \StMet,
	\qquad
	\partial_t \phi_{FLRW}
		:= \sqrt{\frac{2}{3}} \scale^{-1}(t),
		\qquad
	\nabla \phi_{FLRW}
		:= 0,
\end{align}
where $\StMet$ is the round metric on $\mathbb{S}^3$ from Def.\ \ref{D:ROUNDMETRIC}
and by our conventions, the \emph{scale factor} $\scale = \scale(t)$
satisfies the following normalization conditions:
\begin{align} \label{E:SCALEFACTORINITIALCONDITONS}
	\scale(0)
	& = 1,
	&&
	\scale'(0)
	= 0,
\end{align}
where throughout the paper, 
\begin{align} \label{E:SCALEPRIME}
	\scale'(t) := \frac{d}{dt} \scale(t).
\end{align}

\begin{remark}[\textbf{Notation involving powers of} $\scale$]
	Throughout the paper, we often write $\scale^p(t)$ instead of $(\scale(t))^p$.
	Moreover, we often write $\scale^p$ instead of $\scale^p(t)$ when the time $t$ is clear
	from context.
\end{remark}

In order for the fields in \eqref{E:FLRW} to actually be a solution to the Einstein-scalar field system,
the scale factor must verify the well-known Friedmann ODEs. We postpone further discussion
of this until Subsect.\ \ref{SS:SCALEFACTORANDHUBBLE} (see Lemma~\ref{L:FRIEDMANN}).

Later in the article, in stating some equations, we 
will find it convenient to refer to the \emph{Hubble factor} $\Hubble$, 
defined by
\begin{align} \label{E:HUBBLEFACTOR}
	\Hubble
	= \Hubble(t)
	& := \frac{\scale'(t)}{\scale(t)}.
\end{align}

\subsection{Analysis of the scale factor and the Hubble factor}
\label{SS:SCALEFACTORANDHUBBLE}
In this section, we derive some properties of the scale factor and the Hubble factor.

\begin{lemma}[\textbf{Friedmann's\footnote{The ODEs of Lemma~\ref{L:FRIEDMANN} are equivalent to
(but not the same as)
what are usually referred to as ``Friedmann's equations'' in the literature.} equations}]
	\label{L:FRIEDMANN}
	Let $\scale = \scale(t)$ be the scale factor appearing in the expression 
	\eqref{E:FLRW} for the FLRW metric and scalar field,
	subject to the initial conditions \eqref{E:SCALEFACTORINITIALCONDITONS}.
	Then $(\gfour_{FLRW}, \phi_{FLRW})$
	verify the Einstein-scalar field equations \eqref{E:EINSTEINSF}-\eqref{E:WAVEMODEL}
	if and only if $\scale$ verifies the following ODEs:
	\begin{subequations}
	\begin{align}
		(\scale')^2
		& = 1 - \scale^{4/3},
			\label{E:FRIEDMANNFIRSTORDER} \\
		\scale''
		& = - \frac{2}{3} \scale^{1/3}.
		\label{E:FRIEDMANNSECONDORDER}
	\end{align}
	\end{subequations}
	Moreover, any solution to the ODE \eqref{E:FRIEDMANNSECONDORDER} 
	that verifies \eqref{E:FRIEDMANNFIRSTORDER} at time $0$
	must in fact verify \eqref{E:FRIEDMANNFIRSTORDER} during
	its maximal interval of classical existence.
	
	Finally, equation \eqref{E:FRIEDMANNSECONDORDER} implies the following ODE for
	the Hubble factor defined in \eqref{E:HUBBLEFACTOR}:
	\begin{align} \label{E:HUBBLEODE}
		\Hubble'
		& = - \frac{2}{3} \scale^{-2/3} - \Hubble^2.
	\end{align}
\end{lemma}

\begin{proof}
	Standard computations yield that the metric
	$\gfour_{FLRW}$ defined in \eqref{E:FLRW} has the following Ricci curvature components,
	where $x^0 := t$ is the time coordinate and
	$\lbrace x^1, x^2, x^3 \rbrace$ are arbitrary local coordinates on $\mathbb{S}^3$:
	\begin{align} \label{E:FLRWRIC00}
		\Ricfour_{00}
		& = - \frac{\scale''}{\scale}
				+ 
				\frac{2}{3} \frac{(\scale')^2}{\scale^2},
				 \\
		\Ricfour_{i0}
		& = 0,
			\label{E:FLRWRICI0} \\
		\Ricfour_{ij}
		& = \left\lbrace
					\frac{1}{3} \frac{\scale''}{\scale^{1/3}}
					+ 
					\frac{2}{9}
			\right\rbrace
			\StMet_{ij}.
		\label{E:FLRWRICIJ}
	\end{align}
	Next, using \eqref{E:FLRW}
	and \eqref{E:FLRWRIC00}-\eqref{E:FLRWRICIJ},
	we compute that the scalar curvature of $\gfour_{FLRW}$
	is
	\begin{align} \label{E:SCALARCURVATUREFLRW}
		\Rfour
		& = 
			2 \frac{\scale''}{\scale}
			-
			\frac{2}{3} \frac{(\scale')^2}{\scale^2}
			+
			\frac{2}{3} \frac{1}{\scale^{2/3}}.
	\end{align}
	From \eqref{E:FLRW} and \eqref{E:FLRWRIC00}-\eqref{E:SCALARCURVATUREFLRW},
	we compute that the components of the Einstein tensor
	$\Ricfour_{\mu \nu} - \frac{1}{2} \Rfour \gfour_{\mu \nu}$
	of the FLRW metric are
	\begin{align} \label{E:EINSTEINTENSORFLRW00}
		\Ricfour_{00} - \frac{1}{2} \Rfour \gfour_{00}
		& =  
				\frac{1}{3} \frac{(\scale')^2}{\scale^2}
				+
				\frac{1}{3} \frac{1}{\scale^{2/3}},
			\\
		\Ricfour_{0i} 
		- 
		\frac{1}{2} \Rfour \gfour_{0i}
		& = 0,
			\label{E:EINSTEINTENSORFLRW0I} \\
		\Ricfour_{ij} 
		- 
		\frac{1}{2} \Rfour \gfour_{ij}
		& = \left\lbrace
					- 
					\frac{2}{3} \frac{\scale''}{\scale^{1/3}}
					+
					\frac{1}{3} \frac{(\scale')^2}{\scale^{4/3}}
					- 
					\frac{1}{9}
			\right\rbrace
			\StMet_{ij}.
		\label{E:EINSTEINTENSORFLRWIJ}
	\end{align}
	Next, from \eqref{E:EMTSCALARFIELD} and \eqref{E:FLRW},
	we compute that the components of the energy momentum tensor 
	of the FLRW metric/scalar field are
	\begin{align}
		\Tfour_{00}
		& = \frac{1}{3} \frac{1}{\scale^2},
			\label{E:FLRWENMOMENTUM00} \\
		\Tfour_{0i}
		& = 0,
			\label{E:FLRWENMOMENTUM0I} \\
		\Tfour_{ij}
		& = \frac{1}{3} \frac{1}{\scale^{4/3}} \StMet_{ij}.
		\label{E:FLRWENMOMENTUMIJ}
	\end{align}
	From \eqref{E:EINSTEINTENSORFLRW00}, \eqref{E:FLRWENMOMENTUM00},
	and Einstein's equation \eqref{E:EINSTEINSF},
	we conclude \eqref{E:FRIEDMANNFIRSTORDER}.
	Similarly, from \eqref{E:EINSTEINTENSORFLRWIJ}, \eqref{E:FLRWENMOMENTUMIJ},
	and Einstein's equation \eqref{E:EINSTEINSF},
	we compute that
	\begin{align} \label{E:ALMOSTFRIEDMANNSECONDORDER}
		\scale''
		& = \frac{1}{2} \frac{(\scale')^2}{\scale}
		- 
		\frac{1}{6} \scale^{1/3}
		-
		\frac{1}{2} \frac{1}{\scale}.
		\end{align}
		From \eqref{E:FRIEDMANNFIRSTORDER} and \eqref{E:ALMOSTFRIEDMANNSECONDORDER},
		we conclude \eqref{E:FRIEDMANNSECONDORDER}.
		We have thus shown that Einstein's equations imply
		\eqref{E:FRIEDMANNFIRSTORDER}-\eqref{E:FRIEDMANNSECONDORDER}.
		Moreover, it is easy to extend the above argument
		to conclude that if $\scale(t)$ verifies the ODEs
		\eqref{E:FRIEDMANNFIRSTORDER}-\eqref{E:FRIEDMANNSECONDORDER}, 
		then the FLRW metric/scalar field defined in \eqref{E:FLRW}
		are solutions to the Einstein-scalar field equations \eqref{E:EINSTEINSF}-\eqref{E:WAVEMODEL}.
	
	Next, we show that
	solutions to \eqref{E:FRIEDMANNSECONDORDER} also
	verify \eqref{E:FRIEDMANNFIRSTORDER} if
	\eqref{E:FRIEDMANNFIRSTORDER} holds at time $0$.
	To see this, we note that
	\eqref{E:FRIEDMANNFIRSTORDER}
	is equivalent to $f(t) := (\scale'(t))^2 + \scale^{4/3}(t) - 1 = 0$
	and that $f'(t) = 2 \scale'(t) \left\lbrace \scale''(t) + (2/3) \scale^{1/3}(t) \right\rbrace$.
	It follows that \eqref{E:FRIEDMANNSECONDORDER} implies that $f'(t) \equiv 0$ 
	and thus $f(t) \equiv 0$ if $f(0) = 0$,
	which is the desired result.
		
	To complete the proof of the lemma, we need only to show that	
	equation \eqref{E:FRIEDMANNSECONDORDER}
	implies equation \eqref{E:HUBBLEODE}.
	The desired result follows from a simple calculation based on
	definition \eqref{E:HUBBLEFACTOR}.
	
\end{proof}

\begin{remark}
	In the remainder of the paper,
	$\scale(t)$ denotes the FLRW scale factor, that is, the
	solution to the ODEs \eqref{E:FRIEDMANNFIRSTORDER}-\eqref{E:FRIEDMANNSECONDORDER}
	subject to the initial conditions \eqref{E:SCALEFACTORINITIALCONDITONS}.
	We stress that by Lemma~\ref{L:FRIEDMANN}, 
	under \eqref{E:SCALEFACTORINITIALCONDITONS},
	\eqref{E:FRIEDMANNFIRSTORDER} follows from \eqref{E:FRIEDMANNSECONDORDER}.
\end{remark}

Our next goal is to reveal the basic quantitative and qualitative properties of $\scale$,
especially near times where $\scale$ vanishes.
We provide these properties in Lemma~\ref{L:ANALYSISOFFRIEDMANN}.
We start by providing the following preliminary lemma.

\begin{lemma}[\textbf{Preliminary analysis connected to the scale factor}]
	\label{L:SCALEFACTORODEANTIDERIVATIVE}
	Let $\antiscale:[0,1) \rightarrow \mathbb{R}$ be the solution to 
	the ODE initial value problem
	\begin{align} \label{E:SCALEFACTORODEANTIDERIVATIVE}
		\antiscale'(a)
		& = \frac{1}{\sqrt{1 - a^{4/3}}},
		&&
		\antiscale(0)
		= 0.
	\end{align}
	Then $\antiscale$ is strictly increasing on $[0,1)$
	and there exists a real number $M$ verifying
	\begin{align}	\label{E:MDEF}
		0 
		<
		M 
		&:= 
		\int_{a=0}^1
			\frac{1}{\sqrt{1 - a^{4/3}}}
		\, da		
		< \infty
	\end{align}
	such that $\antiscale$ extends to a function (also denoted by $\antiscale$)
	of class $C([0,1]) \cap C^2([0,1)) \cap C^{\infty}((0,1))$
	with $\antiscale([0,1]) = [0,M]$.
	Moreover, there exists a function $F$ such that 
	$F(0) = 0$, $F'(0) = 0$, and
	$\antiscale(a) 
		= a 
		+ F(a)
	$.
	
	Let $\inverseantiscale:[0,M] \rightarrow [0,1]$
	denote the inverse function of $\antiscale$.
	Then there is a unique extension of $\inverseantiscale$ 
	(also denoted by $\inverseantiscale$)
	to an element of
	$C^2([0,2M]) \cap C^{\infty}((0,2M))$ 
	such that for $y \in [0,M]$, we have the 
	symmetry property 
	$\inverseantiscale(M+y) = \inverseantiscale(M-y)$.
	Moreover, there exists a function $\widetilde{F}$ such that
	$\widetilde{F}(0) = 0$, 
	$\widetilde{F}'(0) = 0$,
	and
	$\inverseantiscale(b) 
	= b + \widetilde{F}(b)$.
\end{lemma}

\begin{proof}
	All aspects of the lemma regarding $\antiscale$
	are easy to derive from the ODE \eqref{E:SCALEFACTORODEANTIDERIVATIVE}
	and by using the fact that RHS~\eqref{E:SCALEFACTORODEANTIDERIVATIVE}
	is integrable over the domain $a \in [0,1]$.
	
	All aspects of the lemma regarding $\inverseantiscale$
	are also easy to derive, except for the fact that
	$\inverseantiscale \in C^2([0,2M]) \cap C^{\infty}((0,2M))$,
	which we now derive. The aspect that requires careful checking is that 
	$\inverseantiscale$ is $C^{\infty}$ at the midpoint $M$.
	To prove this fact, we repeatedly differentiate the equation
	$\inverseantiscale \circ \antiscale(a) = a$ with respect to $a$
	and use the ODE \eqref{E:SCALEFACTORODEANTIDERIVATIVE}
	to compute that for $p \in [0,1)$, we have
	\begin{align} \label{E:INVERSEANTISCALEFIRSTTWODERIVATIVES}
		(\inverseantiscale)'(p)
		& = \sqrt{1 - ((\inverseantiscale)(p))^{4/3}},
			\qquad
		(\inverseantiscale)''(p)
		= - \frac{2}{3} ((\inverseantiscale)(p))^{1/3}.
	\end{align}
	From \eqref{E:INVERSEANTISCALEFIRSTTWODERIVATIVES},
	the fact that $\lim_{p \uparrow M} \inverseantiscale(p) = 1$,
	and a straightforward induction argument,
	we deduce that 
	the left derivatives of $\inverseantiscale$ at all orders exist at $M$
	and that $(\inverseantiscale)^{(k)}(M) = 0$ for $k$ odd, where
	$(\inverseantiscale)^{(k)}$ is the $k^{th}$ derivative of $\inverseantiscale$.
	Since $\inverseantiscale(M+y) = \inverseantiscale(M-y)$ by construction,
	it follows that $\inverseantiscale$ is $C^{\infty}$ at $M$. 
	
\end{proof}

With the help of Lemmas~\ref{L:FRIEDMANN}-\ref{L:SCALEFACTORODEANTIDERIVATIVE}, we now derive
the properties of the scale factor that are relevant for our proof of stable blowup.
See Fig.~\ref{F:SCALEFACTOR} on pg.~\pageref{F:SCALEFACTOR} for the graph of $\scale(t)$.

\begin{lemma}[\textbf{Analysis of the FLRW scale factor via Friedmann's equations}]
	\label{L:ANALYSISOFFRIEDMANN}
	Let $\scale(t)$ be the solution to the ODE \eqref{E:FRIEDMANNSECONDORDER}
	under the initial conditions \eqref{E:SCALEFACTORINITIALCONDITONS}
	and note that equation \eqref{E:FRIEDMANNFIRSTORDER} also holds
	by Lemma~\ref{L:FRIEDMANN}. 
	Then there exist times $0 < \TCrunch = - \TBang$
	such that $\scale(t)$ is a classical solution for $t \in [\TBang,\TCrunch]$
	and such that $\scale:[\TBang,\TCrunch] \rightarrow [0,1]$. In fact,
	$\TCrunch$ is equal to the number $M$ from Lemma~\ref{L:SCALEFACTORODEANTIDERIVATIVE}.
	Moreover, $\scale \in C^{\infty}((\TBang,\TCrunch)) \cap C^2([\TBang,\TCrunch])$,
	$\scale$ is even, $\scale$ strictly increases from $0$ to $1$
	on $[\TBang,0]$, and $\scale$ strictly decreases from $1$ to $0$ on $[0,\TCrunch]$.
	In addition, with $\mathcal{O}(f(t))$ denoting a continuous function that is bounded in magnitude by 
	$\leq C |f(t)|$ for $t \in [\TBang,\TCrunch]$, we have
	\begin{subequations}
	\begin{align} \label{E:SCALEFACTORASYMPTOTICSNEARBANG}
		\scale(t)
		& = (t - \TBang) 
			+ \mathcal{O}((t - \TBang)^2),
		&&
		t \in [\TBang,0],
			\\
		\scale(t)
		& = (\TCrunch - t)
			+ \mathcal{O}((\TCrunch - t)^2),
		&&
		t \in [0,\TCrunch].
		\label{E:SCALEFACTORASYMPTOTICSNEARCRUNCH}
	\end{align}
	\end{subequations}
	
	Moreover, $\scale'$ is an odd function with
	\begin{subequations}
	\begin{align} \label{E:TIMEDERIVATIVESCALEFACTORASYMPTOTICSNEARBANG}
		\scale'(t)
		& = 1 
				+ \mathcal{O}((t - \TBang)),
		&&
		t \in [\TBang,0],
			\\
		\scale'(t)
		& = - 1 
				+ \mathcal{O}((\TCrunch - t)),
		&&
		t \in [0,\TCrunch].
		\label{E:TIMEDERIVATIVESCALEFACTORASYMPTOTICSNEARCRUNCH}
	\end{align}
	\end{subequations}
	
	
	Finally, if $q > 0$ is a real number, then
	\begin{align} \label{E:SCALEFACTORLOGPOWERBOUND}
		|\ln \scale(t)|
		\leq \frac{1}{q e}
		\scale^{-q}(t).
	\end{align}

\end{lemma}

\begin{remark}[\textbf{The role of the estimate} \eqref{E:SCALEFACTORLOGPOWERBOUND}]
	Later in the article, we will use the estimate \eqref{E:SCALEFACTORLOGPOWERBOUND}
	to deduce that $|\ln \scale(t)| \lesssim \frac{1}{\sqrt{\varepsilon}} \scale^{-\sqrt{\varepsilon}}(t)$,
	where $\varepsilon > 0$ is a small bootstrap parameter. We employ this simple estimate mainly out of convenience, 
	as it will allow for a unified presentation of many of our estimates.
\end{remark}

\begin{proof}[Proof of Lemma~\ref{L:ANALYSISOFFRIEDMANN}]
The initial conditions \eqref{E:SCALEFACTORINITIALCONDITONS} and 
the ODEs \eqref{E:FRIEDMANNFIRSTORDER}-\eqref{E:FRIEDMANNSECONDORDER}
are invariant under the reflection $t \rightarrow - t$. It follows that
$\scale$ is even. Thus, in the rest of the proof, we derive the properties of
$\scale$ only for non-positive times.

Next, we note that it is straightforward to deduce from the 
initial conditions \eqref{E:SCALEFACTORINITIALCONDITONS}, 
the ODEs \eqref{E:FRIEDMANNFIRSTORDER}-\eqref{E:FRIEDMANNSECONDORDER},
and basic existence theory for ODEs 
that $\scale$ is strictly increasing in a half-interval of the form
$[-\updelta,0]$ for some $\updelta > 0$.
From this fact and 
the ODE \eqref{E:FRIEDMANNFIRSTORDER},
we deduce that
\begin{align} \label{E:SCALEFACTORMORECONVENIENTODE}
	\frac{d}{dt} (\antiscale \circ \scale)
	& = 1
\end{align}
on any interval of the form $[T,0)$ 
such that $T < 0$ and such that $\scale([T,0)) \subset [0,1)$,
since $\antiscale \in C^2([0,1))$.
From \eqref{E:SCALEFACTORMORECONVENIENTODE},
the above remarks, 
and the properties of $\antiscale$ shown in Lemma~\ref{L:SCALEFACTORODEANTIDERIVATIVE},
it follows that there exists a ``least negative'' time $\TBang < 0$ 
such that
\begin{align} \label{E:COMPOSEDFUNCTIONVANISHESATBANG}
	\antiscale \circ \scale(\TBang) = 0
\end{align}
and such that $\antiscale \circ \scale(t) > 0$ for $\TBang < t \leq 0$.
Moreover, from \eqref{E:COMPOSEDFUNCTIONVANISHESATBANG} and 
Lemma~\ref{L:SCALEFACTORODEANTIDERIVATIVE}, it follows that $\scale(\TBang) = 0$.
In addition, from \eqref{E:SCALEFACTORMORECONVENIENTODE} and \eqref{E:COMPOSEDFUNCTIONVANISHESATBANG},
we see that for $t \in [\TBang,0)$, we have
$\antiscale \circ \scale(t) = t - \TBang$
and thus
\begin{align} \label{E:SCALESEMIEXPLICIT}
	\scale(t) 
	& = \inverseantiscale(t - \TBang).
\end{align}
From \eqref{E:SCALESEMIEXPLICIT},
the initial condition $\scale(0) = 1$,
and Lemma~\ref{L:SCALEFACTORODEANTIDERIVATIVE},
it follows that $- \TBang = M$, where $M > 0$ is as in \eqref{E:MDEF}.
From this fact,
\eqref{E:SCALESEMIEXPLICIT},
and the properties of $\inverseantiscale$ shown in Lemma~\ref{L:SCALEFACTORODEANTIDERIVATIVE},
we conclude \eqref{E:SCALEFACTORASYMPTOTICSNEARBANG} and \eqref{E:TIMEDERIVATIVESCALEFACTORASYMPTOTICSNEARBANG}.

\eqref{E:SCALEFACTORLOGPOWERBOUND} is a simple consequence of the fact that
$\scale(t) \in (0,1]$ for $t \in (\TBang,\TCrunch)$
and the fact that the function 
$\displaystyle
f(y) := 
\frac{\ln y}{y^q}$
is non-negative and
bounded from above by $\displaystyle \frac{1}{e q}$ on the domain $y \in [1,\infty)$.
\end{proof}

We will use the following two simple corollaries of Lemma~\ref{L:ANALYSISOFFRIEDMANN}
when we derive a priori estimates.

\begin{corollary}[\textbf{Estimates for time integrals involving the scale factor}]
\label{C:SCALEFACTORTIMEINTEGRALS}
Let $p \in \mathbb{R}$. There exist constants $C > 0$
(which can vary from line to line) 
that do not depend on $p$
and constants $C_p > 0$ that can depend on $|p|$ in 
a continuous increasing fashion 
such that the following
integral estimates hold for $0 \leq t_1 \leq t_2 \leq \TCrunch$:
\begin{align} \label{E:SCALEFACTORTIMEINTEGRALS}
	\int_{s=t_1}^{t_2}
		\scale^p(s)
	\, ds	
	& \leq
	\begin{cases}
	\frac{C_p}{|1 + p|} \scale^{p + 1}(t_2),
	& \mbox{if } p < - 1, \\
		\left|\ln \scale(t_2) \right| + C, & \mbox{if } p = - 1, \\
		\frac{C_p}{1 + p} \scale^{p + 1}(t_1),
		& \mbox{if } p > - 1.
	\end{cases}
\end{align}

Moreover, we have
\begin{align} \label{E:EXPONENTIATEDSCALEFACTORTIMEINTEGRALS}
	\exp
	\left(
		\int_{s=0}^t
			\scale^p(s)
		\, ds	
	\right)
	& \leq
	\begin{cases}
		\frac{C}{\scale(t)} & \mbox{if } p = - 1, \\
		\exp
		\left(
			\frac{C_p}{1 + p}
		\right) & \mbox{if } p > - 1.
	\end{cases}
\end{align}

\end{corollary}

\begin{proof}
	Throughout this proof, it is understood that the constants $C$ and $C_p$ have the properties
	stated in the corollary and are allowed to vary from line to line.
	We first prove \eqref{E:SCALEFACTORTIMEINTEGRALS} in the case $p < - 1$.
	First, we use Lemma~\ref{L:ANALYSISOFFRIEDMANN}
	(in particular \eqref{E:SCALEFACTORASYMPTOTICSNEARCRUNCH})
	to deduce that the following estimates hold for $t \in [0,\TCrunch)$:
	\begin{align} \label{E:SCALEFACTORUPPERBOUND}
		\scale(t) 
		& \leq \TCrunch - t
		+ 
		C (\TCrunch - t)^2
		\leq C (\TCrunch - t),
			\\
		\frac{1}{\scale(t)}
		& \leq 
		\frac{1 + C (\TCrunch - t)}{\TCrunch - t}
		\leq \frac{C}{\TCrunch - t}.
		\label{E:SCALEFACTORRECIPROCALESTIMATE}
	\end{align}
	From Lemma~\ref{L:ANALYSISOFFRIEDMANN} and \eqref{E:SCALEFACTORUPPERBOUND}-\eqref{E:SCALEFACTORRECIPROCALESTIMATE},
	we deduce that
	\begin{align}
	\int_{s=t_1}^{t_2}
		\scale^p(s)
	\, ds	
	\leq
	C_p
	\int_{s=t_1}^{t_2}
		\frac{1}{(\TCrunch - s)^{|p|}}
	\, ds	
	& \leq
		\frac{C_p}{|1 + p|}
		(\TCrunch - t_2)^{p+1}
			\\
& \ \
	\leq
	\frac{C_p}{|1 + p|}
	\scale^{p + 1}(t_2),
	\notag
	\end{align}	
	which yields \eqref{E:SCALEFACTORTIMEINTEGRALS} in the case $p < - 1$.
	
	In the remaining cases $p=-1$ and $p > -1$, the estimate \eqref{E:SCALEFACTORTIMEINTEGRALS}  
	follows similarly with the help of \eqref{E:SCALEFACTORUPPERBOUND}-\eqref{E:SCALEFACTORRECIPROCALESTIMATE}, 
	and we omit the details.
	
	\eqref{E:EXPONENTIATEDSCALEFACTORTIMEINTEGRALS} then immediately follows from \eqref{E:SCALEFACTORTIMEINTEGRALS}
	and the fact that $|\ln \scale(t)| = \ln \frac{1}{\scale(t)}$ (since $0 < \scale(t) \leq 1$ for $t \in [0,\TCrunch)$).
	
\end{proof}

\begin{corollary}[\textbf{The monotonicity becomes strong near the singularities}]
\label{C:MONOTONICITYTIME}
	Consider the function 
	\[
	\mbox{\upshape sign}(t)
	=
	\begin{cases}
		1 & \mbox{if } t \geq 0, \\
		-1 & \mbox{if } t < 0.
	\end{cases}
	\]
	There exists a constant $\uptau$ with
	$0 < \uptau < \TCrunch$ such that
	\begin{subequations}
	\begin{align} \label{E:QUANTITATIVEMONOTONICITYATLATETIMES}
		- 1 
		& \leq \mbox{\upshape sign}(t) \scale'(t)
		\leq - \frac{6}{7},
		&& 
		t \in [\TBang, \TBang + \uptau] \cup [\TCrunch - \uptau, \TCrunch],
			\\
		- 1 
		& \leq \mbox{\upshape sign}(t) (\scale'(t))^3 \leq - \frac{6}{7},
		&& t \in [\TBang, \TBang + \uptau] \cup [\TCrunch - \uptau, \TCrunch].
		\label{E:CUBICQUANTITATIVEMONOTONICITYATLATETIMES}
	\end{align}
	\end{subequations}
	Moreover, there exists a constant $C > 1$ such that
	\begin{align} \label{E:SCALEFACTORBOUNDEDFROMBELOWFORSHORTIMES}
		\frac{1}{C} 
		& \leq \scale(t)
		\leq 1,
		&& t \in [\TBang + \uptau, \TCrunch - \uptau].
	\end{align}
\end{corollary}

\begin{proof}
	The existence of constants $\uptau$ and $C$ such that the estimates stated in 
	\eqref{E:QUANTITATIVEMONOTONICITYATLATETIMES}-\eqref{E:CUBICQUANTITATIVEMONOTONICITYATLATETIMES}
	and \eqref{E:SCALEFACTORBOUNDEDFROMBELOWFORSHORTIMES} hold
	is an easy consequence of Lemma~\ref{L:ANALYSISOFFRIEDMANN}, in particular
	the estimates
	\eqref{E:SCALEFACTORASYMPTOTICSNEARBANG}-\eqref{E:SCALEFACTORASYMPTOTICSNEARCRUNCH}
	and \eqref{E:TIMEDERIVATIVESCALEFACTORASYMPTOTICSNEARBANG}-\eqref{E:TIMEDERIVATIVESCALEFACTORASYMPTOTICSNEARCRUNCH}.
\end{proof}

\section{The Einstein-scalar field equations relative to CMC-transported spatial coordinates gauge}
\label{S:EQUATIONSINCMCTRANSPORTEDGAUGE}
In this section, we provide the 
Einstein-scalar field equations relative to CMC-transported spatial coordinates gauge.
We provide the equations in Prop.~\ref{P:EINSTEININCMC}. We first introduce
some basic geometric concepts.

\subsection{Basic geometric concepts}
\label{SS:BASICGEOMETRIC}
In CMC-transported spatial coordinates gauge,
we decompose the spacetime metric $\gfour$ and its inverse $\gfour^{-1}$ 
relative to the coordinates $(t,x^1,x^2,x^3)$
as follows:
\begin{align}
	\gfour
	& = 
		- n^2 dt \otimes dt 
		+ g_{ab} dx^a \otimes dx^b,
			\label{E:LITTLEGCMCTRANSPORTEDCOORDINTES} \\
	\gfour^{-1}
	& = 
		- n^{-2} \partial_t \otimes \partial_t
		+ g^{ab} \partial_a \otimes \partial_b.
		\label{E:LITTLEGINVERSECMCTRANSPORTEDCOORDINTES}
\end{align}
The scalar function $n > 0$ is known as the lapse.
Note that
\begin{align} \label{E:UNITNORMAL}
	\Nml
	& : = n^{-1} \partial_t
\end{align}
is the future-directed unit normal to $\Sigma_t$.

The second fundamental form $\SecondFund$ is the
$\Sigma_t$-tangent type $\binom{0}{2}$ tensorfield 
defined by requiring that following relation 
holds for all vectorfields $X,Y$ tangent to $\Sigma_t$:
\begin{align} \label{E:SECONDFUNDDEF}
	\gfour(\Dfour_X \Nml, Y) = - \SecondFund(X,Y),
\end{align}
where $\Dfour$ is the Levi--Civita connection of $\gfour$.
It is a standard fact that $\SecondFund$ is symmetric:
\begin{align}
	 \SecondFund(X,Y) =  \SecondFund(Y,X).
\end{align}

Let $\nabla$ denote the Levi--Civita connection of $g$.
The action of the spacetime connection $\Dfour$ can be decomposed into
the action of $\nabla$ and $\SecondFund$ as follows:
\begin{align} \label{E:DDECOMP}
	\Dfour_X Y = \nabla_X Y - \SecondFund(X,Y)\Nml.
\end{align}

When deriving equations and estimates,
we always write $\SecondFund$ in mixed form
(that is, in the form $\SecondFund_{\ j}^i$ with one index up); this leads to equations
with good structures and corresponding good estimates.
Throughout, 
\begin{align} \label{E:TRACESECONDFUND}
	\mbox{\upshape tr} \SecondFund
	& := \SecondFund_{\ a}^a
\end{align}
denotes the pure trace of the type $\binom{1}{1}$ tensorfield $\SecondFund$
and
\begin{align} \label{E:TRFREESECONDFUND}
	\hat{\SecondFund}_{\ j}^i
	& := \SecondFund_{\ j}^i
		- 
		\frac{1}{3} \SecondFund_{\ a}^a
		\ID_{\ j}^i
\end{align}
denotes its trace-free part, where $\ID$ is the identity transformation.

\subsection{The Einstein-scalar field equations relative to CMC-transported spatial coordinates}
We now provide the Einstein-scalar field equations relative to CMC-transported spatial coordinates gauge.

\begin{proposition} [\textbf{The Einstein-scalar field equations relative 
	to CMC-transported spatial coordinates}]
\label{P:EINSTEININCMC}
Let $\scale(t)$ denote the scale factor of the FLRW metric
and let
$\displaystyle \Hubble(t) = \frac{\scale'(t)}{\scale(t)}$
denote the FLRW Hubble factor, as defined in \eqref{E:HUBBLEFACTOR}.
In CMC-transported spatial coordinates normalized by 
\begin{align}\label{E:CMCNORMALIZATIONCHOICE}
	 \SecondFund_{\ a}^a(t,x)
	& \equiv - \Hubble(t),
\end{align}
which implies that
$\SecondFund_{\ j}^i 
	= \FreeSecondFund_{\ j}^i
		-
		\frac{1}{3} \Hubble \ID_{\ j}^i
$,
the Einstein-scalar field 
equations \eqref{E:EINSTEINSF}-\eqref{E:WAVEMODEL}
are equivalent to the following system of equations.

The \textbf{Hamiltonian and momentum constraint equations} 
are respectively:
\begin{subequations}
\begin{align}
		\ScalarCurArg{g} 
		- 
		\SecondFund_{\ b}^a \SecondFund_{\ a}^b 
		+ 
		\underbrace{\Hubble^2}_{(\SecondFund_{\ a}^a)^2} 
		& = \overbrace{(n^{-1} \partial_t \phi)^2 + \nabla^a \phi \nabla_a \phi}^{2 \Tfour(\Nml,\Nml)}, 
			\label{E:HAMILTONIAN} \\
		\nabla_a \FreeSecondFund_{\ i}^a 
		& = 
		\underbrace{- n^{-1} \partial_t \phi \nabla_i \phi}_{- \Tfour(\Nml,\partial_i)}, \label{E:MOMENTUM}
\end{align}
\end{subequations}
where $\ScalarCurArg{g}$ denotes the scalar curvature of $g$.

The \textbf{metric evolution equations} are:
\begin{subequations}
\begin{align}
	\partial_t g_{ij} 
	& = - 2 n g_{ia} \SecondFund_{\ j}^a, 
	\label{E:PARTIALTGCMC} \\
	\partial_t \SecondFund_{\ j}^i 
	& = - \nabla^i \nabla_j n
		+ n 
			\left\lbrace 
				\Ric[g]_{\ j}^i 
				+ \underbrace{\SecondFund_{\ a}^a}_{- \Hubble} \SecondFund_{\ j}^i 
				\underbrace{- \nabla^i \phi \nabla_j \phi}_{- T_{\ j}^i + (1/2)\ID_{\ j}^i (\gfour^{-1})^{\alpha \beta} \Tfour_{\alpha \beta}} 
			\right\rbrace,  \label{E:PARTIALTKCMC}
\end{align}
\end{subequations}
where $\Ric[g]$ denotes the Ricci curvature of $g$.

The \textbf{scalar field wave equation} is:
\begin{align} \label{E:SCALARFIELDCMC}
	\overbrace{- n^{-1} \partial_t(n^{-1} \partial_t \phi)}^{- \Dfour_{\Nml} \Dfour_{\Nml} \phi} 
		+ \nabla^a \nabla_a \phi 
		& = \overbrace{\Hubble n^{-1} \partial_t \phi}^{- \SecondFund_{\ a}^a \Dfour_{\Nml} \phi} 
		- n^{-1} \nabla^a n \nabla_a \phi.	
\end{align}

The \textbf{elliptic lapse equation} is
\begin{align} \label{E:ORIGINALLAPSE}
	\nabla^a \nabla_a n
	& = \Hubble'
			+ n
			\left\lbrace
				\overbrace{\hat{\SecondFund}_{\ b}^a \hat{\SecondFund}_{\ a}^b
				+
				\frac{1}{3} \Hubble^2}^{\SecondFund_{\ b}^a \SecondFund_{\ a}^b}
				+
				\overbrace{
				(n^{-1} \partial_t \phi)^2}^{\Tfour(\Nml,\Nml) + \frac{1}{2} (\gfour^{-1})^{\alpha \beta}\Tfour_{\alpha \beta}}
			\right\rbrace,
\end{align}
which by virtue of \eqref{E:HAMILTONIAN}
is equivalent to
\begin{align} \label{E:LAPSE}
	\nabla^a \nabla_a n
	& = n \ScalarCurArg{g}
		+ \Hubble'
		+ n \Hubble^2
		- n \nabla^a \phi \nabla_a \phi.
\end{align}

Moreover, the gauge condition \eqref{E:CMCNORMALIZATIONCHOICE} 
and the constraint equations \eqref{E:HAMILTONIAN}-\eqref{E:MOMENTUM}
are preserved by the flow of the remaining equations 
if they are verified by the data.

\end{proposition}

\begin{proof}
	The proposition is standard and can be obtained by making
	straightforward modifications to the analysis carried out in
	(for example) \cite[Section 6.2]{aS2010}.
\end{proof}

\section{The rescaled variables, differential operators, and the Einstein-scalar field equations for the rescaled variables}
\label{S:RESCALEDVARSANDEQNS}
In the rest of the article, we will find it convenient to work with solution variables
that are rescaled by various powers of the FLRW scale factor $\scale$. In this section, we define the rescaled
variables and some differential operators connected to them. 
In particular, we define $\Sigma_t$-projected Lie derivative operators,
which we will later use when differentiating the equations.
Finally, we provide the Einstein-scalar field equations
in CMC-transported spatial coordinates gauge in terms of the rescaled variables.

\subsection{The time-rescaled variables}
\label{SS:RESCALEDVARS}

\begin{definition}[\textbf{Time-rescaled variables}] 
		\label{D:RESCALEDVARIABLES}
		In terms of the solution variables
		$(g,\SecondFund,n,\partial_t \phi, \nabla \phi)$
		appearing in the equations of Prop.~\ref{P:EINSTEININCMC},
		we define the following time-rescaled variables:
		\begin{subequations}
		\begin{align}
			\newg
			& := \scale^{-2/3} g,
			\qquad 
			\newg^{-1} 
			= \scale^{2/3} g^{-1},
				\label{E:RESCALEDMETRIC} \\
		\FreeNewSec
		& := \scale \FreeSecondFund,
			\label{E:RESCALEDSECONDFUND} \\
		\newlapse
		& := \scale^{-4/3}(n-1),
			\label{E:RESCALEDLAPSE} \\
		\newtimescalar
		& := n^{-1} \scale \partial_t \phi - \sqrt{\frac{2}{3}},
			\qquad
		\newspacescalar
		:= \nabla \phi.
		\label{E:RESCALEDSCALARFIELD}
		\end{align}
		\end{subequations}
\end{definition}	
Note that for the FLRW solution,
we have $\newg = \StMet$ 
(where $\StMet$ is the round metric on $\mathbb{S}^3$ from Def.\ \ref{D:ROUNDMETRIC}),
$\newg^{-1} = \StMet^{-1}$, $\FreeNewSec = 0$, $\newlapse = 0$, and $\newspacescalar = 0$.

\begin{remark}[\textbf{We do not estimate} $\phi$ \textbf{itself}]
	\label{R:WEDONOTESTIMATEPHI}
	Note that we have not introduced a rescaled variable corresponding to $\phi$ itself.
	This is because we never need to estimate $\phi$ itself, the reason being that it does not
	appear in equations \eqref{E:EINSTEINSF}-\eqref{E:WAVEMODEL};
	only its first derivatives appear.
\end{remark}

\subsection{The $\Sigma_t$-projection tensorfield and $\Sigma_t$-projected Lie derivatives}
\label{SS:PROJECTIONTENSORFIELD}
Some of our constructions will
rely on the type $\binom{1}{1}$ tensorfield $\SigmatProject$
that $\gfour$-orthogonally projects onto $\Sigma_t$.
Relative to arbitrary coordinates, 
the components of $\SigmatProject$ are as follows:
\begin{align} \label{E:SIGMATPROJECT}
	\SigmatProject_{\ \beta}^{\alpha}
	& = 
		\Nml^{\alpha} \Nml_{\beta}	
		+
		\delta_{\ \beta}^{\alpha},
\end{align}
where $\Nml$ is the future-directed unit normal to $\Sigma_t$ (see \eqref{E:UNITNORMAL})
and $\delta_{\ \beta}^{\alpha}$ is the standard Kronecker delta.

If $\xi$
is a type $\binom{l}{m}$ spacetime tensorfield,
then we define $\SigmatProject \xi$ to be the type $\binom{l}{m}$
tensorfield obtained by projecting
all components of $\xi$ onto $\Sigma_t$ using $\SigmatProject$.
We say that $\xi$ is $\Sigma_t$-tangent if 
$\SigmatProject \xi = \xi$. In coordinates, 
this can be expressed as
$\SigmatProject_{\ \widetilde{\mu}_1}^{\mu_1} \cdots \SigmatProject_{\ \widetilde{\mu}_l}^{\mu_l}
	\SigmatProject_{\ \nu_1}^{\widetilde{\nu}_1} \cdots \SigmatProject_{\ \nu_m}^{\widetilde{\nu}_m} 
	\xi_{\widetilde{\nu}_1 \cdots \widetilde{\nu}_m}^{\widetilde{\mu}_1 \cdots \widetilde{\mu}_l}
	= \xi_{\nu_1 \cdots \nu_m}^{\mu_1 \cdots \mu_l}
$.
It is easy to see that relative to CMC-transported spatial coordinates,
$\xi$ is $\Sigma_t$-tangent if and only 
\begin{align} \label{E:XIISSIGMATTANGENT}
	\xi_{\nu_1 \cdots \nu_m}^{\mu_1 \cdots \mu_l} 
	& = 0,
	&&	\mbox{\upshape if any of $\mu_1,\cdots,\mu_l,\nu_1,\cdots,\nu_m$ are $0$}.
\end{align}
For this reason, when referring to a $\Sigma_t$-tangent tensor, we typically display only
its spatial indices. That is, we identify $\xi$ with
$\xi_{j_1 \cdots j_m}^{i_1 \cdots i_l}$.

If $\mathbf{V}$ is a spacetime vectorfield,
then we define the $\Sigma_t$-projected Lie derivative 
of $\xi$ as follows:
\begin{align} \label{E:SIGMATPROJECT}
	(\SigmatLie_{\mathbf{V}} \xi)_{\nu_1 \cdots \nu_m}^{\mu_1 \cdots \mu_l}
	& :=
	(\SigmatProject \Lie_{\mathbf{V}} \xi)_{\nu_1 \cdots \nu_m}^{\mu_1 \cdots \mu_l}
	=
	\SigmatProject_{\ \widetilde{\mu}_1}^{\mu_1} \cdots \SigmatProject_{\ \widetilde{\mu}_l}^{\mu_l}
	\SigmatProject_{\ \nu_1}^{\widetilde{\nu}_1} \cdots \SigmatProject_{\ \nu_m}^{\widetilde{\nu}_m}
	(\Lie_{\mathbf{V}} \xi)_{\widetilde{\nu}_1 \cdots \widetilde{\nu}_m}^{\widetilde{\mu}_1 \cdots \widetilde{\mu}_l}.
\end{align}

\subsection{Notation and differential operators involving the rescaled variables}
\label{SS:RECALEDVARSNOTATION}
Since $\scale$ is a function of $t$, it follows from 
definition \eqref{E:RESCALEDMETRIC} that the Levi--Civita connection of
$\newg$ is the same as that of $g$. Hence, we can denote the connection
by $\nabla$ without any danger of confusion.

We use the following notation to simply the presentation of some
of our formulas.
	
\begin{definition}[\textbf{Inner products with respect to} $\newg$]
\label{D:INNERPRODUCTS}
If 
$\xi$
and
$\eta$
are type $\binom{l}{m}$ $\Sigma_t$-tangent tensorfields,
then
\begin{align} \label{E:INNERPRODUCTS}
	\langle \xi,\eta \rangle_{\newg}
	& := 
		\newg_{a_1 a_1'} \cdots \newg_{a_l a_l'}
		(\newg^{-1})^{b_1 b_1'} \cdots (\newg^{-1})^{b_m b_m'}
		\xi_{b_1 \cdots b_m}^{a_1 \cdots a_l}
		\eta_{b_1' \cdots b_m'}^{a_1' \cdots a_l'}
\end{align}
denotes the inner product of $\xi$ and $\eta$ with respect to $\newg$.
\end{definition}

\begin{definition}[\textbf{Traces}]
If $\xi$ is a type $\binom{0}{2}$ $\Sigma_t$-tangent tensorfield, 
then
\begin{align} \label{E:GTRACE}
	\mytr \xi
	& := (\newg^{-1})^{ab} \xi_{ab}
\end{align}
denotes its trace with respect to $\newg$.

If $\xi$ is a type $\binom{1}{1}$ $\Sigma_t$-tangent tensorfield, 
then
\begin{align} \label{E:PURETRACE}
	\mbox{\upshape tr} \xi
	& := \xi_{\ a}^a
\end{align}
denotes its pure trace.
\end{definition}

\begin{definition}[\textbf{Musical notation}]
	\label{D:MUSICAL}
	If $\xi$ is a $\Sigma_t$-tangent one-form, then
	$\xi^{\#}$ denotes its vectorfield dual, obtained by
	raising the index with the inverse time-rescaled metric $\newg^{-1}$.
	If $\xi$ is a symmetric type $\binom{0}{2}$ $\Sigma_t$-tangent one-form, then
	$\xi^{\#}$ denotes the type $\binom{1}{1}$ tensorfield obtained by
	raising one index with the inverse time-rescaled metric $\newg^{-1}$. 
\end{definition}

\begin{definition}[\textbf{Spatial Laplacian and divergence operators}]
	\label{D:CONNECTIONDIFFOPS}
	If $\xi$ is a $\Sigma_t$-tangent tensorfield, 
	then we define its $\newg$-Laplacian as follows:
	\begin{align} \label{E:RESCALEDMETRICLAPLACIAN}
		\Delta_{\newg} \xi
		& := (\newg^{-1})^{ab} \nabla_a \nabla_b \xi.
	\end{align}	
	
	If $\xi$ is a type $\binom{1}{1}$ $\Sigma_t$-tangent tensorfield, 
	then $\Gdiv \xi$ is the $\Sigma_t$-tangent one-form 
	with the following components relative to arbitrary coordinates on $\Sigma_t$:
	\begin{align} \label{E:DIVERGENCEOFTYPE11}
			(\Gdiv \xi)_j
			& := \nabla_a \xi_{\ j}^a.
	\end{align}
	Similarly, if $\xi$ is a type $\binom{1}{1}$ $\Sigma_t$-tangent tensorfield, then
	$\Gdiv^{\#} \xi$ is the $\Sigma_t$-tangent vectorfield
	with the following components:
	\begin{align} \label{E:SHARPDIVERGENCEOFTYPE11}
			(\Gdiv^{\#} \xi)^i
			& := (\newg^{-1})^{ab}\nabla_a \xi_{\ b}^i.
	\end{align}
	
	If $\xi$ is a symmetric type $\binom{0}{2}$ $\Sigma_t$-tangent tensorfield, 
	then $\Gdiv \xi$ is the $\Sigma_t$-tangent one-form 
	with the following components relative to arbitrary coordinates on $\Sigma_t$:
	\begin{align} \label{E:DIVERGENCEOFTYPE02}
			(\Gdiv \xi)_j
			& := (\newg^{-1})^{ab} \nabla_a \xi_{b j}.
	\end{align}
	Similarly, if $\xi$ is a symmetric type $\binom{0}{2}$ $\Sigma_t$-tangent tensorfield, 
	$\Gdiv^{\#} \xi$ is the $\Sigma_t$-tangent vectorfield
	with the following components:
	\begin{align} \label{E:SHARPDIVERGENCEOFTYPE02}
			(\Gdiv^{\#} \xi)^i
			& := (\newg^{-1})^{ab} (\newg^{-1})^{ic} \nabla_a \xi_{bc}.
	\end{align}
\end{definition}

\begin{definition}[\textbf{Curvatures of } $\newg$]
	\label{D:CURVATURES}
	The Riemann curvature $\Riem_{i j k l}$ of $\newg$
	is determined by the following formula, 
	which holds for all $\Sigma_t$-tangent vectors $X$
	relative to arbitrary coordinates on $\Sigma_t$:
	\begin{align} \label{E:RIEMANNBIGG}
	\nabla_i \nabla_j X_k 
	- 
	\nabla_j \nabla_i X_k 	
	& = (\newg^{-1})^{lm} \Riem_{i j k l} X_m.
\end{align}

We define $\Ric_{ij} := (\newg^{-1})^{ab}\Riem_{i a j b}$
to be the Ricci curvature of $\newg$
\textbf{in type $\binom{0}{2}$ form}.
We define $\ScalarCur := (\newg^{-1})^{ab}\Ric_{ab}$
to be the scalar curvature of $\newg$.
\end{definition}

\begin{remark} 
	\label{R:CURVATURERESCALING}
	Recall that $\ScalarCurArg{g}$ denotes the scalar curvature of the \emph{non-rescaled} Riemannian metric $g$
	and that $\Ric[g]$ denotes the type $\binom{0}{2}$ Ricci curvature of $g$.
	It is straightforward to check that
	$\Ric_{ij} = (\Ric[g])_{ij}$,
	$(\Ric^{\# })_{\ j}^i = \scale^{2/3} g^{ia} (\Ric[g])_{aj}$,
	and
	$\ScalarCur = \scale^{2/3} \ScalarCurArg{g}$.
\end{remark}

In our analysis, we will use two kinds of pointwise norms:
one with respect to the metric $\newg$ and the other with
respect to the round metric $\StMet$ on $\Sigma_t$.
Here we provide the definition of the pointwise norm with respect to $\newg$.

\begin{definition}[\textbf{Pointwise norm relative to} $\newg$]
	\label{D:POINTWISENORM}
	Let $\newg$ denote the rescaled spatial metric from Def.\ \ref{D:RESCALEDVARIABLES}.
	If $\xi$ 
	is a type $\binom{l}{m}$ $\Sigma_t$-tangent tensor,
	then we define the norm 
	$|\xi|_{\newg} \geq 0$ by
	\begin{align} 
		|\xi|_{\newg}^2
		& := 
		\newg_{a_1 a_1'} \cdots \newg_{a_l a_l'}
		(\newg^{-1})^{b_1 b_1'} \cdots (\newg^{-1})^{b_m b_m'}
		\xi_{b_1 \cdots b_m}^{a_1 \cdots a_l}
		\xi_{b_1' \cdots b_m'}^{a_1' \cdots a_l'}.
		\label{E:POINTWISENORM}
	\end{align}
\end{definition}

\begin{remark}
	The norm $|\xi|_{\StMet}$ is defined by replacing the metric
	$\newg$ on RHS~\eqref{E:POINTWISENORM} with the round metric $\StMet$ on $\Sigma_t$,
	which we formally construct in Subsect.\ \ref{SS:EXTENDINGSIGMA0TANGENTTENSORFIELDSTOSIGMATTANGENTTENSORFIELDS}.
\end{remark}

Both of the following elliptic operators will play a role in our analysis of 
the rescaled lapse $\newlapse$.
	
\begin{definition}[\textbf{The elliptic operators $\mathscr{L}$ and $\widetilde{\mathscr{L}}$}]
	\label{D:ELLIPTICOPS}
	We define the operator $\mathscr{L}$ by
	\begin{subequations}
	\begin{align}
		\mathscr{L}
		& := 
			\scale^{8/3} \GLap
			- \scale^{4/3} f,
				\label{E:HIGHORDERELLIPTICOPERATOR} \\
		f 
		& := (\scale')^2
				+ 
				\frac{2}{3} \scale^{4/3}
				+ |\FreeNewSec|_{\newg}^2
				+ 2 \sqrt{\frac{2}{3}} \newtimescalar
				+ \newtimescalar^2,
			\label{E:ERRORTERMHIGHORDERELLIPTICOPERATOR}
	\end{align}
	\end{subequations}
	and the operator $\widetilde{\mathscr{L}}$ by
	\begin{subequations}
	\begin{align}
		\widetilde{\mathscr{L}}
		& := 
			\scale^{4/3} \GLap 
			- \widetilde{f},
				\label{E:LOWORDERELLIPTICOPERATOR} \\
		\widetilde{f} 
		& := 
			(\scale')^2
			+ 
			\frac{2}{3} \scale^{4/3}
			+ 
			\scale^{4/3} 
			\left\lbrace
				\ScalarCur
				- 
				\frac{2}{3}
			\right\rbrace
			-
			\scale^{4/3} |\newspacescalar|_{\newg}^2.
			\label{E:ERRORTERMLOWORDERELLIPTICOPERATOR}
	\end{align}
	\end{subequations}
\end{definition}

\subsection{The Einstein-scalar field equations in terms of the time-rescaled variables}
\label{SS:EQNSFORRESCALEDVARS}
In the next proposition, we derive the Einstein-scalar field equations
relative to CMC-transported spatial coordinates gauge for the time-rescaled variables.

\begin{proposition}[\textbf{The Einstein-scalar field equations for the time-rescaled variables}]
	\label{P:RESCALEDVARIABLES}
	Let $g_{ij}$, $\SecondFund_{\ j}^i$, $n$, $\partial_t \phi$, $\nabla \phi$ be
	a solution to the equations of Prop.~\ref{P:EINSTEININCMC}.
	Then the corresponding time-rescaled variables of 
	Def.\ \ref{D:RESCALEDVARIABLES} verify the following equations.
	
	The \textbf{rescaled constraint equations} are:
	\begin{subequations}
	\begin{align}
		\scale^{4/3} 
		\left\lbrace
			\ScalarCur
			- 
			\frac{2}{3}
		\right\rbrace
		- |\FreeNewSec|_{\newg}^2 
		& = 
			2 \sqrt{\frac{2}{3}} \newtimescalar
			+
			\newtimescalar^2 
			+ 
			\scale^{4/3} |\newspacescalar|_{\newg}^2, 
			\label{E:RENORMALIZEDHAMILTONIAN} \\
		\Gdiv \FreeNewSec 
		& = 
			- \sqrt{\frac{2}{3}} \newspacescalar
			- \newtimescalar \newspacescalar,
			\label{E:RENORMALIZEDMOMENTUM}
				\\
		\Gdiv^{\#} \FreeNewSec 
		& = 
			- 
			\sqrt{\frac{2}{3}} \newspacescalar^{\#}
			- 
			\newtimescalar \newspacescalar^{\#}.
			\label{E:ALTERNATERENORMALIZEDMOMENTUM}
\end{align}
\end{subequations}

	The \textbf{rescaled metric evolution equations} are:
	\begin{subequations}
	\begin{align}
		\SigmatLie_{\partial_t} \newg
		& = 
			- 2 \scale^{-1} \newg \cdot \FreeNewSec
			- 2 \scale^{1/3} \newlapse \newg \cdot \FreeNewSec
			+
			\frac{2}{3} \scale^{1/3} \scale' \newlapse \newg,
			\label{E:EVOLUTIONMETRICRENORMALIZED} \\
	\SigmatLie_{\partial_t} \newg^{-1}
		& = 2 \scale^{-1} \newg^{-1} \cdot \FreeNewSec
		+
		2 \scale^{1/3} \newlapse \newg^{-1} \cdot \FreeNewSec
		- 
		\frac{2}{3} \scale^{1/3} \scale' \newlapse \newg^{-1},
			\label{E:EVOLUTIONINVERSEMETRICRENORMALIZED} \\
	\SigmatLie_{\partial_t} \FreeNewSec
	& = 
		- 
		\scale^{5/3} \nabla^{\#} \nabla \newlapse
		+
		\scale^{1/3}
		(1 + \scale^{4/3} \newlapse) 
		\left\lbrace
			\Ric^{\# } - \frac{2}{9} \ID
		\right\rbrace
		\label{E:EVOLUTIONSECONDFUNDRENORMALIZED} \\
	& \ \
		+
		\frac{1}{3} (\scale')^2 \scale^{1/3} \newlapse \ID
		+ 
		\frac{2}{9} \scale^{5/3} \newlapse \ID
		- \scale^{1/3} \scale' \newlapse \FreeNewSec 
		- \scale^{1/3} \newspacescalar^{\#} \otimes \newspacescalar
		- \scale^{5/3} \newlapse \newspacescalar^{\#} \otimes \newspacescalar,
		\notag 
\end{align}
\end{subequations}
where $\ID$ is the identity transformation.

The \textbf{rescaled scalar field evolution equations} are:
\begin{subequations}
\begin{align}  \label{E:WAVEEQUATIONRENORMALIZED}
		\partial_t \newtimescalar
		  & =
			(1 + \scale^{4/3} \newlapse) \scale^{1/3} \Gdiv \newspacescalar
			-
			\sqrt{\frac{2}{3}} \scale' \scale^{1/3} \newlapse
			- \scale^{1/3} \scale' \newlapse \newtimescalar
			- \scale^{5/3} \nabla^{\#} \newlapse \cdot \newspacescalar,
			\\
		\SigmatLie_{\partial_t} \newspacescalar
		& = \scale^{-1} (1 + \scale^{4/3} \newlapse) \nabla \newtimescalar
			+ \sqrt{\frac{2}{3}} \scale^{1/3} \nabla \newlapse
			+ \scale^{1/3} \newtimescalar \nabla \newlapse.
			\label{E:EVOLUTIONSPACESCALARRENORMALIZED}
\end{align}
\end{subequations}
With $\mathscr{L}$ as in \eqref{E:HIGHORDERELLIPTICOPERATOR},
the \textbf{rescaled elliptic lapse equation} is:
\begin{align} \label{E:LAPSEPDERENORMALIZEDHIGHERDERIVATIVES}
	\mathscr{L} \newlapse
	& = 2 \sqrt{\frac{2}{3}} \newtimescalar
			+
			|\FreeNewSec|_{\newg}^2
			+
			\newtimescalar^2.
\end{align}

In addition, 
with $\widetilde{\mathscr{L}}$ as in \eqref{E:LOWORDERELLIPTICOPERATOR},
$\newlapse$ verifies the elliptic equation
\begin{align}	\label{E:LAPSEPDERENORMALIZEDLOWERDERIVATIVES}
	\widetilde{\mathscr{L}} \newlapse
	& = \left\lbrace
				\ScalarCur - \frac{2}{3}
			\right\rbrace
			-
			|\newspacescalar|_{\newg}^2.
\end{align}

\end{proposition}

\begin{remark}
	Note that we have rewritten the scalar field wave equation
	as the first-order system
	\eqref{E:WAVEEQUATIONRENORMALIZED}-\eqref{E:EVOLUTIONSPACESCALARRENORMALIZED}.
	This formulation is convenient for our subsequent analysis.
\end{remark}

\begin{remark}
Equation \eqref{E:LAPSEPDERENORMALIZEDHIGHERDERIVATIVES}
is a consequence of the lapse equation \eqref{E:ORIGINALLAPSE}
while 
equation \eqref{E:LAPSEPDERENORMALIZEDLOWERDERIVATIVES}
is a consequence of the lapse equation \eqref{E:LAPSE}.
Both \eqref{E:LAPSEPDERENORMALIZEDHIGHERDERIVATIVES} and 
\eqref{E:LAPSEPDERENORMALIZEDLOWERDERIVATIVES} are important for our analysis.
\end{remark}

\begin{proof}[Proof of Prop.~\ref{P:RESCALEDVARIABLES}]
	The equations are straightforward consequences of
	Prop.~\ref{P:EINSTEININCMC}
	and Def.\ \ref{D:RESCALEDVARIABLES},
	once one takes into account
	Lemma~\ref{L:FRIEDMANN}
	and
	Remark~\ref{R:CURVATURERESCALING}.
	As an example, we derive equation \eqref{E:EVOLUTIONSECONDFUNDRENORMALIZED}
	in detail. The remaining equations can be derived using similar but simpler
	arguments and we omit those details.
	To proceed, we first note the identity
	$\SecondFund_{\ j}^i 
	= \FreeSecondFund_{\ j}^i
		-
		\frac{1}{3} \Hubble \ID_{\ j}^i
	$.
	Multiplying 
	\eqref{E:PARTIALTKCMC} 
	by $\scale$, appealing to definition \eqref{E:HUBBLEFACTOR}, 
	and using this identity,
	we obtain
	\begin{align} \label{E:FIRSTSTEPEVOLUTIONSECONDFUNDRENORMALIZED}
	\partial_t (\scale \FreeSecondFund_{\ j}^i)
	& = - \scale g^{ia} \nabla_a \nabla_j n
		+ n 
			\left\lbrace 
				\scale g^{ia} \Ric[g]_{a j} 
				- 
				\scale' \FreeSecondFund_{\ j}^i
				+
				\frac{1}{3} \scale^{-1} (\scale')^2 \ID_{\ j}^i 
				- 
				\scale g^{ia} \nabla_a \phi \nabla_j \phi
			\right\rbrace
				\\
	& \ \
		+
		\scale' \FreeSecondFund_{\ j}^i
		+
		\frac{1}{3} \scale'' \ID_{\ j}^i
		-
		\frac{1}{3} \scale^{-1} (\scale')^2 \ID_{\ j}^i.
		\notag
\end{align}
From \eqref{E:FIRSTSTEPEVOLUTIONSECONDFUNDRENORMALIZED},
\eqref{E:FRIEDMANNSECONDORDER},
Def.\ \ref{D:RESCALEDVARIABLES},
and Remark~\ref{R:CURVATURERESCALING},
we deduce
	\begin{align} \label{E:SECONDSTEPEVOLUTIONSECONDFUNDRENORMALIZED}
	\partial_t \FreeNewSec_{\ j}^i
	& = - \scale^{5/3} (\newg^{-1})^{ia} \nabla_a \nabla_j \newlapse
		\\
	& \ \
			+ 
			(1 + \scale^{4/3} \newlapse) 
			\left\lbrace 
				\scale^{1/3} (\Ric^{\# })_{\ j}^i
				- 
				\scale^{-1} \scale' \FreeNewSec_{\ j}^i
				+
				\frac{1}{3} \scale^{-1} (\scale')^2 \ID_{\ j}^i
				- 
				\scale^{1/3} (\newg^{-1})^{ia} \newspacescalar_a \newspacescalar_j 
			\right\rbrace
			\notag	\\
	& \ \
		+
		\scale^{-1} \scale' \FreeNewSec_{\ j}^i
		-
		\frac{2}{9} \scale^{1/3} \ID_{\ j}^i
		-
		\frac{1}{3} \scale^{-1} (\scale')^2 \ID_{\ j}^i.
		\notag
\end{align}
From \eqref{E:SECONDSTEPEVOLUTIONSECONDFUNDRENORMALIZED}
and simple computations, we arrive at the desired equation
\eqref{E:EVOLUTIONSECONDFUNDRENORMALIZED}.
	
\end{proof}

\section{The Lie-transported frame and related constructions}
\label{S:LIETRANSPORTEDFRAME}
In our analysis, we do not derive estimates for the components of tensorfields relative to the 
transported spatial coordinate frame. Instead, we derive estimates for the components of tensorfields
relative to a frame constructed by transporting the frame
$\mathscr{Z}$ from Sect.\ \ref{S:GEOMETRYBACKGROUND} 
along the flow lines of the unit normal $\Nml$
in a manner such that the frame remains $\Sigma_t$-tangent.
In this section, we construct the transported frame and the corresponding co-frame.
We also construct a transported version of the round metric $\StMet$.

\begin{remark}
	\label{R:S3ISSIGMA0}
	Throughout this section, we identify
	the manifold $\mathbb{S}^3$ from Sect.\ \ref{S:GEOMETRYBACKGROUND} with
	the initial data hypersurface $\Sigma_0$.
\end{remark}

\subsection{Lie-transported tensorfields}
\label{SS:EXTENDINGSIGMA0TANGENTTENSORFIELDSTOSIGMATTANGENTTENSORFIELDS}
Let $\xi$ be a $\Sigma_0$-tangent tensorfield defined only along $\Sigma_0$
and let $\SigmatProject$ be the $\Sigma_t$-projection tensorfield from Subsect.\ \ref{SS:PROJECTIONTENSORFIELD}.
We can extend $\xi$ to a $\Sigma_t$-tangent tensorfield defined on spacetime
by solving the following equations (where the initial conditions $\xi|_{\Sigma_0}$ are known):
\begin{align} 
	\SigmatLie_{\Nml} \xi
	& = 0,
	\label{E:TRANSPORTXI}
		\\
	\SigmatProject \xi
	& = \xi.	
\label{E:XIREMAINSSIGAMTTANGENT}
\end{align}

It is easy to see 
that relative to CMC-transported coordinates
$(t,x^1,x^2,x^3)$ (with $x^0 = t$),
\eqref{E:TRANSPORTXI}-\eqref{E:XIREMAINSSIGAMTTANGENT}
are equivalent to
\begin{align} 
	\partial_t \xi_{j_1 \cdots j_m}^{i_1 \cdots i_l} 
	& = 0,
	&& \label{E:COMPONENTSOFXIDONOTCHANGE}
		\\
		\xi_{\nu_1 \cdots \nu_m}^{\mu_1 \cdots \mu_l} 
	& = 0,
	&&	\mbox{\upshape if any of $\mu_1,\cdots,\mu_l,\nu_1,\cdots,\nu_m$ are $0$}.
\label{E:XIREMAINSSIGMATTANGENT}
\end{align}

\begin{definition}[\textbf{Lie-transported spatial frame, co-frame, and round metric}]
\label{D:LIETRANSPORTEDSPATIALFRAME}
We define
\begin{align} \label{E:AGAINROUNDMETRICKILLING}
	\mathscr{Z} := \lbrace Z_{(1)}, Z_{(2)}, Z_{(3)} \rbrace
\end{align}
to be the vectorfields
obtained by extending (in the above fashion)
the $\Sigma_0$-tangent
vectorfields defined in \eqref{E:ROUNDMETRICKILLING}
(see Remark~\ref{R:S3ISSIGMA0}).

Similarly, 
\begin{align}
	\Theta := \lbrace \theta^{(1)}, \theta^{(2)}, \theta^{(3)} \rbrace
\end{align}
denotes the one-forms obtained by extending the one-forms
from \eqref{E:DUALTOROUNDMETRICKILLING}.

Finally, $\StMet$ denotes the metric obtained by
extending the round metric $\StMet$ on $\Sigma_0$
from Def.\ \ref{D:ROUNDMETRIC}.
\end{definition}

\begin{remark}
Technically, it is an abuse of notation to use the symbol
$\mathscr{Z}$ to denote both the set \eqref{E:ZSET}
and the set
\eqref{E:AGAINROUNDMETRICKILLING}. 
However, this should not cause any confusion.
Similar remarks apply to the set $\Theta$ and the metric $\StMet$.
\end{remark}

From \eqref{E:XIREMAINSSIGMATTANGENT}-\eqref{E:COMPONENTSOFXIDONOTCHANGE},
it is easy to see that the basic properties verified by
$\lbrace Z_{(A)} \rbrace_{A=1,2,3}$,
$\lbrace \theta^{(A)} \rbrace_{A=1,2,3}$,
and
$\StMet$ along $\Sigma_0$
in fact hold in all of spacetime.
For example, the identity
$\SigmatLie_{Z_{(A)}} \StMet = 0$
holds in all of spacetime, not just along
$\Sigma_0$. 
Similarly, the Lie bracket relation
\eqref{E:FRAMEVECTORFIELDLIEBRACKET}
holds in all of spacetime.
We will use these basic facts in the rest of the paper
without further comment.

\subsection{Metrics and connection coefficients relative to the Lie-transported spatial frame}
\label{SS:METRICSANDCONNECTIONRELATIVETOLIETRANSPORTEDFRAME}
Relative to the time coordinate $t$ and the spatial frame of Def.\ \ref{D:LIETRANSPORTEDSPATIALFRAME},
we can decompose the spacetime metric 
and its inverse 
(see \eqref{E:LITTLEGCMCTRANSPORTEDCOORDINTES}-\eqref{E:LITTLEGINVERSECMCTRANSPORTEDCOORDINTES})
as follows:
\begin{align}
	\gfour
	& = 
		- n^2 dt \otimes dt 
		+ g_{AB} \theta^{(A)} \otimes \theta^{(B)},
			\label{E:LITTLEGCMCFRAME} \\
	\gfour^{-1}
	& = 
		- n^{-2} \partial_t \otimes \partial_t
		+ g^{AB} Z_{(A)} \otimes Z_{(B)},
		\label{E:LITTLEGINVERSECMCFRAME}
\end{align}
where $g_{AB} = g(Z_{(A)},Z_{(B)})$
and $g^{AB} = g^{-1}(\theta^{(A)},\theta^{(B)})$.
Note that from \eqref{E:DUALTOROUNDMETRICKILLING}, 
it follows that $g^{AB}$, viewed as a $3 \times 3$ matrix, is the inverse of the $3 \times 3$ matrix $g_{AB}$.

Similarly, we can decompose the time-rescaled spatial metric $\newg$ and its inverse $\newg^{-1}$ (see Def.\ \ref{D:RESCALEDVARIABLES})
as follows:
\begin{align}
	\newg 
	& = \newg_{AB} \theta^{(A)} \otimes \theta^{(B)},
		\label{E:RESCALEDSPATIALMETRICCOFRAMEDECOMPOSITION} \\
	\newg^{-1}
	& = (\newg^{-1})^{AB} Z_{(A)} \otimes Z_{(B)},
	\label{E:RESCALEDINVERSESPATIALMETRICFRAMEDECOMPOSITION}
\end{align}
where  
$\newg_{AB} = \newg(Z_{(A)},Z_{(B)})$
and $(\newg^{-1})^{AB} = \newg^{-1}(\theta^{(A)},\theta^{(B)})$.
Note that $(\newg^{-1})^{AB}$, viewed as a $3 \times 3$ matrix, is the inverse of the $3 \times 3$ matrix $\newg_{AB}$.

More generally, if
$\xi_{j_1 \cdots j_m}^{i_1 \cdots i_l}$
is a type $\binom{l}{m}$ $\Sigma_t$-tangent tensorfield,
$\theta^1, \cdots, \theta^l \in \Theta$,
and $Z_1, \cdots, Z_m \in \mathscr{Z}$
(see Def.\ \ref{D:LIETRANSPORTEDSPATIALFRAME}),
then we use the notation
\begin{align} \label{E:FRAMECONTRACTIONNOTATION}
\xi_{Z_1 \cdots Z_m}^{\theta^1 \cdots \theta^l}
	:= 
	\xi_{b_1 \cdots b_m}^{a_1 \cdots a_l}
	\theta_{a_1}^1 \cdots \theta_{a_l}^l
	Z_1^{b_1} \cdots Z_m^{b_m}
\end{align}	
to denote the contraction of $\xi$ against
	$\theta^1, \cdots, \theta^l$ 
	and
	$Z_1, \cdots, Z_m$.

In our analysis, we will encounter the connection coefficients of the Lie-transported
frame relative to the metric $\newg$. We now define them.

\begin{definition} [\textbf{Connection coefficients of the Lie-transported frame}]
	\label{D:CONNECTIONCOEFFICIENTS}
	Let $\nabla$ denote the Levi--Civita connection of $\newg$
	and let $\lbrace Z_{(A)} \rbrace_{A=1,2,3}$
	denote the $\Sigma_t$-tangent frame from Def.\ \ref{D:LIETRANSPORTEDSPATIALFRAME}.
	We define the connection coefficients
	$\Gamma_{A \ B}^{\ C}$ of the frame by demanding that
	the following identity holds:
	\begin{align} \label{E:FRAMECOVARIANTDERIVATIVES}
		\nabla_{Z_{(A)}} Z_{(B)} 
		& = \Gamma_{A \ B}^{\ C} Z_{(C)}.
	\end{align}
\end{definition}

In the next lemma, we compute the
$\Gamma_{A \ B}^{\ C}$ in terms of the 
derivatives of the frame components of $\newg$
with respect to the frame vectorfields.

\begin{remark}
	In the remainder of the article, if $X$ is a vectorfield and $f$
	is a scalar function, then $X f := X^{\alpha} \partial_{\alpha} f$ 
	denotes the derivative of $f$ in the direction of $X$. 
	In particular, relative to arbitrary local coordinates on $\Sigma_t$,
	we have 
	$Z_{(A)} \newg_{BC} := Z_{(A)}^a \partial_a (\newg_{bc}Z_{(B)}^b Z_{(C)}^c)$.
\end{remark}

\begin{lemma}[\textbf{Connection coefficients of the frame}]
	\label{L:CONNECTIONCOEFFICIENTS}
	We have the identity
	\begin{align} \label{E:CONNECTIONCOEFFICIENTSFIRSTFORMULA}
		\Gamma_{A \ B}^{\ D} = (\newg^{-1})^{CD} \Gamma_{A C B},
	\end{align}
	where
	\begin{align} \label{E:FRAMECONNECTIONEXPRESSION}
		\Gamma_{A C B}
		& = 
			\frac{1}{2}
			\left\lbrace
				Z_{(A)} \newg_{BC}
				+ Z_{(B)} \newg_{AC}
				- Z_{(C)} \newg_{AB}
			\right\rbrace
			+ 
			\frac{1}{3}
			\left\lbrace
				 \epsilon_{ABD} \newg_{CD}
				- \epsilon_{BCD} \newg_{AD}
				- \epsilon_{ACD} \newg_{BD}
			\right\rbrace
				\\
		& = \frac{1}{3} \epsilon_{ABC}
				+  \Gamma_{A C B}^{\triangle},
				\label{E:ALTERNATEFRAMECONNECTIONEXPRESSION}
				\\
		\Gamma_{A C B}^{\triangle} 
			& = 
			\frac{1}{2}
			\left\lbrace
					\SigmatLie_{Z_{(A)}} [\newg - \StMet]_{BC}
				+ \SigmatLie_{Z_{(B)}} [\newg - \StMet]_{AC}
				- \SigmatLie_{Z_{(C)}} [\newg - \StMet]_{AB}
			\right\rbrace
				\label{E:ERRORTERMALTERNATEFRAMECONNECTIONEXPRESSION} 
				\\
		& \ \
			+ 
			\frac{1}{3}
			\left\lbrace
					\epsilon_{ACD} [\newg - \StMet]_{BD}
					+ 
					\epsilon_{BCD} [\newg - \StMet]_{AD}
					+
					\epsilon_{ABD} [\newg - \StMet]_{CD}
			\right\rbrace,
				\notag
		\end{align}
		$\StMet$ is the round metric of Def.\ \ref{D:LIETRANSPORTEDSPATIALFRAME},
		and $\epsilon_{ABC}$ is the fully antisymmetric symbol normalized by $\epsilon_{123} = 1$.
\end{lemma}

\begin{proof}
	We first prove \eqref{E:FRAMECONNECTIONEXPRESSION}.
	From \eqref{E:FRAMECOVARIANTDERIVATIVES}, we deduce that
	$\newg(\nabla_{Z_{(A)}}Z_{(B)},Z_{(C)})
	= \Gamma_{A C B} 
	$,
	where
	$\Gamma_{A C B} := \newg_{CD} \Gamma_{A \ B}^{\ D}$.
	Using also $\nabla'$s torsion-free property
	$
	\SigmatLie_{Z_{(A)}} Z_{(B)}
	=
	[Z_{(A)},Z_{(B)}] 
	= 
	\nabla_{Z_{(A)}} Z_{(B)}
	-
	\nabla_{Z_{(B)}} Z_{(A)}
	$,
	we obtain
	\begin{align} \label{E:GAMMAALLLOWERFIRSTCOMPUTATION}
		\Gamma_{A C B}
		& = \frac{1}{2} \newg(\nabla_{Z_{(A)}}Z_{(B)} + \nabla_{Z_{(B)}}Z_{(A)},Z_{(C)})
			+
			\frac{1}{2} \newg([Z_{(A)},Z_{(B)}], Z_{(C)}).
	\end{align}
	Next, using the Leibniz rule for $\nabla$ 
	and the fact that $\nabla \newg = 0$,
	we deduce
	$Z_{(A)} \newg_{BC}
	= 
	\nabla_{Z_{(A)}} 
	\left\lbrace
		\newg(Z_{(B)},Z_{(C)}) 
	\right\rbrace
	= 
	\newg(\nabla_{Z_{(A)}} Z_{(B)},Z_{(C)})
	+
	\newg(Z_{(B)},\nabla_{Z_{(A)}} Z_{(C)})
	$.
	It follows that
	\begin{align} \label{E:METRICFRAMECOMPONENTSDERIVATIVEIDENTITY}
	Z_{(A)} \newg_{BC}
	+ Z_{(B)} \newg_{AC}
	- Z_{(C)} \newg_{AB}
	& = 
	\newg(\nabla_{Z_{(A)}} Z_{(B)},Z_{(C)})
	+
	\newg(\nabla_{Z_{(B)}} Z_{(A)},Z_{(C)})
		\\
	& \ \
	+
	\newg(\nabla_{Z_{(A)}} Z_{(C)} - \nabla_{Z_{(C)}} Z_{(A)},Z_{(B)})
	+
	\newg(\nabla_{Z_{(B)}} Z_{(C)} - \nabla_{Z_{(C)}} Z_{(B)},Z_{(A)}).
	\notag
	\end{align}
	From \eqref{E:METRICFRAMECOMPONENTSDERIVATIVEIDENTITY}, 
	\eqref{E:GAMMAALLLOWERFIRSTCOMPUTATION},
	and the torsion-free property of $\nabla$,
	we deduce
	\begin{align} \label{E:GAMMAALLLOWERSECONDCOMPUTATION}
		\Gamma_{A C B}
		& = \frac{1}{2} 
				\left\lbrace
					Z_{(A)} \newg_{BC}
					+ Z_{(B)} \newg_{AC}
					- Z_{(C)} \newg_{AB}
				\right\rbrace
					\\
		& \ \
			+
			\frac{1}{2} \newg([Z_{(A)},Z_{(B)}], Z_{(C)})
			-
			\frac{1}{2} \newg([Z_{(A)},Z_{(C)}],Z_{(B)})
			-
			\frac{1}{2} \newg([Z_{(B)},Z_{(C)}],Z_{(A)}).
			\notag
	\end{align}
	Using the identity \eqref{E:FRAMEVECTORFIELDLIEBRACKET}
	to evaluate the vectorfield commutators on the second line of RHS~\eqref{E:GAMMAALLLOWERSECONDCOMPUTATION},
	we arrive at the desired identity \eqref{E:FRAMECONNECTIONEXPRESSION} (the first line only).
	
	Our proof of \eqref{E:ALTERNATEFRAMECONNECTIONEXPRESSION}-\eqref{E:ERRORTERMALTERNATEFRAMECONNECTIONEXPRESSION}
	relies on the identity
	\begin{align} \label{E:ZAGBCID}
	Z_{(A)} \newg_{BC}
	& = \SigmatLie_{Z_{(A)}}\left\lbrace [\newg - \StMet](Z_{(B)},Z_{(C)}) \right\rbrace 
		\\
	& = [\SigmatLie_{Z_{(A)}} (\newg - \StMet)](Z_{(B)},Z_{(C)})
		+ \frac{2}{3} \epsilon_{ABD} [\newg - \StMet](Z_{(D)},Z_{(C)})
		+ \frac{2}{3} \epsilon_{ACD} [\newg - \StMet](Z_{(B)},Z_{(D)}),
		\notag
	\end{align}
	which follows from 
	the identities $\StMet(Z_{(B)},Z_{(C)}) = \delta_{BC}$,
	$\SigmatLie_{Z_{(A)}} \StMet = 0$,
	the Leibniz rule for Lie derivatives, and \eqref{E:FRAMEVECTORFIELDLIEBRACKET}.
	Specifically, we use \eqref{E:ZAGBCID} to substitute for the first three terms in braces on RHS \eqref{E:FRAMECONNECTIONEXPRESSION}.
	We then expand $\newg$ as $\newg = \StMet + (\newg - \StMet)$ in all terms.
	The desired identities 
	\eqref{E:ALTERNATEFRAMECONNECTIONEXPRESSION}-\eqref{E:ERRORTERMALTERNATEFRAMECONNECTIONEXPRESSION}
	then follow from straightforward calculations.
	
\end{proof}

\section{First variation formulas and commutation identities}
\label{S:FIRSTVARANDCOMMUTATION}
In this section, we provide some standard variation and commutation identities
that we later use when differentiating the equations.

\subsection{First variation formulas}
\label{SS:FIRSTVARIATION}

\begin{lemma}[\textbf{First variation formulas}]
	Let $\Ric$ and $\ScalarCur$ be the curvatures of $\newg$ from Def.\ \ref{D:CURVATURES}.
	Relative to arbitrary local coordinates on $\Sigma_t$,
	the following identities hold for
	$\mathbf{V} \in \lbrace \partial_t \rbrace \cup \mathscr{Z}$,
	where $\mathscr{Z}$ is the Lie-transported spatial 
	frame from Def.\ \ref{D:LIETRANSPORTEDSPATIALFRAME}:
	\begin{subequations}
	\begin{align}
		\SigmatLie_{\mathbf{V}} \Ric_{ij}
		& =
			- \frac{1}{2} \GLap \SigmatLie_{\mathbf{V}} \newg_{ij}
			- \frac{1}{2} (\newg^{-1})^{ab} \nabla_i \nabla_j \SigmatLie_{\mathbf{V}} \newg_{ab}
			+ \frac{1}{2} (\newg^{-1})^{ab} \nabla_i \nabla_a \SigmatLie_{\mathbf{V}} \newg_{bj}
			+ \frac{1}{2} (\newg^{-1})^{ab} \nabla_j \nabla_a \SigmatLie_{\mathbf{V}} \newg_{ib}	
			\label{E:RICCIFIRSTVARIATION} 	\\
		& \ \
			- \Ric_{ij} (\newg^{-1})^{ab} \SigmatLie_{\mathbf{V}} \newg_{ab}
			- \newg_{ij} (\newg^{-1})^{ab} (\newg^{-1})^{cd} \Ric_{ac} \SigmatLie_{\mathbf{V}} \newg_{bd} 
				\notag \\
		& \ \
			+ \frac{3}{2} (\newg^{-1})^{ab} \Ric_{ia} \SigmatLie_{\mathbf{V}} \newg_{bj}
			+ \frac{3}{2} (\newg^{-1})^{ab} \Ric_{aj} \SigmatLie_{\mathbf{V}} \newg_{ib}
			\notag
			\\
		& \ \
			+ \frac{1}{2} \newg_{ij} \ScalarCur (\newg^{-1})^{ab} \SigmatLie_{\mathbf{V}} \newg_{ab}
			- \frac{1}{2} \ScalarCur \SigmatLie_{\mathbf{V}} \newg_{ij},
			\notag \\
		\SigmatLie_{\mathbf{V}} (\Ric^{\#})_{\ j}^i
			& =
			- \frac{1}{2} (\newg^{-1})^{ia} \GLap \SigmatLie_{\mathbf{V}} \newg_{aj}
			- \frac{1}{2} (\newg^{-1})^{ab} (\newg^{-1})^{ic} \nabla_c \nabla_j \SigmatLie_{\mathbf{V}} \newg_{ab}
				\label{E:RICSHARPFIRSTVARIATION}	\\
		& \ \
			+ \frac{1}{2} (\newg^{-1})^{ic} (\newg^{-1})^{ab} \nabla_c \nabla_a \SigmatLie_{\mathbf{V}} \newg_{bj}
			+ \frac{1}{2} (\newg^{-1})^{ic} (\newg^{-1})^{ab} \nabla_j \nabla_a \SigmatLie_{\mathbf{V}} \newg_{cb}
			\notag \\
		& \ \
			- (\newg^{-1})^{ic}  (\newg^{-1})^{ab} \Ric_{cj} \SigmatLie_{\mathbf{V}} \newg_{ab}
			- \ID_{\ j}^i (\newg^{-1})^{ab} (\newg^{-1})^{cd} \Ric_{ac} \SigmatLie_{\mathbf{V}} \newg_{bd} 
				\notag \\
	& \ \
			+ \frac{3}{2} (\newg^{-1})^{ic} (\newg^{-1})^{ab} \Ric_{ca} \SigmatLie_{\mathbf{V}} \newg_{bj}
			+ \frac{1}{2} (\newg^{-1})^{ic} (\newg^{-1})^{ab} \Ric_{aj} \SigmatLie_{\mathbf{V}} \newg_{cb}
			\notag
			\\
		& \ \
			+ \frac{1}{2} \ID_{\ j}^i \ScalarCur (\newg^{-1})^{ab} \SigmatLie_{\mathbf{V}} \newg_{ab}
			- \frac{1}{2} (\newg^{-1})^{ia} \ScalarCur \SigmatLie_{\mathbf{V}} \newg_{aj},
			\notag 
	\end{align}
	\end{subequations}
	where $\ID$ denotes the identity transformation.
\end{lemma}

\begin{proof}
	The following identity is standard
	(see, for example, \cite{bCpLnL2006}*{Ch.~2}, and note the different curvature index conventions used there):
	\begin{align}
		\SigmatLie_{\mathbf{V}} \Ric_{ij}
		& = 
			- \frac{1}{2} \GLap \SigmatLie_{\mathbf{V}} \newg_{ij}
			- \frac{1}{2} (\newg^{-1})^{ab} \nabla_i \nabla_j \SigmatLie_{\mathbf{V}} \newg_{ab}
			+ \frac{1}{2} (\newg^{-1})^{ab} \nabla_i \nabla_a \SigmatLie_{\mathbf{V}} \newg_{bj}
			+ \frac{1}{2} (\newg^{-1})^{ab} \nabla_j \nabla_a \SigmatLie_{\mathbf{V}} \newg_{ib}	
				\label{E:FIRSTRICCIFIRSTVARIATION} 
					\\
		& \ \
			- (\newg^{-1})^{ab} (\newg^{-1})^{cd} \Riem_{aicj} \SigmatLie_{\mathbf{V}} \newg_{bd}
			+ \frac{1}{2} (\newg^{-1})^{ab} \Ric_{ia} \SigmatLie_{\mathbf{V}} \newg_{bj}
			+ \frac{1}{2} (\newg^{-1})^{ab} \Ric_{aj} \SigmatLie_{\mathbf{V}} \newg_{ib},
			\notag
	\end{align}
	where $\Riem$ and $\Ric$ are as in Def.\ \ref{D:CURVATURES}.
	Next, we note the following identity, valid in three spatial dimensions
	(see, for example, \cite{bCpLnL2006}*{Ch.~1}, 
	and note the different curvature index conventions used there):
	\begin{align} \label{E:RIEMANNINTERMSOFRICCI}
		\Riem_{ijkl}
		& = \frac{\ScalarCur}{2}(\newg_{il} \newg_{jk} - \newg_{ik} \newg_{jl})
			+ \Ric_{ik} \newg_{jl}
			+ \Ric_{jl} \newg_{ik}
			- \Ric_{il} \newg_{jk}
			- \Ric_{jk} \newg_{il}.
	\end{align}
	Using \eqref{E:RIEMANNINTERMSOFRICCI} to substitute for the term
	$\Riem_{aicj}$ in \eqref{E:FIRSTRICCIFIRSTVARIATION},
	we conclude \eqref{E:RICCIFIRSTVARIATION}.
	
	\eqref{E:RICSHARPFIRSTVARIATION} then follows from
	\eqref{E:RICCIFIRSTVARIATION}, the Leibniz rule for Lie derivatives,
	and the standard identity 
	$\SigmatLie_{\mathbf{V}} (\newg^{-1})^{ij} = - (\newg^{-1})^{ia} (\newg^{-1})^{jb} \SigmatLie_{\mathbf{V}} \newg_{ab}$.
\end{proof}

\subsection{Commutation identities}
\label{SS:COMMUTATIONIDENTITIES}

\begin{lemma}[\textbf{Commutation identities}]
	\label{L:COMMUTATIONIDENTITIES}
	Let $\vec{I}$ be a $\mathscr{Z}$-multi-index (see Subsubsect.\ \ref{SSS:COORDINATES}).
	The following (schematically depicted) commutation identities 
	(involving the differential operators of Def.\ \ref{D:CONNECTIONDIFFOPS})
	hold for scalar functions $f$ and type $\binom{l}{m}$ $\Sigma_t-$tangent tensorfields $\xi$ (with $l + m \geq 1$),
	where on the RHSs, we have omitted all tensorial contractions and numerical coefficients
	(which are not important for our analysis)
	in order to condense the presentation:
	\begin{subequations}
	\begin{align}
		[\nabla, \SigmatLie_{\partial_t}] f 
		= [\nabla, \SigmatLie_{\mathscr{Z}}^{\vec{I}}] f
		& = 0,
			\label{E:FUNCTIONNABLALIEZCOMMUTATOR} \\
		[\nabla^2, \SigmatLie_{\mathscr{Z}}^{\vec{I}}] f
		& = \sum_{L=1}^{|\vec{I}|}
				\mathop{\sum_{\vec{I}_1 + \cdots + \vec{I}_{L+1} = \vec{I}}}_{|\vec{I}_a| \geq 1 \mbox{\upshape \ for } 1 \leq a \leq L} 
				(\newg^{-1})^L
				\underbrace{
				(\SigmatLie_{\mathscr{Z}}^{\vec{I}_1} \newg)
				\cdots
				(\SigmatLie_{\mathscr{Z}}^{\vec{I}_{L-1}} \newg)
				}_{\mbox{absent if $L=1$}}
				(\nabla \SigmatLie_{\mathscr{Z}}^{\vec{I}_L} \newg)
				(\nabla \SigmatLie_{\mathscr{Z}}^{\vec{I}_{L+1}} f),
					\label{E:NABLASQUAREDLIEZCOMMUTATOR} \\
		[\GLap, \SigmatLie_{\mathscr{Z}}^{\vec{I}}] f
		& = \sum_{L=1}^{|\vec{I}|}
				\sum_{i_1 + i_2 = 1}
				\mathop{\sum_{\vec{I}_1 + \cdots + \vec{I}_{L+1} = \vec{I}}}_{|\vec{I}_a| \geq 1 \mbox{\upshape \ for } 1 \leq a \leq L} 
				(\newg^{-1})^{L+1}
				\underbrace{
				(\SigmatLie_{\mathscr{Z}}^{\vec{I}_1} \newg)
				\cdots
				(\SigmatLie_{\mathscr{Z}}^{\vec{I}_{L-1}} \newg)
				}_{\mbox{absent if $i_1=L=1$}}
				(\nabla^{i_1} \SigmatLie_{\mathscr{Z}}^{\vec{I}_L} \newg)
				(\nabla^{i_2 + 1} \SigmatLie_{\mathscr{Z}}^{\vec{I}_{L+1}} f),
				\label{E:LAPLACIANLIEZCOMMUTATOR}
	\end{align}
	\end{subequations}
	
	\begin{subequations}
	\begin{align}
	[\nabla, \SigmatLie_{\mathscr{Z}}^{\vec{I}}] \xi
		& = \sum_{L=1}^{|\vec{I}|}
				\mathop{\sum_{\vec{I}_1 + \cdots + \vec{I}_{L+1} = \vec{I}}}_{|\vec{I}_a| \geq 1 \mbox{\upshape \ for } 1 \leq a \leq L} 
				(\newg^{-1})^L
				\underbrace{
				(\SigmatLie_{\mathscr{Z}}^{\vec{I}_1} \newg)
				\cdots
				(\SigmatLie_{\mathscr{Z}}^{\vec{I}_{L-1}} \newg)
				}_{\mbox{absent if $L=1$}}
				(\nabla \SigmatLie_{\mathscr{Z}}^{\vec{I}_L} \newg)
				(\SigmatLie_{\mathscr{Z}}^{\vec{I}_{L+1}} \xi),
					\label{E:TENSORNABLALIEZCOMMUTATOR} 
					\\
		[\nabla^2, \SigmatLie_{\mathscr{Z}}^{\vec{I}}] \xi
		& = \sum_{L=1}^{|\vec{I}|}
				\mathop{\sum_{i_1 + i_2 + i_3 = 2}}_{i_1, i_3 \leq 1}
				\mathop{\sum_{\vec{I}_1 + \cdots + \vec{I}_{L+1} = \vec{I}}}_{|\vec{I}_a| \geq 1 \mbox{\upshape \ for } 1 \leq a \leq L} 
				(\newg^{-1})^L
				\underbrace{
				(\SigmatLie_{\mathscr{Z}}^{\vec{I}_1} \newg)
				\cdots
				(\nabla^{i_1}\SigmatLie_{\mathscr{Z}}^{\vec{I}_{L-1}} \newg)
				}_{\mbox{absent if $L=1$}}
				(\nabla^{i_2} \SigmatLie_{\mathscr{Z}}^{\vec{I}_L} \newg)
				(\nabla^{i_3} \SigmatLie_{\mathscr{Z}}^{\vec{I}_{L+1}} \xi),
					\label{E:TENSORNABLALSQUAREIEZCOMMUTATOR} 
					\\
		[\GLap, \SigmatLie_{\mathscr{Z}}^{\vec{I}}] \xi
		& = \sum_{L=1}^{|\vec{I}|}
				\mathop{\sum_{i_1 + i_2 + i_3 = 2}}_{i_1 \leq 1}
				\mathop{\sum_{\vec{I}_1 + \cdots + \vec{I}_{L+1} = \vec{I}}}_{|\vec{I}_a| \geq 1 \mbox{\upshape \ for } 1 \leq a \leq L} 
				(\newg^{-1})^{L+1}
				\underbrace{
				(\SigmatLie_{\mathscr{Z}}^{\vec{I}_1} \newg)
				\cdots
				(\nabla^{i_1} \SigmatLie_{\mathscr{Z}}^{\vec{I}_{L-1}} \newg)
				}_{\mbox{absent if $L=1$}}
				(\nabla^{i_2} \SigmatLie_{\mathscr{Z}}^{\vec{I}_L} \newg)
				(\nabla^{i_3} \SigmatLie_{\mathscr{Z}}^{\vec{I}_{L+1}} \xi).
					\label{E:TENSORNABLALSQUAREIEZCOMMUTATOR} 
	\end{align}
	\end{subequations}
	
	Moreover, if $\xi$ is a type $\binom{1}{1}$ $\Sigma_t-$tangent tensorfield, then
	\begin{subequations}
	\begin{align}
		[\Gdiv,\SigmatLie_{\mathscr{Z}}^{\vec{I}}] \xi
		& =	\sum_{L=1}^{|\vec{I}|}
				\mathop{\sum_{\vec{I}_1 + \cdots + \vec{I}_{L+1} = \vec{I}}}_{|\vec{I}_a| \geq 1 \mbox{\upshape \ for } 1 \leq a \leq L} 
				(\newg^{-1})^L
				\underbrace{
				(\SigmatLie_{\mathscr{Z}}^{\vec{I}_1} \newg)
				\cdots
				(\SigmatLie_{\mathscr{Z}}^{\vec{I}_{L-1}} \newg)
				}_{\mbox{absent if $L=1$}}
				(\nabla \SigmatLie_{\mathscr{Z}}^{\vec{I}_L} \newg)
				(\SigmatLie_{\mathscr{Z}}^{\vec{I}_{L+1}} \xi),
					\label{E:GIDIVLIEZCOMMUTATOR} \\
		[\Gdiv^{\#},\SigmatLie_{\mathscr{Z}}^{\vec{I}}] \xi
		& = 
			\sum_{L=1}^{|\vec{I}|}
				\sum_{i_1 + i_2 = 1}
				\mathop{\sum_{\vec{I}_1 + \cdots + \vec{I}_{L+1} = \vec{I}}}_{|\vec{I}_a| \geq 1 \mbox{\upshape \ for } 1 \leq a \leq L} 
				(\newg^{-1})^{L+1}
				\underbrace{
				(\SigmatLie_{\mathscr{Z}}^{\vec{I}_1} \newg)
				\cdots
				(\SigmatLie_{\mathscr{Z}}^{\vec{I}_{L-1}} \newg)
				}_{\mbox{absent if $L=1$}}
				(\nabla^{i_1} \SigmatLie_{\mathscr{Z}}^{\vec{I}_L} \newg)
				(\nabla^{i_2} \SigmatLie_{\mathscr{Z}}^{\vec{I}_{L+1}} \xi).
			\label{E:GIDIVSHARPLIEZCOMMUTATOR}
	\end{align}
	\end{subequations}
	
	Finally, if $\xi$ is a symmetric type $\binom{0}{2}$ tensorfield,
	then the following commutation identity holds:
		\begin{align} \label{E:NABLALIEDETAILEDCOMMUTATOR}
		[\nabla, \SigmatLie_{\partial_t}] \xi
		& = \newg^{-1} (\nabla \SigmatLie_{\partial_t} \newg) \xi.
	\end{align}
	
\end{lemma}

\begin{proof}[Discussion of proof]
	The identities stated in the lemma are standard
	results from differential geometry that can be proved 
	by using arguments similar to the ones given in
	\cite{jS2016b}*{Ch.~8}; we omit the lengthy but
	standard calculations.
\end{proof}

\begin{lemma}[\textbf{Identity for} $\SigmatLie_{\mathscr{Z}}^{\vec{I}} (\Ric^{\#})$ \textbf{and} 
	$\SigmatLie_{\mathscr{Z}}^{\vec{I}} \ScalarCur$]
	\label{L:RICCILIEDERIVATIVESCHEMATIC}
	Let $\Ric$ and $\ScalarCur$ be the curvatures of $\newg$ from Def.\ \ref{D:CURVATURES}
	and let $\vec{I}$ be a $\mathscr{Z}$-multi-index (see Subsubsect.\ \ref{SSS:COORDINATES}) with 
	$1 \leq |\vec{I}|$.
	The following identities hold:
	\begin{align} \label{E:RICCILIEDERIVATIVESCHEMATIC}
		\SigmatLie_{\mathscr{Z}}^{\vec{I}} (\Ric^{\#})
		& = 
			- \frac{1}{2} (\GLap \SigmatLie_{\mathscr{Z}}^{\vec{I}} \newg)^{\#}
			- \frac{1}{2} \nabla^{\#} \nabla \mytr \SigmatLie_{\mathscr{Z}}^{\vec{I}} \newg
			+ \frac{1}{2} \nabla^{\#} \Gdiv \SigmatLie_{\mathscr{Z}}^{\vec{I}} \newg
			+ \frac{1}{2} \nabla (\Gdiv \SigmatLie_{\mathscr{Z}}^{\vec{I}} \newg)^{\#}
				\\
		& \ \
			+ \RicErrorInhom{\vec{I}},
			\notag
			\\
	\SigmatLie_{\mathscr{Z}}^{\vec{I}} \ScalarCur
		& = 
			- \mbox{\upshape tr} (\GLap \SigmatLie_{\mathscr{Z}}^{\vec{I}} \newg)^{\#}
			+ \mbox{\upshape tr} \nabla^{\#} \Gdiv \SigmatLie_{\mathscr{Z}}^{\vec{I}} \newg
			+ \mbox{\upshape tr} \RicErrorInhom{\vec{I}},
				\label{E:SCALARCURLIEDERIVATIVESCHEMATIC}
\end{align}
where $\RicErrorInhom{\vec{I}}$ is a type $\binom{1}{1}$ $\Sigma_t-$tangent tensorfield,
$\mbox{\upshape tr}$ denotes a pure trace (not involving a metric),
and $\RicErrorInhom{\vec{I}}$
has the following schematic form 
(which is accurate up to constant coefficients and the specification of tensorial contractions):
\begin{align}
	\RicErrorInhom{\vec{I}}
	& = \sum_{L=1}^{|\vec{I}|-1}
				\mathop{\sum_{i_1 + i_2 + i_3 = 2}}_{i_1, i_3 \leq 1}
				\mathop{\sum_{\vec{I}_1 + \cdots + \vec{I}_{L+1} = \vec{I}}}_{|\vec{I}_a| \geq 1 \mbox{\upshape \ for } 1 \leq a \leq L+1} 
				(\newg^{-1})^{L+2}
				\underbrace{
				(\SigmatLie_{\mathscr{Z}}^{\vec{I}_1} \newg)
				\cdots
				(\nabla^{i_1}\SigmatLie_{\mathscr{Z}}^{\vec{I}_{L-1}} \newg)
				}_{\mbox{absent if $L=1$}}
				(\nabla^{i_2} \SigmatLie_{\mathscr{Z}}^{\vec{I}_L} \newg)
				(\nabla^{i_3} \SigmatLie_{\mathscr{Z}}^{\vec{I}_{L+1}} \newg)
			\label{E:RICINHOMERROR} \\
	& \ \ 
			+
			\sum_{L=1}^{|\vec{I}|}
			\mathop{\sum_{\vec{I}_1 + \cdots + \vec{I}_L = \vec{I}}}_{|\vec{I}_a| \geq 1 \mbox{\upshape \ for } 1 \leq a \leq L} 
			\sum_{p = 0}^1
			(\newg^{-1})^L
			(\ID)^p
			\Ric^{\#} 
			(\Lie_{\mathscr{Z}}^{\vec{I}_1} \newg)
			\cdots
			(\Lie_{\mathscr{Z}}^{\vec{I}_L} \newg).
				\notag
	\end{align}
	In the formula \eqref{E:RICINHOMERROR},
	the first sum on the RHS is absent if $|\vec{I}| = 1$
	and $\ID$ denotes the identity transformation.
\end{lemma}

\begin{proof}
	\eqref{E:RICCILIEDERIVATIVESCHEMATIC} follows from a straightforward argument involving 
	induction in $|\vec{I}|$, the first variation formula \eqref{E:RICSHARPFIRSTVARIATION},
	the commutation identity \eqref{E:TENSORNABLALSQUAREIEZCOMMUTATOR} with $\xi := \Lie_Z \newg$,
	and the schematic identity $\SigmatLie_Z \newg^{-1} = \newg^{-2} \SigmatLie_Z \newg$;
	we omit the details.
	\eqref{E:SCALARCURLIEDERIVATIVESCHEMATIC}
	then follows from taking the trace of \eqref{E:RICCILIEDERIVATIVESCHEMATIC}.
\end{proof}

\section{The commuted equations}
\label{S:COMMUTEDEQUATIONS}
In this section, we commute the equations of Prop.~\ref{P:RESCALEDVARIABLES}
with the operators
$\SigmatLie_{\mathscr{Z}}^{\vec{I}}$
and sort the error terms into two classes:
``borderline terms,'' whose behavior as $t \downarrow 0$
is borderline with respect to the energy estimates we will prove,
and ``junk terms,'' whose behavior as $t \downarrow 0$
is sub-critical with respect to the energy estimates we will prove.

\begin{notation}[\textbf{$*$ notation}]
	\label{N:STARNOTATION}
	In a few of the formulas below, 
	we use the schematic notation $\xi*\eta$
	to denote a tensor formed out of some
	constant-coefficient linear combination of natural contractions of
	the tensor $\xi$ against the tensor $\eta$.
	For such terms, neither the precise nature of the contractions
	nor the value of the numerical constants will be important for our
	subsequent analysis.
\end{notation}

\begin{proposition}[\textbf{The $\Lie_{\mathscr{Z}}^{\vec{I}}$-commuted equations}]
	\label{P:ICOMMUTEDEQNS}
For each $\mathscr{Z}$-multi-index $\vec{I}$ with $1 \leq |\vec{I}|$,
solutions to the equations of Prop.~\ref{P:RESCALEDVARIABLES}
verify the following equations
(see Subsubsect.\ \ref{SSS:COORDINATES} regarding the vectorfield multi-index notation).

The \textbf{commuted rescaled momentum constraint equations} are:
	\begin{subequations}
	\begin{align}
		\Gdiv \SigmatLie_{\mathscr{Z}}^{\vec{I}} \FreeNewSec 
		& = - \sqrt{\frac{2}{3}} \SigmatLie_{\mathscr{Z}}^{\vec{I}} \newspacescalar
				+ \CommutedMomBorderInhomDown{\vec{I}}
			\label{E:COMMUTEDRENORMALIZEDMOMENTUM},
			\\
		\Gdiv^{\#} \SigmatLie_{\mathscr{Z}}^{\vec{I}} \FreeNewSec 
		& = - \sqrt{\frac{2}{3}} (\SigmatLie_{\mathscr{Z}}^{\vec{I}} \newspacescalar)^{\#}
			+ \CommutedMomBorderInhomUp{\vec{I}}
			\label{E:ALTERNATECOMMUTEDRENORMALIZEDMOMENTUM},	
\end{align}
\end{subequations}
where
\begin{subequations}
\begin{align}
	\CommutedMomBorderInhomDown{\vec{I}}
	& := - [\SigmatLie_{\mathscr{Z}}^{\vec{I}},\Gdiv] \FreeNewSec
			- \sum_{\vec{I}_1 + \vec{I}_2 = \vec{I}}
				(\mathscr{Z}^{\vec{I}_1} \newtimescalar) 
				\SigmatLie_{\mathscr{Z}}^{\vec{I}_2} \newspacescalar,
			\label{E:BORDERLINECOMMUTEDRENORMALIZEDMOMENTUM}	\\
	\CommutedMomBorderInhomUp{\vec{I}}
	& := - [\SigmatLie_{\mathscr{Z}}^{\vec{I}},\Gdiv^{\#}] \FreeNewSec
			- \sqrt{\frac{2}{3}} 
				\mathop{\sum_{\vec{I}_1 + \vec{I}_2 = \vec{I}}}_{|\vec{I}_2| \leq |\vec{I}|-1}
				(\SigmatLie_{\mathscr{Z}}^{\vec{I}_1} \newg^{-1})
				\SigmatLie_{\mathscr{Z}}^{\vec{I}_2} \newspacescalar
				- 
				\sum_{\vec{I}_1 + \vec{I}_2 + \vec{I}_3 = \vec{I}}
				(\mathscr{Z}^{\vec{I}_1} \newtimescalar) 
				(\SigmatLie_{\mathscr{Z}}^{\vec{I}_2} \newg^{-1})
				\cdot
				\SigmatLie_{\mathscr{Z}}^{\vec{I}_3} \newspacescalar.
				\label{E:BORDERLINEALTERNATECOMMUTEDRENORMALIZEDMOMENTUM}
\end{align}
\end{subequations}
	
	The \textbf{commuted rescaled metric evolution equations} are:
	\begin{subequations}
	\begin{align}
		\SigmatLie_{\partial_t}  \SigmatLie_{\mathscr{Z}}^{\vec{I}} \newg
		& =
			\scale^{-1} \CommutedMetBorderInhom{\vec{I}}
			+ 
			\scale^{1/3} \CommutedMetJunkInhom{\vec{I}},
			\label{E:COMMUTEDEVOLUTIONMETRICRENORMALIZED} \\
	\SigmatLie_{\partial_t} \SigmatLie_{\mathscr{Z}}^{\vec{I}} \newg^{-1}
		& =
			\scale^{-1} \CommutedInvMetBorderInhom{\vec{I}}
			+
			\scale^{1/3} \CommutedInvMetJunkInhom{\vec{I}},
			\label{E:COMMUTEDEVOLUTIONINVERSEMETRICRENORMALIZED} 
	\end{align}
	\end{subequations}
	
	\begin{subequations}
	\begin{align}
	\SigmatLie_{\partial_t} \nabla \SigmatLie_{\mathscr{Z}}^{\vec{I}} \newg
		& =
			- 2 \scale^{-1}(1 + \scale^{4/3} \newlapse) \newg \cdot \nabla \SigmatLie_{\mathscr{Z}}^{\vec{I}} \FreeNewSec
			\label{E:GRADCOMMUTEDEVOLUTIONMETRICRENORMALIZED} \\
		& \ \
			+ \frac{2}{3} \scale' \scale^{1/3} (\nabla \mathscr{Z}^{\vec{I}} \newlapse) \otimes \newg
			+ \scale^{-1} \CommutedGradMetBorderInhom{\vec{I}}
			+ \scale^{1/3} \CommutedGradMetJunkInhom{\vec{I}},
			\notag
				\\
	\SigmatLie_{\partial_t} \SigmatLie_{\mathscr{Z}}^{\vec{I}} \FreeNewSec
	& = 
		\frac{1}{2} 
		\scale^{1/3}(1 + \scale^{4/3} \newlapse)
		\left\lbrace
			- \GLap (\SigmatLie_{\mathscr{Z}}^{\vec{I}} \newg)^{\#}
			-  (\nabla^2 \mytr \SigmatLie_{\mathscr{Z}}^{\vec{I}} \newg)^{\#}
			+ \nabla^{\#} \Gdiv \SigmatLie_{\mathscr{Z}}^{\vec{I}} \newg
			+ \nabla (\Gdiv \SigmatLie_{\mathscr{Z}}^{\vec{I}} \newg)^{\#}
		\right\rbrace
			\label{E:COMMUTEDEVOLUTIONSECONDFUNDRENORMALIZED} \\
	& \ \
		- \scale^{5/3} \nabla^{\#} \nabla \mathscr{Z}^{\vec{I}} \newlapse
		+ \frac{1}{3} (\scale')^2 \scale^{1/3} (\mathscr{Z}^{\vec{I}} \newlapse) \ID
		+ \scale^{1/3} \RicErrorInhom{\vec{I}}
		\notag \\
	& \ \
		+ \scale^{1/3} \CommutedSecFunBorderInhom{\vec{I}}
		+ \scale^{1/3} \CommutedSecFunJunkInhom{\vec{I}},
		\notag
\end{align}
\end{subequations} 
where $\ID$ denotes the identity transformation, 
$\RicErrorInhom{\vec{I}}$ is defined by \eqref{E:RICINHOMERROR},
and
\begin{subequations}
\begin{align}
	\CommutedMetBorderInhom{\vec{I}}
	& =	- 2 \sum_{\vec{I}_1 + \vec{I}_2 = \vec{I}}
					(\SigmatLie_{\mathscr{Z}}^{\vec{I}_1} \newg) 
					\cdot 
					\SigmatLie_{\mathscr{Z}}^{\vec{I}_2} \FreeNewSec
			- 2 \scale^{4/3}
					(\mathscr{Z}^{\vec{I}} \newlapse) 
					\newg \cdot \FreeNewSec
			\label{E:ICOMMUTEDMETRICBORDERTERMS} \\
	& \ \
				+ 	
					\frac{2}{3} 
					\scale'
					\scale^{4/3}
					(\mathscr{Z}^{\vec{I}} \newlapse) 
					\newg,
				\notag \\
	\CommutedInvMetBorderInhom{\vec{I}}
	& =	 2 \sum_{\vec{I}_1 + \vec{I}_2 = \vec{I}}
					(\SigmatLie_{\mathscr{Z}}^{\vec{I}_1} \newg^{-1}) 
					\cdot 
					\SigmatLie_{\mathscr{Z}}^{\vec{I}_2} \FreeNewSec
			+ 
		 		2 \scale^{4/3}
		 			(\mathscr{Z}^{\vec{I}} \newlapse) 
					\newg^{-1} 
					\cdot \FreeNewSec
				\label{E:ICOMMUTEDIVERSEMETRICBORDERTERMS} \\
	& \ \
		 - 	\frac{2}{3} 
					\scale'
					\scale^{4/3}
					(\mathscr{Z}^{\vec{I}} \newlapse) 
					\newg^{-1},
					\notag 
			\\
	\CommutedGradMetBorderInhom{\vec{I}}
	& = - 2 \scale^{4/3} (\nabla \mathscr{Z}^{\vec{I}} \newlapse) \otimes (\newg \cdot \FreeNewSec)
			- 2  \mathop{\mathop{\sum_{\vec{I}_1 + \vec{I}_2 = \vec{I}}}_{i_1 + i_2 = 1}}_{|\vec{I}_2| + i_2 \leq |\vec{I}|}
					(\nabla^{i_1} \SigmatLie_{\mathscr{Z}}^{\vec{I}_1} \newg) * \nabla^{i_2} 
					\SigmatLie_{\mathscr{Z}}^{\vec{I}_2} \FreeNewSec
			\label{E:ICOMMUTEDGRADMETRICBORDERTERMS} \\
	& \ \
			+ 
			\newg^{-1} * \scale \nabla \SigmatLie_{\partial_t} \newg * \SigmatLie_{\mathscr{Z}}^{\vec{I}} \newg,
			\notag \\
	\CommutedGradMetJunkInhom{\vec{I}}
	& = \frac{2}{3} \scale' 
				\mathop{\mathop{\sum_{\vec{I}_1 + \vec{I}_2 = \vec{I}}}_{i_1 + i_2 = 1}}_{i_1 + |\vec{I_1}| \leq |\vec{I}|} 
				(\nabla^{i_1} \mathscr{Z}^{\vec{I}_1} \newlapse) 
				\otimes
				\nabla^{i_2} \SigmatLie_{\mathscr{Z}}^{\vec{I}_2} \newg
					\label{E:ICOMMUTEDGRADMETRICJUNKTERMS}		\\
	& \  \
				- 2 \mathop{\mathop{\sum_{\vec{I}_1 + \vec{I}_2 = \vec{I}}}_{i_1 + i_2 = 1}}_{i_1 + |\vec{I_1}|, i_3 + |\vec{I_3}| \leq |\vec{I}|} 
							(\nabla^{i_1} \mathscr{Z}^{\vec{I}_1} \newlapse) 
							\otimes
							(\nabla^{i_2} \SigmatLie_{\mathscr{Z}}^{\vec{I}_2} \newg)
							*
							\nabla^{i_3} \SigmatLie_{\mathscr{Z}}^{\vec{I}_3} \FreeNewSec,
							\notag	
\end{align}
\end{subequations}

\begin{subequations}
\begin{align}
	\CommutedMetJunkInhom{\vec{I}}
	& = -2  \mathop{\sum_{\vec{I}_1 + \vec{I}_2 + \vec{I}_3 = \vec{I}}}_{|\vec{I}_1| \leq |\vec{I}|-1}
					(\mathscr{Z}^{\vec{I}_1} \newlapse) 
					(\SigmatLie_{\mathscr{Z}}^{\vec{I}_2} \newg) 
					\cdot \SigmatLie_{\mathscr{Z}}^{\vec{I}_3} \FreeNewSec
			+ 	\frac{2}{3} 
					\scale'
					\mathop{\sum_{\vec{I}_1 + \vec{I}_2 = \vec{I}}}_{|\vec{I}_1| \leq |\vec{I}|-1}
					(\mathscr{Z}^{\vec{I}_1} \newlapse) 
					\SigmatLie_{\mathscr{Z}}^{\vec{I}_2} \newg,
			\label{E:ICOMMUTEDMETRICJUNKTERMS} \\
	\CommutedInvMetJunkInhom{\vec{I}}
	& = 2  \mathop{\sum_{\vec{I}_1 + \vec{I}_2 + \vec{I}_3 = \vec{I}}}_{|\vec{I}_1| \leq |\vec{I}|-1}
					(\mathscr{Z}^{\vec{I}_1} \newlapse) 
					(\SigmatLie_{\mathscr{Z}}^{\vec{I}_2} \newg^{-1}) 
					\cdot \SigmatLie_{\mathscr{Z}}^{\vec{I}_3} \FreeNewSec
			- 	\frac{2}{3} 
					\scale'
					\mathop{\sum_{\vec{I}_1 + \vec{I}_2 = \vec{I}}}_{|\vec{I}_1| \leq |\vec{I}|-1}
					(\mathscr{Z}^{\vec{I}_1} \newlapse) 
					\SigmatLie_{\mathscr{Z}}^{\vec{I}_2} \newg^{-1},
			\label{E:ICOMMUTEDINVERSEMETRICJUNKTERMS} \\
	\CommutedSecFunBorderInhom{\vec{I}}
	& = - \scale' 
				(\mathscr{Z}^{\vec{I}} \newlapse) \FreeNewSec, 
		 \label{E:ICOMMUTEDSECONDFUNDBORDERTERMS} \\
	\CommutedSecFunJunkInhom{\vec{I}}
	& = 
		- \scale' 
				\mathop{\sum_{\vec{I}_1 + \vec{I}_2 = \vec{I}}}_{|\vec{I}_1| \leq |\vec{I}|-1}
				(\mathscr{Z}^{\vec{I}_1} \newlapse) \SigmatLie_{\mathscr{Z}}^{\vec{I}_2} \FreeNewSec
		- \scale^{4/3} \newg^{-1} [\nabla^2, \SigmatLie_{\mathscr{Z}}^{\vec{I}}] \newlapse
		- \scale^{4/3} [\newg^{-1}, \SigmatLie_{\mathscr{Z}}^{\vec{I}}] \nabla^2 \newlapse
			\label{E:ICOMMUTEDSECONDFUNDJUNKTERMS} \\
	& \ \
		+ \scale^{4/3} 
			[\SigmatLie_{\mathscr{Z}}^{\vec{I}}, \newlapse] 
			\Ric^{\# }
		- 
				\sum_{\vec{I}_1 + \vec{I}_2 + \vec{I}_3 = \vec{I}} 
					(\SigmatLie_{\mathscr{Z}}^{\vec{I}_1} \newg^{-1}) 
					*
					(\SigmatLie_{\mathscr{Z}}^{\vec{I}_2} \newspacescalar)
					*
					\SigmatLie_{\mathscr{Z}}^{\vec{I}_3} \newspacescalar
			\notag
				\\
		& \ \	
				-  
					\scale^{4/3}
					\sum_{\vec{I}_1 + \vec{I}_2 + \vec{I}_3 + \vec{I}_4 = \vec{I}} 
					(\SigmatLie_{\mathscr{Z}}^{\vec{I}_1} \newg^{-1}) 
					*
					(\mathscr{Z}^{\vec{I}_2} \newlapse)
					*
					(\SigmatLie_{\mathscr{Z}}^{\vec{I}_3} \newspacescalar)
					*
					\SigmatLie_{\mathscr{Z}}^{\vec{I}_4} \newspacescalar.
					\notag
\end{align}
\end{subequations}

The \textbf{commuted rescaled scalar field evolution equations are}:
\begin{subequations}
\begin{align}  \label{E:COMMUTEDWAVEEQUATIONRENORMALIZED}
		\partial_t \mathscr{Z}^{\vec{I}} \newtimescalar
		& =
			(1 + \scale^{4/3} \newlapse) \scale^{1/3} \Gdiv \SigmatLie_{\mathscr{Z}}^{\vec{I}} \newspacescalar 
			- \sqrt{\frac{2}{3}} \scale' \scale^{1/3} \mathscr{Z}^{\vec{I}} \newlapse
			+ \scale^{1/3} \CommutedTimeSfBorderInhom{\vec{I}}
			+ \scale^{1/3} \CommutedTimeSfJunkInhom{\vec{I}},
				\\
		\SigmatLie_{\partial_t} \SigmatLie_{\mathscr{Z}}^{\vec{I}} \newspacescalar
		& = \scale^{-1} (1 + \scale^{4/3} \newlapse) \nabla \mathscr{Z}^{\vec{I}} \newtimescalar
			+ \sqrt{\frac{2}{3}} \scale^{1/3} \nabla \mathscr{Z}^{\vec{I}} \newlapse
			+ \scale^{1/3} \CommutedSpaceSfBorderInhom{\vec{I}}
			+ \scale^{1/3} \CommutedSpaceSfJunkInhom{\vec{I}},
			\label{E:COMMUTEDSPACEDERIVATIVESWAVEEQUATIONRENORMALIZED}
\end{align}
\end{subequations}
where
\begin{subequations}
\begin{align} 
	\CommutedTimeSfBorderInhom{\vec{I}}
	& = 
	- \scale' 
		(\mathscr{Z}^{\vec{I}} \newlapse) \newtimescalar,
		\label{E:ICOMMUTEDTIMESFBORDERLINETERMS}  
		\\
	\CommutedTimeSfJunkInhom{\vec{I}}
	& = 
		- \scale' 
		\mathop{\sum_{\vec{I}_1 + \vec{I}_2 = \vec{I}}}_{|\vec{I}_1| \leq |\vec{I}|-1}
		(\mathscr{Z}^{\vec{I}_1} \newlapse) \mathscr{Z}^{\vec{I}_2} \newtimescalar
		+	
		[\SigmatLie_{\mathscr{Z}}^{\vec{I}}, \Gdiv] \newspacescalar
		+ 
		\scale^{4/3} [\SigmatLie_{\mathscr{Z}}^{\vec{I}}, \newlapse \Gdiv] \newspacescalar  
	 \label{E:ICOMMUTEDTIMESFJUNKTERMS} 
			\\
	 & \ \
	 		- \scale^{4/3} 
	 	 	\sum_{\vec{I}_1 + \vec{I}_2 + \vec{I}_3 = \vec{I}}
	 		(\nabla \mathscr{Z}^{\vec{I}_1} \newlapse)
	 		*
			(\SigmatLie_{\mathscr{Z}}^{\vec{I}_2} \newg^{-1}) 
	 		*
			\SigmatLie_{\mathscr{Z}}^{\vec{I}_3} \newspacescalar,
	 		\notag \\
	\CommutedSpaceSfBorderInhom{\vec{I}}
	& = \newtimescalar \nabla \mathscr{Z}^{\vec{I}} \newlapse,
		\label{E:ICOMMUTEDSPACESFBORDERLINETERMS} \\
	\CommutedSpaceSfJunkInhom{\vec{I}}
	& = 
		[\SigmatLie_{\mathscr{Z}}^{\vec{I}},\newlapse] \nabla \newtimescalar
		+ 
		[\SigmatLie_{\mathscr{Z}}^{\vec{I}}, \newtimescalar] \nabla \newlapse.
		\label{E:ICOMMUTEDSPACESFJUNKTERMS}
\end{align}
\end{subequations}

Finally, the \textbf{commuted rescaled elliptic lapse equations} are:
\begin{subequations}
\begin{align} 
	\mathscr{L} \mathscr{Z}^{\vec{I}} \newlapse 
	& = 	
			2 \sqrt{\frac{2}{3}} \mathscr{Z}^{\vec{I}} \newtimescalar
			+ \CommutedLapseHighBorderInhom{\vec{I}}
			+ \scale^{4/3} \CommutedLapseHighJunkInhom{\vec{I}},
			\label{E:COMMUTEDLAPSEPDERENORMALIZEDHIGHERDERIVATIVES}	\\
	\widetilde{\mathscr{L}} \mathscr{Z}^{\vec{I}} \newlapse
	& = \CommutedLapseLowBorderInhom{\vec{I}}
			+ \scale^{4/3} \CommutedLapseLowJunkInhom{\vec{I}},
		\label{E:COMMUTEDLAPSEPDERENORMALIZEDLOWERDERIVATIVES}
\end{align}
\end{subequations}
where
\begin{subequations}
\begin{align} 
	\CommutedLapseHighBorderInhom{\vec{I}}
	& := \sum_{\vec{I}_1 + \vec{I}_2 = \vec{I}}
				(\SigmatLie_{\mathscr{Z}}^{\vec{I}_1} \FreeNewSec) 
		  	\cdot 
		  	\SigmatLie_{\mathscr{Z}}^{\vec{I}_2} \FreeNewSec
		  +
		  \sum_{\vec{I}_1 + \vec{I}_2 = \vec{I}}
				(\mathscr{Z}^{\vec{I}_1} \newtimescalar) 
		  	\mathscr{Z}^{\vec{I}_2}\newtimescalar,
			\label{E:ICOMMUTEDLAPSEHIGHBORDERTERMS}	 \\
	\CommutedLapseLowBorderInhom{\vec{I}}
 & := 
		- \mbox{\upshape tr} (\GLap \SigmatLie_{\mathscr{Z}}^{\vec{I}} \newg)^{\#}
			+ \mbox{\upshape tr} \nabla^{\#} \Gdiv \SigmatLie_{\mathscr{Z}}^{\vec{I}} \newg
			+ \mbox{\upshape tr} \RicErrorInhom{\vec{I}}
				\label{E:ICOMMUTEDLAPSELOWBORDERTERMS} \\
	& \ \
		- 
		\sum_{\vec{I}_1 + \vec{I}_2 + \vec{I}_3 = \vec{I}}
 			(\SigmatLie_{\mathscr{Z}}^{\vec{I}_1} \newg^{-1})
			*
 			(\SigmatLie_{\mathscr{Z}}^{\vec{I}_2} \newspacescalar)
			*
 			\SigmatLie_{\mathscr{Z}}^{\vec{I}_3} \newspacescalar,
			\notag \\
 \CommutedLapseHighJunkInhom{\vec{I}}
 & := \scale^{4/3} [\GLap, \mathscr{Z}^{\vec{I}}] \newlapse
 		+ [\mathscr{Z}^{\vec{I}}, |\FreeNewSec|_{\newg}^2] \newlapse
 		+ 2 \sqrt{\frac{2}{3}} 
				[\mathscr{Z}^{\vec{I}}, \newtimescalar] \newlapse
 			\label{E:ICOMMUTEDLAPSEHIGHJUNKTERMS}	\\
 	& \ \ 
 		+   [\mathscr{Z}^{\vec{I}}, \newtimescalar^2] \newlapse,
 			\notag \\
 \CommutedLapseLowJunkInhom{\vec{I}}
 & := [\GLap, \mathscr{Z}^{\vec{I}}] \newlapse
 			+
			[\mathscr{Z}^{\vec{I}},\ScalarCur] \newlapse
			-
			[\mathscr{Z}^{\vec{I}},|\newspacescalar|_{\newg}^2] \newlapse,
 			\label{E:ICOMMUTEDLAPSELOWJUNKTERMS}
 \end{align}
\end{subequations}
and $\RicErrorInhom{\vec{I}}$ is defined by \eqref{E:RICINHOMERROR}.
\end{proposition}

\begin{proof}
	\eqref{E:COMMUTEDRENORMALIZEDMOMENTUM} follows 
	in a straightforward fashion from applying 
	$\SigmatLie_{\mathscr{Z}}^{\vec{I}}$
	to equation \eqref{E:RENORMALIZEDMOMENTUM}.
	Similarly, \eqref{E:ALTERNATECOMMUTEDRENORMALIZEDMOMENTUM}
	follows from \eqref{E:ALTERNATERENORMALIZEDMOMENTUM},
	\eqref{E:COMMUTEDEVOLUTIONMETRICRENORMALIZED}
	follows from \eqref{E:EVOLUTIONMETRICRENORMALIZED},
	\eqref{E:COMMUTEDEVOLUTIONINVERSEMETRICRENORMALIZED}
	follows from \eqref{E:EVOLUTIONINVERSEMETRICRENORMALIZED},
	\eqref{E:COMMUTEDWAVEEQUATIONRENORMALIZED} 
	follows from \eqref{E:WAVEEQUATIONRENORMALIZED},
	\eqref{E:COMMUTEDSPACEDERIVATIVESWAVEEQUATIONRENORMALIZED}
	follows from \eqref{E:EVOLUTIONSPACESCALARRENORMALIZED},
	and \eqref{E:COMMUTEDLAPSEPDERENORMALIZEDHIGHERDERIVATIVES} 
	follows from \eqref{E:LAPSEPDERENORMALIZEDHIGHERDERIVATIVES}.
	
	To derive \eqref{E:GRADCOMMUTEDEVOLUTIONMETRICRENORMALIZED}, 
	we apply the operator
	$\nabla \SigmatLie_{\mathscr{Z}}^{\vec{I}}$
	to equation \eqref{E:EVOLUTIONMETRICRENORMALIZED}.
	All terms appearing on RHS~\eqref{E:GRADCOMMUTEDEVOLUTIONMETRICRENORMALIZED}
	are straightforward to deduce except for the ones generated by the commutator term 
	$[\SigmatLie_{\partial_t}, \nabla] \SigmatLie_{\mathscr{Z}}^{\vec{I}} \newg$.
	To decompose this commutator term, 
	we use \eqref{E:NABLALIEDETAILEDCOMMUTATOR}
	with $\xi := \SigmatLie_{\mathscr{Z}}^{\vec{I}} \newg$,
	thereby obtaining
	the schematic identity
	$
	[\SigmatLie_{\partial_t}, \nabla] \SigmatLie_{\mathscr{Z}}^{\vec{I}} \newg
	= \newg^{-1} *\nabla \SigmatLie_{\partial_t} \newg * \SigmatLie_{\mathscr{Z}}^{\vec{I}} \newg 
	$,
	which leads to the presence of the last
	product on RHS~\eqref{E:ICOMMUTEDGRADMETRICBORDERTERMS} 
	(which is schematically depicted).

To derive \eqref{E:COMMUTEDEVOLUTIONSECONDFUNDRENORMALIZED},
we apply the operator
$\nabla \SigmatLie_{\mathscr{Z}}^{\vec{I}}$
to equation \eqref{E:EVOLUTIONSECONDFUNDRENORMALIZED}.
All terms appearing on RHS~\eqref{E:COMMUTEDEVOLUTIONSECONDFUNDRENORMALIZED}
are straightforward to deduce except for the ones generated
when derivatives fall on 
$\Ric^{\# }$, for which we use 
Lemma~\ref{L:RICCILIEDERIVATIVESCHEMATIC}.

To derive \eqref{E:COMMUTEDLAPSEPDERENORMALIZEDLOWERDERIVATIVES},
we apply the operator $\SigmatLie_{\mathscr{Z}}^{\vec{I}}$
to equation \eqref{E:LAPSEPDERENORMALIZEDLOWERDERIVATIVES}
and use \eqref{E:SCALARCURLIEDERIVATIVESCHEMATIC}
to decompose the derivatives of the scalar curvature source term.
\end{proof}

\section{Energy currents and approximate monotonicity in divergence form via divergence identities}
\label{S:ENERGYCURRENTSANDDIVIDENTITIES}
In this section, we define various energy current vectorfields
and use them to derive divergence identities verified by solutions.
These identities, when integrated over spacetime domains of the form $[0,t] \times \mathbb{S}^3$,
form the starting point for our $L^2$-type analysis of solutions.

\subsection{An auxiliary spacetime metric and a divergence identity connected to it}
In deriving various identities, we will find it convenient to 
rely on the following auxiliary spacetime metric.

\begin{definition}[\textbf{Auxiliary spacetime metric}]
\label{D:ENERGYESTIMATESPACETIMEMETRIC}
We define the Lorentzian metric $\EnergyEstimatesMetric$ 
relative to CMC-transported spatial coordinates
as follows:
\begin{align} \label{E:ENERGYESTIMATESPACETIMEMETRIC}
	\EnergyEstimatesMetric
	& := -dt^2 + \newg,
\end{align}
where $\newg$ is the time-rescaled spatial metric defined by \eqref{E:RESCALEDMETRIC}.
\end{definition}

\begin{definition}[\textbf{Covariant divergence relative to} $\EnergyEstimatesMetric$]
	\label{D:DIVERGENCERELATIVETOAUXSPACETIMEMETRIC}
	If $\Jfour$ is a spacetime vectorfield, then
	$\Divfour_{\EnergyEstimatesMetric} \Jfour$
	denotes its covariant divergence relative to 
	the Levi--Civita connection of $\EnergyEstimatesMetric$.
\end{definition}

\begin{lemma}[\textbf{Divergence formula involving the auxiliary spacetime metric}]
	\label{L:SIMPLEDIVFORMULA}
	Let $\Jfour$ be any spacetime vectorfield 
	and consider the decomposition
	$\Jfour = J^0 \partial_t + J$, where $J$ is $\Sigma_t-$tangent.  
	Then
	\begin{align} \label{E:SIMPLEDIVFORMULA}
		\Divfour_{\EnergyEstimatesMetric} \Jfour
		& = \partial_t J^0
			+ \Gdiv J
			+  \scale' \scale^{1/3} \newlapse J^0.
	\end{align}
\end{lemma}

\begin{proof}
From \eqref{E:ENERGYESTIMATESPACETIMEMETRIC},
we deduce that 
$|\mbox{\upshape det} \EnergyEstimatesMetric| = \mbox{\upshape det} \newg$
(relative to CMC-transported spatial coordinates).
Since 
$\Divfour_{\EnergyEstimatesMetric} \Jfour =
\frac{1}{|\mbox{\upshape det} \EnergyEstimatesMetric|^{1/2}}
\partial_{\alpha}
\left(
	|\mbox{\upshape det} \EnergyEstimatesMetric|^{1/2} \Jfour^{\alpha}
\right)
$,
since
$\Gdiv J =
\frac{1}{(\mbox{\upshape det} \newg)^{1/2}}
\partial_a
\left(
	(\mbox{\upshape det} \newg)^{1/2} J^a
\right)
$,
and since $|\mbox{\upshape det} \EnergyEstimatesMetric| = \mbox{\upshape det} \newg$,
it follows that
$\Divfour_{\EnergyEstimatesMetric} \Jfour =
\frac{1}{(\mbox{\upshape det} \newg)^{1/2}}
\partial_{\alpha}
\left(
	(\mbox{\upshape det} \newg)^{1/2} \Jfour^{\alpha}
\right)
=
\frac{1}{2} 
\left(
	\partial_t \ln (\mbox{\upshape det} \newg)
\right)
J^0
+
\partial_t J^0
+
\Gdiv J
$.
To complete the proof of \eqref{E:SIMPLEDIVFORMULA},
we need only to note the standard determinant differentiation formula
$
\partial_t \ln (\mbox{\upshape det} \newg)
= \mbox{\upshape tr}_{\newg} \SigmatLie_{\partial_t} \newg
$
and to use equation \eqref{E:EVOLUTIONMETRICRENORMALIZED}
to substitute for $\SigmatLie_{\partial_t} \newg$,
thereby obtaining
$
\partial_t \ln (\mbox{\upshape det} \newg)
= 2 \scale' \scale^{1/3} \newlapse
$.
In obtaining this formula, we have also used
the identity $\mbox{\upshape tr}_{\newg} \newg = 3$.
\end{proof}

\subsection{Volume forms}
\label{SS:VOLUMEFORM}
In this subsection, we exhibit some basic properties of the
volume forms of the round metric $\StMet$ and
the rescaled spatial metric $\newg$.

\begin{definition}[\textbf{Volume forms}] 
	\label{D:VOLUMEFORM}
	Let $\StMet$ be the round metric on $\Sigma_t$ from Subsect.\ \ref{SS:EXTENDINGSIGMA0TANGENTTENSORFIELDSTOSIGMATTANGENTTENSORFIELDS}
	and let $\newg$ be the rescaled spatial metric defined by \eqref{E:RESCALEDMETRIC}.
	We denote the volume forms\footnote{Relative to arbitrary local coordinates
	$\lbrace x^a \rbrace_{a=1,2,3}$ on $\Sigma_t$,
	we have $d \Sttvol = \sqrt{\mbox{\upshape det} \StMet} dx^1 \wedge dx^2 \wedge dx^3$ 
	and
	$d \tvol = \sqrt{\mbox{\upshape det} \newg} dx^1 \wedge dx^2 \wedge dx^3$.} 
	associated to these metrics as follows:
	\begin{subequations}
	\begin{align} 
		& d \Sttvol,	
		\label{E:STMETVOLUMEFORM} \\
		& d \tvol. \label{E:VOLUMEFORM}
	\end{align}
	\end{subequations}
\end{definition}

\begin{lemma}[\textbf{Evolution equations for the volume forms}]
	\label{L:VOLUMEFORMTIMEDERIVATIVE}
	The volume forms
	$d \Sttvol$
	and $d \tvol$ from Def.\ \ref{D:VOLUMEFORM}
	verify the following evolution equations:
	\begin{align} \label{E:STMETVOLUMEFORMTIMEDERIVATIVE}
		\SigmatLie_{\partial_t} d \Sttvol
		& = 0,
			\\
		\SigmatLie_{\partial_t} d \tvol
		& = \scale' \scale^{1/3} \newlapse d \tvol.
		 \label{E:VOLUMEFORMTIMEDERIVATIVE}
	\end{align}
\end{lemma}

\begin{proof}
	To prove \eqref{E:VOLUMEFORMTIMEDERIVATIVE}, we use
	the following standard identity:
	$\SigmatLie_{\partial_t} d \tvol
	= \frac{1}{2} (\mbox{\upshape tr}_{\newg} \SigmatLie_{\partial_t} \newg) d \tvol
	$. 
	Also using \eqref{E:EVOLUTIONMETRICRENORMALIZED} and the fact that $\mbox{\upshape tr}_{\newg} \newg = 3$,
	we conclude \eqref{E:VOLUMEFORMTIMEDERIVATIVE}.
	\eqref{E:STMETVOLUMEFORMTIMEDERIVATIVE} follows similarly since
	$\SigmatLie_{\partial_t} \StMet = 0$ by construction.
\end{proof}

\subsection{Energy currents}
\label{SS:ENERGYCURRENTS}
In this subsection, we define energy currents,
which are vectorfields that we will use for bookkeeping
when integrating by parts.

\begin{definition}[\textbf{Energy Currents}]
\label{D:ENERGYCURRENTS}
Let $\vec{I}$ be a $\mathscr{Z}$-multi-index.
To each triplet 
$(\SigmatLie_{\mathscr{Z}}^{\vec{I}} \FreeNewSec,\SigmatLie_{\mathscr{Z}}^{\vec{I}} \newg,\mathscr{Z}^{\vec{I}} \newlapse)$,
we associate the metric energy current
$\Jfour_{(Metric)}
[\SigmatLie_{\mathscr{Z}}^{\vec{I}} \FreeNewSec,\SigmatLie_{\mathscr{Z}}^{\vec{I}} \newg,\nabla \mathscr{Z}^{\vec{I}} \newlapse]$,
which is the spacetime vectorfield with the following components:
\begin{subequations}
\begin{align}
	J_{(Metric)}^0[\SigmatLie_{\mathscr{Z}}^{\vec{I}} \FreeNewSec,\SigmatLie_{\mathscr{Z}}^{\vec{I}} \newg]
	& := |\SigmatLie_{\mathscr{Z}}^{\vec{I}} \FreeNewSec|_{\newg}^2
			+ \frac{1}{4} \scale^{4/3} |\nabla \SigmatLie_{\mathscr{Z}}^{\vec{I}} \newg|_{\newg}^2,
			\label{E:METRICCURRENT0} \\
	J_{(Metric)}
	[\SigmatLie_{\mathscr{Z}}^{\vec{I}} \FreeNewSec,\SigmatLie_{\mathscr{Z}}^{\vec{I}} \newg,\nabla \mathscr{Z}^{\vec{I}} \newlapse]
	& := 
		2 \scale^{5/3} \SigmatLie_{\mathscr{Z}}^{\vec{I}} \FreeNewSec \cdot \nabla^{\#} \mathscr{Z}^{\vec{I}} \newlapse
			\label{E:METRICCURRENTSPACE} \\
	& \ \
		+ \scale^{1/3} (1 + \scale^{4/3} \newlapse) 
			(\SigmatLie_{\mathscr{Z}}^{\vec{I}} \FreeNewSec)^{\#}
			\cdot 
			\nabla^{\#} \SigmatLie_{\mathscr{Z}}^{\vec{I}} \newg
				\notag \\
	& \ \
			+ 
			\scale^{1/3} (1 + \scale^{4/3} \newlapse) 
			\SigmatLie_{\mathscr{Z}}^{\vec{I}} \FreeNewSec 
			\cdot 
			\nabla^{\#} \mytr \SigmatLie_{\mathscr{Z}}^{\vec{I}} \newg
			\notag \\
	& \ \ 
		- \scale^{1/3} (1 + \scale^{4/3} \newlapse) 
			\SigmatLie_{\mathscr{Z}}^{\vec{I}} \FreeNewSec
			\cdot
			(\Gdiv \SigmatLie_{\mathscr{Z}}^{\vec{I}} \newg)^{\#} 
			\notag \\
	& \ \
			- \scale^{1/3} (1 + \scale^{4/3} \newlapse) 
			(\SigmatLie_{\mathscr{Z}}^{\vec{I}} \FreeNewSec
			\cdot\Gdiv \SigmatLie_{\mathscr{Z}}^{\vec{I}} \newg)^{\#}.
		\notag
\end{align}
\end{subequations}

Similarly, to each pair
$(\mathscr{Z}^{\vec{I}} \newtimescalar,\SigmatLie_{\mathscr{Z}}^{\vec{I}} \newspacescalar)$,
we associate the scalar field energy current
$\Jfour_{(Sf)}[\mathscr{Z}^{\vec{I}} \newtimescalar,\SigmatLie_{\mathscr{Z}}^{\vec{I}} \newspacescalar]$,
which is the spacetime vectorfield with the following components:
\begin{subequations}
\begin{align}
	J_{(Sf)}^0[\mathscr{Z}^{\vec{I}} \newtimescalar,\SigmatLie_{\mathscr{Z}}^{\vec{I}} \newspacescalar]
	& := (\mathscr{Z}^{\vec{I}} \newtimescalar)^2 
		+ \scale^{4/3} |\SigmatLie_{\mathscr{Z}}^{\vec{I}} \newspacescalar|_{\newg}^2,
			\label{E:SFCURRENT0} 
			\\
	J_{(Sf)}[\mathscr{Z}^{\vec{I}} \newtimescalar,\SigmatLie_{\mathscr{Z}}^{\vec{I}} \newspacescalar]
	& := 
		- 2 (1 + \scale^{4/3} \newlapse) \scale^{1/3} 
		(\mathscr{Z}^{\vec{I}} \newtimescalar) 
		(\SigmatLie_{\mathscr{Z}}^{\vec{I}} \newspacescalar)^{\#}.
		\label{E:SFCURRENTSPACE}
\end{align}
\end{subequations}

\end{definition}

\subsection{Divergence of the currents}
\label{SS:DIVOFCURRENTS}
In this subsection, we use the $\Lie_{\mathscr{Z}}^{\vec{I}}$-commuted Einstein-scalar field equations
to compute the divergence of the energy currents
defined in Def.\ \ref{D:ENERGYCURRENTS}.

\begin{lemma}[\textbf{Divergence of the metric energy current}]
\label{L:METRICCURRENTDIV}	
	Let $\Jfour_{(Metric)}[\cdots]$ be the energy current from Def.\ \ref{D:ENERGYCURRENTS}
	and let $\vec{I}$ be any $\mathscr{Z}$-multi-index. 
	For solutions to the $\Lie_{\mathscr{Z}}^{\vec{I}}$-commuted equations 
	of Prop.~\ref{P:ICOMMUTEDEQNS}, 
	the following identity holds:
	\begin{align} \label{E:DIVMETRICCURRENT}
		\Divfour_{\EnergyEstimatesMetric} 
		\Jfour_{(Metric)}[\SigmatLie_{\mathscr{Z}}^{\vec{I}} \FreeNewSec,\SigmatLie_{\mathscr{Z}}^{\vec{I}} \newg,\nabla \mathscr{Z}^{\vec{I}} 		
			\newlapse]
		& = \frac{1}{3} \scale' \scale^{1/3} |\nabla \SigmatLie_{\mathscr{Z}}^{\vec{I}} \newg|_{\newg}^2 
			\\
		& \ \
			+ \frac{1}{3} \scale' \scale^{5/3} 
				\langle
					\nabla \mathscr{Z}^{\vec{I}} \newlapse,
					\nabla \mytr \SigmatLie_{\mathscr{Z}}^{\vec{I}} \newg
				\rangle_{\newg}
			\notag \\
		& \ \
			- 2 \sqrt{\frac{2}{3}} \scale^{5/3}
				\SigmatLie_{\mathscr{Z}}^{\vec{I}} \newspacescalar
				\cdot 
				\nabla^{\#} \mathscr{Z}^{\vec{I}} \newlapse
			\notag \\
		& \ \
			- \sqrt{\frac{2}{3}} \scale^{1/3} (1 + \scale^{4/3} \newlapse) 
			  \nabla^{\#} \mytr \SigmatLie_{\mathscr{Z}}^{\vec{I}} \newg
			  \cdot \SigmatLie_{\mathscr{Z}}^{\vec{I}} \newspacescalar
			\notag \\
		& \ \ 
		+ 2 \sqrt{\frac{2}{3}} \scale^{1/3} (1 + \scale^{4/3} \newlapse) (\Gdiv \SigmatLie_{\mathscr{Z}}^{\vec{I}} \newg)^{\#} 
			\cdot \SigmatLie_{\mathscr{Z}}^{\vec{I}} \newspacescalar 
			\notag 
			\\
		& \ \ 
			+ \MetricCurrentBorder{\vec{I}}
			+ \MetricCurrentJunk{\vec{I}},
			\notag 
\end{align}
\begin{subequations}
\begin{align}
	\MetricCurrentBorder{\vec{I}}
		& := 
 				\left[
					(\SigmatLie_{\partial_t} \newg) \otimes \newg^{-1} 
				\right]
				\cdot 
				\SigmatLie_{\mathscr{Z}}^{\vec{I}} \FreeNewSec 
				\otimes
				\SigmatLie_{\mathscr{Z}}^{\vec{I}} \FreeNewSec
				+
				\left[
					\newg \otimes (\SigmatLie_{\partial_t} \newg^{-1})
				\right]
				\cdot 
				\SigmatLie_{\mathscr{Z}}^{\vec{I}} \FreeNewSec 
				\otimes
				\SigmatLie_{\mathscr{Z}}^{\vec{I}} \FreeNewSec
			\label{E:METRICCURRENTBORDER}  \\
		& \ \
			+ \frac{1}{4} \scale^{4/3} 
			\left[
				\SigmatLie_{\partial_t} \newg^{-1} \otimes \newg^{-1} \otimes \newg^{-1}
			\right]
			\cdot
			(\nabla 
			\otimes 
			\nabla)
		  \left\lbrace
		  \SigmatLie_{\mathscr{Z}}^{\vec{I}} \newg 
			\otimes
			\SigmatLie_{\mathscr{Z}}^{\vec{I}} \newg
			\right\rbrace
			\notag \\
		& \ \
			+ 
			\frac{1}{2} \scale^{4/3} 
			\left[
				\newg^{-1} \otimes \SigmatLie_{\partial_t} \newg^{-1} \otimes \newg^{-1}
			\right]
			\cdot
			(\nabla 
			\otimes 
			\nabla)
			\left\lbrace
			\SigmatLie_{\mathscr{Z}}^{\vec{I}} \newg 
			\otimes
			\SigmatLie_{\mathscr{Z}}^{\vec{I}} \newg
			\right\rbrace
			\notag \\
 		& \ \
				+
				2 \scale^{1/3} 
				\langle 
					\SigmatLie_{\mathscr{Z}}^{\vec{I}} \FreeNewSec, 
					\CommutedSecFunBorderInhom{\vec{I}} 
				\rangle_{\newg}
			\notag \\
		& \ \
 			+ \frac{1}{2} 
			\scale^{1/3} 
			\langle 
				\nabla \SigmatLie_{\mathscr{Z}}^{\vec{I}} \newg, 
				\CommutedGradMetBorderInhom{\vec{I}} 
			\rangle_{\newg}
			\notag \\
	& \ \
			+ \scale^{1/3} (1 + \scale^{4/3} \newlapse) 
				\nabla^{\#} \mytr \SigmatLie_{\mathscr{Z}}^{\vec{I}} \newg
				\cdot 
				\CommutedMomBorderInhomDown{\vec{I}}
			\notag \\
		& \ \ 
				- \scale^{1/3} (1 + \scale^{4/3} \newlapse) 
			  (\Gdiv \SigmatLie_{\mathscr{Z}}^{\vec{I}} \newg)^{\#}
			  \cdot 
			  \CommutedMomBorderInhomDown{\vec{I}}
			  \notag \\
		& \ \
				- 
				\scale^{1/3} (1 + \scale^{4/3} \newlapse) 
			  \Gdiv \SigmatLie_{\mathscr{Z}}^{\vec{I}} \newg
			  \cdot 
			  \CommutedMomBorderInhomUp{\vec{I}}
			  \notag \\
		&  \ \
			+ 2 \scale^{5/3} 
				\nabla^{\#} \mathscr{Z}^{\vec{I}} \newlapse
				\cdot
				\CommutedMomBorderInhomDown{\vec{I}},
			\notag \\
	\MetricCurrentJunk{\vec{I}}
		& := 
			\scale' \scale^{1/3}  
			\newlapse
			\left|
				\SigmatLie_{\mathscr{Z}}^{\vec{I}} \FreeNewSec
			\right|_{\newg}^2
			+ \frac{1}{4} \scale' \scale^{5/3}
			\newlapse 
			\left|
				\nabla \SigmatLie_{\mathscr{Z}}^{\vec{I}} \newg
			\right|_{\newg}^2
			\label{E:METRICCURRENTJUNK} \\
		& \ \
				+
				2 \scale^{1/3} 
				\langle 
					\SigmatLie_{\mathscr{Z}}^{\vec{I}} \FreeNewSec, 
					\CommutedSecFunJunkInhom{\vec{I}} 
				\rangle_{\newg}
				+
				2 \scale^{1/3} 
				\langle 
					\SigmatLie_{\mathscr{Z}}^{\vec{I}} \FreeNewSec, 
					\RicErrorInhom{\vec{I}}
				\rangle_{\newg}
					\notag \\
		&  \ \	
				+ \frac{1}{2} 
					\scale^{5/3} 
					\langle 
						\nabla \SigmatLie_{\mathscr{Z}}^{\vec{I}} \newg, \CommutedGradMetJunkInhom{\vec{I}} 
					\rangle_{\newg}
		\notag \\
		& \ \
		+ \scale^{5/3} \nabla \newlapse
			\cdot
			\left\lbrace
				(\SigmatLie_{\mathscr{Z}}^{\vec{I}} \FreeNewSec)^{\#}
				\cdot 
				\nabla^{\#} \SigmatLie_{\mathscr{Z}}^{\vec{I}} \newg
			\right\rbrace
		+ \scale^{5/3} \nabla \newlapse \cdot \SigmatLie_{\mathscr{Z}}^{\vec{I}} \FreeNewSec 
			\cdot \nabla^{\#} \mytr \SigmatLie_{\mathscr{Z}}^{\vec{I}} \newg
			\notag \\
		& \ \ 
			- \scale^{5/3} \nabla \newlapse 
				\cdot \SigmatLie_{\mathscr{Z}}^{\vec{I}} \FreeNewSec
				\cdot (\Gdiv \SigmatLie_{\mathscr{Z}}^{\vec{I}} \newg)^{\#} 
			- \scale^{5/3} 
				\nabla^{\#} \newlapse 
				\cdot 
				\SigmatLie_{\mathscr{Z}}^{\vec{I}} \FreeNewSec
				\cdot
				\Gdiv \SigmatLie_{\mathscr{Z}}^{\vec{I}} \newg.
				\notag
	\end{align}
	\end{subequations}
\end{lemma}

\begin{proof}
	We start by using the identity \eqref{E:SIMPLEDIVFORMULA},
	where $J_{(Metric)}^0[\cdots]$ and $J_{(Metric)}[\cdots]$ 
	are defined by
	\eqref{E:METRICCURRENT0}
	and \eqref{E:METRICCURRENTSPACE}.
	The product $\scale' \scale^{1/3} \newlapse J^0$
	on RHS~\eqref{E:SIMPLEDIVFORMULA}
	leads to the presence of
	the first two products on RHS~\eqref{E:METRICCURRENTJUNK}.
	We now consider the terms in
	$\Gdiv J_{(Metric)}$ when the operator $\nabla$ 
	falls on the factors of $\newlapse$
	on RHS~\eqref{E:METRICCURRENTSPACE}
	(these factors are multiplied by $\scale^{5/3}$).
	We place all four of these error term products on RHS~\eqref{E:METRICCURRENTJUNK}.
	We now consider the products 
	in $\partial_t J_{(Metric)}^0$
	when the
	operator $\partial_t = \SigmatLie_{\partial_t}$
	falls on the
	factors of $\newg$ and $\newg^{-1}$
	inherent in the definition of the norms $|\cdot|_{\newg}$
	on RHS~\eqref{E:METRICCURRENT0}.
	We place these products on RHS~\eqref{E:METRICCURRENTBORDER}
	(they are the first four).
	Moreover, when the
	operator $\partial_t$
	falls on the factor $\scale^{4/3}$
	in the second product on RHS~\eqref{E:METRICCURRENT0},
	we place the resulting product as the first one on RHS~\eqref{E:DIVMETRICCURRENT}.
	
	We now consider the terms in
	$\partial_t J_{(Metric)}^0$
	and $\Gdiv J$
	when $\partial_t = \SigmatLie_{\partial_t}$
	and $\nabla$ fall on the top-order factors
	$\SigmatLie_{\mathscr{Z}}^{\vec{I}} \FreeNewSec$,
	$\SigmatLie_{\mathscr{Z}}^{\vec{I}} \newg$,
	and
	$\nabla \mathscr{Z}^{\vec{I}} \newlapse$.
	When
	$\SigmatLie_{\partial_t}$
	falls on
	$\SigmatLie_{\mathscr{Z}}^{\vec{I}} \FreeNewSec$
	(respectively $\SigmatLie_{\mathscr{Z}}^{\vec{I}} \newg$),
	we use
	\eqref{E:COMMUTEDEVOLUTIONSECONDFUNDRENORMALIZED} 
	(respectively \eqref{E:GRADCOMMUTEDEVOLUTIONMETRICRENORMALIZED})
	for substitution.
	We note the following critically important cancellation:
	since $\mbox{\upshape tr} \SigmatLie_{\mathscr{Z}}^{\vec{I}} \FreeNewSec = 0$, 
	we have 
	$\frac{1}{3} (\scale')^2 \scale^{1/3} (\mathscr{Z}^{\vec{I}} \newlapse) \ID \cdot \SigmatLie_{\mathscr{Z}}^{\vec{I}} \FreeNewSec = 0$;
	this product is generated by the second term on the second line of RHS~\eqref{E:COMMUTEDEVOLUTIONSECONDFUNDRENORMALIZED}.
	Similarly, when $\Gdiv$ falls on any of the factors
	$\SigmatLie_{\mathscr{Z}}^{\vec{I}} \FreeNewSec$
	on RHS~\eqref{E:METRICCURRENTSPACE},
	we use
	\eqref{E:COMMUTEDRENORMALIZEDMOMENTUM}-\eqref{E:ALTERNATECOMMUTEDRENORMALIZEDMOMENTUM}
	for substitution.
	These substitutions lead to the 
	presence of the terms on lines two through five of RHS~\eqref{E:DIVMETRICCURRENT}
	as well as the products 
	on RHS~\eqref{E:METRICCURRENTBORDER} and RHS~\eqref{E:METRICCURRENTJUNK}
	involving inhomogeneous terms 
	such as $\TopCommutedGradMetBorderInhom{\vec{I}}$.
	Moreover, we stress that this procedure allows us 
	to \emph{completely eliminate} all products
	involving the first derivatives of
	$\SigmatLie_{\mathscr{Z}}^{\vec{I}} \FreeNewSec$,
	$\nabla \SigmatLie_{\mathscr{Z}}^{\vec{I}} \newg$,
	or
	$\nabla \mathscr{Z}^{\vec{I}} \newlapse$.
\end{proof}

The next lemma lies at the heart of our energy estimates.
Its proof relies on the observation of \emph{special cancellations}
that allow us to replace damaging error terms with favorable ones.
In particular, the products
$
\frac{4}{3} \scale' \scale^{1/3} |\SigmatLie_{\mathscr{Z}}^{\vec{I}} \newspacescalar|_{\newg}^2 
+ 
\scale' \scale^3  |\nabla \mathscr{Z}^{\vec{I}} \newlapse|_{\newg}^2
$
on RHS~\eqref{E:DIVSFCURRENT} will yield, in our energy estimates,
spacetime integrals that provide control over the lapse and its derivatives
near the singularity; 
see the third and fourth terms on RHS~\eqref{E:FUNDAMENTALENERGYINTEGRALINEQUALITY}.
For this, it is important that
$\scale'(t) \approx - 1$ for $t$ near $\TCrunch$
(see Lemma~\ref{L:ANALYSISOFFRIEDMANN}).

\begin{lemma}[\textbf{Approximate monotonicity in divergence form: The divergence of the scalar field-lapse energy current}]
\label{L:SFCURRENTDIV}
	Let $\Jfour_{(Sf)}[\cdots]$ be the energy current from Def.\ \ref{D:ENERGYCURRENTS}
	and let $\vec{I}$ be any $\mathscr{Z}$-multi-index. 
	For solutions to the $\Lie_{\mathscr{Z}}^{\vec{I}}$-commuted equations 
	of Prop.~\ref{P:ICOMMUTEDEQNS}, 
	the following identity holds:
	\begin{align} \label{E:DIVSFCURRENT}
		\Divfour_{\EnergyEstimatesMetric} 
		& \Jfour_{(Sf)}[\mathscr{Z}^{\vec{I}} \newtimescalar,\SigmatLie_{\mathscr{Z}}^{\vec{I}} \newspacescalar]
		+
		\scale' \scale^3 \Gdiv 
			\left\lbrace 
				(\nabla^{\#} \mathscr{Z}^{\vec{I}} \newlapse) \mathscr{Z}^{\vec{I}} \newlapse
			\right\rbrace
				\\
		& = \frac{4}{3} \scale' \scale^{1/3} |\SigmatLie_{\mathscr{Z}}^{\vec{I}} \newspacescalar|_{\newg}^2 
				+ 
				\scale' \scale^3  |\nabla \mathscr{Z}^{\vec{I}} \newlapse|_{\newg}^2
				\notag \\
		& \ \
			+
	\left\lbrace
				(\scale')^3 \scale^{5/3}
				+ \frac{2}{3} \scale' \scale^3
				+ \scale' \scale^{5/3} |\FreeNewSec|_{\newg}^2
				+ 2 \sqrt{\frac{2}{3}} \scale' \scale^{5/3} \newtimescalar
				+ \scale' \scale^{5/3} \newtimescalar^2 
	\right\rbrace
	\left|
		\mathscr{Z}^{\vec{I}} \newlapse
	\right|^2
			 \notag \\
		& \ \	
				+ 2 \sqrt{\frac{2}{3}} \scale^{5/3} (\SigmatLie_{\mathscr{Z}}^{\vec{I}} \newspacescalar)^{\#} 
				\cdot \nabla \mathscr{Z}^{\vec{I}} \newlapse
				\notag  \\
		& \ \ 
			+ \SfCurrentBorder{\vec{I}}
			+ \SfCurrentJunk{\vec{I}},
			\notag 
	\end{align}
	\begin{subequations}
	\begin{align}
	\SfCurrentBorder{\vec{I}}
		& := 
				2 \scale^{1/3} (\mathscr{Z}^{\vec{I}} \newtimescalar) \CommutedTimeSfBorderInhom{\vec{I}}
				+ 2 \scale^{5/3} (\SigmatLie_{\mathscr{Z}}^{\vec{I}} \newspacescalar)^{\#} \cdot \CommutedSpaceSfBorderInhom{\vec{I}}
				\label{E:SFCURRENTBORDER}  \\
		& \ \
			+ \scale' \scale^{1/3} (\mathscr{Z}^{\vec{I}} \newlapse) \CommutedLapseHighBorderInhom{\vec{I}}
			 \notag \\
		& \ \
				+
				\scale^{4/3} \SigmatLie_{\partial_t} \newg^{-1} 
				\cdot 
				\SigmatLie_{\mathscr{Z}}^{\vec{I}} \newspacescalar 
				\otimes
				\SigmatLie_{\mathscr{Z}}^{\vec{I}} \newspacescalar,
				\notag \\
	\SfCurrentJunk{\vec{I}}
		& := 
			\scale' \scale^{1/3} \newlapse (\mathscr{Z}^{\vec{I}} \newtimescalar)^2 
			+ \scale' \scale^{5/3} \newlapse |\SigmatLie_{\mathscr{Z}}^{\vec{I}} \newspacescalar|_{\newg}^2
			\label{E:SFCURRENTJUNK} \\
		& \ \
			- 2 \scale^{5/3} 
				(\mathscr{Z}^{\vec{I}} \newtimescalar) 
				(\SigmatLie_{\mathscr{Z}}^{\vec{I}} \newspacescalar)^{\#}	
				\cdot \nabla \newlapse
				\notag \\
		& \ \
			+ 2 \scale^{1/3} (\mathscr{Z}^{\vec{I}} \newtimescalar) \CommutedTimeSfJunkInhom{\vec{I}}
			+ 2 \scale^{5/3} (\SigmatLie_{\mathscr{Z}}^{\vec{I}} \newspacescalar)^{\#} \cdot \CommutedSpaceSfJunkInhom{\vec{I}}
			+ \scale' \scale^{5/3} (\mathscr{Z}^{\vec{I}} \newlapse) \CommutedLapseHighJunkInhom{\vec{I}}.
			\notag 
	\end{align}
	\end{subequations}
\end{lemma}

\begin{proof}
	Most steps are similar to the proof of \eqref{E:DIVMETRICCURRENT}, so we only sketch the proof.
	For example, when time derivatives fall on the
	factors
	$
	\mathscr{Z}^{\vec{I}} \newtimescalar
	$
	or
	$
		\SigmatLie_{\mathscr{Z}}^{\vec{I}} \newspacescalar
	$
	on RHS~\eqref{E:SFCURRENT0},
	we use equations
	\eqref{E:COMMUTEDWAVEEQUATIONRENORMALIZED}-\eqref{E:COMMUTEDSPACEDERIVATIVESWAVEEQUATIONRENORMALIZED}
	to substitute, in analogy with the way that we used
	equations 
	\eqref{E:COMMUTEDEVOLUTIONSECONDFUNDRENORMALIZED}-\eqref{E:GRADCOMMUTEDEVOLUTIONMETRICRENORMALIZED}
	to substitute in our proof of \eqref{E:DIVMETRICCURRENT}.
	The new feature of the present proof is that we 
	use equation
	\eqref{E:COMMUTEDLAPSEPDERENORMALIZEDHIGHERDERIVATIVES}
	to express the term
	$
	\scale' \scale^3 \Gdiv 
			\left\lbrace 
				(\nabla^{\#} \mathscr{Z}^{\vec{I}} \newlapse) \mathscr{Z}^{\vec{I}} \newlapse
			\right\rbrace
	$
	on LHS~\eqref{E:DIVSFCURRENT} as follows:
	\begin{align}
	& \scale' \scale^3  |\nabla \mathscr{Z}^{\vec{I}} \newlapse|_{\newg}^2
		\label{E:KEYDIVERGENCESUBSTUTION} \\
	&
	+
	\left\lbrace
				(\scale')^3 \scale^{5/3}
				+ \frac{2}{3} \scale' \scale^3
				+ \scale' \scale^{5/3} |\FreeNewSec|_{\newg}^2
				+ 2 \sqrt{\frac{2}{3}} \scale' \scale^{5/3} \newtimescalar
				+ \scale' \scale^{5/3} \newtimescalar^2 
	\right\rbrace
	\left|
		\mathscr{Z}^{\vec{I}} \newlapse
	\right|^2
		\notag \\
	& \ \
	+
	\left\lbrace
			2 \sqrt{\frac{2}{3}} \scale' \scale^{1/3} \mathscr{Z}^{\vec{I}} \newtimescalar
			+ 
			\scale' \scale^{1/3} \CommutedLapseHighBorderInhom{\vec{I}}
			+ 
			\scale' \scale^{5/3} \CommutedLapseHighJunkInhom{\vec{I}}
	\right\rbrace
	\mathscr{Z}^{\vec{I}} \newlapse.
		\notag
\end{align}
We now highlight the three most important aspects of the expression 
\eqref{E:KEYDIVERGENCESUBSTUTION}. 
First, the first product in \eqref{E:KEYDIVERGENCESUBSTUTION}
is precisely the negative-semi-definite 
second product on RHS~\eqref{E:DIVSFCURRENT}
(recall that by Lemma~\ref{L:ANALYSISOFFRIEDMANN}, $\scale'(t) \leq 0$ for $t \in [0,\TCrunch]$).
Next, the second line of
\eqref{E:KEYDIVERGENCESUBSTUTION} yields the 
negative-semi-definite products on the third line of RHS~\eqref{E:DIVSFCURRENT}.
Finally, there is a critically important cancellation that occurs.
Specifically, the first product
$2 \sqrt{\frac{2}{3}} \scale' \scale^{1/3} 
(\mathscr{Z}^{\vec{I}} \newtimescalar)
\mathscr{Z}^{\vec{I}} \newlapse
$
on the last line of \eqref{E:KEYDIVERGENCESUBSTUTION}
\emph{exactly cancels} the product
$
-
2 \sqrt{\frac{2}{3}} \scale' \scale^{1/3} 
(\mathscr{Z}^{\vec{I}} \newtimescalar)
\mathscr{Z}^{\vec{I}} \newlapse
$
that arises when
$\partial_t$ falls on the factor 
$(\mathscr{Z}^{\vec{I}} \newtimescalar)^2$
on RHS~\eqref{E:SFCURRENT0}
and we use equation \eqref{E:COMMUTEDWAVEEQUATIONRENORMALIZED}
to substitute for
$\partial_t \mathscr{Z}^{\vec{I}} \newtimescalar$.
We place the second product
$\scale' \scale^{1/3} (\mathscr{Z}^{\vec{I}} \newlapse) \CommutedLapseHighBorderInhom{\vec{I}}$
from the last line of \eqref{E:KEYDIVERGENCESUBSTUTION}
on RHS~\eqref{E:SFCURRENTBORDER}.
The remaining product 
$\scale' \scale^{5/3} (\mathscr{Z}^{\vec{I}} \newlapse) \CommutedLapseHighJunkInhom{\vec{I}}$
from \eqref{E:KEYDIVERGENCESUBSTUTION}
is a negligible error term that we place on
RHS~\eqref{E:SFCURRENTJUNK}.

\end{proof}

\section{Integrals, norms, and energies}
\label{S:INTEGRALSNORMSANDENERGIES}
In this section, 
we define the norms and energies that we use to control the solution.

\subsection{Integrals and norms}
\label{SS:INTEGRALSANDNORMS}

\begin{definition}[$L^2$ norms \textbf{with respect to two volume forms}]
\label{D:L2NORMS}
Let $d \Sttvol$ and $d \tvol$ be the volume forms on $\Sigma_t$ from
Def.\ \ref{D:VOLUMEFORM}. If $f$ is a scalar function defined along $\Sigma_t$, then we define
\begin{align} \label{E:L2NORMS}
		\| 
			f
		\|_{L_{\StMet}^2(\Sigma_t)}^2
		:=
		\int_{\Sigma_t}
			f^2 
		\, d \Sttvol,
	&&
		\| 
			f
		\|_{L_{\newg}^2(\Sigma_t)}^2
		:=
		\int_{\Sigma_t}
			f^2 
		\, d \tvol.
\end{align}
\end{definition}

\begin{definition} [\textbf{$H^M$ and $C^M$ norms}]  \label{D:NORMS}
	Let $\StMet$ be the round metric on $\Sigma_t$ from Subsect.\ \ref{SS:EXTENDINGSIGMA0TANGENTTENSORFIELDSTOSIGMATTANGENTTENSORFIELDS}
	and let $\newg$ be the rescaled spatial metric of Def.\ \ref{D:RESCALEDVARIABLES}.
	We define the following norms for $\Sigma_t-$tangent tensorfields $\xi$.
\begin{subequations}
\begin{align}
	\| \xi \|_{H_{\StMet}^M(\Sigma_t)}^2
	& := 
		\sum_{|\vec{I}| \leq M} 
		\left\| \left|\SigmatLie_{\mathscr{Z}}^{\vec{I}} \xi \right|_{\StMet} \right \|_{L_{\StMet}^2(\Sigma_t)}^2, 
	&& 
	\| \xi \|_{H_{\newg}^M(\Sigma_t)}^2 
		:= \sum_{|\vec{I}| \leq M} 
		\left\| \left|\SigmatLie_{\mathscr{Z}}^{\vec{I}} \xi \right|_{\newg} \right \|_{L_{\newg}^2(\Sigma_t)}^2, 
			\label{E:SOBOLEVNORMDEFS} \\
	\| \xi \|_{\dot{H}_{\StMet}^M(\Sigma_t)}^2
	& := 
		\sum_{|\vec{I}| = M} 
		\left\| \left|\SigmatLie_{\mathscr{Z}}^{\vec{I}} \xi \right|_{\StMet} \right \|_{L_{\StMet}^2(\Sigma_t)}^2, 
	&& 
	\| \xi \|_{\dot{H}_{\newg}^M(\Sigma_t)}^2 
		:= \sum_{|\vec{I}| = M} 
		\left\| \left|\SigmatLie_{\mathscr{Z}}^{\vec{I}} \xi \right|_{\newg} \right \|_{L_{\newg}^2(\Sigma_t)}^2, 
			\label{E:HOMSOBOLEVNORMDEFS} \\
	\| \xi \|_{C_{\StMet}^M(\Sigma_t)}^2
	& := \sum_{|\vec{I}| \leq M} \sup_{p \in \Sigma_t} \left|\SigmatLie_{\mathscr{Z}}^{\vec{I}} \xi(p) \right|_{\StMet}^2, 
	&& \| \xi \|_{C_{\newg}^M(\Sigma_t)}^2
		:= \sum_{|\vec{I}| \leq M} \sup_{p \in \Sigma_t} 
		\left| \SigmatLie_{\mathscr{Z}}^{\vec{I}} 
		\xi(p) \right|_{\newg}^2,
		\label{E:SUPNORMDEFS}
			\\
	\| \xi \|_{\dot{C}_{\StMet}^M(\Sigma_t)}^2
	& := \sum_{|\vec{I}| = M} \sup_{p \in \Sigma_t} \left|\SigmatLie_{\mathscr{Z}}^{\vec{I}} \xi(p) \right|_{\StMet}^2, 
	&& \| \xi \|_{\dot{C}_{\newg}^M(\Sigma_t)}^2
		:= \sum_{|\vec{I}| = M} \sup_{p \in \Sigma_t} 
		\left| \SigmatLie_{\mathscr{Z}}^{\vec{I}} 
		\xi(p) \right|_{\newg}^2.
		\label{E:HOMSUPNORMDEFS}
\end{align}
\end{subequations}
\end{definition}

Our bootstrap assumption for the solution (see Subsect.\ \ref{SS:BOOTSTRAPASSUMPTIONS}) 
is for the following solution norm;
see Remark~\ref{R:NUMBEROFDERIVATIVES} regarding the number of derivatives that it controls.

\begin{definition}[\textbf{High-norm for the solution}]
	\label{D:HIGHNORM}
	We define
	\begin{align}
		\highnorm{16}(t)
		& :=
			\left\|
				\FreeNewSec
			\right\|_{H_{\StMet}^{16}(\Sigma_t)}
			+
			\scale^{2/3}(t)
			\left\|
				\newg
			\right\|_{\dot{H}_{\StMet}^{17}(\Sigma_t)}
			+
			\left\|
				\newg - \StMet
			\right\|_{H_{\StMet}^{16}(\Sigma_t)}
			+
			\left\|
				\newg^{-1} - \StMet^{-1}
			\right\|_{H_{\StMet}^{16}(\Sigma_t)}
			\label{E:HIGHNORM}	\\
			& \ \
			+
			\sum_{L=1}^4
			\scale^{(2/3)L}(t)
			\left\|
				\newlapse
			\right\|_{\dot{H}_{\StMet}^{14 + L}(\Sigma_t)}
			+
			\left\|
				\newlapse
			\right\|_{H_{\StMet}^{14}(\Sigma_t)}
			+
			\left\|
				\newtimescalar
			\right\|_{H_{\StMet}^{16}(\Sigma_t)}
			+
			\scale^{2/3}(t)
			\left\|
				\newspacescalar
			\right\|_{H_{\StMet}^{16}(\Sigma_t)}
			+
			\left\|
				\newspacescalar
			\right\|_{H_{\StMet}^{15}(\Sigma_t)}
				\notag 
					\\
		& \ \
			+
			\left\|
				\FreeNewSec
			\right\|_{C_{\StMet}^{14}(\Sigma_t)}
			+
			\left\|
				\newg - \StMet
			\right\|_{C_{\StMet}^{14}(\Sigma_t)}
			+
			\left\|
				\newg^{-1} - \StMet^{-1}
			\right\|_{C_{\StMet}^{14}(\Sigma_t)}
				\notag \\
		& \ \
		+
			\left\|
				\newlapse
			\right\|_{C_{\StMet}^{12}(\Sigma_t)}
			+
			\left\|
				\newtimescalar
			\right\|_{C_{\StMet}^{14}(\Sigma_t)}
			+
			\left\|
				\newspacescalar
			\right\|_{C_{\StMet}^{13}(\Sigma_t)}.
			\notag
	\end{align}
\end{definition}

\begin{remark}
	Note that $\highnorm{16}(t) \equiv 0$ for the FLRW solution.
	Thus, $\highnorm{16}(t)$ is a measure of the perturbed solution's
	deviation from the FLRW solution.
\end{remark}

\begin{remark}[\textbf{On the number of derivatives}]
	\label{R:NUMBEROFDERIVATIVES}
	In the present article, our solution norm \eqref{E:HIGHNORM}
	involves more derivatives than the corresponding solution norms in
	\cites{iRjS2014a,iRjS2014b}. We have chosen to increase the number of 
	derivatives in the present article because it allows us to simplify some of
	the analysis.
\end{remark}

\begin{definition}[\textbf{Summed pointwise norms and seminorms}]
	\label{D:SUMMEDPOINTWISESEMINORMS}
	If $M \geq 0$ is an integer and $\mathfrak{D}$ is a differential operator 
	(typically $\mathfrak{D} = \SigmatLie_{\partial_t}$ or $\mathfrak{D} = \nabla^L$
	for some exponent $L$), 
	then
	\begin{align}
			\left|
				\SigmatLie_{\mathscr{Z}}^{\leq M} \xi
			\right|_{\newg}
			& 
			:=
			\sum_{|\vec{I}| \leq M}
			\left|
				\SigmatLie_{\mathscr{Z}}^{\vec{I}} \xi
			\right|_{\newg},
			&&
			\left|
				\mathfrak{D} \SigmatLie_{\mathscr{Z}}^{\leq M} \xi
			\right|_{\newg}
			:=
			\sum_{|\vec{I}| \leq M}
			\left|
				\mathfrak{D} \SigmatLie_{\mathscr{Z}}^{\vec{I}} \xi
			\right|_{\newg}.
	\end{align}
	Similarly, if $M \geq 1$, then
	\begin{align}
			\left|
				\SigmatLie_{\mathscr{Z}}^{[1,M]} \xi
			\right|_{\newg}
			& 
			:=
			\sum_{1 \leq |\vec{I}| \leq M}
			\left|
				\SigmatLie_{\mathscr{Z}}^{\vec{I}} \xi
			\right|_{\newg},
			&&
			\left|
				\mathfrak{D} \SigmatLie_{\mathscr{Z}}^{[1,M]} \xi
			\right|_{\newg}
			:=
			\sum_{1 \leq |\vec{I}| \leq M}
			\left|
				\mathfrak{D} \SigmatLie_{\mathscr{Z}}^{\vec{I}} \xi
			\right|_{\newg}.
	\end{align}
	We typically write
	$
	\left|
		\mathscr{Z}^{[1,M]} f
	\right|
	$
	instead of
	$\left|
				\SigmatLie_{\mathscr{Z}}^{[1,M]} f
			\right|_{\newg}
	$
	when $f$ is a scalar function.
	Moreover, we sometimes write 
	$
	\left|
		\mathscr{Z} f
	\right|
	$
	instead of
	$
	\left|
		\mathscr{Z}^{[1,1]} f
	\right|
	$.
	We also define similar seminorms in which the metric $\newg$
	is replaced by the metric $\StMet$.
\end{definition}

\begin{definition}[\textbf{Summed sup-seminorms}]
	\label{D:SUMMEDSUPNORMS}
	If $M \geq 0$ is an integer and $\mathfrak{D}$ is a differential operator 
	(typically $\mathfrak{D} = \SigmatLie_{\partial_t}$ or $\mathfrak{D} = \nabla^L$
	for some exponent $L$),
	then
	\begin{align}
			\left\|
				\SigmatLie_{\mathscr{Z}}^{\leq M} \xi
			\right\|_{C_{\newg}^0(\Sigma_t)}
			& 
			:=
			\sum_{|\vec{I}| \leq M}
			\sup_{p \in \Sigma_t}
			\left|
				\SigmatLie_{\mathscr{Z}}^{\vec{I}} \xi(p)
			\right|_{\newg},
			&&
			\left\|
				\mathfrak{D} \SigmatLie_{\mathscr{Z}}^{\leq M} \xi
			\right\|_{C_{\newg}^0(\Sigma_t)}
			:=
			\sum_{|\vec{I}| \leq M}
			\sup_{p \in \Sigma_t}
			\left|
				\mathfrak{D} \SigmatLie_{\mathscr{Z}}^{\vec{I}} \xi(p)
			\right|_{\newg}.
	\end{align}
	Similarly, if $M \geq 1$, then
	\begin{align}
			\left\|
				\SigmatLie_{\mathscr{Z}}^{[1,M]} \xi
			\right\|_{C_{\newg}^0(\Sigma_t)}
			& 
			:=
			\sum_{1 \leq |\vec{I}| \leq M}
			\sup_{p \in \Sigma_t}
			\left\|
				\SigmatLie_{\mathscr{Z}}^{\vec{I}} \xi(p)
			\right\|_{C_{\newg}^0(\Sigma_t)},
			&&
			\left\|
				\mathfrak{D} \SigmatLie_{\mathscr{Z}}^{[1,M]} \xi
			\right\|_{C_{\newg}^0(\Sigma_t)}
			:=
			\sum_{1 \leq |\vec{I}| \leq M}
			\sup_{p \in \Sigma_t}
			\left|
				\mathfrak{D} \SigmatLie_{\mathscr{Z}}^{\vec{I}} \xi(p)
			\right|_{\newg}.
	\end{align}
	We analogously define similar seminorms in which the metric $\newg$
	is replaced by the metric $\StMet$
	or in which the sup norm $\| \cdot \|_{C_{\newg}^0(\Sigma_t)}$
	is replaced with the Lebesgue norm
	$\| \cdot \|_{L_{\newg}^2(\Sigma_t)}$
	or
	$\| \cdot \|_{L_{\StMet}^2(\Sigma_t)}$.
\end{definition}

The norms and seminorms in the next definition complement those of
Def.\ \ref{D:SUMMEDSUPNORMS}.

\begin{definition}[\textbf{Spatial derivative norms and seminorms}]
	\label{D:SPATIALDERIVATIVENORMSANDSEMINORMS}
\label{D:SCHEMATICDIFF}
		$\difarg{M}$ denotes an arbitrary $M^{th}$-order differential operator corresponding to repeated differentiation
		with respect to the connection $\nabla$ and/or the $\Sigma_t$-projected Lie derivative operators
		$\SigmatLie_Z$ with $Z \in \mathscr{Z}$, where no more than two differentiations
		with respect to $\nabla$ are taken.\footnote{We only need to consider at most two $\nabla$ differentiations since
		this is the maximum number of $\nabla$ differentiations that occur in the Einstein-scalar field equations.}
		That is, $\difarg{M}$ is any operator of the form
		\begin{align}
			\difarg{M} 
			= 
		\prod_{i=1}^M O_i,
		\end{align}
		where $O_i \in \lbrace \nabla, \SigmatLie_{Z_{(1)}}, \SigmatLie_{Z_{(2)}}, \SigmatLie_{Z_{(3)}} \rbrace$
		and no more than two of the $O_i$ are equal to $\nabla$.
		
		If $M \geq 0$ is an integer and $\xi$ is a $\Sigma_t$-tangent tensorfield, then
		\begin{align} \label{E:DTOMSUMMED}
			\left\|
				\dif^{\leq M} \xi
			\right\|_{C_{\newg}^0(\Sigma_t)}
			:=
			\sum_{M' \leq M}
				\sup_{p \in \Sigma_t}
				\left|
					\dif^{M'} \xi(p)
				\right|_{\newg},
		\end{align}
		where $\sum_{M' \leq M}$ means that the sum is over all operators of type $\difarg{M'}$ with $M' \leq M$.
		Similarly,
		if $M \geq 1$, then
		\begin{align} \label{E:DFROMONETOMSUMMED}
			\left\|
				\dif^{[1,M]} \xi
			\right\|_{C_{\newg}^0(\Sigma_t)}
			:=
			\sum_{1 \leq M' \leq M}
				\sup_{p \in \Sigma_t}
				\left|
					\dif^{M'} \xi(p)
				\right|_{\newg},
		\end{align}
		where $\sum_{1 \leq M' \leq M}$ means that the sum is over all operators of type $\difarg{M'}$ with $1 \leq M' \leq M$.
		
		We similarly define 
		$
		\left\|
				\dif^{\leq M} \xi
			\right\|_{C_{\StMet}^0(\Sigma_t)}
		$
		and
		$
		\left\|
				\dif^{[1,M]} \xi
			\right\|_{C_{\StMet}^0(\Sigma_t)}
		$	
		by replacing the metric $\newg$ in \eqref{E:DTOMSUMMED}-\eqref{E:DFROMONETOMSUMMED}
		with $\StMet$.
		
\end{definition}

\subsection{Energies}
\label{SS:ENERGIES}
To control the solution and to derive a priori estimates for the norm 
$\highnorm{16}$, we will primarily rely on the following energies.

\begin{definition}[\textbf{Energies}]
\label{D:ENERGIES}
	Let $\Jfour_{(Metric)}[\cdots]$ and $\Jfour_{(Sf)}[\cdots]$ be the energy currents from
	Def.\ \ref{D:ENERGYCURRENTS} and let $\smallparameter > 0$ be a real parameter. We define
	\begin{subequations}
	\begin{align}
		\Metricenergy{M}(t)
		& :=
			\sum_{1 \leq |\vec{I}| \leq M}
			\int_{\Sigma_t}	
				J_{(Metric)}^0[\SigmatLie_{\mathscr{Z}}^{\vec{I}} \FreeNewSec,\SigmatLie_{\mathscr{Z}}^{\vec{I}} \newg]
			\, d \tvol,
				\label{E:METRICENERGY}
				\\
		\Sfenergy{M}(t)
		& :=
			\sum_{1 \leq |\vec{I}| \leq M}
			\int_{\Sigma_t}	
				J_{(Sf)}^0[\mathscr{Z}^{\vec{I}} \newtimescalar, \SigmatLie_{\mathscr{Z}}^{\vec{I}}\newspacescalar]
			\, d \tvol,
			\label{E:SFENERGY}
				\\
		\Totalenergy{M}{\smallparameter}(t)
		& := 
			\smallparameter \Metricenergy{M}(t)
			+
			\Sfenergy{M}(t),
			\label{E:TOTALENERGY}
				\\
	\SupTotalenergy{M}{\smallparameter}(t)
	&
	:= 
	\sup_{s \in [0,t]}
	\Totalenergy{M}{\smallparameter}(s).
	\label{E:SUPTOTALENERGY}
	\end{align}
	\end{subequations}
\end{definition}

We now quantify the coerciveness of the energies.

\begin{lemma}[\textbf{Coerciveness of the energies}]
\label{L:ENERGYCOERCIVENESS}
The energies 
$\Metricenergy{M}(t)$
and $\Sfenergy{M}(t)$
from Def.\ \ref{D:ENERGIES}
are coercive in the following sense:
\begin{subequations}
\begin{align}
	\Metricenergy{M}(t)
	& = 
			\sum_{1 \leq |\vec{I}| \leq M}
			\left\|
				\SigmatLie_{\mathscr{Z}}^{\vec{I}} \FreeNewSec
			\right\|_{L_{\newg}^2(\Sigma_t)}^2
			+ \frac{1}{4} 
				\sum_{1 \leq |\vec{I}| \leq M}
				\scale^{4/3} 
				\left\|
					\nabla \SigmatLie_{\mathscr{Z}}^{\vec{I}} \newg
				\right\|_{L_{\newg}^2(\Sigma_t)}^2,
				\label{E:METRICENERGYCOERCIVENESS}	\\
	\Sfenergy{M}(t)
	& = 
			\sum_{1 \leq |\vec{I}| \leq M}
			\left\|
				\mathscr{Z}^{\vec{I}} \newtimescalar
			\right\|_{L_{\newg}^2(\Sigma_t)}^2
			+ 
			\scale^{4/3}(t)
			\sum_{1 \leq |\vec{I}| \leq M}
			\left\|
				\SigmatLie_{\mathscr{Z}}^{\vec{I}} \newspacescalar
			\right\|_{L_{\newg}^2(\Sigma_t)}^2.
			\label{E:SFENERGYCOERCIVENESS}
\end{align}
\end{subequations}

\end{lemma}

\begin{proof}
	The lemma is a trivial consequence of
	Def.\ \ref{D:ENERGIES} and the definitions
	\eqref{E:METRICCURRENT0} and \eqref{E:SFCURRENT0}.
\end{proof}

\section{Assumptions on the data and bootstrap assumptions}
\label{S:DATAASSUMPTIONSANDBOOTSTRAPASSUMPTIONS}
In this short section, we provide our smallness assumptions on the initial
data and state the bootstrap assumptions that we will use
when deriving a priori estimates.

\subsection{Smallness assumptions on the data}
\label{SS:DATAASSUMPTIONS}
We assume that the data verify the following smallness assumptions:
\begin{align} \label{E:SMALLDATA}
	\highnorm{16}(0) := \varepsilon^2,
\end{align}
where $\highnorm{16}$ is defined in \eqref{E:HIGHNORM}
and $\varepsilon < 1$ is a small positive number. 
We will shrink the allowable size of $\varepsilon$ 
as the paper progresses.

\begin{remark}
	The smallness of 
	$
	\left\|
		\SigmatLie_{\mathscr{Z}}^{\leq 16} (\newg^{-1} - \StMet^{-1})
	\right\|_{L_{\StMet}^2(\Sigma_0)}
	$
	implied by \eqref{E:SMALLDATA}
	is redundant in the sense that it could be derived
	as an easy consequence of the smallness
	for 
	$
	\left\|
		\SigmatLie_{\mathscr{Z}}^{\leq 16} (\newg - \StMet)
	\right\|_{L_{\StMet}^2(\Sigma_0)}
	$
	implied by \eqref{E:SMALLDATA}.
	In addition, the smallness of
	$
	\left\|
		\mathscr{Z}^{\leq 18} \newlapse
	\right\|_{L_{\StMet}^2(\Sigma_0)}
	$
	implied by \eqref{E:SMALLDATA}
	is redundant in the sense that it could be
	derived from the smallness of the remaining variables
	by virtue of elliptic estimates of the type 
	that we derive in Prop.~\ref{P:BOUNDFORLAPSEANDBELOWTOPMETRICINTERMSOFENERGIES}.
\end{remark}

\subsection{Bootstrap assumptions}
\label{SS:BOOTSTRAPASSUMPTIONS}
We assume that the following estimates hold for $t \in [0,\Tboot)$:
\begin{align} 
		\highnorm{16}(t)
		& \leq \varepsilon \scale^{-\upsigma}(t),
		\label{E:HIGHNORMBOOTSTRAP}
\end{align}
where $\highnorm{16}(t)$ is defined in \eqref{E:HIGHNORM},
$\varepsilon$ is as in \eqref{E:SMALLDATA},
and $\upsigma > 0$ is a small bootstrap parameter 
that is constrained in particular by
\begin{align} \label{E:PARAMETERCONSTRAINT}
	0 < \varepsilon^{1/2} \leq \upsigma < 1.
\end{align}
We will adjust the allowable size of $\varepsilon$
and $\upsigma$ throughout our analysis.

\begin{remark}[\textbf{We tacitly assume the bootstrap assumptions and sufficient smallness of the bootstrap parameters}]
	\label{R:TACITBOOTSTRAP}
	In the remainder of the article, 
	we will assume that the above bootstrap assumptions hold
	and that $\varepsilon$ and $\upsigma$ are sufficiently small.
	In particular, in statements of propositions and lemmas in which we derive
	estimates, we do not explicitly state that we are assuming the
	bootstrap assumptions. Moreover, it is understood that all of the estimates
	that we derive hold on the domain $[0,\Tboot) \times \mathbb{S}^3$.
\end{remark}

\begin{remark}[\textbf{Quadratically small data out of convenience}]
	\label{R:QUADRATICSMALLNESSOFDATA}
	Note that our bootstrap assumption \eqref{E:HIGHNORMBOOTSTRAP}
	states that $\highnorm{16}(t)$ is \emph{linearly} small in
	$\varepsilon$ while our data assumption \eqref{E:SMALLDATA}
	for $\highnorm{16}(0)$ posits \emph{quadratic} smallness.
	This is mainly for technical convenience.
\end{remark}

\section{Preliminary inequalities and weak sup-norm estimates}
\label{S:PRELIMINEQUALITIES}
In this section, we derive some preliminary estimates.
The estimates form the starting point for our
pointwise analysis of solutions and the error terms
in the $\SigmatLie_{\mathscr{Z}}^{\vec{I}}$-commuted equations.

\subsection{Preliminary comparison estimates}
In this subsection, we prove some comparison estimates.
The main result is Lemma~\ref{L:PRELIMINARYCOMPARISON}.
We start by providing a simple identity.

\begin{lemma}[\textbf{Covariant derivatives in terms of Lie derivatives}]
	\label{L:VECTORFIELDCOVARIANTDERIVATIVEID}
	For any three vectorfields $X,Y,Z \in \mathscr{Z}$ 
	(see \eqref{E:AGAINROUNDMETRICKILLING}),
	we have the following identity:
	\begin{align} \label{E:VECTORFIELDCOVARIANTDERIVATIVEID}
	\newg(\nabla_X Y,Z)
	 & = \frac{1}{2}  
	 	\left\lbrace
	 		Y \otimes Z \cdot \SigmatLie_X \newg
	 		+
	 		Z \otimes X \cdot \SigmatLie_Y \newg
	 		- 
	 		X \otimes Y \cdot \SigmatLie_Z \newg
	 	\right\rbrace
	 	 \\
	 	& \ \
	 	+
	 	\frac{1}{2}  
		\left\lbrace
			\newg(\SigmatLie_X Y,Z)
			+
			\newg(\SigmatLie_Y Z,X)
			-
			\newg(\SigmatLie_Z X,Y)
		\right\rbrace.
		\notag
	\end{align}
\end{lemma}
\begin{proof}
	First, using the Leibniz rule and the torsion-free property of $\nabla$, we
	compute that
	\begin{align} \label{E:LEIBNIZTORSIONFREECALCULATION}
	\newg(\nabla_Y X,Z) + \newg(X,\nabla_Y Z)
	& 
	= \nabla_Y [\newg(X,Z)] 
	= \SigmatLie_Y [\newg(X,Z)]
		\\
	& =
	X \otimes Z \cdot \SigmatLie_Y \newg
	+
	\newg(\SigmatLie_Y X,Z)
	+
	\newg(X,\SigmatLie_Y Z)
		\notag \\
	& =
	X \otimes Z \cdot \SigmatLie_Y \newg
	+
	\newg(\nabla_Y X,Z)
	-
	\newg(\nabla_X Y,Z)
	+
	\newg(X,\nabla_Y Z)
	-
	\newg(X,\nabla_Z Y).
	\notag
	\end{align}
	From \eqref{E:LEIBNIZTORSIONFREECALCULATION},
	we deduce that
	$
	\newg(\nabla_X Y,Z)
	= 	X \otimes Z \cdot \SigmatLie_Y \newg
	-
	\newg(X,\nabla_Z Y)
	$
	and hence
	\begin{align} \label{E:EASYLEIBNIZTORSIONFREECALCULATION}
	2 \newg(\nabla_X Y,Z) 
	=
	X \otimes Z \cdot \SigmatLie_Y \newg
	+ 
	\newg(\nabla_X Y,Z) 
	- 
	\newg(\nabla_Z Y,X).
	\end{align}
	Using similar arguments, we decompose the last two terms on RHS~\eqref{E:EASYLEIBNIZTORSIONFREECALCULATION} as follows:
	\begin{align} \label{E:SECONDLEIBNIZTORSIONFREECALCULATION}
		\newg(\nabla_X Y,Z) 
		- 
		\newg(\nabla_Z Y,X)
		&  
		= \nabla_X [\newg(Y,Z)]
		- \nabla_Z [\newg(Y,X)]
		- \newg(\SigmatLie_X Z,Y)
			\\
	& = Y \otimes Z \cdot \SigmatLie_X \newg
		+ \newg(\SigmatLie_X Y,Z)
		+ \newg(\SigmatLie_X Z,Y)
		- X \otimes Y \cdot \SigmatLie_Z \newg
			\notag \\
	& \ \
		+ \newg(\SigmatLie_Y Z, X)
		- \newg(\SigmatLie_Z X,Y) 
		- \newg(\SigmatLie_X Z,Y).
		\notag
	\end{align}
	Substituting RHS~\eqref{E:SECONDLEIBNIZTORSIONFREECALCULATION}
	for the last two terms on RHS~\eqref{E:EASYLEIBNIZTORSIONFREECALCULATION}
	and noting that 
	$- \newg(\SigmatLie_Z X,Y) 
		- \newg(\SigmatLie_X Z,Y)
	=0
	$,
	we arrive at \eqref{E:VECTORFIELDCOVARIANTDERIVATIVEID}.
\end{proof}

We now compare the strength of the norms
$|\cdot|_{\StMet}$ and $|\cdot|_{\newg}$.

\begin{lemma}[\textbf{Preliminary comparison estimates}]
	\label{L:PRELIMINARYCOMPARISON}
	There exist constants $C >1$ and $c > 1$ such that 
	the following estimates hold for $\Sigma_t$-tangent tensors $\xi$:
	\begin{align} \label{E:GNORMSTMETRICCOMPARISON}
		\frac{1}{C} \scale^{c \upsigma}
			\left|
				\xi	
			\right|_{\StMet}
		& \leq
		\left|
			\xi
		\right|_{\newg}	
		\leq C \scale^{- c \upsigma}
			\left|
				\xi	
			\right|_{\StMet}.
	\end{align}
	
	Moreover, let $\xi$ be a type $\binom{l}{m}$ $\Sigma_t-$tangent tensor,
	let $\theta^1, \cdots, \theta^l \in \Theta$
	and
	$Z_1, \cdots, Z_m \in \mathscr{Z}$
	(see Def.\ \ref{D:LIETRANSPORTEDSPATIALFRAME}),
	and let
	$\xi_{Z_1 \cdots Z_m}^{\theta^1 \cdots \theta^l}
	:= 
	\xi_{b_1 \cdots b_m}^{a_1 \cdots a_l}
	\theta_{a_1}^1 \cdots \theta_{a_l}^l
	Z_1^{b_1} \cdots Z_m^{b_m}
	$
	denote the contraction of $\xi$ against
	$\theta^1, \cdots, \theta^l$ 
	and
	$Z_1, \cdots, Z_m$.
	The following pointwise estimate holds:
	\begin{align} \label{E:GNORMCONTRACTIONCOMPARISON}
		\frac{1}{C} \scale^{c \upsigma}
			\mathop{\sum_{\theta^1,\cdots,\theta^l \in \Theta}}_{Z_1,\cdots,Z_m \in \mathscr{Z}}
				\left|
					\xi_{Z_1 \cdots Z_m}^{\theta^1 \cdots \theta^l}	
				\right|
		& \leq
		\left|
			\xi
		\right|_{\newg}	
		\leq C \scale^{- c \upsigma}
			\mathop{\sum_{\theta^1,\cdots,\theta^l \in \Theta}}_{Z_1,\cdots,Z_m \in \mathscr{Z}}
				\left|
					\xi_{Z_1 \cdots Z_m}^{\theta^1 \cdots \theta^l}	
				\right|.
	\end{align}
	
	Finally, if $\nabla^L$ is the $L^{th}$ order covariant derivative operator,
	then for $L = 1,2$, we have
	\begin{align} \label{E:OPERATORCOMPARISON}
		\left|
			\nabla^L \xi
		\right|_{\newg}
		& \leq C
			\scale^{- c \upsigma} 
			\left|
				\SigmatLie_{\mathscr{Z}}^{\leq L} \xi
			\right|_{\newg},
		&&
		\left|
			\SigmatLie_{\mathscr{Z}}^L \xi
		\right|_{\newg}
		\leq 
		C \scale^{- c \upsigma} 	
		\left|
			\nabla^{\leq L} \xi
		\right|_{\newg}.
	\end{align}
\end{lemma}

\begin{proof}
	We start by making a basic remark (which is justified by our proof) 
	that will later be important for our proof of Cor.~\ref{C:IMPROVEMENTLEMMAOPERATORCOMPARISON}:
	our proof of the lemma relies only on the bounds
	$|\newg^{-1} - \StMet^{-1}|_{\StMet} \lesssim \scale^{- c \upsigma}$
	and
	$\left|
			\SigmatLie_{\mathscr{Z}}^{\leq 2}  (\newg - \StMet)
	\right|_{\StMet}
	\lesssim \scale^{- c \upsigma}
	$,
	which are a simple consequence 
	of \eqref{E:HIGHNORMBOOTSTRAP}.
	Throughout the proof, we refer to these bounds as ``the bootstrap assumptions.''
	
	We first prove \eqref{E:GNORMSTMETRICCOMPARISON}.
	We give the proof only for type $\binom{0}{1}$
	tensors since the case of general $\binom{l}{m}$
	tensors can be handled similarly.
	From Cauchy--Schwarz relative to $\StMet$, we deduce
	$\left| \xi \right|_{\newg}^2 
		= 
		(\StMet^{-1})^{ab} \xi_a \xi_b
		+
		\left\lbrace
			(\newg^{-1})^{ab} 
			- 
			(\StMet^{-1})^{ab}
		\right\rbrace
		\xi_a \xi_b
	\leq 
	\left\lbrace
		1 
		+
		C
		|\newg^{-1} - \StMet^{-1}|_{\StMet}
	\right\rbrace
	|\xi|_{\StMet}^2
$.
	The second inequality in \eqref{E:GNORMSTMETRICCOMPARISON} then follows from
	the bootstrap assumptions.
	To derive the first inequality in 
	\eqref{E:GNORMSTMETRICCOMPARISON},
	we first use Cauchy--Schwarz relative to $\newg$
	to deduce
	$\left| \xi \right|_{\StMet}^2 
		= (\StMet^{-1})^{ab} \xi_a \xi_b
		\lesssim |\StMet^{-1}|_{\newg} |\xi|_{\newg}^2$.
	We then use Cauchy--Schwarz relative to $\StMet$ 
	and the bootstrap assumptions
	to obtain
	$|\StMet^{-1}|_{\newg}
	= \left\lbrace
			\newg_{ab} \newg_{cd} (\StMet^{-1})^{ab} (\StMet^{-1})^{cd}
	\right\rbrace^{1/2}
	\lesssim |\newg|_{\StMet}
	\lesssim 1 + |\newg - \StMet|_{\StMet}
	\lesssim \scale^{- c \upsigma}
	$.
	Combining this estimate with the previous one,
	we conclude the desired lower bound \eqref{E:GNORMSTMETRICCOMPARISON}.
	
	Inequality \eqref{E:GNORMCONTRACTIONCOMPARISON}
	then follows from
	\eqref{E:GNORMSTMETRICCOMPARISON}
	and the following
	simple consequence of Cor.~\ref{C:STMETRICFRAMEEXPANDED}:
	$
	|\xi|_{\StMet}
	\approx
	\mathop{\sum_{\theta^1,\cdots,\theta^l \in \Theta}}_{Z_1,\cdots,Z_m \in \mathscr{Z}}
	\left|
	\xi_{Z_1 \cdots Z_m}^{\theta^1 \cdots \theta^l}	
	\right|
	$.
	
	We now prove \eqref{E:OPERATORCOMPARISON} in the case $L = 1$.
	Again, we give the proof only for type $\binom{0}{1}$
	tensorfields since the case of general $\binom{l}{m}$
	tensorfields can be handled similarly.
	We first show that
	$
	\left|
		\nabla Z
	\right|_{\newg}
	\lesssim 
	\scale^{- c \upsigma} 
	$.
	To this end,
	we consider the type $\binom{0}{2}$ tensorfield
	$\nabla Z_{\flat}$ that is $\newg$-dual to $\nabla Z$.
	By \eqref{E:GNORMCONTRACTIONCOMPARISON},
	the identities \eqref{E:FRAMEVECTORFIELDLIEBRACKET} and \eqref{E:VECTORFIELDCOVARIANTDERIVATIVEID},
	Cauchy--Schwarz relative to $\StMet$, 
	the fact that $|Z|_{\StMet} = 1$ for $Z \in \mathscr{Z}$,
	and the bootstrap assumptions,
	we obtain the desired bound as follows:
	$
	\left|
		\nabla Z
	\right|_{\newg}
	=
	\left|
		\nabla Z_{\flat}
	\right|_{\newg}
	\lesssim
	\scale^{- c \upsigma}
	\sum_{Z_1, Z_2 \in \mathscr{Z}}
	\left|
		Z_2 \cdot \nabla_{Z_1} Z_{\flat}
	\right|
	=
	\scale^{- c \upsigma}
	\sum_{Z_1, Z_2 \in \mathscr{Z}}
	\left|
		\newg(\nabla_{Z_1} Z, Z_2)
	\right|
	\lesssim 
	\scale^{- c \upsigma}
	\left|
		\SigmatLie_{\mathscr{Z}}^{\leq 1} \newg
	\right|_{\StMet}
	$
	$
	\lesssim 
	\scale^{- c \upsigma}
	$.
	Next, we note the following identity, written in schematic form (without regard for constant coefficients)
	and valid for $Z \in \mathscr{Z}$ and any $\Sigma_t$-tangent tensor $\xi$:
	$Z \cdot \nabla \xi  = \SigmatLie_Z \xi + \xi \cdot \nabla Z$.
	Thus, using Cauchy--Schwarz relative to $\newg$, 
	\eqref{E:GNORMSTMETRICCOMPARISON}, 
	the fact that $|Z|_{\StMet} = 1$ when $Z \in \mathscr{Z}$,
	and the bound
	$
	\left|
		\nabla Z
	\right|_{\newg}
	\lesssim 
	\scale^{- c \upsigma} 
	$,
	we deduce
	$
	\left|
		\SigmatLie_{\mathscr{Z}}^{\leq 1} \xi
	\right|_{\newg}
	\lesssim
	|\xi|_{\newg}
	+
	\sum_{Z \in \mathscr{Z}}
	\left|
		Z
	\right|_{\newg}
	\left|
		\nabla \xi
	\right|_{\newg}
	+
	\sum_{Z \in \mathscr{Z}}
	\left|
		\nabla Z
	\right|_{\newg}
	\left|
		\xi
	\right|_{\newg}
	\lesssim
	\scale^{- c \upsigma}
	\left|
		\nabla^{\leq 1} \xi
	\right|_{\newg}
	$,
	which is the second inequality stated in \eqref{E:OPERATORCOMPARISON} in the case $L=1$.
	To obtain the first inequality stated in \eqref{E:OPERATORCOMPARISON} in the case $L=1$,
	we use the schematic identity
	$Z \cdot \nabla \xi  = \SigmatLie_Z \xi + \xi \cdot \nabla Z$,
	\eqref{E:GNORMSTMETRICCOMPARISON},
	\eqref{E:GNORMCONTRACTIONCOMPARISON}, 
	the identity $|Z|_{\StMet} = 1$,
	and the bound 
	$
	\left|
		\nabla Z
	\right|_{\newg}
	\lesssim 
	\scale^{- c \upsigma} 
	$
	to deduce that
	$
	\left|
		\nabla \xi
	\right|_{\StMet}
	\lesssim
	\sum_{Z_1,Z_2 \in \mathscr{Z}}
		\left|Z_1 \otimes Z_2 \cdot \nabla \xi \right|
	$	
	$
	\lesssim
	\sum_{Z \in \mathscr{Z}}
		\left|
			\SigmatLie_Z \xi 
		\right|_{\StMet}
		+
		\sum_{Z  \in \mathscr{Z}}
		\left|
			\xi
		\right|_{\StMet}
		\left|
			\nabla Z
		\right|_{\StMet}
	$
	$
	\lesssim
	\left\lbrace
		1 
		+ 
		C \scale^{- c\upsigma} 
		\left|
			\nabla Z
		\right|_{\newg}
	\right\rbrace
	\left|
			\SigmatLie_{\mathscr{Z}}^{\leq 1} \xi
	\right|_{\StMet}
	$
	$
	\lesssim 
	\scale^{- c\upsigma}
	\left|
		\SigmatLie_{\mathscr{Z}}^{\leq 1} \xi
	\right|_{\StMet}
	$.
	The desired first inequality in \eqref{E:OPERATORCOMPARISON}
	in the case $L=1$ now follows from the previous estimate
	and \eqref{E:GNORMSTMETRICCOMPARISON}.
	
	To prove the first estimate stated in \eqref{E:OPERATORCOMPARISON}
	in the case $L=2$, we use the first estimate in \eqref{E:OPERATORCOMPARISON}
	in the already proven case $L=1$
	with $\nabla \xi$ in the role of $\xi$ to deduce that
	$
	\left|
			\nabla^2 \xi
	\right|_{\newg}
		\lesssim
			\scale^{- c \upsigma} 
			\left|
				\SigmatLie_{\mathscr{Z}}^{\leq 1} \nabla \xi
			\right|_{\newg}
			= 
			\scale^{- c \upsigma} 
			\left|
				\nabla \SigmatLie_{\mathscr{Z}}^{\leq 1}  \xi
			\right|_{\newg}
			+
		 \scale^{- c \upsigma} 
			\sum_{Z \in \mathscr{Z}}
			\left|
				[\nabla,\SigmatLie_Z]  \xi
			\right|_{\newg}
		\lesssim 
		\scale^{- c \upsigma} 
			\left|
				\SigmatLie_{\mathscr{Z}}^{\leq 2}  \xi
			\right|_{\newg}
			+
		 \scale^{- c \upsigma} 
			\sum_{Z \in \mathscr{Z}}
			\left|
				[\nabla, \SigmatLie_Z]  \xi
			\right|_{\newg}
	$.
	To obtain the desired bound,
	we need only to show that
	$
	\left|
		[\nabla, \SigmatLie_Z]  \xi
	\right|_{\newg}
	\lesssim 
	\scale^{- c \upsigma} |\xi|_{\newg}
	$.
	To this end, we use
	the formula \eqref{E:TENSORNABLALIEZCOMMUTATOR} (with $\vec{I}=1$),
	the first estimate in \eqref{E:OPERATORCOMPARISON}
	in the case $L=1$,
	\eqref{E:GNORMSTMETRICCOMPARISON},
	and the bootstrap assumptions
	to deduce the desired bound as follows:
	$
	\left|
		[\nabla, \SigmatLie_Z]  \xi
	\right|_{\newg}
	\lesssim 
	\sum_{Z \in \mathscr{Z}}
	\left|
		\nabla \SigmatLie_Z \newg
	\right|_{\newg}
	\left|
		\xi
	\right|_{\newg}
	\lesssim 
	\scale^{- c \upsigma} 
	\left|
		\SigmatLie_{\mathscr{Z}}^{\leq 2} (\newg - \StMet)
	\right|_{\newg}
	\left|
		\xi
	\right|_{\newg}
	\lesssim 
	\scale^{- c \upsigma} 
	\left|
			\SigmatLie_{\mathscr{Z}}^{\leq 2} (\newg - \StMet)
	\right|_{\StMet}
	\left|
		\xi
	\right|_{\newg}
	\lesssim 
	\scale^{- c \upsigma} 
	\left|
		\xi
	\right|_{\newg}
	$.
	To prove the second estimate stated in \eqref{E:OPERATORCOMPARISON}
	in the case $L=2$,
	we first use 
	the second estimate in \eqref{E:OPERATORCOMPARISON}
	in the case $L=1$ to deduce
	$
	\sum_{Z_1, Z_2 \in \mathscr{Z}}
	\left|
		\SigmatLie_{Z_1} \SigmatLie_{Z_2} \xi
	\right|_{\newg}
	$
	$
	\lesssim
	\scale^{- c \upsigma} 
	\left|
		\SigmatLie_{\mathscr{Z}}^{\leq 1} \xi
	\right|_{\newg}
	+
	\scale^{- c \upsigma} 
	\sum_{Z \in \mathscr{Z}}
	\left|
		\nabla \SigmatLie_Z \xi
	\right|_{\newg}
	$
	$
	\lesssim
	\scale^{- c \upsigma} 
	\left|
		\nabla^{\leq 1} \xi
	\right|_{\newg}
	+
	\scale^{- c \upsigma} 
	\sum_{Z \in \mathscr{Z}}
	\left|
		\SigmatLie_Z \nabla \xi
	\right|_{\newg}
	+
	\scale^{- c \upsigma} 
	\sum_{Z \in \mathscr{Z}}
	\left|
		[\nabla, \SigmatLie_Z] \xi
	\right|_{\newg}
	$
	$
	\lesssim
	\scale^{- c \upsigma}
	\left|
		\nabla^{\leq 2} \xi
	\right|_{\newg}
	+
	\scale^{- c \upsigma} 
	\sum_{Z \in \mathscr{Z}}
	\left|
		[\nabla, \SigmatLie_Z] \xi
	\right|_{\newg}
	$.
	To finish the proof,
	we need only to use the bound
	$
	\left|
		[\nabla, \SigmatLie_Z] \xi
	\right|_{\newg}
	\lesssim 
	\scale^{- c \upsigma} |\xi|_{\newg}
	$
	proved just above.
\end{proof}

\subsection{Preliminary curvature estimates}
\label{SS:PRELIMINARYCURVATURE}
In this subsection, we derive preliminary estimates for the curvatures of
the rescaled metric $\newg$. We recall that the type $\binom{0}{4}$ Riemann curvature of $\newg$ is denoted by $\Riem$
and that its type $\binom{1}{1}$ Ricci tensor is denoted by $\Ric^{\#}$.

\begin{lemma}[\textbf{Preliminary curvature estimates}]
	\label{L:RIEMANNCURVATUREAPPROXALLINDICESDOWN}
	There exist polynomials in two real variables 
	with strictly positive coefficients,
	all of which are schematically denoted by $\Pol$,
	such that the following estimates hold:
\begin{subequations}	
\begin{align} \label{E:RIEMANNCURVATUREAPPROXALLINDICESDOWN}
		\Riem
		& = 
				\frac{1}{18} \newg \owedge \newg
				+ 
				\Riem_{\triangle},
				\\
	\Ric^{\#}
		& = 
				\frac{2}{9} \ID
				+ 
				\Ric_{\triangle}^{\#},
			\label{E:RICCIESTONEUPONEDOWN} 
\end{align}
\end{subequations}
where
\begin{subequations}
\begin{align}
	\left|
		\Riem_{\triangle}
	\right|_{\StMet}
	& \leq
		\Pol
		\left(
				\left|
				\SigmatLie_{\mathscr{Z}}^{\leq 1}
				\left\lbrace
					\newg - \StMet
				\right\rbrace
			\right|_{\StMet},
			|\newg^{-1} - \StMet^{-1}|_{\StMet}
		\right)
		\left\lbrace
			\left|
				\SigmatLie_{\mathscr{Z}}^{\leq 2}
				\left\lbrace
					\newg - \StMet
				\right\rbrace
			\right|_{\StMet}
			+
			\left|
				\newg^{-1} - \StMet^{-1}
			\right|_{\StMet}
		\right\rbrace,
		\label{E:ERRORRIEMANNCURVATUREAPPROXALLINDICESDOWN}
			\\
	\left|
		\Ric_{\triangle}^{\#}
	\right|_{\StMet}
	& \leq
		\Pol
		\left(
				\left|
				\SigmatLie_{\mathscr{Z}}^{\leq 1}
				\left\lbrace
					\newg - \StMet
				\right\rbrace
			\right|_{\StMet},
			|\newg^{-1} - \StMet^{-1}|_{\StMet}
		\right)
		\left\lbrace
			\left|
				\SigmatLie_{\mathscr{Z}}^{\leq 2}
				\left\lbrace
					\newg - \StMet
				\right\rbrace
			\right|_{\StMet}
			+
			\left|
				\newg^{-1} - \StMet^{-1}
			\right|_{\StMet}
		\right\rbrace.
		\label{E:ERRORRICONEUPONEDOWN}
	\end{align}
	\end{subequations}
	In \eqref{E:RICCIESTONEUPONEDOWN},
	$\ID$ denotes the identity transformation,
	while the tensor $\newg \owedge \newg$ on RHS \eqref{E:RIEMANNCURVATUREAPPROXALLINDICESDOWN}
	has the following components relative to arbitrary coordinates on $\Sigma_t$:
	$
	(\newg \owedge \newg)_{ijkl} 
	:= 
	2 \newg_{ik} \newg_{jl}
	-
	2 \newg_{il} \newg_{jk}
	$.
\end{lemma}
\begin{proof}
	We give the proof of \eqref{E:RIEMANNCURVATUREAPPROXALLINDICESDOWN} and 
	\eqref{E:ERRORRIEMANNCURVATUREAPPROXALLINDICESDOWN}. The proof of
	\eqref{E:RICCIESTONEUPONEDOWN} and \eqref{E:ERRORRICONEUPONEDOWN}
	is similar and we omit those details.
	Throughout the proof, we use the notation $\Riem(Z_{(A)},Z_{(B)},Z_{(C)},Z_{(D)})$
	to denote the components of $\Riem$ relative to the 
	Lie-transported frame of Def.\ \ref{D:LIETRANSPORTEDSPATIALFRAME}.
	We now note the following standard formula:
	\begin{align} \label{E:RIEMANNCURVATUREOFTHEFRAME}
		\Riem(Z_{(A)},Z_{(B)},Z_{(C)},Z_{(D)})
		& 
		:=
		\newg(\nabla_{Z_{(B)}} (\nabla_{Z_{(A)}} Z_{(C)}), Z_{(D)})
		-
		\newg(\nabla_{Z_{(A)}} (\nabla_{Z_{(B)}} Z_{(C)}), Z_{(D)})
		\\
		& \ \ 
		+ \newg(\nabla_{[Z_{(A)},Z_{(B)}]} Z_{(C)}, Z_{(D)}).
			\notag
	\end{align}
	To proceed, we first note that by
	Cor.~\ref{C:STMETRICFRAMEEXPANDED},
	the desired bounds follow once we
	show that for $A,B,C,D = 1,2,3$, we have
	\[
	\Riem(Z_{(A)},Z_{(B)},Z_{(C)},Z_{(D)})
	= \frac{1}{18} \newg \owedge \newg(Z_{(A)},Z_{(B)},Z_{(C)},Z_{(D)})
	+
	\Riem_{\triangle}(Z_{(A)},Z_{(B)},Z_{(C)},Z_{(D)}),
	\]
	where 
	$|\Riem_{\triangle}(Z_{(A)},Z_{(B)},Z_{(C)},Z_{(D)})| \leq$ 
	RHS \eqref{E:ERRORRIEMANNCURVATUREAPPROXALLINDICESDOWN}.
	To this end, we first compute RHS \eqref{E:RIEMANNCURVATUREOFTHEFRAME}
	using 
	\eqref{E:FRAMEVECTORFIELDLIEBRACKET},
	\eqref{E:FRAMECOVARIANTDERIVATIVES},
	\eqref{E:ALTERNATEFRAMECONNECTIONEXPRESSION}-\eqref{E:ERRORTERMALTERNATEFRAMECONNECTIONEXPRESSION},
	and the identity
	$Z_{(C)} (\newg^{-1})^{AB} 
	= - (\newg^{-1})^{AA'} (\newg^{-1})^{BB'} Z_{(C)} \newg_{A'B'}
	$.
	We then use \eqref{E:ZAGBCID} to substitute for $Z_{(C)} \newg_{A'B'}$
	in the previous formula.
	Next, we expand $\newg = \StMet + (\newg - \StMet)$
	and substitute this expansion for $\newg$ everywhere in
	our computations for RHS \eqref{E:RIEMANNCURVATUREOFTHEFRAME}.
	Similarly, we expand $\newg^{-1} = \StMet^{-1} + (\newg^{-1} - \StMet^{-1})$
	and substitute. 
	The only non-small term that arises is
	$\frac{1}{18} \StMet \owedge \StMet (Z_{(A)},Z_{(B)},Z_{(C)},Z_{(D)})$,
	which is generated by tensor products of the first term on RHS \eqref{E:ALTERNATEFRAMECONNECTIONEXPRESSION}
	with itself. We replace
	$\frac{1}{18} \StMet \owedge \StMet (Z_{(A)},Z_{(B)},Z_{(C)},Z_{(D)})$
	with $\frac{1}{18} \newg \owedge \newg (Z_{(A)},Z_{(B)},Z_{(C)},Z_{(D)})$
	(which leads to the first term on RHS~\eqref{E:RIEMANNCURVATUREAPPROXALLINDICESDOWN})
	plus an error term that we incorporate into 
	$\Riem_{\triangle}(Z_{(A)},Z_{(B)},Z_{(C)},Z_{(D)})$.
	From these steps, we obtain that 
	$\Riem_{\triangle}(Z_{(A)},Z_{(B)},Z_{(C)},Z_{(D)})$
	is a sum of products involving
	$\StMet$,
	$\StMet^{-1}$,
	$\SigmatLie_{\mathscr{Z}}^{\leq 2}(\newg - \StMet)$,
	and
	$\newg^{-1} - \StMet^{-1}$,
	where each product contains at least one factor of either
	$\SigmatLie_{\mathscr{Z}}^{\leq 2}(\newg - \StMet)$
	or
	$\newg^{-1} - \StMet^{-1}$,
	and the factor $\SigmatLie_{\mathscr{Z}}^2(\newg - \StMet)$
	appears either linearly or not at all.
	The desired bound for 
	$\Riem_{\triangle}(Z_{(A)},Z_{(B)},Z_{(C)},Z_{(D)})$
	now follows from these computations,
	Cauchy--Schwarz relative to $\StMet$,
	the identities $|\StMet|_{\StMet} = |\StMet^{-1}|_{\StMet} = \sqrt{3}$,
	and the identity $|Z_{(A)}|_{\StMet} = 1$ for $A=1,2,3$.
\end{proof}

\subsection{Weak sup-norm estimates}
\label{SS:WEAKSUPNORM}
The following lemma provides ``weak'' sup-norm estimates,
which follow as a direct consequence of the bootstrap assumption \eqref{E:HIGHNORMBOOTSTRAP}. 
That is, the proof of the lemma does not rely
on the Einstein-scalar field equations.
In Sect.\ \ref{S:STRONGC0ESTIMATES}, we will use these weak sup norm estimates 
and the Einstein-scalar field equation
to prove stronger (i.e., less singular with respect to $t$) estimates.

\begin{lemma} [\textbf{Weak sup-norm estimates}] 
\label{L:WEAKSUPNORM}
The following estimates hold (see Def.\ \ref{D:SPATIALDERIVATIVENORMSANDSEMINORMS} regarding the notation):
\begin{subequations}
\begin{align} \label{E:WEAKSUPNORM}
	&
	\left\|
		\dif^{\leq 14} \FreeNewSec
	\right\|_{C_{\newg}^0(\Sigma_t)}
	+
	\left\|
		\dif^{\leq 14} (\newg - \StMet)
	\right\|_{C_{\newg}^0(\Sigma_t)}
	+
	\left\|
		\dif^{\leq 14} 
		(\newg^{-1} - \StMet^{-1})
	\right\|_{C_{\newg}^0(\Sigma_t)}
		\\
	& \ \ 
	+
	\left\|
		\dif^{\leq 12} \newlapse
	\right\|_{C^0(\Sigma_t)}
	+
	\left\|
		\dif^{\leq 14} \newtimescalar
	\right\|_{C^0(\Sigma_t)}
	+
	\left\|
		\dif^{\leq 13} \newspacescalar
	\right\|_{C_{\newg}^0(\Sigma_t)}
	\lesssim \varepsilon \scale^{- c \upsigma}(t),
		\notag \\
 &	
	\left\|
		\dif^{\leq 14} \FreeNewSec
	\right\|_{C_{\StMet}^0(\Sigma_t)}
	+
	\left\|
		\dif^{\leq 14} (\newg - \StMet)
	\right\|_{C_{\StMet}^0(\Sigma_t)}
	+
	\left\|
		\dif^{\leq 14} 
		(\newg^{-1} - \StMet^{-1})
	\right\|_{C_{\StMet}^0(\Sigma_t)}
		\label{E:WEAKSUPNORMSTMET} \\
& \ \
	+
	\left\|
		\dif^{\leq 12} \newlapse
	\right\|_{C^0(\Sigma_t)}
	+
	\left\|
		\dif^{\leq 14} \newtimescalar
	\right\|_{C^0(\Sigma_t)}
		+
	\left\|
		\dif^{\leq 13} \newspacescalar
	\right\|_{C_{\StMet}^0(\Sigma_t)}
	\lesssim \varepsilon \scale^{- c \upsigma}(t).
	\notag
\end{align}
\end{subequations}

Moreover, we have
\begin{subequations}
\begin{align} \label{E:WEAKRICCISUPNORM}
	\left\|
		\Ric^{\# } - \frac{2}{9} \ID 
	\right\|_{C_{\newg}^0(\Sigma_t)},
		\,
	\left\|
		\dif^{[1,12]} \Ric^{\# } 
	\right\|_{C_{\newg}^0(\Sigma_t)}
	& \lesssim \varepsilon \scale^{- c \upsigma}(t),
		\\
	\left\|
		\Ric^{\# } - \frac{2}{9} \ID 
	\right\|_{C_{\StMet}^0(\Sigma_t)},
		\,
	\left\|
		\dif^{[1,12]} \Ric^{\# } 
	\right\|_{C_{\StMet}^0(\Sigma_t)}
	& \lesssim \varepsilon \scale^{- c \upsigma}(t).
		\label{E:WEAKRICCISUPNORMSTMET}
\end{align} 
\end{subequations}
\end{lemma}

\begin{proof}
	To obtain \eqref{E:WEAKSUPNORM}-\eqref{E:WEAKSUPNORMSTMET},
	we first repeatedly use the commutation identities of Lemma~\ref{L:COMMUTATIONIDENTITIES}
	to commute any factors of $\nabla$ that might appear in an operator
	$\dif^M$ to the front (so that the $\nabla$ act last).
	The desired inequalities \eqref{E:WEAKSUPNORM}-\eqref{E:WEAKSUPNORMSTMET}
	now follow from these commutation identities,
	the bootstrap assumption \eqref{E:HIGHNORMBOOTSTRAP},
	Lemma~\ref{L:PRELIMINARYCOMPARISON},
	and the fact that in all estimates,
	we can replace $\SigmatLie_Z \newg$
	with $\SigmatLie_Z (\newg - \StMet)$ 
	whenever $Z \in \mathscr{Z}$ 
	(since $\SigmatLie_Z \StMet = 0$).
	
	To prove \eqref{E:WEAKRICCISUPNORMSTMET},
	we first use
	\eqref{E:RICCIESTONEUPONEDOWN},
	\eqref{E:ERRORRICONEUPONEDOWN},
	and \eqref{E:WEAKSUPNORMSTMET}
	to deduce that
	$\left\| 
		\Ric^{\# } - \frac{2}{9} \ID
	\right\|_{C^0_{\StMet}(\Sigma_t)}
	\lesssim
	\left\|
		\Riem_{\triangle \newg}
	\right\|_{C^0_{\StMet}(\Sigma_t)}
	\lesssim \varepsilon \scale^{-c \upsigma}(t)
	$
	as desired.
	To obtain the bound
	$\left\| 
		\dif^{[1,12]} \Ric^{\# } 
	\right\|_{C^0_{\StMet}(\Sigma_t)} 
	\lesssim 
	\varepsilon \scale^{-c \upsigma}$,
	we first note the basic fact that
	$
	\left\| 
		\dif^{[1,12]} \Ric^{\# } 
	\right\|_{C^0_{\StMet}(\Sigma_t)}
	=
	\left\| 
		\dif^{[1,12]} \left(\Ric^{\# } - \frac{2}{9} \ID \right)
	\right\|_{C^0_{\StMet}(\Sigma_t)}
	$.
	We then use the commutation identities of Lemma~\ref{L:COMMUTATIONIDENTITIES}
	to commute any factors of $\nabla$ that might appear in an operator
	$\dif^M$ to the front (so that the $\nabla$ act last).
	Using these identities,
	\eqref{E:OPERATORCOMPARISON},
	and the already proven bounds \eqref{E:WEAKSUPNORMSTMET} and
	$\left\| 
		\Ric^{\# } - \frac{2}{9} \ID
	\right\|_{C^0_{\StMet}(\Sigma_t)}
	\lesssim \varepsilon \scale^{-c \upsigma}(t)
	$,
	we deduce that
	$
	\left\| 
		\dif^{[1,12]} \left(\Ric^{\# } - \frac{2}{9} \ID \right)
	\right\|_{C^0_{\StMet}(\Sigma_t)}
	\lesssim 
	\scale^{- c \upsigma}(t)
	\left\| 
		\SigmatLie_{\mathscr{Z}}^{\leq 12} 
		\left(\Ric^{\# } - \frac{2}{9} \ID \right)
	\right\|_{C^0_{\StMet}(\Sigma_t)}
	\lesssim
	\varepsilon
	\scale^{- c \upsigma}(t)
	+
	\scale^{- c \upsigma}(t)
	\left\| 
		\SigmatLie_{\mathscr{Z}}^{[1,12]} \Ric^{\#} 
	\right\|_{C^0_{\StMet}(\Sigma_t)}
	$.
	Next, we use the formulas
	\eqref{E:RICCILIEDERIVATIVESCHEMATIC}-\eqref{E:RICINHOMERROR}
	to algebraically express
	$\SigmatLie_{\mathscr{Z}}^{[1,12]} \Ric^{\#}$
	in terms of the undifferentiated quantity $\Ric^{\# }$
	and the covariant and Lie derivatives of $\newg$,
	where the covariant derivatives act last, and we replace all instances 
	of $\SigmatLie_Z \newg$
	with $\SigmatLie_Z (\newg - \StMet)$ 
	(which is possible since $\SigmatLie_Z \StMet = 0$).
	Again using \eqref{E:OPERATORCOMPARISON}
	and the already proven bounds \eqref{E:WEAKSUPNORMSTMET}
	and
	$\left\| 
		\Ric^{\# } - \frac{2}{9} \ID
	\right\|_{C^0_{\StMet}(\Sigma_t)}
	\lesssim \varepsilon \scale^{-c \upsigma}(t)
	$,
	we conclude that
	$
	\left\| 
		\dif^{[1,12]} \left(\Ric^{\# } - \frac{2}{9} \ID \right)
	\right\|_{C^0_{\StMet}(\Sigma_t)}
	\lesssim 
	\varepsilon
	\scale^{- c \upsigma}(t)
	+
	\scale^{-c \upsigma}(t) 
	\left\| 
		\SigmatLie_{\mathscr{Z}}^{\leq 14} (\newg - \StMet)
	\right\|_{C^0_{\StMet}(\Sigma_t)}
	\lesssim
	\varepsilon
	\scale^{- c \upsigma}(t)
	$
	as desired.
	We have thus proved \eqref{E:WEAKRICCISUPNORMSTMET}.
	
	\eqref{E:WEAKRICCISUPNORM} then follows from \eqref{E:WEAKRICCISUPNORMSTMET} and \eqref{E:GNORMSTMETRICCOMPARISON}.
	
\end{proof}

The following corollary is an immediate consequence of Lemma~\ref{L:WEAKSUPNORM}.

\begin{corollary}[\textbf{Weak sup-norm bounds along $\Sigma_0$}]
	\label{C:WEAKSMALLDATA}
	The following bounds hold on $\Sigma_0$
	(see Def.\ \ref{D:SPATIALDERIVATIVENORMSANDSEMINORMS} regarding the notation):
	\begin{align} \label{E:WEAKSMALLDATA}
	\left\|
		\dif^{\leq 14} \FreeNewSec
	\right\|_{C_{\newg}^0(\Sigma_0)}
	+
	\left\|
		\dif^{\leq 14} (\newg - \StMet)
	\right\|_{C_{\newg}^0(\Sigma_0)}
	+
	\left\|
		\dif^{\leq 14} (\newg^{-1} - \StMet^{-1})
	\right\|_{C_{\newg}^0(\Sigma_0)}
	&
		\\	
 \ \
		+
	\left\|
		\dif^{\leq 12} \newlapse
	\right\|_{C^0(\Sigma_0)}
	+
	\left\|
		\dif^{\leq 14} \newtimescalar
	\right\|_{C^0(\Sigma_0)}
		+
	\left\|
		\dif^{\leq 13} \newspacescalar
	\right\|_{C_{\newg}^0(\Sigma_0)}
	& \lesssim \varepsilon,
		\notag \\
	\left\|
		\dif^{\leq 14} \FreeNewSec
	\right\|_{C_{\StMet}^0(\Sigma_0)}
	+
	\left\|
		\dif^{\leq 14} (\newg - \StMet)
	\right\|_{C_{\StMet}^0(\Sigma_0)}
	+
	\left\|
		\dif^{\leq 14} (\newg^{-1} - \StMet^{-1})
	\right\|_{C_{\StMet}^0(\Sigma_0)}
	&
		\label{E:WEAKSMALLDATASTMETR} \\	
 \ \
		+
	\left\|
		\dif^{\leq 12} \newlapse
	\right\|_{C^0(\Sigma_0)}
	+
	\left\|
		\dif^{\leq 14} \newtimescalar
	\right\|_{C^0(\Sigma_0)}
		+
	\left\|
		\dif^{\leq 13} \newspacescalar
	\right\|_{C_{\StMet}^0(\Sigma_0)}
	& \lesssim \varepsilon.
	\notag
\end{align}
\end{corollary}

\subsection{Volume form estimates}
\label{SS:VOLUMEFORMESTIMATES}
In this subsection, we derive some simple bounds for the volume forms of Def.\ \ref{D:VOLUMEFORM}.

\begin{lemma}[\textbf{Volume form estimates}]
	\label{L:VOLUMEFORMCOMPARISON}
	The following pointwise estimates hold for the volume forms of Def.\ \ref{D:VOLUMEFORM}:
	\begin{align} \label{E:VOLUMEFORMCOMPARISON}
		d \varpi_{\StMet}
		 & =  
		 \left\lbrace 1 + \mathcal{O}(\varepsilon) \right\rbrace d \tvol,
		&
		\SigmatLie_{\partial_t} d \tvol 
		& = \mathcal{O}(\varepsilon) \scale^{1/3 - c \upsigma} d \tvol.
	\end{align}
\end{lemma}

\begin{proof}
From \eqref{E:VOLUMEFORMTIMEDERIVATIVE} and
\eqref{E:WEAKSUPNORMSTMET},
we deduce that
$\SigmatLie_{\partial_t} d \tvol(t,x) 
= \mathcal{O}(\varepsilon) \scale^{1/3 - c \upsigma}(t) d \tvol(t,x)$,
which yields the second estimate in \eqref{E:VOLUMEFORMCOMPARISON}.
Integrating this estimate with respect to time and using Gronwall's inequality,
we deduce the pointwise bound 
$d \tvol(t,x) = \left\lbrace 1 + \mathcal{O}(\varepsilon) \right\rbrace d \tvol(0,x)$. 
Thus, the first estimate in
\eqref{E:VOLUMEFORMCOMPARISON} will follow once we show that
$d \tvol(0,x) = \left\lbrace 1 + \mathcal{O}(\varepsilon) \right\rbrace d \varpi_{\StMet}(x)$.
The desired bound is an easy consequence of the facts that
$d \tvol(0,x)$ and $d \varpi_{\StMet}(0,x)$ are proportional,
that
$d \tvol(0,x)(Z_{(1)},Z_{(2)},Z_{(3)})$
is equal to
$\mbox{\upshape det} M$,
	where
	$M$ is the $3 \times 3$ matrix with entries $M_{AB}$ defined by
	$M_{AB}
	:= \left( 
		\newg(0,x)(Z_{(A)},Z_{(B)})
	\right)_{1 \leq A,B \leq 3}
	$,
	the identity
	$d \varpi_{\StMet}(Z_{(1)},Z_{(2)},Z_{(3)}) = 1$,
	and the small-data assumption \eqref{E:SMALLDATA}.
\end{proof}

\section{Strong sup-norm estimates and Sobolev embedding}
\label{S:STRONGC0ESTIMATES}
		In this section, we use the Einstein-scalar field equations
		and the bootstrap assumptions to prove
		prove ``strong'' sup-norm estimates at the lower derivative levels.
		We also provide Sobolev embedding estimates that allow us to
		obtain sup-norm bounds for tensorfields $\xi$ in terms of the $L^2$ norms
		of the $\Sigma_t$-projected Lie derivatives of $\xi$ with respect to the
		elements of $\mathscr{Z}$.
		The estimates of this section are less singular than the 
		``weak'' sup-norm estimates that we derived in Subsect.\ \ref{SS:WEAKSUPNORM}.
		They are essential
		for closing the energy estimates, 
		for exhibiting the AVTD nature of the solution (as stated in Theorem~\ref{T:VERYROUGH}),
		and for proving convergence results near the singularities.

\subsection{Maximum principle estimates}
In deriving sup-norm estimates for the lapse, we will 
use the maximum principle-type estimates provided by the following lemma.

\begin{lemma}[\textbf{Maximum principle}]
	\label{L:MAXIMUMPRINCIPLEESTIMATE}
	Let $\widetilde{\mathscr{L}}$ be the elliptic operator defined by
	\eqref{E:LOWORDERELLIPTICOPERATOR}.
	Classical solutions $u$ to the PDE
	\begin{align} \label{E:LOWORDERELLIPTICPDE}
		\widetilde{\mathscr{L}} u
		& = U
	\end{align}
	verify the estimate
	\begin{align} \label{E:MAXIMUMPRINCIPLEESTIMATE}
		\left\|
			u
		\right\|_{C^0(\Sigma_t)}
		& \leq
			C
			\left\|
				U
			\right\|_{C^0(\Sigma_t)}.
	\end{align}
\end{lemma}
\begin{proof}
	We first note that by \eqref{E:FRIEDMANNFIRSTORDER},
	\eqref{E:WEAKSUPNORM},
	and
	\eqref{E:WEAKRICCISUPNORM}, the term $\widetilde{f}$
	defined in \eqref{E:ERRORTERMLOWORDERELLIPTICOPERATOR}
	verifies the bounds 
	$
	\frac{2}{3} - C \varepsilon
	\leq
	\widetilde{f} \leq \frac{2}{3} + C \varepsilon$.
	When $u$ achieves its max along $\Sigma_t$, we have $\Delta_{\newg} u \leq 0$.
	At such points, we have $\widetilde{f} u + U \leq 0$ and, in view of the estimates for $\widetilde{f}$,
	$u \leq C |U| 
	\leq C 
	\left\|
		U
	\right\|_{C^0(\Sigma_t)}
	$.
	Similarly, when
	$u$ achieves its min, we have $\Delta_{\newg} u \geq 0$
	and thus 
	$u \geq - C |U| 
	\geq 
	-
	C 
	\left\|
		U
	\right\|_{C^0(\Sigma_t)}
	$.
	We have therefore proved the lemma.
\end{proof}

\subsection{Base-level strong sup-norm estimates}
\label{SS:BASELEVELSTRONG}
In the next lemma, we initiate our proof of the strong sup-norm
estimates by deriving them at the lowest derivative level
for the time-rescaled metric and second fundamental form.

\begin{lemma}[\textbf{Key base-level strong sup-norm estimates}]
	\label{L:NODERIVATIVESPARTIALTGANDKSTRONGSUPNORM}
	The following estimates hold:
	\begin{align} 
		\left\|
			\newg - \StMet
		\right\|_{C_{\newg}^0(\Sigma_t)},
			\,
		\left\|
			\newg^{-1} - \StMet^{-1}
		\right\|_{C_{\newg}^0(\Sigma_t)}
		& \lesssim 
			\sqrt{\varepsilon}
			\scale^{-c \sqrt{\varepsilon}}(t),
			 \label{E:NODERIVATIVESGMINUSSTANDARDSUPNORM} \\
		\left\|
			\newg - \StMet
		\right\|_{C_{\StMet}^0(\Sigma_t)},
			\,
		\left\|
			\newg^{-1} - \StMet^{-1}
		\right\|_{C_{\StMet}^0(\Sigma_t)}
		& \lesssim 
			\sqrt{\varepsilon} \scale^{-c \sqrt{\varepsilon}}(t),
			 \label{E:NODERIVATIVESGMINUSSTANDARDSUPNORMSTMET} \\
		\left\|
			\SigmatLie_{\partial_t} \newg
		\right\|_{C_{\newg}^0(\Sigma_t)},
			\,
		\left\|
			\SigmatLie_{\partial_t} \newg^{-1}
		\right\|_{C_{\newg}^0(\Sigma_t)}
		& \lesssim \varepsilon \scale^{-1}(t),
			\label{E:NODERIVATIVESPARTIALTGSTRONGSUPNORM} \\
		\left\|
			\SigmatLie_{\partial_t} \newg
		\right\|_{C_{\StMet}^0(\Sigma_t)},
			\,
		\left\|
			\SigmatLie_{\partial_t} \newg^{-1}
		\right\|_{C_{\StMet}^0(\Sigma_t)}
		& \lesssim \varepsilon \scale^{-1}(t),
			\label{E:NODERIVATIVESPARTIALTGSTRONGSUPNORMSTMET} 
				\\
		\left\|
			\FreeNewSec
		\right\|_{C_{\newg}^0(\Sigma_t)}
		& \lesssim \varepsilon,
		\label{E:NODERIVATIVESKSTRONGPOINTWISESUPNORM}
			\\
		\left\|
			\FreeNewSec
		\right\|_{C_{\StMet}^0(\Sigma_t)}
		& \lesssim \varepsilon.
		\label{E:NODERIVATIVESKSTRONGPOINTWISESUPNORMSTMET}
\end{align}
\end{lemma}
\begin{proof}
We first prove \eqref{E:NODERIVATIVESKSTRONGPOINTWISESUPNORM}.
From the symmetry property $\newg_{ia} \FreeNewSec_{\ j}^a= \newg_{ja} \FreeNewSec_{\ i}^a$,
we derive
	$
\partial_t
	\left(
	\left|
		\FreeNewSec
	\right|_{\newg}^2
	\right)
= 2 \newg_{ab} (\newg^{-1})^{cd} \FreeNewSec_{\ c}^a \SigmatLie_{\partial_t} \FreeNewSec_{\ d}^b
$. It follows that
$
\left|
	\partial_t
	\left|
		\FreeNewSec
	\right|_{\newg}
\right|
\lesssim 
\left|
	\SigmatLie_{\partial_t} \FreeNewSec
\right|_{\newg}
$.
To bound
$
\left|
	\SigmatLie_{\partial_t} \FreeNewSec
\right|_{\newg}
$,
we observe that 
RHS~\eqref{E:EVOLUTIONSECONDFUNDRENORMALIZED}
is the product of $\scale^{1/3}$
times tensors that, by Lemma~\ref{L:WEAKSUPNORM},
are bounded in the norm
$|\cdot|_{\newg}$ by 
$\lesssim \varepsilon \scale^{- c \upsigma}$.
Thus, we have
$
\left|
	\SigmatLie_{\partial_t} \FreeNewSec
\right|_{\newg}
\lesssim
\varepsilon 
\scale^{1/3 - c \upsigma}
$.
Since $0 \leq \scale \leq 1$,
we deduce that the following bound holds for $\upsigma$ sufficiently small:
$
\left|
	\partial_t
	\left|
		\FreeNewSec
	\right|_{\newg}
\right|
\lesssim
\varepsilon 
$.
Integrating in time from $0$ to $t$, using the above estimates, 
and using the small-data bound
$
\left\|
	\FreeNewSec
\right\|_{C_{\newg}^0(\Sigma_0)}
\leq \varepsilon 
$ yielded by Cor.~\ref{C:WEAKSMALLDATA},
we conclude \eqref{E:NODERIVATIVESKSTRONGPOINTWISESUPNORM}.

The proof of \eqref{E:NODERIVATIVESKSTRONGPOINTWISESUPNORMSTMET}
is based on the identity
$
\partial_t
	\left(
	\left|
		\FreeNewSec
	\right|_{\StMet}^2
	\right)
= 2 \StMet_{ab} (\StMet^{-1})^{cd} \FreeNewSec_{\ c}^a \SigmatLie_{\partial_t} \FreeNewSec_{\ d}^b
$
and is similar to the proof of \eqref{E:NODERIVATIVESKSTRONGPOINTWISESUPNORM};
we omit the details.

To prove \eqref{E:NODERIVATIVESPARTIALTGSTRONGSUPNORM}
for $\SigmatLie_{\partial_t} \newg$,
we first use
\eqref{E:EVOLUTIONMETRICRENORMALIZED}
to deduce that
$\SigmatLie_{\partial_t} \newg
= - 2 \scale^{-1} \newg \cdot \FreeNewSec
$
plus error terms that are
the product of $\scale^{1/3}$
and tensors that, by Lemma~\ref{L:WEAKSUPNORM},
are bounded in the norm
$|\cdot|_{\newg}$ by 
$\lesssim \varepsilon \scale^{- c \upsigma}$.
Also using \eqref{E:NODERIVATIVESKSTRONGPOINTWISESUPNORM}, 
we find that
$
	\left|
		\SigmatLie_{\partial_t} \newg
	\right|_{\newg}
	\lesssim 
	\varepsilon \scale^{-1}
$,
which yields
the desired bound \eqref{E:NODERIVATIVESPARTIALTGSTRONGSUPNORM}
for $\SigmatLie_{\partial_t} \newg$.
The proof of \eqref{E:NODERIVATIVESPARTIALTGSTRONGSUPNORM}
for $\SigmatLie_{\partial_t} \newg^{-1}$
follows similarly from equation \eqref{E:EVOLUTIONINVERSEMETRICRENORMALIZED}
and we omit the details.

The proof of \eqref{E:NODERIVATIVESPARTIALTGSTRONGSUPNORMSTMET}
is similar to the proof of \eqref{E:NODERIVATIVESPARTIALTGSTRONGSUPNORM}
but relies on the bound \eqref{E:NODERIVATIVESKSTRONGPOINTWISESUPNORMSTMET}
in place of \eqref{E:NODERIVATIVESKSTRONGPOINTWISESUPNORM}; we omit the details.

To derive \eqref{E:NODERIVATIVESGMINUSSTANDARDSUPNORMSTMET}
for $\newg - \StMet$,
we first use \eqref{E:EVOLUTIONMETRICRENORMALIZED} to deduce that
$\SigmatLie_{\partial_t} (\newg - \StMet)
= 
- 
2 \scale^{-1} (\newg - \StMet) \cdot \FreeNewSec
- 
2 \scale^{-1} \StMet \cdot \FreeNewSec
$
plus error terms that, 
by Lemma~\ref{L:WEAKSUPNORM},
are bounded in the norm
$|\cdot|_{\StMet}$ by
$
\lesssim
\varepsilon 
\scale^{1/3 - c \upsigma}
$.
Also using \eqref{E:NODERIVATIVESKSTRONGPOINTWISESUPNORMSTMET}, we
find that
$
	\left|
		\partial_t
		\left|
			\newg - \StMet
		\right|_{\StMet}
	\right|
	\leq
	c \varepsilon \scale^{-1} 
	\left|
		\newg - \StMet
	\right|_{\StMet}
	+ 
	C \varepsilon \scale^{-1}
$.
From the previous inequality, 
Cor.~\ref{C:SCALEFACTORTIMEINTEGRALS},
Cor.~\ref{C:WEAKSMALLDATA},
and Gronwall's inequality, 
we obtain the pointwise bound
$\left|
		\newg - \StMet
	\right|_{\StMet}
\lesssim \varepsilon (1 + |\ln \scale|) \scale^{-c \varepsilon}
$.
From this bound and \eqref{E:SCALEFACTORLOGPOWERBOUND},
we find that
$\left|
		\newg - \StMet
	\right|_{\StMet}
\lesssim \sqrt{\varepsilon} \scale^{-c \sqrt{\varepsilon}}
$,
which yields
\eqref{E:NODERIVATIVESGMINUSSTANDARDSUPNORMSTMET}
for $\newg - \StMet$.
The proof of \eqref{E:NODERIVATIVESGMINUSSTANDARDSUPNORMSTMET}
for $\newg^{-1} - \StMet^{-1}$
is based on equation
\eqref{E:EVOLUTIONINVERSEMETRICRENORMALIZED}
and is similar; we omit the details.

To derive \eqref{E:NODERIVATIVESGMINUSSTANDARDSUPNORM}
for $\newg - \StMet$,
we use Cauchy--Schwarz relative to $\StMet$
and \eqref{E:NODERIVATIVESGMINUSSTANDARDSUPNORMSTMET}
to deduce that
$
	\left|
		\newg - \StMet
	\right|_{\newg}
\lesssim
\left|
	\newg - \StMet
\right|_{\StMet}
\left|
	\newg^{-1} 
\right|_{\StMet}
\lesssim 
\sqrt{\varepsilon} \scale^{-c \sqrt{\varepsilon}}
$
as desired.
The estimate \eqref{E:NODERIVATIVESGMINUSSTANDARDSUPNORM}
for $\newg^{-1} - \StMet^{-1}$ can be proved in a similar fashion
and we omit the details.

\end{proof}

The following corollary is an easy consequence of Lemma~\ref{L:NODERIVATIVESPARTIALTGANDKSTRONGSUPNORM}.

\begin{corollary}[\textbf{Preliminary sup-norm estimates}]
	\label{C:SUPNORMALMOSTREADYTOBEGRONWALLED}
	The following estimates hold for $\Sigma_t$-tangent tensorfields $\xi$:
	\begin{align}
	\left\|
		\xi
	\right\|_{C^0_{\newg}(\Sigma_t)}
	& 
	\leq
		\left\|
			\xi
		\right\|_{C^0_{\newg}(\Sigma_0)}
		+
		c \varepsilon
		\int_{s=0}^t
				\scale^{-1}(s)
				\left\|
					\xi
				\right\|_{C^0_{\newg}(\Sigma_s)}
			\, ds
		+
		\int_{s=0}^t
			\left\|
				\SigmatLie_{\partial_t} \xi
			\right\|_{C^0_{\newg}(\Sigma_s)}
		\, ds,
			\label{E:SUPNORMALMOSTREADYTOBEGRONWALLED} \\
		\left\|
			\xi
		\right\|_{C^0_{\newg}(\Sigma_t)}
		& 
		\leq
		C
		\scale^{-c \varepsilon}(t)
		\left\lbrace
			\left\|
				\xi
			\right\|_{C^0_{\newg}(\Sigma_0)}
			+
			\int_{s=0}^t
				\left\|
					\SigmatLie_{\partial_t} \xi
				\right\|_{C^0_{\newg}(\Sigma_s)}
			\, ds
		\right\rbrace,
		 \label{E:SUPNORMGRONWALLED}
		\\
		\left\|
			\xi
		\right\|_{C^0_{\StMet}(\Sigma_t)}
		& 
		\leq
		\left\|
			\xi
		\right\|_{C^0_{\StMet}(\Sigma_0)}
		+
		\int_{s=0}^t
			\left\|
				\SigmatLie_{\partial_t} \xi
			\right\|_{C^0_{\StMet}(\Sigma_s)}
			\, ds.
		 \label{E:STMETSUPNORMGRONWALLED}
	\end{align}
\end{corollary}
\begin{proof}
	The estimate \eqref{E:STMETSUPNORMGRONWALLED} is a
	simple consequence of the fundamental theorem of calculus
	and the fact that $\SigmatLie_{\partial_t} \StMet = \SigmatLie_{\partial_t} \StMet^{-1} = 0$.
	
	To derive \eqref{E:SUPNORMALMOSTREADYTOBEGRONWALLED},
	we first use \eqref{E:NODERIVATIVESPARTIALTGSTRONGSUPNORM}
	and the Cauchy--Schwarz inequality relative to $\newg$
	to deduce the pointwise bound
	$
	\left|
	\partial_t
	\left|
		\xi
	\right|_{\newg}
\right|
\leq
c \varepsilon
\scale^{-1}
\left|
	\xi
\right|_{\newg}
+
\left|
	\SigmatLie_{\partial_t} \xi
\right|_{\newg}
$.
From this bound and the fundamental theorem of calculus, we conclude \eqref{E:SUPNORMALMOSTREADYTOBEGRONWALLED}.

\eqref{E:SUPNORMGRONWALLED} then follows from \eqref{E:SUPNORMALMOSTREADYTOBEGRONWALLED},
Gronwall's inequality, and \eqref{E:EXPONENTIATEDSCALEFACTORTIMEINTEGRALS} with $p=1$.
\end{proof}

\subsection{Strong sup-norm estimates and Sobolev embedding}
In the next proposition, we provide the main results of this section.

\begin{proposition}[\textbf{Strong sup-norm estimates}]
	\label{P:STRONGSUPNORMESTIMATES}
	The following estimates hold:
	\begin{subequations}
	\begin{align}
		\left\|
			\SigmatLie_{\mathscr{Z}}^{\leq 10} (\newg - \StMet)
		\right\|_{C_{\StMet}^0(\Sigma_t)}
		& \lesssim \sqrt{\varepsilon} \scale^{-c \sqrt{\varepsilon}}(t),
			\label{E:METRICSTRONGSUPNROMSTMET} \\
		\left\|
			\SigmatLie_{\mathscr{Z}}^{\leq 10} (\newg^{-1} - \StMet^{-1})
		\right\|_{C_{\StMet}^0(\Sigma_t)}
		& \lesssim \sqrt{\varepsilon} \scale^{-c \sqrt{\varepsilon}}(t),
			\label{E:INVERSEMETRICSTRONGSUPNROMSTMET} \\
		\left\|
			\SigmatLie_{\mathscr{Z}}^{\leq 10} (\newg - \StMet)
		\right\|_{C_{\newg}^0(\Sigma_t)}
		& \lesssim \sqrt{\varepsilon} \scale^{-c \sqrt{\varepsilon}}(t),
			\label{E:METRICSTRONGSUPNROM} \\
		\left\|
			\SigmatLie_{\mathscr{Z}}^{\leq 10} (\newg^{-1} - \StMet^{-1})
		\right\|_{C_{\newg}^0(\Sigma_t)}
		& \lesssim \sqrt{\varepsilon} \scale^{-c \sqrt{\varepsilon}}(t),
			\label{E:INVERSEMETRICSTRONGSUPNROM} \\
		\left\|
			\SigmatLie_{\partial_t} \newg
		\right\|_{C_{\newg}^0(\Sigma_t)}
		& \lesssim \varepsilon\scale^{-1}(t),
		\label{E:AGAINNODERIVATIVESPARTIALTGSTRONGSUPNORM}
			\\
		\left\|
			\SigmatLie_{\partial_t} \newg^{-1}
		\right\|_{C_{\newg}^0(\Sigma_t)}
		& \lesssim \varepsilon \scale^{-1}(t),
				\label{E:AGAINNODERIVATIVESPARTIALTGINVERSESTRONGSUPNORM} \\
		\left\|
			\SigmatLie_{\mathscr{Z}}^{\leq 10} \FreeNewSec
		\right\|_{C_{\StMet}^0(\Sigma_t)}
		& \lesssim \varepsilon,
			\label{E:KSTRONGSUPNROMSTMETRIC} \\
		\left\|
			\SigmatLie_{\partial_t} \SigmatLie_{\mathscr{Z}}^{\leq 10} \FreeNewSec
		\right\|_{C_{\StMet}^0(\Sigma_t)}
		& \lesssim \varepsilon \scale^{1/3 -c \upsigma}(t),
			\label{E:NOLOSSTIMEDERIVATIVESKSTRONGSUPNROMSTMETRIC}
				\\
		\left\|
			\FreeNewSec
		\right\|_{C_{\newg}^0(\Sigma_t)}
		& \lesssim \varepsilon,
			\label{E:NOLOSSKSTRONGSUPNROM} \\
		\left\|
			\SigmatLie_{\mathscr{Z}}^{[1,10]} \FreeNewSec
		\right\|_{C_{\newg}^0(\Sigma_t)}
		& \lesssim \varepsilon \scale^{-c \sqrt{\varepsilon}}(t),
			\label{E:KSTRONGSUPNROM} 
			\\
		\left\|
			\mathscr{Z}^{\leq 10} \newtimescalar
		\right\|_{C^0(\Sigma_t)}
		& \lesssim \varepsilon,
			\label{E:SCALARFIELDTIMESTRONGSUPNROM} \\
		\left\|
			\partial_t \mathscr{Z}^{\leq 10} \newtimescalar
		\right\|_{C^0(\Sigma_t)}
		& \lesssim \varepsilon \scale^{1/3 - c \upsigma}(t),
			\label{E:TIMEDERIVATIVEOFSCALARFIELDTIMESTRONGSUPNROM} \\
		\left\|
			\SigmatLie_{\mathscr{Z}}^{\leq 9} \newspacescalar
		\right\|_{C_{\newg}^0(\Sigma_t)}
		& \lesssim \sqrt{\varepsilon} \scale^{-c \sqrt{\varepsilon}}(t),
			\label{E:SCALARFIELDSPACESTRONGSUPNROM} \\
		\left\|
			\mathscr{Z}^{\leq 8} \newlapse
		\right\|_{C^0(\Sigma_t)}
		& \lesssim \sqrt{\varepsilon} \scale^{-c \sqrt{\varepsilon}}(t),
			\label{E:LAPSESTRONGSUPNROM}
				\\
		\left\|
				\SigmatLie_{\mathscr{Z}}^{\leq 8}
				\left\lbrace
					\Ric^{\# } - \frac{2}{9} \ID
				\right\rbrace
		\right\|_{C_{\newg}^0(\Sigma_t)}
		& \lesssim \sqrt{\varepsilon} \scale^{-c \sqrt{\varepsilon}}(t).
			\label{E:RICCISTRONGSUPNROM}
	\end{align}	
	\end{subequations}
\end{proposition}

In the middle of our proof of Prop.~\ref{P:STRONGSUPNORMESTIMATES}, we will also 
prove the following corollary.

\begin{corollary}[\textbf{Improvement of Lemma~\ref{L:PRELIMINARYCOMPARISON}}]
	\label{C:IMPROVEMENTLEMMAOPERATORCOMPARISON}
	There exist constants $C >1$ and $c > 1$ such that 
	the following estimates hold for $\Sigma_t$-tangent tensors $\xi$:
	\begin{align} \label{E:IMPROVEDGNORMSTMETRICCOMPARISON}
		\frac{1}{C} \scale^{c \sqrt{\varepsilon}}
			\left|
				\xi	
			\right|_{\StMet}
		& \leq
		\left|
			\xi
		\right|_{\newg}	
		\leq C \scale^{- c \sqrt{\varepsilon}}
			\left|
				\xi	
			\right|_{\StMet}.
	\end{align}
	
	Moreover, let $\xi$ be a type $\binom{l}{m}$ $\Sigma_t-$tangent tensor.
	Let $\theta^1, \cdots, \theta^l \in \Theta$
	and
	$Z_1, \cdots, Z_m \in \mathscr{Z}$
	(see Def.\ \ref{D:LIETRANSPORTEDSPATIALFRAME})
	and let
	$\xi_{Z_1 \cdots Z_m}^{\theta^1 \cdots \theta^l}
	:= 
	\xi_{b_1 \cdots b_m}^{a_1 \cdots a_l}
	\theta_{a_1}^1 \cdots \theta_{a_l}^l
	Z_1^{b_1} \cdots Z_m^{b_m}
	$
	denote the contraction of $\xi$ against
	$\theta^1, \cdots, \theta^l$ 
	and
	$Z_1, \cdots, Z_m$.
	The following pointwise estimate holds:
	\begin{align} \label{E:IMPROVEDGNORMCONTRACTIONCOMPARISON}
		\frac{1}{C} \scale^{c \sqrt{\varepsilon}}
			\mathop{\sum_{\theta^1,\cdots,\theta^l \in \Theta}}_{Z_1,\cdots,Z_m \in \mathscr{Z}}
				\left|
					\xi_{Z_1 \cdots Z_m}^{\theta^1 \cdots \theta^l}	
				\right|
		& \leq
		\left|
			\xi
		\right|_{\newg}	
		\leq C \scale^{- c \sqrt{\varepsilon}}
			\mathop{\sum_{\theta^1,\cdots,\theta^l \in \Theta}}_{Z_1,\cdots,Z_m \in \mathscr{Z}}
				\left|
					\xi_{Z_1 \cdots Z_m}^{\theta^1 \cdots \theta^l}	
				\right|.
	\end{align}
	
	Finally, let $\nabla^L$ be the $L^{th}$ order covariant derivative operator.
	Then for $L = 1,2$, we have
	\begin{align} \label{E:IMPROVEDOPERATORCOMPARISON}
		\left|
			\nabla^L \xi
		\right|_{\newg}
		& \leq C
			\scale^{- c \sqrt{\varepsilon}} 
			\left|
				\SigmatLie_{\mathscr{Z}}^{\leq L} \xi
			\right|_{\newg},
		&&
		\left|
			\SigmatLie_{\mathscr{Z}}^L \xi
		\right|_{\newg}
		\leq 
		C \scale^{- c \sqrt{\varepsilon}} 	
		\left|
			\nabla^{\leq L} \xi
		\right|_{\newg}.
	\end{align}
\end{corollary}

\begin{proof}[Proof of Prop.~\ref{P:STRONGSUPNORMESTIMATES} and Cor.~\ref{C:IMPROVEMENTLEMMAOPERATORCOMPARISON}]
	We start by noting that the order in which we prove the estimates is important.
	
	We now recall that
	\eqref{E:AGAINNODERIVATIVESPARTIALTGSTRONGSUPNORM},
	\eqref{E:AGAINNODERIVATIVESPARTIALTGINVERSESTRONGSUPNORM},
	and \eqref{E:NOLOSSKSTRONGSUPNROM}
	were proved in Lemma~\ref{L:NODERIVATIVESPARTIALTGANDKSTRONGSUPNORM}.

	To prove 
	\eqref{E:KSTRONGSUPNROMSTMETRIC}
	and
	\eqref{E:NOLOSSTIMEDERIVATIVESKSTRONGSUPNROMSTMETRIC},
	we observe that
	for $1 \leq |\vec{I}| \leq 10$,
	all products on RHS~\eqref{E:COMMUTEDEVOLUTIONSECONDFUNDRENORMALIZED}
	are equal to the product of $\scale^{1/3}$ 
	and factors that, by virtue of \eqref{E:WEAKSUPNORMSTMET} and \eqref{E:WEAKRICCISUPNORMSTMET},
	are $\lesssim \varepsilon \scale^{- c \upsigma}(t)$
	in the norm $\| \cdot \|_{C_{\StMet}^0(\Sigma_t)}$.
	The same remarks hold for RHS~\eqref{E:EVOLUTIONSECONDFUNDRENORMALIZED},
	which is relevant for the case $\vec{I} = 0$.
	This yields \eqref{E:NOLOSSTIMEDERIVATIVESKSTRONGSUPNROMSTMETRIC}.
	\eqref{E:KSTRONGSUPNROMSTMETRIC} then follows from
	\eqref{E:NOLOSSTIMEDERIVATIVESKSTRONGSUPNROMSTMETRIC},
	\eqref{E:STMETSUPNORMGRONWALLED}, 
	the small-data bound \eqref{E:WEAKSMALLDATASTMETR},
	and the fact that $0 \leq \scale \leq 1$.
	
	The estimates
	\eqref{E:SCALARFIELDTIMESTRONGSUPNROM}-\eqref{E:TIMEDERIVATIVEOFSCALARFIELDTIMESTRONGSUPNROM}
	can be proved using similar arguments based on 
	equations 
	\eqref{E:WAVEEQUATIONRENORMALIZED} and \eqref{E:COMMUTEDWAVEEQUATIONRENORMALIZED}
	and we omit the details.
	
	We now prove \eqref{E:METRICSTRONGSUPNROMSTMET}.
	The case $|\vec{I}| = 0$ has already been treated in Lemma~\ref{L:NODERIVATIVESPARTIALTGANDKSTRONGSUPNORM}.
	We now show that for $1 \leq |\vec{I}| \leq 10$, the terms
	$
	\CommutedMetBorderInhom{\vec{I}}
	$
	and
	$
	\CommutedMetJunkInhom{\vec{I}}
	$
	on RHS~\eqref{E:COMMUTEDEVOLUTIONMETRICRENORMALIZED}
	verify
	\begin{align} \label{E:BORDERINHOMTERMESTIMATEPARTIALTGSTRONGSUPNROM}
		\left| 
			\CommutedMetBorderInhom{\vec{I}}
		\right|_{\StMet}
		& \lesssim
			\varepsilon
			\scale^{-1}
			\left|
				\SigmatLie_{\mathscr{Z}}^{[1,|\vec{I}|]} \newg
			\right|_{\StMet}
			+ 
			\varepsilon,
				\\
		\left| 
			\CommutedMetJunkInhom{\vec{I}}
		\right|_{\StMet}
		& \lesssim \varepsilon \scale^{- c \upsigma}.
		\label{E:JUNKINHOMTERMESTIMATEPARTIALTGSTRONGSUPNROM}
	\end{align}
	To prove~\eqref{E:BORDERINHOMTERMESTIMATEPARTIALTGSTRONGSUPNROM},
	we use \eqref{E:KSTRONGSUPNROMSTMETRIC}
	to bound the first sum on RHS~\eqref{E:ICOMMUTEDMETRICBORDERTERMS}
	in the norm $|\cdot|_{\StMet}$
	by $\lesssim$ the first term on RHS~\eqref{E:BORDERINHOMTERMESTIMATEPARTIALTGSTRONGSUPNROM}.
	We then use the estimate \eqref{E:WEAKSUPNORMSTMET}
	to bound the remaining two sums on RHS~\eqref{E:ICOMMUTEDMETRICBORDERTERMS}
	in the norm $|\cdot|_{\StMet}$
	by $\lesssim \varepsilon$.
	We have thus proved \eqref{E:BORDERINHOMTERMESTIMATEPARTIALTGSTRONGSUPNROM}.
	Similarly, the estimate \eqref{E:WEAKSUPNORMSTMET}
	yields that RHS~\eqref{E:ICOMMUTEDMETRICJUNKTERMS}
	is bounded in the norm $|\cdot|_{\StMet}$
	by $\lesssim \varepsilon \scale^{- c \upsigma}$,
	which yields \eqref{E:JUNKINHOMTERMESTIMATEPARTIALTGSTRONGSUPNROM}.
	Next, we use equation \eqref{E:COMMUTEDEVOLUTIONMETRICRENORMALIZED},
	the estimates
	\eqref{E:BORDERINHOMTERMESTIMATEPARTIALTGSTRONGSUPNROM}-\eqref{E:JUNKINHOMTERMESTIMATEPARTIALTGSTRONGSUPNROM},
	the small-data bound \eqref{E:WEAKSMALLDATASTMETR},
	and Cor.~\ref{C:MONOTONICITYTIME} to deduce,
	by integrating in time, the following estimate:
	\begin{align} \label{E:LIEPARTIALTCOMMUTEDMETRICSUPNORMALMOSTREADYTOBEGRONWALLED}
	\left\|
		\SigmatLie_{\mathscr{Z}}^{[1,|\vec{I}|]} \newg
	\right\|_{C^0_{\StMet}(\Sigma_t)}
	& \leq
		C (1 + |\ln \scale(t)|) \varepsilon
		+
		c \varepsilon
		\int_{s=0}^t
				\scale^{-1}(s)
				\left\|
					\SigmatLie_{\mathscr{Z}}^{[1,|\vec{I}|]} \newg
				\right\|_{C^0_{\StMet}(\Sigma_s)}
			\, ds.
	\end{align}
	From \eqref{E:LIEPARTIALTCOMMUTEDMETRICSUPNORMALMOSTREADYTOBEGRONWALLED},
	Gronwall's inequality,
	and \eqref{E:EXPONENTIATEDSCALEFACTORTIMEINTEGRALS},
	we obtain
	$
	\left\|
		\SigmatLie_{\mathscr{Z}}^{[1,|\vec{I}|]} \newg
	\right\|_{C^0_{\StMet}(\Sigma_t)}
	\lesssim
	\varepsilon (1 + |\ln \scale(t)|) \scale^{-c \varepsilon}(t)
	$.
	From this bound and \eqref{E:SCALEFACTORLOGPOWERBOUND},
	we conclude the desired estimate \eqref{E:METRICSTRONGSUPNROMSTMET}.
	The estimate \eqref{E:INVERSEMETRICSTRONGSUPNROMSTMET} 
	can be proved in the same way 
	based on the evolution equation \eqref{E:COMMUTEDEVOLUTIONINVERSEMETRICRENORMALIZED}
	(the case $\vec{I} = 0$ having been treated in Lemma~\ref{L:NODERIVATIVESPARTIALTGANDKSTRONGSUPNORM})
	and we omit the details.
	
	To deduce Cor.~\ref{C:IMPROVEMENTLEMMAOPERATORCOMPARISON},
	we simply repeat the proof of Lemma~\ref{L:PRELIMINARYCOMPARISON}, 
	but now using the estimates 
	\eqref{E:METRICSTRONGSUPNROMSTMET} and \eqref{E:INVERSEMETRICSTRONGSUPNROMSTMET}
	in place of the bootstrap assumptions/estimates 
	that we used in the original proof of the lemma
	(see the first sentence of the proof of Lemma~\ref{L:PRELIMINARYCOMPARISON}).
	
	The estimates \eqref{E:METRICSTRONGSUPNROM} and \eqref{E:INVERSEMETRICSTRONGSUPNROM}
	now follow from \eqref{E:METRICSTRONGSUPNROMSTMET}-\eqref{E:INVERSEMETRICSTRONGSUPNROMSTMET}
	and \eqref{E:IMPROVEDGNORMSTMETRICCOMPARISON}.
	Similarly, \eqref{E:KSTRONGSUPNROM} follows from
	\eqref{E:KSTRONGSUPNROMSTMETRIC} and \eqref{E:IMPROVEDGNORMSTMETRICCOMPARISON}.
	
	We now prove \eqref{E:SCALARFIELDSPACESTRONGSUPNROM}.
	Using \eqref{E:WEAKSUPNORM}, we see that for $1 \leq |\vec{I}| \leq 9$,
	all terms on RHS~\eqref{E:COMMUTEDSPACEDERIVATIVESWAVEEQUATIONRENORMALIZED}
	except for $\scale^{-1} \nabla \mathscr{Z}^{\vec{I}} \newtimescalar$
	are bounded in the norm 
	$\| \cdot \|_{C_{\newg}^0(\Sigma_t)}$
	by $\lesssim \varepsilon$.
	Moreover, 
	from \eqref{E:SCALARFIELDTIMESTRONGSUPNROM}
	and Cor.~\ref{C:IMPROVEMENTLEMMAOPERATORCOMPARISON},
	we deduce
	$\left\|
		\scale^{-1} \nabla \mathscr{Z}^{\vec{I}} \newtimescalar
	\right\|_{C_{\newg}^0(\Sigma_t)}
	\lesssim 
	\varepsilon \scale^{-1- c \sqrt{\varepsilon}}(t)
	$.
	Combining these estimates, we 
	deduce
	$
		\left\|
			\SigmatLie_{\partial_t} \SigmatLie_{\mathscr{Z}}^{[1,|\vec{I}|]} \newspacescalar
		\right\|_{C_{\newg}^0(\Sigma_t)}
		\lesssim
		\varepsilon \scale^{-1- c \sqrt{\varepsilon}}(t)
	$.
	Using this estimate,
	inequality \eqref{E:SUPNORMGRONWALLED}
	with $\xi := \SigmatLie_{\mathscr{Z}}^{\vec{I}} \newspacescalar$,
	and the small-data bound \eqref{E:WEAKSMALLDATA},
	we deduce
	$
	\left\|
		\SigmatLie_{\mathscr{Z}}^{[1,|\vec{I}|]} \newspacescalar
	\right\|_{C^0_{\newg}(\Sigma_t)}
	\lesssim
	\scale^{-c \varepsilon}(t)
		\left\lbrace
			\varepsilon
				+
			\varepsilon
			\int_{s=0}^t
				\scale^{-1 - c \sqrt{\varepsilon}}(s)
			\, ds
		\right\rbrace
	$.
	The desired estimate \eqref{E:SCALARFIELDSPACESTRONGSUPNROM}
	for 
	$\left\|
		\SigmatLie_{\mathscr{Z}}^{[1,9]} \newspacescalar
	\right\|_{C^0_{\newg}(\Sigma_t)}$
	now follows from the previous estimate
	and \eqref{E:SCALEFACTORTIMEINTEGRALS} with $p = -1 - c \sqrt{\varepsilon}$.

	We now prove \eqref{E:RICCISTRONGSUPNROM}.
	From \eqref{E:RICCIESTONEUPONEDOWN},
	\eqref{E:ERRORRICONEUPONEDOWN},
	\eqref{E:METRICSTRONGSUPNROMSTMET},
	\eqref{E:INVERSEMETRICSTRONGSUPNROMSTMET},
	and \eqref{E:IMPROVEDGNORMSTMETRICCOMPARISON},
	we deduce that
	$\left\| 
		\Ric^{\# } - \frac{2}{9} \ID
	\right\|_{C^0_{\newg}(\Sigma_t)}
	\lesssim
	\scale^{-c \sqrt{\varepsilon}}(t)
	\left\|
		\Ric_{\triangle}
	\right\|_{C^0_{\StMet}(\Sigma_t)}
	\lesssim \sqrt{\varepsilon} \scale^{-c \sqrt{\varepsilon}}(t)
	$
	as desired.
	To obtain the bound
	$\left\| 
		\SigmatLie_{\mathscr{Z}}^{\vec{I}} \Ric^{\# } 
	\right\|_{C^0_{\newg}(\Sigma_t)} 
	\lesssim 
	\sqrt{\varepsilon} \scale^{-c \sqrt{\varepsilon}}$
	for $1 \leq |\vec{I}| \leq 8$, we use the identities
	\eqref{E:RICCILIEDERIVATIVESCHEMATIC} and \eqref{E:RICINHOMERROR},
	the estimate \eqref{E:METRICSTRONGSUPNROM},
	the estimate
	$
	\left\| 
		\Ric^{\# } 
	\right\|_{C^0_{\newg}(\Sigma_t)}
	\lesssim 
	1 + \sqrt{\varepsilon} \scale^{-c \sqrt{\varepsilon}}(t)
	$
	(which follows from the bound for $\Ric^{\# } - \frac{2}{9} \ID$ proved above),
	and Cor.~\ref{C:IMPROVEMENTLEMMAOPERATORCOMPARISON}.
	We have thus proved \eqref{E:RICCISTRONGSUPNROM}.
	
	We now prove \eqref{E:LAPSESTRONGSUPNROM}.
	By \eqref{E:MAXIMUMPRINCIPLEESTIMATE},
	it suffices to show that for $1 \leq |\vec{I}| \leq 8$, the norms
	$
	\left\|
		\cdot
	\right\|_{C_{\newg}^0(\Sigma_t)}
	$
	of the inhomogeneous terms on RHS~\eqref{E:COMMUTEDLAPSEPDERENORMALIZEDLOWERDERIVATIVES} are
	$\lesssim \sqrt{\varepsilon} \scale^{-c \sqrt{\varepsilon}}(t)$
	and that the inhomogeneous terms on RHS~\eqref{E:LAPSEPDERENORMALIZEDLOWERDERIVATIVES} verify the same bound.
	Using 
	\eqref{E:RICINHOMERROR},
	\eqref{E:METRICSTRONGSUPNROM},
	\eqref{E:INVERSEMETRICSTRONGSUPNROM},
	\eqref{E:SCALARFIELDSPACESTRONGSUPNROM},
	\eqref{E:RICCISTRONGSUPNROM},
	and Cor.~\ref{C:IMPROVEMENTLEMMAOPERATORCOMPARISON},
	we see that the inhomogeneous term
	$\CommutedLapseLowBorderInhom{\vec{I}}$
	on RHS~\eqref{E:COMMUTEDLAPSEPDERENORMALIZEDLOWERDERIVATIVES}
	verifies the desired bound
	$
	\left\|
		\CommutedLapseLowBorderInhom{\vec{I}}
	\right\|_{C^0(\Sigma_t)}
	\lesssim \varepsilon \scale^{-c \sqrt{\varepsilon}}(t)
	$
	and that the inhomogeneous terms on RHS~\eqref{E:LAPSEPDERENORMALIZEDLOWERDERIVATIVES} verify the same bound.
	Finally, from 
	the commutation formula \eqref{E:LAPLACIANLIEZCOMMUTATOR},
	\eqref{E:WEAKSUPNORM},
	\eqref{E:WEAKRICCISUPNORM},
	the fact that $\ScalarCur$ is the pure trace of $\Ric^{\#}$,
	and Cor.~\ref{C:IMPROVEMENTLEMMAOPERATORCOMPARISON},
	we deduce that the inhomogeneous term $\scale^{4/3} \CommutedLapseLowJunkInhom{\vec{I}}$ 
	on RHS~\eqref{E:COMMUTEDLAPSEPDERENORMALIZEDLOWERDERIVATIVES} 
	verifies
	$
	\scale^{4/3}(t)
	\left\|
		\CommutedLapseLowJunkInhom{\vec{I}}
	\right\|_{C^0(\Sigma_t)}
	\lesssim
	\varepsilon \scale^{4/3 - c \upsigma}(t)
	\lesssim
	\varepsilon
	$
	as desired.
	We have therefore proved \eqref{E:LAPSESTRONGSUPNROM},
	which completes the proof of the proposition and corollary.
	
\end{proof}

\subsection{Sobolev embedding}
\label{SS:SOBOLEVEMBEDDING}
In Lemma~\ref{L:SOBOLEV} below, we prove our main Sobolev embedding result.
We first provide a simple comparison lemma.

\begin{lemma}[\textbf{Background $L^2$ norm and geometric $L^2$ norm comparison estimates}]
	\label{L:L2NORMSCOMPARISONESTIMATES}
	The following estimates 
	hold for $\Sigma_t-$tangent tensorfields $\xi$:
	\begin{align} \label{E:L2NORMSCOMPARISONESTIMATES}
		\| \xi \|_{L^2_{\StMet}(\Sigma_t)}
		& \lesssim \scale^{- c \sqrt{\varepsilon}}(t) \| \xi \|_{L^2_{\newg}(\Sigma_t)},
		&
		\| \xi \|_{L^2_{\newg}(\Sigma_t)}
		& \lesssim C \scale^{- c \sqrt{\varepsilon}}(t) \| \xi \|_{L^2_{\StMet}(\Sigma_t)}.
	\end{align}
\end{lemma}

\begin{proof}
	The estimates in \eqref{E:L2NORMSCOMPARISONESTIMATES} are
	a simple consequence of
	Lemma~\ref{L:VOLUMEFORMCOMPARISON} and Cor.~\ref{C:IMPROVEMENTLEMMAOPERATORCOMPARISON}.
\end{proof}

\begin{lemma}[\textbf{Sobolev embedding}]
	\label{L:SOBOLEV}
	The following estimates hold for any $\Sigma_t$-tangent tensorfield $\xi$:
	\begin{subequations}
	\begin{align}
		\left\|
			\xi
		\right\|_{C_{\StMet}^0(\Sigma_t)}
		& 
		\lesssim 
		\left\|
			\SigmatLie_{\mathscr{Z}}^{\leq 2} \xi
		\right\|_{L_{\StMet}^2(\Sigma_t)},
			\label{E:STSOBOLEV} \\
		\left\|
			\xi
		\right\|_{C_{\newg}^0(\Sigma_t)}
		& \lesssim 
		\scale^{- c \sqrt{\varepsilon}}(t)
		\left\|
			\SigmatLie_{\mathscr{Z}}^{\leq 2} \xi
		\right\|_{L_{\newg}^2(\Sigma_t)}.
		 \label{E:SOBOLEV}
	\end{align}
	\end{subequations}
\end{lemma}

\begin{proof}
	\eqref{E:STSOBOLEV} follows from standard Sobolev embedding on $(\mathbb{S}^3,\StMet)$.
	
	To obtain \eqref{E:SOBOLEV}, we first use
	\eqref{E:STSOBOLEV} and Cor.~\ref{C:IMPROVEMENTLEMMAOPERATORCOMPARISON}
	to deduce 
	$\left\|
		\xi
	\right\|_{C_{\newg}^0(\Sigma_t)}
	:=
	\left\|
		|\xi|_{\newg}
	\right\|_{C^0(\Sigma_t)}
	\lesssim
	\scale^{- c \sqrt{\varepsilon}}(t)
	\left\|
		|\xi|_{\StMet}
	\right\|_{C^0(\Sigma_t)}
	=
	\scale^{- c \sqrt{\varepsilon}}(t)
	\left\|
		\xi
	\right\|_{C_{\StMet}^0(\Sigma_t)}
	\lesssim
	\scale^{- c \sqrt{\varepsilon}}(t) 
	\left\|
		\SigmatLie_{\mathscr{Z}}^{\leq 2} \xi
	\right\|_{L_{\StMet}^2(\Sigma_t)}
	$.
	Using this estimate and the first estimate in \eqref{E:L2NORMSCOMPARISONESTIMATES}, we 
	arrive at the desired bound \eqref{E:SOBOLEV}.
	
\end{proof}

\subsection{Preliminary $L^2$ estimates for tensorfields}
In the next lemma, we provide some preliminary $L^2$ estimates that we will use later in the paper.

\begin{lemma}[\textbf{Preliminary} $L^2$ \textbf{estimates for tensorfields}]
	\label{L:PRELIMINARYL2ESTIMATES}
	The following estimates hold for 
	$\Sigma_t-$tangent tensorfields $\xi$:
	\begin{subequations}
	\begin{align} 
	\left\|
		\xi
	\right\|_{L_{\newg}^2(\Sigma_t)}
	& \leq 
		\left\|
			\xi
		\right\|_{L_{\newg}^2(\Sigma_0)}
	+
	c \varepsilon
	\int_{s=0}^t
		\scale^{-1}(s)
		\left\|
			\xi
		\right\|_{L_{\newg}^2(\Sigma_s)}
	\, ds
	+
	\int_{s=0}^t
		\left\|
			\SigmatLie_{\partial_t} \xi
		\right\|_{L_{\newg}^2(\Sigma_s)}
	\, ds,
			\label{E:BASICGRONWALLREADYESTIMATEFORATENSORFIELD} \\
		\left\|
			\xi
		\right\|_{L^2_{\newg}(\Sigma_t)}
		& 
		\lesssim
		\scale^{-c \varepsilon}(t)
		\left\lbrace
			\left\|
				\xi
			\right\|_{L^2_{\newg}(\Sigma_0)}
			+
			\int_{s=0}^t
				\left\|
					\SigmatLie_{\partial_t} \xi
				\right\|_{L^2_{\newg}(\Sigma_s)}
			\, ds
		\right\rbrace.
		\label{E:BASICGRONWALLESTIMATEFORATENSORFIELD}
	\end{align}
	\end{subequations}
\end{lemma}
\begin{proof}
	First, arguing as in the proof of \eqref{E:SUPNORMALMOSTREADYTOBEGRONWALLED},
	we deduce
	\begin{align} \label{E:PARTIALTXISQUAREDFORMULA}
		\left|
			\partial_t (|\xi|_{\newg}^2)
		\right|
		& 
		\leq 
		c \varepsilon \scale^{-1}
		\left|
			\SigmatLie_{\partial_t} \xi
		\right|_{\newg}^2
		+
		2 \left|
				\langle \xi, \SigmatLie_{\partial_t} \xi \rangle_{\newg}
			\right|.
\end{align}
From \eqref{E:PARTIALTXISQUAREDFORMULA} 
and \eqref{E:VOLUMEFORMCOMPARISON},
we deduce
\begin{align} \label{E:PARTIALTXISQUAQUAREDINEQUALITY}
	\left|
	\frac{d}{dt}
	\int_{\Sigma_t}
		|\xi|_{\newg}^2
	\, d \tvol
	\right|
	& \leq 
	c \varepsilon \scale^{-1}(t)
	\int_{\Sigma_t} 
		|\xi|_{\newg}^2 
	\, d \tvol
		+ 
	2	
	\int_{\Sigma_t}
		\left|
			\xi
		\right|_{\newg}
		\left|
			\SigmatLie_{\partial_t} \xi
		\right|_{\newg}
	\, d \tvol,
\end{align}
from which \eqref{E:BASICGRONWALLREADYESTIMATEFORATENSORFIELD}
follows as a simple consequence.
\eqref{E:BASICGRONWALLESTIMATEFORATENSORFIELD} then follows from
Gronwall's inequality and \eqref{E:EXPONENTIATEDSCALEFACTORTIMEINTEGRALS}
with $p=-1$.

\end{proof}

\section{The fundamental energy integral inequalities}
\label{S:FUNDAMENTALENERGY}
Our main goal in this section is to prove Prop.~\ref{P:FUNDAMENTALENERGYINTEGRALINEQUALITY},
which provides integral inequalities for the combined metric + scalar field energies.
Its proof is located in Subsect.\ \ref{SS:PROOFOFUNDAMENTALENERGY}.
Before proving the proposition, we will establish some preliminary
lemmas in which we separately derive integral inequalities
for the metric and scalar field energies.

\begin{proposition}[\textbf{Fundamental total energy integral inequalities}]
	\label{P:FUNDAMENTALENERGYINTEGRALINEQUALITY}
	Let $1 \leq M \leq 16$.
	There exist a small constant $\smallparameter_* > 0$
	and a large constant $C > 0$ 
	such that the energy $\Totalenergy{M}{\smallparameter_*}(t)$
	defined by \eqref{E:TOTALENERGY}
	obeys the following inequality for $t \in [0,\Tboot)$:
	\begin{align}
		\Totalenergy{M}{\smallparameter_*}(t)
		& +
		\sum_{1 \leq |\vec{I}| \leq M}
		\int_{s=0}^t
				|\scale'(s)| 
				\scale^{1/3}(s)
				\left\|
					\nabla \SigmatLie_{\mathscr{Z}}^{\vec{I}} \newg
				\right\|_{L_{\newg}^2(\Sigma_s)}^2
		\, ds
			\label{E:FUNDAMENTALENERGYINTEGRALINEQUALITY} \\
		& \ \
		+
		\sum_{1 \leq |\vec{I}| \leq M}
		\int_{s=0}^t
			|\scale'(s)| 
			\scale^{1/3}(s)
			\left\|
				\SigmatLie_{\mathscr{Z}}^{\vec{I}} \newspacescalar
			\right\|_{L_{\newg}^2(\Sigma_s)}^2
		\, ds
			\notag \\
		& \ \
		+
		\sum_{1 \leq |\vec{I}| \leq M}
		\int_{s=0}^t
			|\scale'(s)| 
			\scale^3(s)
			\left\|
				\nabla \mathscr{Z}^{\vec{I}} \newlapse
			\right\|_{L_{\newg}^2(\Sigma_s)}^2
		\, ds
			\notag \\
	& \ \
		+
		\sum_{1 \leq |\vec{I}| \leq M}
		\int_{s=0}^t
			|\scale'(s)|^3 
			\scale^{5/3}(s)
			\left\|
				\mathscr{Z}^{\vec{I}} \newlapse
			\right\|_{L_{\newg}^2(\Sigma_s)}^2
			\, ds
			\notag \\
		& \leq
			C \Totalenergy{M}{\smallparameter_*}(0)
			\notag \\
		& \ \
		+
		C \varepsilon
		\int_{s=0}^t
			\scale^{-1}(s)
			\Totalenergy{M}{\smallparameter_*}(s)
		\, ds
		+
		C 
		\int_{s=0}^t
			\scale^{-1/3}(s)
			\Totalenergy{M}{\smallparameter_*}(s)
		\, ds	
			\notag \\
		& \ \
		+
		C
		\sum_{1 \leq |\vec{I}| \leq M}
		\int_{s=0}^t
			\Energyborder{\vec{I}}(s)
		\, ds
		+
		C
		\sum_{1 \leq |\vec{I}| \leq M}
		\int_{s=0}^t
			\Energyjunk{\vec{I}}(s)
		\, ds,
		\notag
	\end{align}
	where
	\begin{subequations}
	\begin{align} \label{E:ENERGYBORDER}
		\Energyborder{\vec{I}}(s)
		& :=
				\scale^{1/3}(s)
				\left\|
					\CommutedMomBorderInhomDown{\vec{I}}
				\right\|_{L_{\newg}^2(\Sigma_s)}^2
				+
				\scale^{1/3}(s)
				\left\|
					\CommutedMomBorderInhomUp{\vec{I}}
				\right\|_{L_{\newg}^2(\Sigma_s)}^2
					\\
			& \ \
				+
				\scale^{1/3}(s)
				\left\|
					\CommutedGradMetBorderInhom{\vec{I}}
				\right\|_{L_{\newg}^2(\Sigma_s)}^2
				+
				\scale^3(s)
				\left\|
					\CommutedSpaceSfBorderInhom{\vec{I}}
				\right\|_{L_{\newg}^2(\Sigma_s)}^2
				\notag \\
			& \ \
				+
				\scale^{-1}(s)
				\left\|
					\CommutedLapseHighBorderInhom{\vec{I}}
				\right\|_{L_{\newg}^2(\Sigma_s)}^2	
				\notag \\
			& \ \
				+
				\frac{1}{\varepsilon}
				\scale^{5/3}(s)
				\left\|
					\CommutedSecFunBorderInhom{\vec{I}}
				\right\|_{L_{\newg}^2(\Sigma_s)}^2
				+
				\frac{1}{\varepsilon}
				\scale^{5/3}(s)
				\left\|
					\CommutedTimeSfBorderInhom{\vec{I}}
				\right\|_{L_{\newg}^2(\Sigma_s)}^2,
				\notag 
	\end{align}
	\begin{align}
		\Energyjunk{\vec{I}}(s)
		& :=
				\scale^{7/3}(s)
				\left\|
					\CommutedGradMetJunkInhom{\vec{I}}
				\right\|_{L_{\newg}^2(\Sigma_s)}^2
				+
				\scale^{7/3}(s)
				\left\|
					\CommutedSpaceSfJunkInhom{\vec{I}}
				\right\|_{L_{\newg}^2(\Sigma_s)}^2
				\label{E:ENERGYJUNK} \\
			& \ \
				+
				\scale(s)
				\left\|
					\CommutedLapseHighJunkInhom{\vec{I}}
				\right\|_{L_{\newg}^2(\Sigma_s)}^2	
				+
				\scale(s)
				\left\|
					\RicErrorInhom{\vec{I}}
				\right\|_{L_{\newg}^2(\Sigma_s)}^2
					\notag \\
			& \ \
				+
				\scale(s)
				\left\|
					\CommutedSecFunJunkInhom{\vec{I}}
				\right\|_{L_{\newg}^2(\Sigma_s)}^2
				+
				\scale(s)
				\left\|
					\CommutedTimeSfJunkInhom{\vec{I}}
				\right\|_{L_{\newg}^2(\Sigma_s)}^2
				\notag
					\\
			& \ \
				+ \scale^4(s)
					\left\|
						\nabla \mathscr{Z}^{\vec{I}} \newlapse
					\right\|_{L_{\newg}^2(\Sigma_s)}^2
				+ 
					\scale^{7/3}(s)
					\left\|
						\mathscr{Z}^{\vec{I}} \newlapse
					\right\|_{L_{\newg}^2(\Sigma_s)}^2,
					\notag
		\end{align}
	\end{subequations}
	and the terms $\CommutedMomBorderInhomDown{\vec{I}},\cdots$ on
	RHSs~\eqref{E:ENERGYBORDER}-\eqref{E:ENERGYJUNK}
	are the error terms from the $\SigmatLie_{\mathscr{Z}}^{\vec{I}}$-commuted equations of
	Prop.~\ref{P:ICOMMUTEDEQNS}.
\end{proposition}

\subsection{Energy integral inequalities for the scalar field and lapse}
\label{SS:SCALARFIELDENERGYINEQ}
In this subsection, we derive integral inequalities for the scalar field
energies. A key point is that the proof is based on 
energy identities exhibiting remarkable structures,
which show in particular that a seemingly damaging term can be re-expressed, up to error terms, 
as a spacetime integral that controls the lapse with a favorable sign.
These structures are effectively already found in the divergence identity of Lemma~\ref{L:SFCURRENTDIV}.
A less careful proof of inequality \eqref{E:SFINTEGRALINEQUALITY} would have led to the presence of
unsigned borderline error
terms with large coefficients of size $C$. In turn, this would have resulted in
singular high-order energy estimates featuring a severe blowup-rate of $\scale^{-C}(t)$ as $t \uparrow \TCrunch$, 
which would have prevented us from closing our bootstrap argument.
Moreover, these arguments also give rise to the spacetime integrals on LHS~\eqref{E:SFINTEGRALINEQUALITY}
that control the time-rescaled lapse variable $\newlapse$.

\begin{lemma}[\textbf{Fundamental combined scalar field and lapse energy integral inequalities}]
	\label{L:SFINTEGRALINEQUALITY}
	Let $1 \leq M \leq 16$.
	There exists a constant $C > 0$ such that
	the energy $\Sfenergy{M}(t)$
	defined by \eqref{E:SFENERGY}
	obeys the following inequality for $t \in [0,\Tboot)$:
	\begin{align} \label{E:SFINTEGRALINEQUALITY}
		\Sfenergy{M}(t)
		& +
		\frac{1}{12}
		\sum_{1 \leq |\vec{I}| \leq M}
		\int_{s=0}^t
			\int_{\Sigma_s}
				|\scale'(s)| 
				\scale^{1/3}(s)
				\left|
					\SigmatLie_{\mathscr{Z}}^{\vec{I}} \newspacescalar
				\right|_{\newg}^2
			\, d \tvol
		\, ds
			\\
		& \ \
		+
		\frac{1}{12}
		\sum_{1 \leq |\vec{I}| \leq M}
		\int_{s=0}^t
			\int_{\Sigma_s}
				|\scale'(s)| 
				\scale^3(s)
				\left|
					\nabla \mathscr{Z}^{\vec{I}} \newlapse
				\right|_{\newg}^2
			\, d \tvol
		\, ds
			\notag \\
		& \ \
		+
		\frac{1}{2}
		\sum_{1 \leq |\vec{I}| \leq M}
		\int_{s=0}^t
			\int_{\Sigma_s}
				|\scale'(s)|^3 
				\scale^{5/3}(s)
				\left|
					\mathscr{Z}^{\vec{I}} \newlapse
				\right|^2
			\, d \tvol
		\, ds
			\notag \\
		& \leq
			\Sfenergy{M}(0)
				\notag \\
		& 
		\ \
		+ 
		\frac{1}{\varepsilon}
		\sum_{1 \leq |\vec{I}| \leq M}
		\int_{s=0}^t
			\int_{\Sigma_s}
				\scale^{5/3}(s)
				\left|
					\CommutedTimeSfBorderInhom{\vec{I}}
				\right|_{\newg}^2
			\, d \tvol
		\, ds
			\notag \\
		& \ \
		+ 
		C
		\sum_{1 \leq |\vec{I}| \leq M}
		\int_{s=0}^t
			\int_{\Sigma_s}
				\scale^3(s)
				\left|
					\CommutedSpaceSfBorderInhom{\vec{I}}
				\right|_{\newg}^2
			\, d \tvol
		\, ds
			\notag \\
	& \ \
		+ 
		C
		\sum_{1 \leq |\vec{I}| \leq M}
		\int_{s=0}^t
			\int_{\Sigma_s}
				\scale^{-1}(s)
				\left|
					\CommutedLapseHighBorderInhom{\vec{I}}
				\right|^2
			\, d \tvol
		\, ds
		\notag \\
	&  \ \
		+ 
		C
		\sum_{1 \leq |\vec{I}| \leq M}
		\int_{s=0}^t
			\int_{\Sigma_s}
				\scale(s)
				\left|
					\CommutedTimeSfJunkInhom{\vec{I}} 
				\right|^2
			\, d \tvol
		\, ds
		\notag
			\\
	&  \ \
		+ 
		C
		\sum_{1 \leq |\vec{I}| \leq M}
		\int_{s=0}^t
			\int_{\Sigma_s}
				\scale^2(s) 
				\left|
					\CommutedSpaceSfJunkInhom{\vec{I}} 
				\right|^2
			\, d \tvol
		\, ds
		\notag \\
		&  \ \
		+ 
		C
		\sum_{1 \leq |\vec{I}| \leq M}
		\int_{s=0}^t
			\int_{\Sigma_s}
				\scale(s) 
				\left|
					\CommutedLapseHighJunkInhom{\vec{I}} 
				\right|^2
			\, d \tvol
		\, ds
		\notag
			\\
	& \ \
			+ 
			C
			\sum_{1 \leq |\vec{I}| \leq M}
			\int_{s=0}^t
				\int_{\Sigma_s}
					\scale^4(s)
					\left|
						\nabla \SigmatLie_{\mathscr{Z}}^{\vec{I}} \newlapse
					\right|_{\newg}^2
				\, d \tvol	
			\, ds 
			\notag \\
	& \ \
			+ 
			C
			\sum_{1 \leq |\vec{I}| \leq M}
			\int_{s=0}^t
				\int_{\Sigma_s}
					\scale^{7/3}(s)
					\left|
						\mathscr{Z}^{\vec{I}} \newlapse
					\right|^2
				\, d \tvol	
			\, ds 
			\notag \\
	& \ \
		+ 
		C \varepsilon
		\int_{s=0}^t
			\scale^{-1}(s)
			\Sfenergy{M}(s)
		\, ds	
		\notag \\
	&  \ \
		+
		C 
		\sum_{1 \leq |\vec{I}| \leq M}
		\int_{s=0}^t
			\scale^{-1/3}(s)
			\Sfenergy{M}(s)
		\, ds,
		\notag 
	\end{align}
	where the terms $\CommutedTimeSfBorderInhom{\vec{I}},\cdots$ on
	RHS~\eqref{E:SFINTEGRALINEQUALITY}
	are the error terms from the $\SigmatLie_{\mathscr{Z}}^{\vec{I}}$-commuted equations of
	Prop.~\ref{P:ICOMMUTEDEQNS}.
	
\end{lemma}

\begin{proof}
Let 
$\Jfour_{(Sf)}[\cdots]$
be the scalar field energy current from Def.\ \ref{D:ENERGYCURRENTS}.
Integrating 
$
\Divfour_{\EnergyEstimatesMetric} 
			\Jfour_{(Sf)}[\mathscr{Z}^{\vec{I}} \newtimescalar, \SigmatLie_{\mathscr{Z}}^{\vec{I}}\newspacescalar]
$
over the spacetime domain 
$[0,t] \times \mathbb{S}^3$
with respect to the measure
$d \tvol \, ds$ of the metric $\EnergyEstimatesMetric$ from Def.\ \ref{D:ENERGYESTIMATESPACETIMEMETRIC},
summing over $1 \leq |\vec{I}| \leq M$,
using the divergence theorem, 
and appealing to 
definition \eqref{E:SFENERGY},
we deduce
\begin{align} \label{E:FIRSTDIVTHMSF}
\Sfenergy{M}(t)
& = 
	\Sfenergy{M}(0)
	+
	\sum_{1 \leq |\vec{I}| \leq M}
	\int_{s=0}^t
		\int_{\Sigma_s}
			\Divfour_{\EnergyEstimatesMetric} 
			\Jfour_{(Sf)}[\mathscr{Z}^{\vec{I}} \newtimescalar, \SigmatLie_{\mathscr{Z}}^{\vec{I}}\newspacescalar]
			\, d \tvol
	\, ds
		\\
& \ \
	+
	\sum_{1 \leq |\vec{I}| \leq M}
	\int_{s=0}^t
		\int_{\Sigma_s}
			\scale' \scale^3 \Gdiv 
			\left\lbrace 
				(\nabla^{\#} \mathscr{Z}^{\vec{I}} \newlapse) \mathscr{Z}^{\vec{I}} \newlapse
			\right\rbrace
		\, d \tvol
	\, ds
	\notag
\end{align}
(note that the last integral on RHS~\eqref{E:FIRSTDIVTHMSF} vanishes).
Next, we use the formula
\eqref{E:DIVSFCURRENT}	
to substitute for the integrands on RHS~\eqref{E:FIRSTDIVTHMSF}.
For convenience, in the remainder of the proof,
we carry out the analysis only for
a fixed multi-index $\vec{I}$ on RHS~\eqref{E:FIRSTDIVTHMSF};
the summations 
$\sum_{1 \leq |\vec{I}| \leq M} \cdots$
are easy to put in at the end of the proof,
and we omit these simple details.

In the rest of the proof, we silently use the 
fundamentally important facts that for $t \in [0,\TCrunch]$,
$\scale'(t) \leq 0$ and $|\scale'(t)|$ is uniformly bounded,
as was shown in Lemma~\ref{L:ANALYSISOFFRIEDMANN}.
To proceed, we bring the integral of the terms
$
\left\lbrace
				(\scale')^3 \scale^{5/3}
				+ 
				\cdots
\right\rbrace
\left|
	\mathscr{Z}^{\vec{I}} \newlapse
\right|^2
$
on the second line of RHS~\eqref{E:DIVSFCURRENT}	
over to LHS~\eqref{E:FIRSTDIVTHMSF},
where the terms 
$\cdots$ above are non-positive.
This yields the last positive definite integral
\begin{align} \label{E:FIRSTGOODSFINTEGRALS}
		\int_{s=0}^t
			\int_{\Sigma_s}
				|\scale'(s)|^3 
				\scale^{5/3}(s)
				\left|
					\mathscr{Z}^{\vec{I}} \newlapse
				\right|_{L_{\newg}^2(\Sigma_s)}^2
			\, d \tvol
		\, ds
\end{align}
on LHS~\eqref{E:SFINTEGRALINEQUALITY},
with the coefficient $1$ in place of the coefficient
$1/2$ written on LHS~\eqref{E:SFINTEGRALINEQUALITY}.
Observe that we have simply discarded the additional
good contribution made by the terms 
$\cdots$ noted above.

Next, we will treat 
the terms on the first and third lines of
RHS~\eqref{E:DIVSFCURRENT}.
Using Young's inequality, we bound the product
on the third line as follows, where $\uptau$
is the constant from Cor.~\ref{C:MONOTONICITYTIME}
and $\mathbf{1}_{I}$ denotes the characteristic function of the time interval $I$:
\begin{align} \label{E:KEYYOUNG}
	\left|
		2 \sqrt{\frac{2}{3}} \scale^{5/3} (\SigmatLie_{\mathscr{Z}}^{\vec{I}} \newspacescalar)^{\#} 
		\cdot 
		\nabla \mathscr{Z}^{\vec{I}} \newlapse
	\right|
	& \leq 
	\mathbf{1}_{[0,\TCrunch - \uptau]} \scale^{1/3} |\SigmatLie_{\mathscr{Z}}^{\vec{I}} \newspacescalar|_{\newg}^2
	+
	\frac{2}{3} \mathbf{1}_{[0,\TCrunch - \uptau]} \scale^3 |\nabla \mathscr{Z}^{\vec{I}} \newlapse|_{\newg}^2
		\\
	& 
	+
	\mathbf{1}_{(\TCrunch - \uptau,\TCrunch)} \scale^{1/3} |\SigmatLie_{\mathscr{Z}}^{\vec{I}} \newspacescalar|_{\newg}^2
	+
	\frac{2}{3} \mathbf{1}_{(\TCrunch - \uptau,\TCrunch)} \scale^3 |\nabla \mathscr{Z}^{\vec{I}} \newlapse|_{\newg}^2.
		\notag
\end{align}
Moreover, using \eqref{E:QUANTITATIVEMONOTONICITYATLATETIMES}-\eqref{E:CUBICQUANTITATIVEMONOTONICITYATLATETIMES} 
and \eqref{E:SCALEFACTORBOUNDEDFROMBELOWFORSHORTIMES},
we deduce that the following bounds hold, 
where $p$ is an arbitrary non-negative constant
and $C_p$ depends on $p$ and the (fixed) constant $\uptau$:
\begin{align} 
	\mathbf{1}_{[0,\TCrunch - \uptau]}
	& \leq
	C_p \scale^p,
		\label{E:EARLYTIMESCALEFACTORESTIMATE} \\
	\mathbf{1}_{(\TCrunch - \uptau,\TCrunch)}
	& \leq \frac{7}{6} |\scale'|.
	\label{E:LATETIMESCALEFACTORESTIMATE}
\end{align}
Thus, using \eqref{E:KEYYOUNG},
\eqref{E:EARLYTIMESCALEFACTORESTIMATE}, 
\eqref{E:LATETIMESCALEFACTORESTIMATE},
and the identities
$\frac{4}{3} - \frac{7}{6} = \frac{1}{6}$
and 
$1 - \frac{2}{3} \frac{7}{6} = \frac{4}{18} > \frac{1}{6}$,
we obtain the following key pointwise estimate for
the terms on the first and third lines of RHS~\eqref{E:DIVSFCURRENT}:
\begin{align} \label{E:KEYYOUNGAPPLICATION}
\frac{4}{3} \scale' \scale^{1/3} |\SigmatLie_{\mathscr{Z}}^{\vec{I}} \newspacescalar|_{\newg}^2 
& 
+ 
\scale' \scale^3  |\nabla \mathscr{Z}^{\vec{I}} \newlapse|_{\newg}^2
+ 
2 \sqrt{\frac{2}{3}} \scale^{5/3} 
\left|
	(\SigmatLie_{\mathscr{Z}}^{\vec{I}} \newspacescalar)^{\#} 
	\cdot 
	\nabla \mathscr{Z}^{\vec{I}} \newlapse
\right|
	\\
& 
\leq
	- 
	\frac{1}{6} |\scale'|
	 \scale^{1/3} |\SigmatLie_{\mathscr{Z}}^{\vec{I}} \newspacescalar|_{\newg}^2
	-
	\frac{1}{6} |\scale'|
	\scale^3 
	|\nabla \mathscr{Z}^{\vec{I}} \newlapse|_{\newg}^2
	\notag \\
& \ \ 
	+ C \scale^{4/3} |\SigmatLie_{\mathscr{Z}}^{\vec{I}} \newspacescalar|_{\newg}^2
	+ C \scale^4  |\nabla \mathscr{Z}^{\vec{I}} \newlapse|_{\newg}^2.
\notag 
\end{align}
We then bring the spacetime integrals of 
the two terms on the first line of
RHS~\eqref{E:KEYYOUNGAPPLICATION}
over to LHS~\eqref{E:FIRSTDIVTHMSF},
which yields the first and second
positive definite spacetime integrals on LHS~\eqref{E:SFINTEGRALINEQUALITY},
but with the coefficients $1/6$ in place of the coefficients
$1/12$ written on LHS~\eqref{E:SFINTEGRALINEQUALITY}.
We use \eqref{E:SFENERGYCOERCIVENESS} to bound the spacetime integral of
the term $C \scale^{4/3} |\SigmatLie_{\mathscr{Z}}^{\vec{I}} \newspacescalar|_{\newg}^2$ 
on the last line of RHS~\eqref{E:KEYYOUNGAPPLICATION} by
$\leq
\int_{s=0}^t
	\Sfenergy{M}(s)
\, ds
$,
which is 
$\leq \mbox{RHS~\eqref{E:SFINTEGRALINEQUALITY}}$
as desired.
We place the spacetime integral of 
the term $C \scale^4  |\nabla \mathscr{Z}^{\vec{I}} \newlapse|_{\newg}^2$
from the last line of RHS~\eqref{E:KEYYOUNGAPPLICATION}
on the fourth-to-last line of RHS~\eqref{E:SFINTEGRALINEQUALITY}.

We now address the spacetime
integrals generated by the borderline error term
$\SfCurrentBorder{\vec{I}}$ defined in
\eqref{E:SFCURRENTBORDER}.
Using 
\eqref{E:EARLYTIMESCALEFACTORESTIMATE},
\eqref{E:LATETIMESCALEFACTORESTIMATE},
and Young's inequality, 
and separately considering the time
intervals $[0,\TCrunch - \uptau]$ 
and $(\TCrunch - \uptau,\TCrunch)$
(as we did on RHS~\eqref{E:KEYYOUNG}),
we pointwise bound the magnitude of the terms on the first and second lines of 
RHS~\eqref{E:SFCURRENTBORDER} by
\begin{align} \label{E:YOUNGINEQUALITYINSFENERGYESTIMATE}
& \leq
\varepsilon \scale^{-1} |\mathscr{Z}^{\vec{I}} \newtimescalar|^2 
+ \frac{1}{\varepsilon} \scale^{5/3} |\CommutedTimeSfBorderInhom{\vec{I}}|^2
+
\frac{1}{13} |\scale'| \scale^{1/3} 
|\SigmatLie_{\mathscr{Z}}^{\vec{I}} \newspacescalar|_{\newg}^2
+ C \scale^3 |\CommutedSpaceSfBorderInhom{\vec{I}}|_{\newg}^2
	\\
& \ \
	+ \frac{1}{2} |\scale'| \scale^{5/3} |\mathscr{Z}^{\vec{I}} \newlapse|^2 
	+ C \scale^{-1} |\CommutedLapseHighBorderInhom{\vec{I}}|^2
	\notag \\
& \ \ 
+ C \scale^{4/3} |\SigmatLie_{\mathscr{Z}}^{\vec{I}} \newspacescalar|_{\newg}^2
+ C \scale^{8/3} |\mathscr{Z}^{\vec{I}} \newlapse|^2.
		\notag
\end{align}
Moreover, using \eqref{E:AGAINNODERIVATIVESPARTIALTGINVERSESTRONGSUPNORM},
\eqref{E:EARLYTIMESCALEFACTORESTIMATE},
and \eqref{E:LATETIMESCALEFACTORESTIMATE},
we bound the magnitude of the product on the third line of 
RHS~\eqref{E:SFCURRENTBORDER}
by 
\begin{align} \label{E:SFCURRENTPARTIALTGBORDERLINEERRORTERM}
& \leq 
C \varepsilon \mathbf{1}_{[0,\TCrunch - \uptau]} \scale^{1/3}  
|\SigmatLie_{\mathscr{Z}}^{\vec{I}} \newspacescalar|_{\newg}^2
+
C \varepsilon \mathbf{1}_{(\TCrunch - \uptau,\TCrunch)} \scale^{1/3}  
|\SigmatLie_{\mathscr{Z}}^{\vec{I}} \newspacescalar|_{\newg}^2
	\\
& \leq 
C \varepsilon |\scale'| \scale^{1/3} 
|\SigmatLie_{\mathscr{Z}}^{\vec{I}} \newspacescalar|_{\newg}^2
+
C \scale^{4/3} |\SigmatLie_{\mathscr{Z}}^{\vec{I}} \newspacescalar|_{\newg}^2.
\notag
\end{align}
We now note that for $\varepsilon$ sufficiently small, 
we can absorb the spacetime integrals of the terms
\begin{align*}
\frac{1}{13} |\scale'| \scale^{1/3} |\SigmatLie_{\mathscr{Z}}^{\vec{I}} \newspacescalar|_{\newg}^2,
	\qquad
\frac{1}{2} |\scale'| \scale^{5/3} |\mathscr{Z}^{\vec{I}} \newlapse|^2,
	\qquad
C \varepsilon |\scale'| \scale^{1/3} 
  |\SigmatLie_{\mathscr{Z}}^{\vec{I}} \newspacescalar|_{\newg}^2
\end{align*}
into the three positive definite integrals 
on LHS~\eqref{E:SFINTEGRALINEQUALITY} that were generated in the second and third paragraphs of the proof.
This procedure reduces the coefficients of the positive integrals
to no less than the values stated on LHS~\eqref{E:SFINTEGRALINEQUALITY}.
We place the spacetime integrals of the term
$ 
\frac{1}{\varepsilon} \scale^{5/3} |\CommutedTimeSfBorderInhom{\vec{I}}|^2
$
directly on RHS~\eqref{E:SFINTEGRALINEQUALITY}.
In view of 
\eqref{E:SFENERGYCOERCIVENESS},
we see that since $1 \leq |\vec{I}| \leq M$,
the spacetime integral of the term
$
\varepsilon \scale^{-1} |\mathscr{Z}^{\vec{I}} \newtimescalar|^2
$
is
$\leq
\varepsilon
\int_{s=0}^t
	\scale^{-1}(s) \Sfenergy{M}(s)
\, ds
$,
which is 
$\leq \mbox{RHS~\eqref{E:SFINTEGRALINEQUALITY}}$
as desired.
Similarly, 
with the exception of
$C \scale^{4/3} |\SigmatLie_{\mathscr{Z}}^{\vec{I}} \newspacescalar|_{\newg}^2$,
we place the spacetime integrals of the products on RHSs~\eqref{E:YOUNGINEQUALITYINSFENERGYESTIMATE} 
and \eqref{E:SFCURRENTPARTIALTGBORDERLINEERRORTERM}
featuring the large coefficient $C$ directly on RHS~\eqref{E:SFINTEGRALINEQUALITY}.
Finally, in view of 
\eqref{E:SFENERGYCOERCIVENESS},
we see that since $1 \leq |\vec{I}| \leq M$,
the spacetime integral of the remaining product
$C \scale^{4/3} |\SigmatLie_{\mathscr{Z}}^{\vec{I}} \newspacescalar|_{\newg}^2$
is
$\leq
\int_{s=0}^t
	\Sfenergy{M}(s)
\, ds
$,
which is 
$\leq \mbox{RHS~\eqref{E:SFINTEGRALINEQUALITY}}$
as desired.

To complete the proof of \eqref{E:SFINTEGRALINEQUALITY},
it remains for us to address the spacetime
integrals generated by the error term $\SfCurrentJunk{\vec{I}}$
defined in \eqref{E:SFCURRENTJUNK}. 
This term is easy to handle
because the products on RHS~\eqref{E:SFCURRENTJUNK} contain
sufficiently large powers of $\scale$, and we can afford to be non-optimal in our estimates.
We first use the bounds
$|\newlapse|, 
|\nabla \newlapse|
\lesssim \sqrt{\varepsilon} \scale^{- c \sqrt{\varepsilon}}
$
(which follow from \eqref{E:LAPSESTRONGSUPNROM} and \eqref{E:IMPROVEDOPERATORCOMPARISON})
and Young's inequality
to derive the (non-optimal) pointwise bound
\begin{align}
\left|
	\SfCurrentJunk{\vec{I}}
\right|
& \lesssim
		\scale^{-1/3} |\mathscr{Z}^{\vec{I}} \newtimescalar|^2 
		+
		\scale^{4/3} |\SigmatLie_{\mathscr{Z}}^{\vec{I}} \newspacescalar|_{\newg}^2
		+
		\scale^{7/3} |\mathscr{Z}^{\vec{I}} \newlapse|^2
			\label{E:POINTWISEBOUNDSFCURRENTJUNK} \\
	& \ \ 
		+
		\scale \left|\CommutedTimeSfJunkInhom{\vec{I}} \right|^2
		+
		\scale^2 \left|\CommutedSpaceSfJunkInhom{\vec{I}} \right|^2
		+
		\scale \left|\CommutedLapseHighJunkInhom{\vec{I}} \right|^2.
			 \notag 
\end{align}
Using \eqref{E:SFENERGYCOERCIVENESS},
we see that the spacetime integral of the first two terms on RHS~\eqref{E:POINTWISEBOUNDSFCURRENTJUNK} is
$\leq
\int_{s=0}^t
	\scale^{-1/3}(s)
	\Sfenergy{M}(s)
\, ds
$,
which is 
$\leq \mbox{RHS~\eqref{E:SFINTEGRALINEQUALITY}}$
as desired.
The spacetime integrals
of the last four terms on RHS~\eqref{E:POINTWISEBOUNDSFCURRENTJUNK}
are manifestly
$\leq \mbox{RHS~\eqref{E:SFINTEGRALINEQUALITY}}$.

We have therefore proved the lemma.
\end{proof}

\subsection{Energy integral inequalities for the metric}
In this subsection, we derive an analog of Lemma~\ref{L:SFINTEGRALINEQUALITY}
for the metric energies. The proof is easier in the sense that
we do not need to observe the same kind of delicate cancellations
that were at the heart of the proof of Lemma~\ref{L:SFINTEGRALINEQUALITY}.

\begin{lemma}[\textbf{Fundamental metric energy integral inequalities}]
	\label{L:METRICINTEGRALINEQUALITY}
	Let $1 \leq M \leq 16$.
	There exists a constant $C > 0$ such that
	the energy $\Metricenergy{M}(t)$
	defined by \eqref{E:METRICENERGY}
	obeys the following inequality for $t \in [0,\Tboot)$:
	\begin{align} \label{E:METRICINTEGRALINEQUALITY}
		\Metricenergy{M}(t)
		& +
		\frac{1}{9}
		\sum_{1 \leq |\vec{I}| \leq M}
		\int_{s=0}^t
			\int_{\Sigma_s}
				|\scale'(s)| 
				\scale^{1/3}(s)
				\left|
					\nabla \SigmatLie_{\mathscr{Z}}^{\vec{I}} \newg
				\right|_{\newg}^2
			\, d \tvol
		\, ds
			\\
		& \leq
			\Metricenergy{M}(0)
			\notag \\
		& \ \
		+ 
		C
		\sum_{1 \leq |\vec{I}| \leq M}
		\int_{s=0}^t
			\int_{\Sigma_s}
				\scale^{1/3}(s)
				\left\lbrace
				\left|
					\CommutedMomBorderInhomDown{\vec{I}}
				\right|_{\newg}^2
				+
				\left|
					\CommutedMomBorderInhomUp{\vec{I}}
				\right|_{\newg}^2
			\right\rbrace
			\, d \tvol
		\, ds
			\notag \\
	& \ \
		+ 
		C
		\sum_{1 \leq |\vec{I}| \leq M}
		\int_{s=0}^t
			\int_{\Sigma_s}
				\scale^{1/3}(s)
				\left|
					\CommutedGradMetBorderInhom{\vec{I}}
				\right|_{\newg}^2
			\, d \tvol
		\, ds
		\notag \\
	& \ \
		+ 
		\frac{C}{\varepsilon}
		\sum_{1 \leq |\vec{I}| \leq M}
		\int_{s=0}^t
			\int_{\Sigma_s}
				\scale^{5/3}(s)
				\left|
					\CommutedSecFunBorderInhom{\vec{I}}
				\right|_{\newg}^2
			\, d \tvol
		\, ds
		\notag \\
	&  \ \
		+ 
		C
		\sum_{1 \leq |\vec{I}| \leq M}
		\int_{s=0}^t
			\int_{\Sigma_s}
				\scale(s)
				\left|
					\CommutedSecFunJunkInhom{\vec{I}} 
				\right|_{\newg}^2
			\, d \tvol
		\, ds
		\notag \\
	&  \ \
		+ 
		C
		\sum_{1 \leq |\vec{I}| \leq M}
		\int_{s=0}^t
			\int_{\Sigma_s}
				\scale(s)
				\left|
					\RicErrorInhom{\vec{I}} 
				\right|_{\newg}^2
			\, d \tvol
		\, ds
		\notag \\
	&  \ \
		+ 
		C
		\sum_{1 \leq |\vec{I}| \leq M}
		\int_{s=0}^t
			\int_{\Sigma_s}
				\scale^2(s)
				\left|
					\CommutedGradMetJunkInhom{\vec{I}} 
				\right|_{\newg}^2
		\, ds
		\notag
			\\
	& \ \
			+ 
			C
			\sum_{1 \leq |\vec{I}| \leq M}
			\int_{s=0}^t
				\scale^{4/3}(s)
				\int_{\Sigma_s}
					\left|
						\SigmatLie_{\mathscr{Z}}^{\vec{I}} \newspacescalar
					\right|_{\newg}^2
				\, d \tvol	
			\, ds 
			\notag \\
	& \ \
			+ 
			C
			\sum_{1 \leq |\vec{I}| \leq M}
			\int_{s=0}^t
				|\scale'(s)|
				\scale^{1/3}(s)
				\int_{\Sigma_s}
					\left|
						\SigmatLie_{\mathscr{Z}}^{\vec{I}} \newspacescalar
					\right|_{\newg}^2
				\, d \tvol	
			\, ds 
			\notag \\
	& \ \
			+ 
			C
			\sum_{1 \leq |\vec{I}| \leq M}
			\int_{s=0}^t
				\scale^4(s)
				\int_{\Sigma_s}
					\left|
						\nabla \SigmatLie_{\mathscr{Z}}^{\vec{I}} \newlapse
					\right|_{\newg}^2
				\, d \tvol	
			\, ds 
			\notag \\
		& \ \
			+ 
			C
			\sum_{1 \leq |\vec{I}| \leq M}
			\int_{s=0}^t
				|\scale'(s)|
				\scale^3(s)
				\int_{\Sigma_s}
					\left|
						\nabla \SigmatLie_{\mathscr{Z}}^{\vec{I}} \newlapse
					\right|_{\newg}^2
				\, d \tvol	
			\, ds 
			\notag \\
	&  \ \
		+
		C \varepsilon
		\sum_{1 \leq |\vec{I}| \leq M}
		\int_{s=0}^t
			\scale^{-1}(s)
			\Metricenergy{M}(s)
		\, ds
		\notag \\
	&  \ \
		+
		C 
		\sum_{1 \leq |\vec{I}| \leq M}
		\int_{s=0}^t
			\scale^{-1/3}(s)
			\Metricenergy{M}(s)
		\, ds,
		\notag 
	\end{align}
	where the terms $\CommutedMomBorderInhomDown{\vec{I}},\cdots$ on
	RHS~\eqref{E:METRICINTEGRALINEQUALITY}
	are the error terms from the $\SigmatLie_{\mathscr{Z}}^{\vec{I}}$-commuted equations of
	Prop.~\ref{P:ICOMMUTEDEQNS}.
\end{lemma}

\begin{proof}
The proof is similar to that of Lemma~\ref{L:SFINTEGRALINEQUALITY} but simpler.
Let 
$\Jfour_{(Metric)}[\SigmatLie_{\mathscr{Z}}^{\vec{I}} \FreeNewSec,\SigmatLie_{\mathscr{Z}}^{\vec{I}} \newg,]$
be the metric energy current from Def.\ \ref{D:ENERGYCURRENTS}.
We integrate
$
\Divfour_{\EnergyEstimatesMetric} 
			\Jfour_{(Metric)}[\SigmatLie_{\mathscr{Z}}^{\vec{I}} \FreeNewSec,\SigmatLie_{\mathscr{Z}}^{\vec{I}} \newg]
$
over the spacetime domain 
$[0,t] \times \mathbb{S}^3$
with respect to the measure
$d \tvol \, ds$ of the metric $\EnergyEstimatesMetric$ from Def.\ \ref{D:ENERGYESTIMATESPACETIMEMETRIC},
sum over $1 \leq |\vec{I}| \leq M$,
use the divergence theorem, 
and appeal to 
definition \eqref{E:METRICENERGY},
thereby deducing that
\begin{align} \label{E:FIRSTDIVTHMMETRIC}
\Metricenergy{M}(t)
& = 
	\Metricenergy{M}(0)
	+
	\sum_{1 \leq |\vec{I}| \leq M}
	\int_{s=0}^t
		\int_{\Sigma_s}
			\Divfour_{\EnergyEstimatesMetric} 
			\Jfour_{(Metric)}[\SigmatLie_{\mathscr{Z}}^{\vec{I}} \FreeNewSec,\SigmatLie_{\mathscr{Z}}^{\vec{I}} \newg,\nabla \mathscr{Z}^{\vec{I}} 
				\newlapse]
		\, d \tvol
	\, ds.
\end{align}
Next, we use the formula
\eqref{E:DIVMETRICCURRENT}	
to substitute for the integrand on RHS~\eqref{E:FIRSTDIVTHMMETRIC}.
For convenience, in the remainder of the proof,
we carry out the analysis only for
a fixed multi-index $\vec{I}$ on RHS~\eqref{E:FIRSTDIVTHMMETRIC};
the summations 
$\sum_{1 \leq |\vec{I}| \leq M} \cdots$
are easy to put in at the end of the proof,
and we omit these simple details.

In the rest of the proof, we silently use the 
fundamentally important facts that for $t \in [0,\TCrunch]$,
$\scale'(t) \leq 0$ and $|\scale'(t)|$ is uniformly bounded,
as was shown in Lemma~\ref{L:ANALYSISOFFRIEDMANN}.
To proceed, we bring the spacetime integral of the first term
on RHS~\eqref{E:DIVMETRICCURRENT}
back to the left,
which yields the good term
\begin{align} \label{E:GOODMETRICINTEGRAL}
\frac{1}{3}
		\int_{s=0}^t
			\int_{\Sigma_s}
				|\scale'(s)| 
				\scale^{1/3}(s)
				\left|
					\nabla \SigmatLie_{\mathscr{Z}}^{\vec{I}} \newg
				\right|_{L_{\newg}^2(\Sigma_s)}^2
			\, d \tvol
		\, ds.
\end{align}
Note that \eqref{E:GOODMETRICINTEGRAL} differs from the second
term on LHS~\eqref{E:METRICINTEGRALINEQUALITY} only in that
the coefficient is $1/3$ in \eqref{E:GOODMETRICINTEGRAL} and $1/9$
in \eqref{E:METRICINTEGRALINEQUALITY}.

Next, we argue as in the proof of \eqref{E:KEYYOUNGAPPLICATION},
using in addition the estimate
$
|\newlapse|
\lesssim \sqrt{\varepsilon} \scale^{- c \sqrt{\varepsilon}}
$
(see \eqref{E:LAPSESTRONGSUPNROM}),
to bound the magnitude of the products on the second through fourth lines of RHS~\eqref{E:DIVMETRICCURRENT}
by
\begin{align} \label{E:METRICENERGYYOUNGAPPLICATION}
& 
\leq
	\frac{1}{9}
		|\scale'| 
		\scale^{1/3}
		\left|
			\nabla \SigmatLie_{\mathscr{Z}}^{\vec{I}} \newg
		\right|_{\newg}^2
		+	
	C |\scale'|
	 \scale^{1/3} |\SigmatLie_{\mathscr{Z}}^{\vec{I}} \newspacescalar|_{\newg}^2
	+
	C
	|\scale'|
	\scale^3 
	|\nabla \mathscr{Z}^{\vec{I}} \newlapse|_{\newg}^2
	\\
& \ \ 
	+ C \scale^{4/3} |\SigmatLie_{\mathscr{Z}}^{\vec{I}} \newspacescalar|_{\newg}^2
	+ C \scale^4  |\nabla \mathscr{Z}^{\vec{I}} \newlapse|_{\newg}^2.
\notag 
\end{align}
We can absorb the spacetime integral of the first 
term in \eqref{E:METRICENERGYYOUNGAPPLICATION}
into the positive definite term \eqref{E:GOODMETRICINTEGRAL},
which reduces the coefficient of 
$1/3$ in \eqref{E:GOODMETRICINTEGRAL} to 
$1/3 - 1/9 = 2/9$.
We place the spacetime integrals of the remaining
terms in \eqref{E:METRICENERGYYOUNGAPPLICATION}
directly on RHS~\eqref{E:METRICINTEGRALINEQUALITY}.

We now bound the borderline error terms 
$\MetricCurrentBorder{\vec{I}}$
defined in \eqref{E:METRICCURRENTBORDER}.
Using the same arguments that we used to prove
\eqref{E:YOUNGINEQUALITYINSFENERGYESTIMATE}
and
\eqref{E:SFCURRENTPARTIALTGBORDERLINEERRORTERM}
together with the estimate
$
|\newlapse|
\lesssim \sqrt{\varepsilon} \scale^{- c \sqrt{\varepsilon}}
$
mentioned above,
we deduce that 
\begin{align} \label{E:METRICCURRENTBORDERLINEERRORPOINTWISEBOUND}
	\left|
		\MetricCurrentBorder{\vec{I}}
	\right|
	& \leq
		\frac{1}{9} 
		|\scale'| 
		\scale^{1/3}
		\left|
			\nabla \SigmatLie_{\mathscr{Z}}^{\vec{I}} \newg
		\right|_{\newg}^2
		+ 
		C \varepsilon 
		\scale^{-1}
		\left|
			\SigmatLie_{\mathscr{Z}}^{\vec{I}} \FreeNewSec
		\right|_{\newg}^2
			\\
	& \ \
		+
		C \varepsilon^{-1} 
		\scale^{5/3}
		\left|
			\CommutedSecFunBorderInhom{\vec{I}}
		\right|_{\newg}^2
		+
		C 
		\scale^{1/3}
		\left|
			\CommutedGradMetBorderInhom{\vec{I}}
		\right|_{\newg}^2
			\notag \\
	& \ \
		+
		C 
		\scale^{1/3}
		\left|
			\CommutedMomBorderInhomDown{\vec{I}}
		\right|_{\newg}^2
		+
		C 
		\scale^{1/3}
		\left|
			\CommutedMomBorderInhomUp{\vec{I}}
		\right|_{\newg}^2
			\notag \\
	& \ \
		+
		C
		\scale^{4/3}
		\left|
			\nabla \SigmatLie_{\mathscr{Z}}^{\vec{I}} \newg
		\right|_{\newg}^2
		+
		C
		|\scale'|
		\scale^3
		\left|
			\nabla \SigmatLie_{\mathscr{Z}}^{\vec{I}} \newlapse
		\right|_{\newg}^2
		+
		C
		\scale^4
		\left|
			\nabla \SigmatLie_{\mathscr{Z}}^{\vec{I}} \newlapse
		\right|_{\newg}^2.
		\notag
\end{align}
We now absorb the spacetime
integral of the product
$
\frac{1}{9} 
		|\scale'| 
		\scale^{1/3}
		\left|
			\nabla \SigmatLie_{\mathscr{Z}}^{\vec{I}} \newg
		\right|_{\newg}^2
$
on RHS~\eqref{E:METRICCURRENTBORDERLINEERRORPOINTWISEBOUND}
into the positive definite term \eqref{E:GOODMETRICINTEGRAL},
which reduces the coefficient of 
$1/3$ in \eqref{E:GOODMETRICINTEGRAL} to its stated value of
$1/3 - 1/9 - 1/9 = 1/9$ on LHS~\eqref{E:METRICINTEGRALINEQUALITY}.
We place the spacetime integrals of the remaining
on RHS~\eqref{E:METRICCURRENTBORDERLINEERRORPOINTWISEBOUND}
directly on RHS~\eqref{E:METRICINTEGRALINEQUALITY}.

To complete the proof of the lemma, it remains for
us to bound the error term
$\MetricCurrentJunk{\vec{I}}$
defined in \eqref{E:METRICCURRENTJUNK}.
Using the same arguments that we used to prove \eqref{E:POINTWISEBOUNDSFCURRENTJUNK},
we derive the (non-optimal) pointwise bound
\begin{align}
\left|
	\MetricCurrentJunk{\vec{I}}
\right|
& \lesssim
			\scale^{-1/3}  
			\left|
				\SigmatLie_{\mathscr{Z}}^{\vec{I}} \FreeNewSec
			\right|_{\newg}^2
			+ 
			\scale^{4/3}
			\left|
				\nabla \SigmatLie_{\mathscr{Z}}^{\vec{I}} \newg
			\right|_{\newg}^2
			\label{E:POINTWISEBOUNDMETRICCURRENTJUNK} \\
		& \ \
				+
				\scale
				\left|
					\CommutedSecFunJunkInhom{\vec{I}} 
				\right|_{\newg}^2
				+
				\scale
				\left|
					\RicErrorInhom{\vec{I}} 
				\right|_{\newg}^2
				+
				\scale^2
				\left|
					\CommutedGradMetJunkInhom{\vec{I}} 
				\right|_{\newg}^2.
				\notag 
\end{align}
Using \eqref{E:METRICENERGYCOERCIVENESS},
we see that the spacetime integral of the first two terms on RHS~\eqref{E:POINTWISEBOUNDMETRICCURRENTJUNK} is
$\leq
C
\int_{s=0}^t
	\scale^{-1/3}(s)
	\Metricenergy{M}(s)
\, ds
$,
which is 
$\leq \mbox{RHS~\eqref{E:METRICINTEGRALINEQUALITY}}$
as desired.
The spacetime integrals
of the last three terms on RHS~\eqref{E:POINTWISEBOUNDMETRICCURRENTJUNK}
are manifestly
$\leq \mbox{RHS~\eqref{E:METRICINTEGRALINEQUALITY}}$.

We have therefore proved the lemma.

\end{proof}

\subsection{Proof of Prop.~\ref{P:FUNDAMENTALENERGYINTEGRALINEQUALITY}}
	\label{SS:PROOFOFUNDAMENTALENERGY}
Let $\smallparameter > 0$ be a parameter.
We add inequality \eqref{E:SFINTEGRALINEQUALITY} to $\smallparameter$
times inequality \eqref{E:METRICINTEGRALINEQUALITY}.
For $\smallparameter$ sufficiently small 
(we denote by $\smallparameter_*$ a fixed sufficiently small choice of $\smallparameter$),
we can absorb (with room to spare) the integrals
$C \smallparameter_*
			\sum_{1 \leq |\vec{I}| \leq M}
			\int_{s=0}^t
				|\scale'(s)|
				\scale^{1/3}(s)
				\int_{\Sigma_s}
					\left|
						\SigmatLie_{\mathscr{Z}}^{\vec{I}} \newspacescalar
					\right|_{\newg}^2
				\, d \tvol	
			\, ds 
$
and
$
C \smallparameter_*
			\sum_{1 \leq |\vec{I}| \leq M}
			\int_{s=0}^t
				|\scale'(s)| 
				\scale^3(s)
				\int_{\Sigma_s}
					\left|
						\nabla \SigmatLie_{\mathscr{Z}}^{\vec{I}} \newlapse
					\right|_{\newg}^2
				\, d \tvol	
			\, ds
$
generated by RHS~\eqref{E:METRICINTEGRALINEQUALITY}
into the first two positive definite spacetime integrals
on LHS~\eqref{E:SFINTEGRALINEQUALITY}.
Moreover,
using \eqref{E:SFENERGYCOERCIVENESS}
and definition~\eqref{E:TOTALENERGY},
we see that the term
$
C \smallparameter_*
			\sum_{1 \leq |\vec{I}| \leq M}
			\int_{s=0}^t
				\scale^{4/3}(s)
				\int_{\Sigma_s}
					\left|
						\SigmatLie_{\mathscr{Z}}^{\vec{I}} \newspacescalar
					\right|_{\newg}^2
				\, d \tvol	
			\, ds 
$
generated by the sixth-from-last term on RHS~\eqref{E:METRICINTEGRALINEQUALITY}
is
$
\lesssim
\int_{s=0}^t
	\Sfenergy{M}(s) 
\, ds
\leq
\int_{s=0}^t
	\scale^{-1/3}(s) \Totalenergy{M}{\smallparameter_*}(s) 
\, ds
$.
The desired inequality \eqref{E:FUNDAMENTALENERGYINTEGRALINEQUALITY}
now follows from the above considerations and the definition \eqref{E:TOTALENERGY} 
of $\Totalenergy{M}{\smallparameter_*}$.
\hfill $\qed$

\section{Pointwise estimates}
\label{S:POINTWISE} 

\subsection{Estimates for the Energy Estimate Error Terms}
\label{SS:POINTWISEFORENERGYESTIMATEERRORTERMS}
In view of the energy integral inequalities of Prop.~\ref{P:FUNDAMENTALENERGYINTEGRALINEQUALITY},
the primary remaining ingredient that we need to derive a priori energy estimates is:
$L^2$-bounds for the error terms in the $\SigmatLie_{\mathscr{Z}}^{\vec{I}}$-commuted equations.
That is, we need estimates that control error integrals on RHS~\eqref{E:FUNDAMENTALENERGYINTEGRALINEQUALITY},
such as the term
$
\scale^{1/3}(s)
\left\|
	\CommutedMomBorderInhomDown{\vec{I}}
\right\|_{L_{\newg}^2(\Sigma_s)}^2
$
from RHS~\eqref{E:ENERGYBORDER},
in terms of $\Totalenergy{M}{\smallparameter_*}(s)$.
In this section, we set up the forthcoming $L^2$ analysis by deriving
pointwise estimates for these error terms. The main point is to identify 
the borderline principal-order error terms and to make sure
that no degenerate coefficient of
$\scale^{-c\sqrt{\varepsilon}}$ 
(which blows up as $t \uparrow \TCrunch$)
appears in the pointwise estimates for these terms;
for such terms, the presence of a coefficient of $\scale^{-c\sqrt{\varepsilon}}$
would have spoiled the Gronwall estimates for the energies.
We provide the desired pointwise estimates in the following proposition.

\begin{remark}[\textbf{``Boxed'' error terms}]
	\label{R:BOXEDERROR}
	We place the borderline principal-order error terms in boxes on the RHSs
	of the estimates of the proposition.
\end{remark}

\begin{proposition}[\textbf{Pointwise estimates for the error terms in the $\SigmatLie_{\mathscr{Z}}^{\vec{I}}$-commuted equations}]
\label{P:POINTWISEESTIMATESFORERRORTERMS}
Let $\vec{I}$ be a $\mathscr{Z}$-multi-index with
$1 \leq |\vec{I}| \leq 16$.
The borderline inhomogeneous terms in the commuted equations
of Prop.~\ref{P:ICOMMUTEDEQNS} 
verify the following pointwise estimates,
where all products involving the operators
$\SigmatLie_{\mathscr{Z}}^{[1,|\vec{I}|-1]}$
or
$\mathscr{Z}^{[1,|\vec{I}|-1]}$
are absent if $|\vec{I}|=1$:
\begingroup
\allowdisplaybreaks
\begin{subequations}
\begin{align}
	\left|
		\CommutedMomBorderInhomUp{\vec{I}}
	\right|_{\newg},
		\,
	\left|
		\CommutedMomBorderInhomDown{\vec{I}}
	\right|_{\newg}
	& \lesssim 
		\sqrt{\varepsilon}
		\scale^{-c\sqrt{\varepsilon}}
		\left|
			\SigmatLie_{\mathscr{Z}}^{[1,|\vec{I}|]} \FreeNewSec
		\right|_{\newg}
		+
		\boxed{
		\varepsilon
		\left|
			\nabla \SigmatLie_{\mathscr{Z}}^{\vec{I}} \newg
		\right|_{\newg}
		}
		+
		\sqrt{\varepsilon}
		\scale^{-c\sqrt{\varepsilon}}
		\left|
			\SigmatLie_{\mathscr{Z}}^{[1,|\vec{I}|]} \newg
		\right|_{\newg}
			\label{E:POINTWISEMOMENTUMCONSTRAINTBORDERLINE} \\
	& \ \
		+
		\sqrt{\varepsilon}
		\scale^{-c\sqrt{\varepsilon}}
		\left|
			\mathscr{Z}^{[1,|\vec{I}|]} \newtimescalar
		\right|
		+
		\boxed{
		\varepsilon
		\left|
			\SigmatLie_{\mathscr{Z}}^{\vec{I}]} \newspacescalar
		\right|_{\newg}
		}
		+
		\sqrt{\varepsilon}
		\scale^{-c\sqrt{\varepsilon}}
		\left|
			\SigmatLie_{\mathscr{Z}}^{[1,|\vec{I}|-1]} \newspacescalar
		\right|_{\newg},
		\notag \\
	  \left|
			\CommutedMetBorderInhom{\vec{I}}
		\right|_{\newg},
			\,
		\left|
			\CommutedInvMetBorderInhom{\vec{I}}
		\right|_{\newg}
		& \lesssim
			\left|
				\SigmatLie_{\mathscr{Z}}^{\vec{I}} \FreeNewSec
			\right|_{\newg}
			+
			\sqrt{\varepsilon} \scale^{-c \sqrt{\varepsilon}}
			\left|
				\SigmatLie_{\mathscr{Z}}^{[1,|\vec{I}| - 1]} \FreeNewSec
			\right|
			+
			\boxed{
			\varepsilon
			\left|
				\SigmatLie_{\mathscr{Z}}^{\vec{I}} \newg
			\right|_{\newg}
			}
				+
			\sqrt{\varepsilon}
			\scale^{-c \sqrt{\varepsilon}}
			\left|
				\SigmatLie_{\mathscr{Z}}^{[1,|\vec{I}|-1]} \newg
			\right|_{\newg}
			\label{E:POINTWISEMETRICBORDERLINE} \\
	& \ \
		+ 
			\scale^{4/3}
			\left|
				\mathscr{Z}^{\vec{I}} \newlapse
			\right|,
			\notag \\
	\left|
		\CommutedGradMetBorderInhom{\vec{I}} 
	\right|_{\newg}
	&	
	\lesssim
	\boxed{
	\varepsilon
	\left|
		\nabla \SigmatLie_{\mathscr{Z}}^{\vec{I}} \newg
	\right|_{\newg}
	}
	+
		\sqrt{\varepsilon}
		\scale^{-c\sqrt{\varepsilon}}
		\left|
			\SigmatLie_{\mathscr{Z}}^{|\vec{I}|} \newg
		\right|_{\newg}
		\label{E:POINTWISEGRADMETRICBORDERLINE} \\
& \ \
	+ 
	\boxed{
	\varepsilon
	\scale^{4/3}
	\left|
		\nabla \mathscr{Z}^{\vec{I}} \newlapse
	\right|_{\newg}
	},
		\notag 
			\\
	\left|
		\CommutedSpaceSfBorderInhom{\vec{I}}	
	\right|_{\newg}
	& \lesssim 
		\boxed{
		\varepsilon
		\left|
		 	\nabla \mathscr{Z}^{\vec{I}} \newlapse
		\right|_{\newg}
		},
			\label{E:POINTWISESPACESFBORDERLINE} \\
	\left|	
		\CommutedSecFunBorderInhom{\vec{I}}		
	\right|_{\newg}
	& \lesssim
		\boxed{
		\varepsilon
		\left|
		 	 \mathscr{Z}^{\vec{I}} \newlapse
		\right|
		},
			\label{E:POINTWISESECONDFUNDBORDERLINE} \\
	\left|	
		\CommutedTimeSfBorderInhom{\vec{I}}
	\right|
	& \lesssim
		\boxed{
		\varepsilon
		\left|
		 	 \mathscr{Z}^{\vec{I}} \newlapse
		\right|
		},
			\label{E:POINTWISETIMESFBORDERLINE} \\
	\left|
		\CommutedLapseHighBorderInhom{\vec{I}}
	\right|
	& \lesssim 
		\boxed{
		\varepsilon
		\left|
			\SigmatLie_{\mathscr{Z}}^{\vec{I}} \FreeNewSec
		\right|_{\newg}
		}
		+
		\sqrt{\varepsilon}
		\scale^{-c \sqrt{\varepsilon}}
		\left|
			\SigmatLie_{\mathscr{Z}}^{[1,|\vec{I}|-1]} \FreeNewSec
		\right|_{\newg}
			\label{E:LAPSEHIGHBORDERINHOMPOINTWISE} \\
	& \ \
		+
		\boxed{
		\varepsilon
		\left|
			\mathscr{Z}^{\vec{I}} \newtimescalar
		\right|_{\newg}
		}
		+
		\sqrt{\varepsilon}
		\scale^{-c \sqrt{\varepsilon}}
		\left|
			\mathscr{Z}^{[1,|\vec{I}|-1]} \newtimescalar
		\right|_{\newg},
		\notag	\\
		\left|
			\CommutedLapseLowBorderInhom{\vec{I}}
		\right|
	& \lesssim 
		\scale^{-c \sqrt{\varepsilon}}
		\left|
			\nabla \SigmatLie_{\mathscr{Z}}^{[1,|\vec{I}|+2]} \newg
		\right|_{\newg}
		+ 
		\sqrt{\varepsilon}
		\scale^{-c \sqrt{\varepsilon}}
		\left|
			\SigmatLie_{\mathscr{Z}}^{[1,|\vec{I}|]} \newspacescalar
		\right|_{\newg}.
		\label{E:LAPSELOWBORDERINHOMPOINTWISE}
\end{align}	
\end{subequations}
\endgroup

Similarly, the ``junk'' inhomogeneous terms in the commuted equations
of Prop.~\ref{P:ICOMMUTEDEQNS} 
verify the following pointwise estimates:
\begingroup
\allowdisplaybreaks
\begin{subequations}
\begin{align}
	  \left|
			\CommutedMetJunkInhom{\vec{I}}
		\right|_{\newg},
			\,
		\left|
			\CommutedInvMetJunkInhom{\vec{I}}
		\right|_{\newg}
		& \lesssim
			\sqrt{\varepsilon}
			\scale^{-c \sqrt{\varepsilon}}
			\left|
				\SigmatLie_{\mathscr{Z}}^{[1,|\vec{I}|]} \FreeNewSec
			\right|_{\newg}
			+
			\sqrt{\varepsilon}
			\scale^{-c \sqrt{\varepsilon}}
			\left|
				\SigmatLie_{\mathscr{Z}}^{[1,|\vec{I}|]} \newg
			\right|_{\newg}
				\\
		& \ \
			+
			\sqrt{\varepsilon}
			\scale^{-c \sqrt{\varepsilon}}
			\left|
				\mathscr{Z}^{[1,|\vec{I}|-1]} \newlapse
			\right|,
			\notag \\
	\left|
		\CommutedGradMetJunkInhom{\vec{I}} 
	\right|_{\newg}	
	& \lesssim 
		\sqrt{\varepsilon}
		\scale^{-c \sqrt{\varepsilon}}
		\left|
			\SigmatLie_{\mathscr{Z}}^{[1,|\vec{I}|]} \FreeNewSec
		\right|_{\newg}
		+
		\sqrt{\varepsilon}
		\scale^{-c \sqrt{\varepsilon}}
		\left|
			\nabla \SigmatLie_{\mathscr{Z}}^{[1,|\vec{I}|]} \newg
		\right|_{\newg}
		+
		\sqrt{\varepsilon}
		\scale^{-c \sqrt{\varepsilon}}
		\left|
			\SigmatLie_{\mathscr{Z}}^{[1,|\vec{I}|]} \newg
		\right|_{\newg}	
			\\
	& \ \
		+
		\sqrt{\varepsilon}
		\scale^{-c \sqrt{\varepsilon}}
		\left|
			\mathscr{Z}^{[1,|\vec{I}|]} \newlapse
		\right|,
		\notag \\
	\left|
		\CommutedSpaceSfJunkInhom{\vec{I}}	
	\right|_{\newg}
	& \lesssim 
		\sqrt{\varepsilon}
		\scale^{-c \sqrt{\varepsilon}}
		\left|
			\mathscr{Z}^{[1,|\vec{I}|]} \newlapse
		\right|
		+
		\sqrt{\varepsilon}
		\scale^{-c \sqrt{\varepsilon}}
		\left|
			\mathscr{Z}^{[1,|\vec{I}|]} \newtimescalar
		\right|,
		\label{E:POINTWISESPACESFJUNK} \\	
	\left|
		\RicErrorInhom{\vec{I}}
	\right|_{\newg}	
	&
	\lesssim 
		\sqrt{\varepsilon} 
		\scale^{-c \sqrt{\varepsilon}}
		\left|
			\nabla \SigmatLie_{\mathscr{Z}}^{[1,|\vec{I}|]} \newg
		\right|_{\newg} 
		+
		\scale^{-c \sqrt{\varepsilon}}
		\left|
			\SigmatLie_{\mathscr{Z}}^{[1,|\vec{I}|]} \newg
		\right|_{\newg},
		\\
	\left|
		\CommutedSecFunJunkInhom{\vec{I}}
	\right|_{\newg}
	& \lesssim
		\sqrt{\varepsilon}
			\scale^{-c \sqrt{\varepsilon}}
			\left|
				\SigmatLie_{\mathscr{Z}}^{[1,|\vec{I}|]} \FreeNewSec
			\right|_{\newg}
		+	
		\sqrt{\varepsilon}
		\scale^{4/3-c \sqrt{\varepsilon}}
		\left|
			\nabla \SigmatLie_{\mathscr{Z}}^{[1,|\vec{I}|]} \newg
		\right|_{\newg}
		+
		\sqrt{\varepsilon}
		\scale^{-c \sqrt{\varepsilon}}
		\left|
			\SigmatLie_{\mathscr{Z}}^{[1,|\vec{I}|]} \newg
		\right|_{\newg}
		\label{E:SECONDFUNDJUNKINHOMPOINTWISE} \\
	& \ \
		+
		\sqrt{\varepsilon}
		\scale^{-c \sqrt{\varepsilon}}
		\left|
			\mathscr{Z}^{[1,|\vec{I}|-1]} \newlapse
		\right|
		+
		\scale^{4/3 - c \sqrt{\varepsilon}}
		\left|
			\nabla^{\leq 1} \mathscr{Z}^{[1,|\vec{I}|]} \newlapse
		\right|_{\newg}
		\notag
			\\
	& \ \
		+
		\sqrt{\varepsilon}
		\scale^{-c \sqrt{\varepsilon}}
		\left|
			\SigmatLie_{\mathscr{Z}}^{[1,|\vec{I}|]} \newspacescalar
		\right|_{\newg},
			\notag 
			\\
	\left|
		\CommutedTimeSfJunkInhom{\vec{I}}
	\right|
	& \lesssim
		\sqrt{\varepsilon}
		\scale^{-c \sqrt{\varepsilon}}
		\left|
			\nabla^{\leq 1} \SigmatLie_{\mathscr{Z}}^{[1,|\vec{I}|]} \newg
		\right|_{\newg}
		\label{E:SFTIMEJUNKINHOMPOINTWISE} \\
	& \ \
		+
		\sqrt{\varepsilon}
		\scale^{-c \sqrt{\varepsilon}}
		\left|
			\mathscr{Z}^{[1,|\vec{I}|-1]} \newlapse
		\right|_{\newg}
		+
		\sqrt{\varepsilon}
		\sum_{L=0}^1
		\scale^{4/3 -c \sqrt{\varepsilon}}
		\left|
			\nabla^L \mathscr{Z}^{[1,|\vec{I}|]} \newlapse
		\right|_{\newg}
		\notag
			\\
	& \ \
		+
		\sqrt{\varepsilon}
		\scale^{-c \sqrt{\varepsilon}}
		\left|
			\mathscr{Z}^{[1,|\vec{I}|]} \newtimescalar
		\right|
		+
		\sqrt{\varepsilon}
		\scale^{-c \sqrt{\varepsilon}}
		\left|
			\SigmatLie_{\mathscr{Z}}^{[1,|\vec{I}|]} \newspacescalar
		\right|_{\newg},
			\notag 
		\\
	\left|
		\CommutedLapseHighJunkInhom{\vec{I}}
	\right|
	& \lesssim
		\sqrt{\varepsilon}
		\scale^{-c \sqrt{\varepsilon}}
		\left|
			\SigmatLie_{\mathscr{Z}}^{[1,|\vec{I}|]} \FreeNewSec
		\right|_{\newg}
		+
		\sqrt{\varepsilon}
		\sum_{L=0}^1
		\scale^{-c \sqrt{\varepsilon} + (4/3)L}
		\left|
			\nabla^L \mathscr{Z}^{[1,|\vec{I}|]} \newg
		\right|_{\newg}
			\label{E:LAPSEHIGHJUNKINHOMPOINTWISE} \\
	& \ \
		+
		\sqrt{\varepsilon}
		\scale^{-c \sqrt{\varepsilon}}
		\left|
			\mathscr{Z}^{[1,|\vec{I}|-1]} \newlapse
		\right|_{\newg}
		+
		\sqrt{\varepsilon}
		\scale^{4/3 -c \sqrt{\varepsilon}}
		\left|
			\nabla^{\leq 1} \mathscr{Z}^{[1,|\vec{I}|]} \newlapse
		\right|_{\newg}
		\notag
			\\
	& \ \
		+
		\sqrt{\varepsilon}
		\scale^{-c \sqrt{\varepsilon}}
		\left|
			\mathscr{Z}^{[1,|\vec{I}|]} \newtimescalar
		\right|,
			\notag \\
	\left|
		\CommutedLapseLowJunkInhom{\vec{I}}
	\right|
	& \lesssim
		\sqrt{\varepsilon}
		\scale^{-c \sqrt{\varepsilon}}
		\left|
			\SigmatLie_{\mathscr{Z}}^{[1,|\vec{I}|+2]} \newg
		\right|_{\newg}
		+
		\sqrt{\varepsilon}
		\scale^{-c \sqrt{\varepsilon}}
		\left|
			\nabla \mathscr{Z}^{[1,|\vec{I}|]} \newlapse
		\right|_{\newg}
		\label{E:LAPSELOWJUNKINHOMPOINTWISE}
			\\
		& \ \
			+
		\sqrt{\varepsilon}
		\scale^{-c \sqrt{\varepsilon}}
		\left|
			\mathscr{Z}^{[1,|\vec{I}|]} \newspacescalar
		\right|_{\newg}.
		\notag
\end{align}	
\end{subequations}
\endgroup

\end{proposition}

\begin{proof}
	The estimate \eqref{E:POINTWISEMOMENTUMCONSTRAINTBORDERLINE} can be proved by taking
	the norm $|\cdot|_{\newg}$ of RHSs \eqref{E:BORDERLINECOMMUTEDRENORMALIZEDMOMENTUM}
	and \eqref{E:BORDERLINEALTERNATECOMMUTEDRENORMALIZEDMOMENTUM}
	and bounding the terms one at a time
	with the help of the Leibniz rule for Lie derivatives, 
	the strong sup-norm estimates of Prop.~\ref{P:STRONGSUPNORMESTIMATES},
	the comparison estimates of Cor.~\ref{C:IMPROVEMENTLEMMAOPERATORCOMPARISON},
	and some commutator estimates that we explain. We treat in detail only the few
	delicate top-order terms that generate the boxed terms on RHS~\eqref{E:POINTWISEMOMENTUMCONSTRAINTBORDERLINE}.
	For these terms, 
	we must carefully avoid the degenerate factor $\scale^{-c \sqrt{\varepsilon}}$
	in the estimates.
	For the remaining terms, we simply allow the presence of the degenerate factor
	$\scale^{-c \sqrt{\varepsilon}}$, which comes as a consequence of the
	sup-norm estimates of
	Prop.~\ref{P:STRONGSUPNORMESTIMATES}.
	We mention that in some of the estimates, 
	allowing the loss of $\scale^{-c \sqrt{\varepsilon}}$ is non-optimal,
	and that we have allowed the loss only so that we can avoid having to carefully track the less delicate error terms.
	However, this non-optimality will not affect our a priori energy estimates.
	We now explain how to bound the delicate top-order terms
	on RHSs \eqref{E:BORDERLINECOMMUTEDRENORMALIZEDMOMENTUM}
	and \eqref{E:BORDERLINEALTERNATECOMMUTEDRENORMALIZEDMOMENTUM},
	which occur when all $|\vec{I}|$ derivatives fall on the factors of $\newspacescalar$
	or all $|\vec{I}|+1$ derivatives fall on factors of $\newg$.
	When all $|\vec{I}|$ derivatives fall on the factors of $\newspacescalar$,
	we use \eqref{E:SCALARFIELDTIMESTRONGSUPNROM} to bound 
	$|\cdot|_{\newg}$-norm of the remaining factors
	by $\lesssim \varepsilon$. Thus, the product's 	$|\cdot|_{\newg}$-norm
	is 
	$
	\lesssim
	\varepsilon
		\left|
			\SigmatLie_{\mathscr{Z}}^{\vec{I}} \newspacescalar
		\right|_{\newg}
	$
	as desired.
	The terms featuring the order $|\vec{I}|+1$ derivatives of $\newg$
	are generated by the commutator terms
	$[\SigmatLie_{\mathscr{Z}}^{\vec{I}},\Gdiv] \FreeNewSec$
	and $[\SigmatLie_{\mathscr{Z}}^{\vec{I}},\Gdiv^{\#}] \FreeNewSec$,
	which we express using the formulas \eqref{E:GIDIVLIEZCOMMUTATOR}-\eqref{E:GIDIVSHARPLIEZCOMMUTATOR}
	with $\xi := \FreeNewSec$.
	From these formulas, we see that up to harmless factors of $\newg^{-1}$ (which verify $|\newg^{-1}|_{\newg} \lesssim 1$),
	the top-order terms are of the schematic form 
	$\nabla \SigmatLie_{\mathscr{Z}}^{\vec{I}} \newg \cdot \FreeNewSec$.
	Thus, using \eqref{E:NOLOSSKSTRONGSUPNROM},
	we find that the product's $|\cdot|_{\newg}$-norm
	is 
	$
	\lesssim
	\varepsilon
		\left|
			\nabla \SigmatLie_{\mathscr{Z}}^{\vec{I}} \newg
		\right|_{\newg}
	$
	as desired.
	The remaining products 
	on RHSs \eqref{E:BORDERLINECOMMUTEDRENORMALIZEDMOMENTUM}
	and \eqref{E:BORDERLINEALTERNATECOMMUTEDRENORMALIZEDMOMENTUM}
	do not contribute to the boxed terms
	on RHS~\eqref{E:POINTWISEMOMENTUMCONSTRAINTBORDERLINE}
	and are easy to bound using the strategy described above,
	where, when necessary, we use Cor.~\ref{C:IMPROVEMENTLEMMAOPERATORCOMPARISON}
	to bound covariant derivatives $\nabla$ in terms of 
	$\SigmatLie_{\mathscr{Z}}$ derivatives
	and Lemma~\ref{L:COMMUTATIONIDENTITIES} to 
	treat products involving commutators; we omit the straightforward details.
	
	The proofs of the remaining estimates in the proposition are similar,
	so we provide only partial details.
	Specifically, we will identify the principal-order borderline
	error terms and show that no degenerate factor of $\scale^{-c \sqrt{\varepsilon}}$
	occurs in these terms; in the remaining error terms, which we do not discuss in detail, 
	we simply allow the loss of the degenerate factor of $\scale^{-c \sqrt{\varepsilon}}$
	and argue as in the previous paragraph. We clarify that 
	the identities of Lemma~\ref{L:RICCILIEDERIVATIVESCHEMATIC}
	are also needed to bound the terms 
	on RHS~\eqref{E:ICOMMUTEDSECONDFUNDJUNKTERMS}
	involving the Lie derivatives of $\Ric^{\# }$
	and the terms on RHS~\eqref{E:ICOMMUTEDLAPSELOWJUNKTERMS}
	involving the derivatives of the scalar curvature $\ScalarCur$.
	
	We now proceed with the analysis of the borderline terms.
	The borderline boxed terms 
	$
	\boxed{
			\varepsilon
			\left|
				\SigmatLie_{\mathscr{Z}}^{\vec{I}} \newg
			\right|_{\newg}
			}
	$
	on RHS~\eqref{E:POINTWISEMETRICBORDERLINE}
	are generated when all $|\vec{I}|$ derivatives fall on the factor $\newg$
	in the sum on the first line of RHS~\eqref{E:ICOMMUTEDMETRICBORDERTERMS}
	and when all $|\vec{I}|$ derivatives fall on the factor $\newg^{-1}$ 
	in the sum on the first line of RHS~\eqref{E:ICOMMUTEDIVERSEMETRICBORDERTERMS}.
	In the first case, this generates a product of the schematic form
	$\SigmatLie_{\mathscr{Z}}^{\vec{I}} \newg \cdot \FreeNewSec$
	while in the second case, by the Leibniz rule and the schematic
	identity $\SigmatLie_Z \newg^{-1} = \newg^{-1} \cdot \newg^{-1} \cdot \SigmatLie_Z \newg$,
	it generates a principal-order product
	of the schematic form 
	$\newg^{-1} \cdot \newg^{-1} \cdot \SigmatLie_{\mathscr{Z}}^{\vec{I}} \newg \cdot \FreeNewSec$
	plus lower-order products.
	In the first case, we use \eqref{E:NOLOSSKSTRONGSUPNROM}
	to bound the product's $|\cdot|_{\newg}$-norm
	by the borderline boxed terms written above as desired. In the second case,
	we bound the principal-order product by the borderline boxed terms written above, 
	and we also use \eqref{E:METRICSTRONGSUPNROM}
	to bound the lower-order products by
	$
	\lesssim 
	\sqrt{\varepsilon}
			\scale^{-c \sqrt{\varepsilon}}
			\left|
				\SigmatLie_{\mathscr{Z}}^{[1,|\vec{I}|-1]} \newg
			\right|_{\newg}
	$
	as desired.
	
	The borderline term 
	$
	\boxed{
	\varepsilon
	\left|
		\nabla \SigmatLie_{\mathscr{Z}}^{\vec{I}} \newg
	\right|_{\newg}
	}
	$
	on RHS~\eqref{E:POINTWISEGRADMETRICBORDERLINE}
	is generated when all $|\vec{I}|+1$ derivatives fall on $\newg$
	in the sum on the first line of RHS~\eqref{E:ICOMMUTEDGRADMETRICBORDERTERMS}.
	Specifically, the term is of the form 
	$
	\nabla \SigmatLie_{\mathscr{Z}}^{\vec{I}} \newg \cdot \FreeNewSec
	$,
	and we can obtain the desired bound for its $|\cdot|_{\newg}$-norm with the help of
	\eqref{E:NOLOSSKSTRONGSUPNROM}.
	The borderline term
	$
	\boxed{
	\varepsilon
	\scale^{4/3}
	\left|
		\nabla \mathscr{Z}^{\vec{I}} \newlapse
	\right|_{\newg}
	}
	$
	on RHS~\eqref{E:POINTWISEGRADMETRICBORDERLINE} is generated by 
	the first product on RHS~\eqref{E:ICOMMUTEDGRADMETRICBORDERTERMS},
	which is of the form 
	$\scale^{4/3} (\nabla \mathscr{Z}^{\vec{I}} \newlapse) \newg \cdot \FreeNewSec$.
	To obtain the desired bound, we again use \eqref{E:NOLOSSKSTRONGSUPNROM}.
	
	The borderline term 
	$
	\boxed{
		\varepsilon
		\left|
		 	\nabla \mathscr{Z}^{\vec{I}} \newlapse
		\right|_{\newg}
		}
	$
	on RHS~\eqref{E:POINTWISESPACESFBORDERLINE} is generated
	by the product
	$\newtimescalar \nabla \mathscr{Z}^{\vec{I}} \newlapse$
	on RHS~\eqref{E:ICOMMUTEDSPACESFBORDERLINETERMS}.
	We use \eqref{E:SCALARFIELDTIMESTRONGSUPNROM}
	to obtain the desired bound.
	
	The borderline term
	$
	\boxed{
		\varepsilon
		\left|
		 	 \mathscr{Z}^{\vec{I}} \newlapse
		\right|
		}
	$
	on RHS~\eqref{E:POINTWISESECONDFUNDBORDERLINE}
	is generated by the term
	$\scale' (\mathscr{Z}^{\vec{I}} \newlapse) \FreeNewSec$
	on RHS~\eqref{E:ICOMMUTEDSECONDFUNDBORDERTERMS}.
	We use \eqref{E:NOLOSSKSTRONGSUPNROM} to obtain the desired bound.

The borderline term
	$
	\boxed{
		\varepsilon
		\left|
		 	 \mathscr{Z}^{\vec{I}} \newlapse
		\right|
		}
 $	
on RHS~\eqref{E:POINTWISETIMESFBORDERLINE}
is generated by the term
$
\scale' (\mathscr{Z}^{\vec{I}} \newlapse) \newtimescalar
$
on RHS~\eqref{E:ICOMMUTEDTIMESFBORDERLINETERMS}.
We use \eqref{E:SCALARFIELDTIMESTRONGSUPNROM} to obtain the desired bound.
	
	The borderline term
	$
	\boxed{
		\varepsilon
		\left|
			\SigmatLie_{\mathscr{Z}}^{\vec{I}} \FreeNewSec
		\right|_{\newg}
		}
	$
	on RHS~\eqref{E:LAPSEHIGHBORDERINHOMPOINTWISE}
	is generated by the first sum
	on RHS~\eqref{E:ICOMMUTEDLAPSEHIGHBORDERTERMS}
	when all $|\vec{I}|$ derivatives fall on $\FreeNewSec$,
	which yields a product of the form
	$\FreeNewSec 
		  	\cdot 
	\SigmatLie_{\mathscr{Z}}^{\vec{I}} \FreeNewSec
	$. 	We use \eqref{E:NOLOSSKSTRONGSUPNROM} to obtain the desired bound.
	The borderline term 
	$\boxed{
		\varepsilon
		\left|
			\mathscr{Z}^{\vec{I}} \newtimescalar
		\right|_{\newg}
		}
	$
	on RHS~\eqref{E:LAPSEHIGHBORDERINHOMPOINTWISE}
	is generated by the second sum
	on RHS~\eqref{E:ICOMMUTEDLAPSEHIGHBORDERTERMS}
	when all $|\vec{I}|$ derivatives fall on $\newtimescalar$,
	which yields a product of the form
	$
	\newtimescalar \mathscr{Z}^{\vec{I}} \newtimescalar
	$.
	We use \eqref{E:SCALARFIELDTIMESTRONGSUPNROM} to obtain the desired bound.

	This completes our proof of the proposition.
\end{proof}

We will use the pointwise estimates provided by the following lemma
to derive $L^2$ estimates for the non-differentiated solution variables
(see Lemma~\ref{L:L2ESTIMATESFORNONDIFFERNTIATED}).

\begin{lemma}[\textbf{Pointwise estimates relevant for energy estimates at the lowest level}]
	\label{L:POINTWISEESTIMATESBASELEVEL}
	The following pointwise estimates hold:
	\begin{subequations}
	\begin{align} \label{E:GNORMPOINTWISEMETRICTIMEDERIVATIVE}
		\left|
			\SigmatLie_{\partial_t} (\newg - \StMet)
		\right|_{\newg},
		\,
		\left|
			\SigmatLie_{\partial_t} (\newg^{-1} - \StMet^{-1})
		\right|_{\newg}
		& \lesssim 
			\scale^{-1}
			\left|
				\FreeNewSec
			\right|_{\newg}
			+
			\scale^{1/3}
			\left|
				\newlapse
			\right|,
				\\
		\left|
			\SigmatLie_{\partial_t} \FreeNewSec
		\right|_{\newg}
		& \lesssim 
			\scale^{1/3 - c \sqrt{\varepsilon}}
			\left|
				\SigmatLie_{\mathscr{Z}}^{\leq 2} (\newg - \StMet)
			\right|_{\newg}
			+
			\scale^{1/3 - c \sqrt{\varepsilon}}
			\left|
				\newg^{-1} - \StMet^{-1}
			\right|_{\newg}
			+
			\sqrt{\varepsilon}
			\scale^{1/3 - c \sqrt{\varepsilon}}
			\left|
				\newspacescalar
			\right|_{\newg}
			\label{E:GNORMPOINTWISESECONDFUNDTIMEDERIVATIVE}
				\\
		& \ \
		  +
			\scale^{5/3 - c \sqrt{\varepsilon}}
			\left|
				\SigmatLie_{\mathscr{Z}}^{\leq 2} \newlapse
			\right|
			+
			\scale^{1/3}
			\left|
				\newlapse
			\right|,
				\notag
					\\
		\left|
			\partial_t \newtimescalar
		\right|
		& \lesssim 
			\scale^{1/3 - c \sqrt{\varepsilon}}
			\left|
				\SigmatLie_{\mathscr{Z}}^{\leq 1} \newspacescalar
			\right|_{\newg}
			+
			\scale^{1/3}
			\left|
				\newlapse
			\right|,
			\label{E:GNORMPOINTWISESFTIMEDERIVATIVETIMEDERIVATIVE}
				\\
		\left|
			\partial_t \newspacescalar
		\right|_{\newg}
		& \lesssim 
			\scale^{- 1 - c \sqrt{\varepsilon}}
			\left|
				\mathscr{Z} \newtimescalar
			\right|
			+
			\scale^{1/3 - c \sqrt{\varepsilon}}
			\left|
				\mathscr{Z} \newlapse
			\right|,
			\label{E:GNORMPOINTWISESFSPACEDERIVATIVETIMEDERIVATIVE}
				\\
		\left|
			\widetilde{\mathscr{L}} \newlapse
		\right|
		& \lesssim
			\scale^{- c \sqrt{\varepsilon}}
			\left|
				\SigmatLie_{\mathscr{Z}}^{\leq 2} (\newg - \StMet)
			\right|_{\newg}
			+
			\scale^{- c \sqrt{\varepsilon}}
			\left|
				\newg^{-1} - \StMet^{-1}
			\right|_{\newg}
			+
			\sqrt{\varepsilon}
			\scale^{- c \sqrt{\varepsilon}}
			\left|
				\newspacescalar
			\right|_{\newg}.
			\label{E:LOWESTORDERELLIPTICLAPSEINHOMOGENEOUSTERMPOINTWISE}
		\end{align}
		\end{subequations}
\end{lemma}

\begin{proof}
	The proof is similar to the proof of Prop.~\ref{P:POINTWISEESTIMATESFORERRORTERMS}, 
	so we only sketch it.
	The main idea is to bound the products in the equations of Prop.~\ref{P:RESCALEDVARIABLES}
	with the help of the strong sup-norm estimates of Prop.~\ref{P:STRONGSUPNORMESTIMATES}
	and the comparison estimates of Cor.~\ref{C:IMPROVEMENTLEMMAOPERATORCOMPARISON}.
	For example, to prove \eqref{E:GNORMPOINTWISESFTIMEDERIVATIVETIMEDERIVATIVE},
	we first use equation \eqref{E:WAVEEQUATIONRENORMALIZED}
	to deduce
	\begin{align} \label{E:ISOLATEDPARTIALTNEWTIMESCALAR}
		|\partial_t \newtimescalar|
		& \lesssim \scale^{1/3} |\nabla \newspacescalar|_{\newg}
			+
			\scale^{1/3} |\newlapse|
			+ \scale^{1/3} |\newlapse| |\nabla \newspacescalar|_{\newg}
			+ \scale^{5/3} |\nabla \newlapse|_{\newg} |\newspacescalar|_{\newg}.
	\end{align}
	From \eqref{E:ISOLATEDPARTIALTNEWTIMESCALAR},
	Prop.~\ref{P:STRONGSUPNORMESTIMATES},
	and Cor.~\ref{C:IMPROVEMENTLEMMAOPERATORCOMPARISON},
	we obtain the desired bound \eqref{E:GNORMPOINTWISESFTIMEDERIVATIVETIMEDERIVATIVE}.
	
	The proof of \eqref{E:LOWESTORDERELLIPTICLAPSEINHOMOGENEOUSTERMPOINTWISE} is similar and is
	based on equation \eqref{E:LAPSEPDERENORMALIZEDLOWERDERIVATIVES},
	but it also relies on \eqref{E:RICCIESTONEUPONEDOWN}, \eqref{E:ERRORRICONEUPONEDOWN} and the fact that
	$\ScalarCur$ is the pure trace of $\Ric^{\#}$.
	The proof of \eqref{E:GNORMPOINTWISEMETRICTIMEDERIVATIVE} is similar and is based 
	on equations \eqref{E:EVOLUTIONMETRICRENORMALIZED}-\eqref{E:EVOLUTIONINVERSEMETRICRENORMALIZED}
	and the simple identities $\SigmatLie_{\partial_t} (\newg - \StMet) = \SigmatLie_{\partial_t} \newg$
	and $\SigmatLie_{\partial_t} (\newg^{-1} - \StMet^{-1}) = \SigmatLie_{\partial_t} \newg^{-1}$.
	The proof of \eqref{E:GNORMPOINTWISESECONDFUNDTIMEDERIVATIVE} 
	is similar and is based on equation \eqref{E:EVOLUTIONSECONDFUNDRENORMALIZED},
	\eqref{E:RICCIESTONEUPONEDOWN}, and \eqref{E:ERRORRICONEUPONEDOWN}.
	The proof of \eqref{E:GNORMPOINTWISESFSPACEDERIVATIVETIMEDERIVATIVE} is similar 
	and is based on equation \eqref{E:EVOLUTIONSPACESCALARRENORMALIZED}.
\end{proof}

\section{Preliminary $L^2$ estimates for the lapse and some below-top-order derivatives}
\label{S:ELLIPTICLAPSEESTIMATES}
In this section, we derive $L^2$ estimates for the lapse
and show that it is controlled by the energies. It turns out that
to obtain these bounds, we must simultaneously
derive similar bounds for the below-top-order derivatives of 
$\newg$ and $\newspacescalar$. The bounds that we derive in this section
yield better estimates (i.e., less singular with respect to $t$)
for $\newg$ and $\newspacescalar$ than the energy estimate \eqref{E:MAINAPRIORIENERGYESTIMATES}
derived below in Cor.\ \ref{C:MAINAPRIORIENERGYESTIMATES}.
However, unlike the estimate \eqref{E:IMPROVEDOPERATORCOMPARISON},
our estimates for $\newg$ and $\newspacescalar$ in this section lose one derivative
(which is permissible below top order).
The main result is provided by the following proposition,
which we prove in 
Subsect.\ \ref{SS:PROOFOFPROPBOUNDFORLAPSEANDBELOWTOPMETRICINTERMSOFENERGIES}.
Before proving the proposition, we will establish a series of preliminary estimates.

\begin{remark}
	Prop.~\ref{P:BOUNDFORLAPSEANDBELOWTOPMETRICINTERMSOFENERGIES} provides
	estimates in the case that the solution variables have been differentiated 
	at least one time.
	In Lemma~\ref{L:L2ESTIMATESFORNONDIFFERNTIATED}, we will
	derive similar estimates for the non-differentiated variables
	using a separate argument.
\end{remark}

\begin{proposition}[\textbf{Control of the lapse and the below-top-order derivatives of}
$\newg$ and $\newspacescalar$ \textbf{in terms of the energies}]
	\label{P:BOUNDFORLAPSEANDBELOWTOPMETRICINTERMSOFENERGIES}
	For $1 \leq M \leq 16$, the following estimates hold:
	\begin{subequations}
	\begin{align}
		\sum_{L=0}^2
		\scale^{4/3 + (2/3)L}
		\left\| 
			\nabla^L \mathscr{Z}^{[1,M]} \newlapse
		\right\|_{L_{\newg}^2(\Sigma_t)}
		& \lesssim 
			\varepsilon^2
		 	\scale^{4/3-c\sqrt{\varepsilon}}(t)
		 	+
			\SupTotalenergy{\smallparameter_*}{M}^{1/2}(t)
			+
			\underbrace{
			\sqrt{\varepsilon} \scale^{-c\sqrt{\varepsilon}}
			\SupTotalenergy{M-1}{\smallparameter_*}^{1/2}(t).
			}_{\mbox{\upshape absent if $M=1$}}
				\label{E:L2HIGHORDERLAPSEINTERMSOFENERGIES} 
\end{align}
Moreover, for $1 \leq M \leq 14$, we have
\begin{align}
		\sum_{L=0}^2
		\scale^{(2/3)L}
		\left\| 
			\nabla^L \mathscr{Z}^{[1,M]} \newlapse
		\right\|_{L_{\newg}^2(\Sigma_t)}
		& \lesssim 
			\varepsilon^2
		 	\scale^{-c\sqrt{\varepsilon}}(t)
		 	+
			\frac{1}{\sqrt{\varepsilon}}
			\scale^{-c\sqrt{\varepsilon}}(t)
			\SupTotalenergy{M+2}{\smallparameter_*}^{1/2}(t).
			\label{E:L2LOWESTORDERLAPSEINTERMSOFENERGIES}
	\end{align}
	\end{subequations}
	
	In addition, for $1 \leq M \leq 16$, we have
	\begin{align} \label{E:COMMUTEDMETRICL2BOUNDSINTERMSOFENERGY}
		\left\|
			\SigmatLie_{\mathscr{Z}}^{[1,M]} \newg
		\right\|_{L_{\newg}^2(\Sigma_t)},
			\,
		\left\|
			\SigmatLie_{\mathscr{Z}}^{[1,M]} \newg^{-1}
		\right\|_{L_{\newg}^2(\Sigma_t)}
		& \lesssim
			\varepsilon^2 \scale^{-c \sqrt{\varepsilon}}(t)
			+	
			\frac{1}{\sqrt{\varepsilon}}
			\scale^{-c \sqrt{\varepsilon}}(t)
			\SupTotalenergy{M}{\smallparameter_*}^{1/2}(t).
\end{align}

Finally, for $1 \leq M \leq 15$, we have
\begin{align} \label{E:BELOWTOPORDERDERIVATIVESOFSPATIALSACALARFIELD}
		\left\|
			\SigmatLie_{\mathscr{Z}}^{[1,M]} \newspacescalar
		\right\|_{L_{\newg}^2(\Sigma_t)}
		\lesssim
		\varepsilon^2 \scale^{-c \sqrt{\varepsilon}}(t)
		+
		\frac{1}{\sqrt{\varepsilon}}
		\scale^{-c \sqrt{\varepsilon}}(t)
		\SupTotalenergy{M+1}{\smallparameter_*}^{1/2}(t).
\end{align}
	
\end{proposition}

\subsection{Preliminary $L^2$ bounds}
\label{SS:PRELININARYL2BOUNDSFORSOMEINHOMOGNEOUSTERMS}

\begin{lemma}[Preliminary $L^2$ \textbf{bounds for some inhomogeneous terms}]
	\label{L:FIRSTL2BOUNDSINHOMOGENEOUS}
	Let $\vec{I}$ be a $\mathscr{Z}$-multi-index with
	$1 \leq |\vec{I}| \leq 16$,
	except in
	\eqref{E:LAPSELOWBORDERINHOMPRELIMINARYL2}
	and 
	\eqref{E:LAPSELOWJUNKINHOMPRELIMINARYL2} 
	where we instead assume that $1 \leq |\vec{I}| \leq 14$.
	The inhomogeneous terms
	from equations 
	\eqref{E:COMMUTEDEVOLUTIONMETRICRENORMALIZED},
	\eqref{E:COMMUTEDEVOLUTIONINVERSEMETRICRENORMALIZED},
	\eqref{E:COMMUTEDWAVEEQUATIONRENORMALIZED},
	\eqref{E:COMMUTEDSPACEDERIVATIVESWAVEEQUATIONRENORMALIZED},
	\eqref{E:COMMUTEDLAPSEPDERENORMALIZEDHIGHERDERIVATIVES},
	and \eqref{E:COMMUTEDLAPSEPDERENORMALIZEDLOWERDERIVATIVES}
	verify the following estimates:
	\begin{subequations}
	\begin{align}
	\left\| 
		\CommutedMetBorderInhom{\vec{I}}
	\right\|_{L_{\newg}^2(\Sigma_t)},
		\,
	\left\| 
		\CommutedInvMetBorderInhom{\vec{I}}
	\right\|_{L_{\newg}^2(\Sigma_t)}
		& \lesssim
		\varepsilon
		\left\| 
			\SigmatLie_{\mathscr{Z}}^{\vec{I}} \newg 
		\right\|_{L_{\newg}^2(\Sigma_t)}
		+
		\underbrace{
		\sqrt{\varepsilon}
		\scale^{-c \sqrt{\varepsilon}}(t)
		\left\| 
			\SigmatLie_{\mathscr{Z}}^{[1,|\vec{I}|-1]} \newg 
		\right\|_{L_{\newg}^2(\Sigma_t)}}_{\mbox{\upshape absent if $|\vec{I}|=1$}}
			\label{E:METRICBORDERINHOMPRELIMINARYL2} \\
	& \ \
		+
		\scale^{4/3}
		\left\|
			\mathscr{Z}^{\vec{I}} \newlapse
		\right\|_{L_{\newg}^2(\Sigma_t)}
			\notag \\
	& \ \
		+
		\SupTotalenergy{|\vec{I}|}{\smallparameter_*}^{1/2}(t)
		+
		\underbrace{
		\sqrt{\varepsilon}
		\scale^{-c \sqrt{\varepsilon}}(t)
		\SupTotalenergy{|\vec{I}|-1}{\smallparameter_*}^{1/2}(t)}_{\mbox{\upshape absent if $|\vec{I}|=1$}},
			\notag \\
		\left\| 
			\CommutedSpaceSfBorderInhom{\vec{I}}
		\right\|_{L_{\newg}^2(\Sigma_t)}
	& \lesssim
		\varepsilon
		\left\| 
			\nabla \mathscr{Z}^{\vec{I}} \newlapse 
		\right\|_{L_{\newg}^2(\Sigma_t)},
			\label{E:SFSPACEBORDERINHOMPRELIMINARYL2} 
				\\
		\left\| 
			\CommutedLapseHighBorderInhom{\vec{I}}
		\right\|_{L_{\newg}^2(\Sigma_t)}
	& \lesssim
		\varepsilon
		\SupTotalenergy{|\vec{I}|}{\smallparameter_*}^{1/2}(t)
		+
		\underbrace{
		\sqrt{\varepsilon}
		\scale^{-c \sqrt{\varepsilon}}(t)
		\SupTotalenergy{|\vec{I}|-1}{\smallparameter_*}^{1/2}(t)}_{\mbox{\upshape absent if $|\vec{I}|=1$}},
			\label{E:LAPSEHIGHBORDERINHOMPRELIMINARYL2} 
			\\
	\left\| 
		\CommutedLapseLowBorderInhom{\vec{I}}
	\right\|_{L_{\newg}^2(\Sigma_t)}
	& \lesssim
		+
		\scale^{-c \sqrt{\varepsilon}}(t)
		\left\| 
			\SigmatLie_{\mathscr{Z}}^{[1,|\vec{I}|+2]} \newg 
		\right\|_{L_{\newg}^2(\Sigma_t)}
		+
		\sqrt{\varepsilon}
		\scale^{-c \sqrt{\varepsilon}}(t)
		\left\| 
			\SigmatLie_{\mathscr{Z}}^{[1,|\vec{I}|]} \newspacescalar
		\right\|_{L_{\newg}^2(\Sigma_t)},
			\label{E:LAPSELOWBORDERINHOMPRELIMINARYL2} 
	\end{align}
	\end{subequations}	
	
	\begin{subequations}
	\begin{align}
	\left\| 
		\CommutedMetJunkInhom{\vec{I}}
	\right\|_{L_{\newg}^2(\Sigma_t)},
		\,
	\left\| 
		\CommutedInvMetJunkInhom{\vec{I}}
	\right\|_{L_{\newg}^2(\Sigma_t)}
		& \lesssim
		\sqrt{\varepsilon}
		\scale^{-c \sqrt{\varepsilon}}(t)
		\left\| 
			\SigmatLie_{\mathscr{Z}}^{[1,|\vec{I}|]} \newg 
		\right\|_{L_{\newg}^2(\Sigma_t)}
			\label{E:METRICJUNKINHOMPRELIMINARYL2} \\
	& \ \
		+
		\underbrace{
		\sqrt{\varepsilon}
		\scale^{-c \sqrt{\varepsilon}}(t)
		\left\|
			\mathscr{Z}^{[1,|\vec{I}|-1]} \newlapse
		\right\|_{L_{\newg}^2(\Sigma_t)}
		}_{\mbox{\upshape absent if $|\vec{I}|=1$}}
			\notag \\
	& \ \
		+
		\sqrt{\varepsilon}
		\scale^{-c \sqrt{\varepsilon}}(t)
		\SupTotalenergy{|\vec{I}|}{\smallparameter_*}^{1/2}(t),
			\notag
				\\
		\left\| 
			\CommutedSpaceSfJunkInhom{\vec{I}}
		\right\|_{L_{\newg}^2(\Sigma_t)}
	& \lesssim
		\sqrt{\varepsilon}
		\scale^{-c \sqrt{\varepsilon}}(t)
		\left\| 
			\mathscr{Z}^{[1,|\vec{I}|]} \newlapse 
		\right\|_{L_{\newg}^2(\Sigma_t)}
		+
		\underbrace{
		\sqrt{\varepsilon}
		\scale^{-c \sqrt{\varepsilon}}(t)
		\SupTotalenergy{|\vec{I}|-1}{\smallparameter_*}^{1/2}(t)}_{\mbox{\upshape absent if $|\vec{I}|=1$}},
			\label{E:SFSPACEJUNKINHOMPRELIMINARYL2} 
				\\
	\left\| 
		\CommutedLapseHighJunkInhom{\vec{I}}
	\right\|_{L_{\newg}^2(\Sigma_t)}
	& \lesssim
		\sqrt{\varepsilon}
		\scale^{- c \sqrt{\varepsilon}}(t)
		\SupTotalenergy{|\vec{I}|}{\smallparameter_*}^{1/2}(t)
		+
		\sqrt{\varepsilon}
		\scale^{-c \sqrt{\varepsilon}}(t)
		\left\|
			\mathscr{Z}^{[1,|\vec{I}|]} \newg
		\right\|_{L_{\newg}^2(\Sigma_t)}
		\label{E:LAPSEHIGHJUNKINHOMPRELIMINARYL2}  \\
	& \ \
		+
		\sqrt{\varepsilon}
		\scale^{-c \sqrt{\varepsilon} + (4/3)}(t)
		\left\|
			\nabla^{\leq 1} \mathscr{Z}^{[1,|\vec{I}|]} \newlapse
		\right\|_{L_{\newg}^2(\Sigma_t)}
			\notag \\
	& \ \
		+
		\underbrace{
		\sqrt{\varepsilon}
		\scale^{-c \sqrt{\varepsilon}}(t)
		\left\|
			\mathscr{Z}^{[1,|\vec{I}|-1]} \newlapse
		\right\|_{L_{\newg}^2(\Sigma_t)}}_{\mbox{\upshape absent if $|\vec{I}|=1$}},
			\notag \\ 
	\left\| 
		\CommutedLapseLowJunkInhom{\vec{I}}
	\right\|_{L_{\newg}^2(\Sigma_t)}
	& \lesssim
			\sqrt{\varepsilon}
			\scale^{-c \sqrt{\varepsilon}}(t)
			\left\| 
				\SigmatLie_{\mathscr{Z}}^{[1,|\vec{I}|+2]} \newg 
			\right\|_{L_{\newg}^2(\Sigma_t)}
			\label{E:LAPSELOWJUNKINHOMPRELIMINARYL2} 	\\
		& \ \
			+
			\sqrt{\varepsilon}
			\scale^{-c \sqrt{\varepsilon}}(t)
			\left\| 
				\nabla \mathscr{Z}^{[1,|\vec{I}|]} \newlapse 
			\right\|_{L_{\newg}^2(\Sigma_t)}
			\notag
				\\
		& \ \
		+
		\sqrt{\varepsilon}
		\scale^{-c \sqrt{\varepsilon}}(t)
		\left\| 
			\SigmatLie_{\mathscr{Z}}^{[1,|\vec{I}|]} \newspacescalar
		\right\|_{L_{\newg}^2(\Sigma_t)}.
		\notag
	\end{align}
	\end{subequations}
\end{lemma}

\begin{proof}
	To prove \eqref{E:METRICBORDERINHOMPRELIMINARYL2},
	we need only to take the norm $\| \cdot \|_{L_{\newg}^2(\Sigma_t)}$ of 
	inequality \eqref{E:POINTWISEMETRICBORDERLINE}
	and to use Lemma~\ref{L:ENERGYCOERCIVENESS} to control the
	norms of the terms on the RHS in terms of $\SupTotalenergy{\cdot}{\smallparameter_*}$,
	the Lie derivatives of $\newg$, and the derivatives of $\newlapse$.
	The proofs of the remaining estimates follow similarly 
	with the help of the pointwise estimates of Prop.~\ref{P:POINTWISEESTIMATESFORERRORTERMS}.
	We omit the details, noting only that we have explicitly placed
	the terms
	$
	\sqrt{\varepsilon}
		\scale^{-c \sqrt{\varepsilon}}(t)
		\left\| 
			\SigmatLie_{\mathscr{Z}}^{[1,|\vec{I}|]} \newspacescalar
		\right\|_{L_{\newg}^2(\Sigma_t)}
	$
	on RHSs \eqref{E:LAPSELOWBORDERINHOMPRELIMINARYL2} and \eqref{E:LAPSELOWJUNKINHOMPRELIMINARYL2}
	instead of using
	Lemma~\ref{L:ENERGYCOERCIVENESS} to bound them by 
	$\lesssim \scale^{-2/3 -c \sqrt{\varepsilon}}(t) \SupTotalenergy{|\vec{I}|}{\smallparameter_*}(t)$.
	
\end{proof}

\begin{lemma}[\textbf{Preliminary $L^2$ bounds for the below-top-order derivatives of $\newg$ and 
$\newspacescalar$}]
\label{L:COMMUTEDMETRICL2BOUNDSINTERMSOFENERGY}
	For $1 \leq |\vec{I}| \leq 16$,
	the following estimates hold on $[0,\Tboot)$:
	\begin{align} \label{E:COMMUTEDMETRICL2BOUNDSINTERMSOFLAPSEANDENERGY}
		&
		\left\|
			\SigmatLie_{\mathscr{Z}}^{\vec{I}} \newg
		\right\|_{L_{\newg}^2(\Sigma_t)},
			\,
		\left\|
			\SigmatLie_{\mathscr{Z}}^{\vec{I}} \newg^{-1}
		\right\|_{L_{\newg}^2(\Sigma_t)}
				\\
		& \lesssim
			\varepsilon^2 \scale^{-c \sqrt{\varepsilon}}(t)
			+	
			\scale^{-c \sqrt{\varepsilon}}(t)
			\int_{s=0}^t
				\scale^{1/3}(s)
				\left\|
					\mathscr{Z}^{[1,|\vec{I}|]} \newlapse
				\right\|_{L_{\newg}^2(\Sigma_s)}
			\, ds
			+	
			\frac{1}{\sqrt{\varepsilon}}
			\scale^{-c \sqrt{\varepsilon}}(t)
			\SupTotalenergy{|\vec{I}|}{\smallparameter_*}^{1/2}(t).
			\notag
\end{align}

Moreover, for $1 \leq |\vec{I}| \leq 15$, we have
\begin{align} \label{E:COMMUTEDSFSPATIALINTERMSOFLAPSEANDENERGY}
		\left\|
			\SigmatLie_{\mathscr{Z}}^{\vec{I}} \newspacescalar
		\right\|_{L_{\newg}^2(\Sigma_t)}
		& \lesssim
			\varepsilon^2 \scale^{-c \sqrt{\varepsilon}}(t)
			+
			\scale^{-c \sqrt{\varepsilon}}(t)
			\int_{s=0}^t
				\scale^{1/3}(s)
				\left\|
					\nabla^{\leq 1} \mathscr{Z}^{[1,|\vec{I}|]} \newlapse
				\right\|_{L_{\newg}^2(\Sigma_s)}
			\, ds
				\\
		& \ \
			+	
			\frac{1}{\sqrt{\varepsilon}}
			\scale^{-c \sqrt{\varepsilon}}(t)
			\SupTotalenergy{|\vec{I}|+1}{\smallparameter_*}^{1/2}(t).
			\notag
\end{align}
\end{lemma}

\begin{proof}
	To prove \eqref{E:COMMUTEDSFSPATIALINTERMSOFLAPSEANDENERGY}, 
	we first use
	\eqref{E:COMMUTEDSPACEDERIVATIVESWAVEEQUATIONRENORMALIZED},
	\eqref{E:LAPSESTRONGSUPNROM},
	\eqref{E:SFSPACEBORDERINHOMPRELIMINARYL2},
	\eqref{E:SFSPACEJUNKINHOMPRELIMINARYL2},
	Lemma~\ref{L:ENERGYCOERCIVENESS},
	and Cor.~\ref{C:IMPROVEMENTLEMMAOPERATORCOMPARISON}
	to deduce that
	\begin{align} \label{E:L2BOUNDFORTIMEDERIVATIVEOFCOMMUTEDSFSPACE} 
	\left\|
		\SigmatLie_{\partial_t} \SigmatLie_{\mathscr{Z}}^{\vec{I}} \newspacescalar 
	\right\|_{L^2_{\newg}(\Sigma_t)}
	& \lesssim
	\scale^{1/3 - c \sqrt{\varepsilon}}(t)
			\left\|
				\nabla^{\leq 1} \mathscr{Z}^{[1,|\vec{I}|]} \newlapse
			\right\|_{L_{\newg}^2(\Sigma_t)}
	+
	\scale^{-1-c \sqrt{\varepsilon}}(t)
	\SupTotalenergy{|\vec{I}|}{\smallparameter_*}^{1/2}(t).
	\end{align}
	Inserting the estimate \eqref{E:L2BOUNDFORTIMEDERIVATIVEOFCOMMUTEDSFSPACE}
	into inequality \eqref{E:BASICGRONWALLESTIMATEFORATENSORFIELD}
	with $\xi := \SigmatLie_{\mathscr{Z}}^{\vec{I}} \newspacescalar$,
	using the small-data bound
	$
	\left\|
			\SigmatLie_{\mathscr{Z}}^{[1,|\vec{I}|]} \newspacescalar 
		\right\|_{L^2_{\newg}(\Sigma_0)}
		\lesssim
		\varepsilon^2
	$
	(which follows from \eqref{E:SMALLDATA} and Lemma~\ref{L:L2NORMSCOMPARISONESTIMATES}),
	using the fact that $\scale$ is decreasing on $[0,\TCrunch)$,
	and using \eqref{E:SCALEFACTORTIMEINTEGRALS},
	which in particular implies 
	$
\int_{s=0}^t
				\scale^{-1 - c \sqrt{\varepsilon}}(s)
				\SupTotalenergy{|\vec{I}|}{\smallparameter_*}^{1/2}(s)
			\, ds
\lesssim \frac{1}{\sqrt{\varepsilon}} \SupTotalenergy{|\vec{I}|}{\smallparameter_*}^{1/2}(t)
$,
	we obtain
	\begin{align}
	\left\|
		\SigmatLie_{\mathscr{Z}}^{[1,|\vec{I}|]} \newspacescalar
	\right\|_{L^2_{\newg}(\Sigma_t)}
	& 
	\leq
		C \varepsilon^2 \scale^{- c \sqrt{\varepsilon}}(t)
		+
		C
			\scale^{- c \sqrt{\varepsilon}}(t)
			\int_{s=0}^t
				\scale^{1/3}(s)
				\left\|
					\nabla^{\leq 1} \mathscr{Z}^{[1,|\vec{I}|]} \newlapse
				\right\|_{L_{\newg}^2(\Sigma_s)}
			\, ds
	\label{E:L2COMMUTEDCSPACESFGRONWALLREADY}	\\
	& \ \ 
			+
			C
			\frac{1}{\sqrt{\varepsilon}} \scale^{- c \sqrt{\varepsilon}}(t) \SupTotalenergy{|\vec{I}|}{\smallparameter_*}^{1/2}(t),
			\notag
\end{align}
which is the desired bound
\eqref{E:COMMUTEDSFSPATIALINTERMSOFLAPSEANDENERGY}.

	
	To prove \eqref{E:COMMUTEDMETRICL2BOUNDSINTERMSOFLAPSEANDENERGY}, 
	we first use
	\eqref{E:COMMUTEDEVOLUTIONMETRICRENORMALIZED},
	\eqref{E:METRICBORDERINHOMPRELIMINARYL2},
	\eqref{E:METRICJUNKINHOMPRELIMINARYL2},
	inequality \eqref{E:BASICGRONWALLREADYESTIMATEFORATENSORFIELD} 
	with $\xi := \SigmatLie_{\mathscr{Z}}^{\vec{I}} \newg$,
	the small-data bound
	$
	\left\|
			\SigmatLie_{\mathscr{Z}}^{\vec{I}} \newg 
		\right\|_{L^2_{\newg}(\Sigma_0)}
		\lesssim
		\varepsilon^2
	$
	(which follows from \eqref{E:SMALLDATA} and Lemma~\ref{L:L2NORMSCOMPARISONESTIMATES}),
	and \eqref{E:SCALEFACTORTIMEINTEGRALS}
	to obtain
	\begin{align}
	\left\|
		\SigmatLie_{\mathscr{Z}}^{\vec{I}} \newg 
	\right\|_{L^2_{\newg}(\Sigma_t)}
	& 
	\leq
		C \varepsilon^2
		+
		c \varepsilon
		\int_{s=0}^t
				\scale^{-1}(s)
				\left\|
					\SigmatLie_{\mathscr{Z}}^{\vec{I}} \newg 
				\right\|_{L^2_{\newg}(\Sigma_s)}
		\, ds
		+
		C \sqrt{\varepsilon}
		\int_{s=0}^t
				\scale^{1/3 - c \sqrt{\varepsilon}}(s)
				\left\|
					\SigmatLie_{\mathscr{Z}}^{\vec{I}} \newg 
				\right\|_{L^2_{\newg}(\Sigma_s)}
		\, ds
	\label{E:L2COMMUTEDMETRICGRONWALLREADY}	\\
		& \ \
			+
		C \sqrt{\varepsilon}
		\int_{s=0}^t
				\scale^{-1 - c \sqrt{\varepsilon}}(s)
				\left\|
					\SigmatLie_{\mathscr{Z}}^{[1,|\vec{I}|-1]} \newg
				\right\|_{L^2_{\newg}(\Sigma_s)}
		\, ds
			+	
			C
			\int_{s=0}^t
				\scale^{1/3 - c \sqrt{\varepsilon}}(s)
				\left\|
					\mathscr{Z}^{[1,|\vec{I}|]} \newlapse
				\right\|_{L_{\newg}^2(\Sigma_s)}
			\, ds
			\notag \\
		& \ \ 
			+
			C
			\frac{1}{\sqrt{\varepsilon}}
			\scale^{- c \sqrt{\varepsilon}}(t)
			\SupTotalenergy{|\vec{I}|}{\smallparameter_*}^{1/2}(s),
			\notag
\end{align}
where the integral involving
$
\left\|
					\SigmatLie_{\mathscr{Z}}^{[1,|\vec{I}|-1]} \newg
				\right\|_{L^2_{\newg}(\Sigma_s)}
$ 
is
absent if $|\vec{I}|=1$
and to obtain the last term on RHS~\eqref{E:L2COMMUTEDMETRICGRONWALLREADY}, 
we have used the bound
$
\int_{s=0}^t
	\scale^{-1}(s)
	\SupTotalenergy{|\vec{I}|}{\smallparameter_*}^{1/2}(s)
\, ds
\leq
C
(1 + |\ln \scale(t)|)
\SupTotalenergy{|\vec{I}|}{\smallparameter_*}^{1/2}(t)
\leq 
\frac{C}{\sqrt{\varepsilon}}
\scale^{- c \sqrt{\varepsilon}}
\SupTotalenergy{|\vec{I}|}{\smallparameter_*}^{1/2}(t)
$,
which is a simple consequence of
Cor.~\ref{C:SCALEFACTORTIMEINTEGRALS} and \eqref{E:SCALEFACTORLOGPOWERBOUND}.
We now derive \eqref{E:COMMUTEDMETRICL2BOUNDSINTERMSOFLAPSEANDENERGY}
by using induction in $|\vec{I}|$. In the base case
$|\vec{I}| = 1$, 
we use inequality \eqref{E:L2COMMUTEDMETRICGRONWALLREADY}, 
Gronwall's inequality,
\eqref{E:EXPONENTIATEDSCALEFACTORTIMEINTEGRALS},
and the fact that $\scale$ is decreasing on $[0,\TCrunch)$
to conclude the bound for $\SigmatLie_{\mathscr{Z}}^{\vec{I}} \newg$ stated in
\eqref{E:COMMUTEDMETRICL2BOUNDSINTERMSOFLAPSEANDENERGY}.
To carry out the induction step, we assume that
\eqref{E:COMMUTEDMETRICL2BOUNDSINTERMSOFLAPSEANDENERGY} has been established for multi-indices of length
$|\vec{I}|-1$. To obtain \eqref{E:COMMUTEDMETRICL2BOUNDSINTERMSOFENERGY} for $\vec{I}$,
we use the induction hypothesis,
Cor.~\ref{C:SCALEFACTORTIMEINTEGRALS},
and the fact that $\scale$ is decreasing on $[0,\TCrunch)$
to bound the first integral on 
the second line of RHS~\eqref{E:L2COMMUTEDMETRICGRONWALLREADY} as follows:
\begin{align} \label{E:L2COMMUTEDMETRICUSINGINDUCTION}
		& 
		C \sqrt{\varepsilon}
		\int_{s=0}^t
				\scale^{-1 - c \sqrt{\varepsilon}}(s)
				\left\|
					\SigmatLie_{\mathscr{Z}}^{[1,|\vec{I}|-1]} \newg
				\right\|_{L^2_{\newg}(\Sigma_s)}
		\, ds
			\\
&   \lesssim
		\varepsilon^2 \scale^{-c \sqrt{\varepsilon}}(t)
			+	
			\scale^{-c \sqrt{\varepsilon}}(t)
			\int_{s=0}^t
				\scale^{1/3}(s)
				\left\|
					\mathscr{Z}^{[1,|\vec{I}|]} \newlapse
				\right\|_{L_{\newg}^2(\Sigma_s)}
			\, ds
			+	
			\int_{s = 0}^1
				\scale^{-1 - c \sqrt{\varepsilon}}(s)
				\SupTotalenergy{|\vec{I}|-1}{\smallparameter_*}^{1/2}(s)
			\, ds
			\notag
				\\
		&   \lesssim
		\varepsilon^2 \scale^{-c \sqrt{\varepsilon}}(t)
			+	
			\scale^{-c \sqrt{\varepsilon}}(t)
			\int_{s=0}^t
				\scale^{1/3}(s)
				\left\|
					\mathscr{Z}^{[1,|\vec{I}|]} \newlapse
				\right\|_{L_{\newg}^2(\Sigma_s)}
			\, ds
			+	
			\frac{1}{\sqrt{\varepsilon}} \scale^{-c \sqrt{\varepsilon}}(t) 
			\SupTotalenergy{|\vec{I}|}{\smallparameter_*}^{1/2}(t).
			\notag
\end{align}
We now substitute RHS~\eqref{E:L2COMMUTEDMETRICUSINGINDUCTION}
for the first integral on the second line of RHS~\eqref{E:L2COMMUTEDMETRICGRONWALLREADY} 
and apply Gronwall's inequality as before,
thereby concluding the bound for $\SigmatLie_{\mathscr{Z}}^{\vec{I}} \newg$
stated in
\eqref{E:COMMUTEDMETRICL2BOUNDSINTERMSOFLAPSEANDENERGY}.
We have therefore closed the induction and 
completed the proof of \eqref{E:COMMUTEDMETRICL2BOUNDSINTERMSOFLAPSEANDENERGY}
for 
$\left\|
	\SigmatLie_{\mathscr{Z}}^{\vec{I}} \newg
\right\|_{L_{\newg}^2(\Sigma_t)}
$.
To obtain the same bound for
$\left\|
			\SigmatLie_{\mathscr{Z}}^{\vec{I}} \newg^{-1} 
\right\|_{L_{\newg}^2(\Sigma_t)}
$,
we use a similar argument based on 
\eqref{E:COMMUTEDEVOLUTIONINVERSEMETRICRENORMALIZED},
\eqref{E:METRICBORDERINHOMPRELIMINARYL2},
\eqref{E:METRICJUNKINHOMPRELIMINARYL2},
and the already proven bounds for
$
\left\|
	\SigmatLie_{\mathscr{Z}}^{\vec{I}} \newg
\right\|_{L_{\newg}^2(\Sigma_t)}
$;
we omit the details.

\end{proof}

\subsection{Basic elliptic estimates}
\label{SS:BASICELLIPTIC}
In this subsection, we prove some simple elliptic estimates
that we will use to control the time-rescaled lapse variable $\newlapse$.

\begin{lemma}[\textbf{Basic elliptic estimates}]
	\label{L:BASICELLIPTICESTIMATES}
	Let $\mathscr{L}$ and $\widetilde{\mathscr{L}}$
	be the elliptic operators from Def.\ \ref{D:ELLIPTICOPS}.
	Solutions $u$ to the elliptic PDE
	\begin{align} \label{E:HIGHBASICELLIPTICPDE}
		\mathscr{L} u
		& = U
	\end{align}
	verify the estimate
	\begin{align} \label{E:HIGHBASICELLIPTICESTIMATE}
		\sum_{L=0}^2
		\scale^{4/3 +(2/3) L}(t)
		\left\|
			\nabla^L u
		\right\|_{L_{\newg}^2(\Sigma_t)}
		& \lesssim
			\left\|
				U
			\right\|_{L_{\newg}^2(\Sigma_t)}.
	\end{align}
	
	Moreover, solutions $u$ to the elliptic PDE
	\begin{align} \label{E:LOWBASICELLIPTICPDE}
		\widetilde{\mathscr{L}} u
		& = U
	\end{align}
	verify the estimate
	\begin{align} \label{E:LOWBASICELLIPTICESTIMATE}
		\sum_{L=0}^2
		\scale^{(2/3) L}(t)
		\left\|
			\nabla^L u
		\right\|_{L_{\newg}^2(\Sigma_t)}
		& \lesssim
			\left\|
				U
			\right\|_{L_{\newg}^2(\Sigma_t)}.
	\end{align}
\end{lemma}

\begin{proof}
	We first prove \eqref{E:HIGHBASICELLIPTICESTIMATE}.
	First, using Prop.~\ref{P:STRONGSUPNORMESTIMATES} 
	and Cor.~\ref{C:IMPROVEMENTLEMMAOPERATORCOMPARISON},
	we deduce
	that the coefficient $f$ defined by
	\eqref{E:ERRORTERMHIGHORDERELLIPTICOPERATOR}
	verifies the bounds 
	$f 
	=
	(\scale')^2 
	+ 
	(2/3) \scale^{4/3} 
	+ 
	\mathcal{O}(\varepsilon)
	$
	and
	$
	|\nabla f|_{\newg}
	\lesssim \varepsilon \scale^{- c \sqrt{\varepsilon}}
	$.
	In particular, also using \eqref{E:FRIEDMANNFIRSTORDER}, we deduce that
	$f \geq 2/3 - C \varepsilon$.
	Next, we multiply equation \eqref{E:HIGHBASICELLIPTICPDE}
	by $\scale^{4/3} u$ and integrate by parts to obtain the estimate
	\begin{align} \label{E:JUSTBELOWTOPORDERLAPSEELLIPTICALMOSTDERIVED}
	\int_{\Sigma_t}
		\scale^4 |\nabla u|_{\newg}^2
	\, d \tvol
	+
	\int_{\Sigma_t}
		\scale^{8/3} f u^2
	\, d \tvol
	\leq
	\int_{\Sigma_t}
		\scale^{4/3} |u| |U|
	\, d \tvol
	\leq 
	\frac{1}{2}
	\int_{\Sigma_t}
		\scale^{8/3} u^2
	\, d \tvol
	+
	\frac{1}{2}
	\int_{\Sigma_t}
		U^2
	\, d \tvol.
	\end{align}
	The desired bounds \eqref{E:HIGHBASICELLIPTICESTIMATE}
	for $u$ and $\nabla u$ now follow easily from \eqref{E:JUSTBELOWTOPORDERLAPSEELLIPTICALMOSTDERIVED}
	the above lower bound for $f$.
	Similarly, to obtain the bound \eqref{E:HIGHBASICELLIPTICESTIMATE}
	for $\nabla^2 u$, 
	we multiply equation \eqref{E:HIGHBASICELLIPTICPDE}
	by $\scale^{8/3} \Delta_{\newg} u$ and integrate by parts
	to obtain
	\begin{align} \label{E:TOPORDERLAPSEELLIPTICALMOSTDERIVED}
	\int_{\Sigma_t}
		\scale^{16/3} |\nabla^2 u|_{\newg}^2
	\, d \tvol
	+
	\int_{\Sigma_t}
		\scale^4 f |\nabla u|_{\newg}
	\, d \tvol
	& \leq
	C
	\int_{\Sigma_t}
		\scale^4 |\nabla f|_{\newg} |u| |\nabla u|_{\newg}
	\, d \tvol
	+
	C
	\int_{\Sigma_t}
		\scale^{8/3} |\nabla^2 u|_{\newg} |U|_{\newg}
	\, d \tvol
		\\
& \leq
	C
	\int_{\Sigma_t}
		\scale^2 |\nabla f|_{\newg} |u|^2 
	\, d \tvol
	+
	\frac{1}{3}
	\int_{\Sigma_t}
		\scale^6 |\nabla f|_{\newg} |\nabla u|_{\newg}^2
	\, d \tvol
	\notag \\
& \ \
	+
	\frac{1}{2}
	\int_{\Sigma_t}
		\scale^{16/3} |\nabla^2 u|_{\newg}^2
	\, d \tvol
	+
	C
	\int_{\Sigma_t}
		|U|^2
	\, d \tvol.
	\notag
	\end{align}
	From \eqref{E:TOPORDERLAPSEELLIPTICALMOSTDERIVED},
	the lower bound on $f$ obtained above, the bound on $|\nabla f|_{\newg}$
	obtained above, and the already established bound
	$
	\scale^{4/3}
	\left\|
		u
	\right\|_{L_{\newg}^2(\Sigma_t)}
	\leq C
	\left\|
		U
	\right\|_{L_{\newg}^2(\Sigma_t)}
	$,
	we easily conclude the 
	desired bound \eqref{E:HIGHBASICELLIPTICESTIMATE}
	for $\nabla^2 u$.
	
	The proof of \eqref{E:LOWBASICELLIPTICESTIMATE}
	is similar 
	and relies on the bound 
	$
	\frac{2}{3} - C \varepsilon
	\leq
	\widetilde{f} \leq \frac{2}{3} + C \varepsilon$
	obtained in the proof of Lemma~\ref{L:MAXIMUMPRINCIPLEESTIMATE},
	where $\widetilde{f}$ is defined in 
	\eqref{E:ERRORTERMLOWORDERELLIPTICOPERATOR}
	and appears the definition \eqref{E:LOWORDERELLIPTICOPERATOR}
	of $\widetilde{\mathscr{L}}$; we omit the details.
\end{proof}

\subsection{Proof of Prop.~\ref{P:BOUNDFORLAPSEANDBELOWTOPMETRICINTERMSOFENERGIES}}
	\label{SS:PROOFOFPROPBOUNDFORLAPSEANDBELOWTOPMETRICINTERMSOFENERGIES}
We first prove \eqref{E:L2HIGHORDERLAPSEINTERMSOFENERGIES}.
Throughout we use silently use the fact that $\scale$ is decreasing on $[0,\TCrunch)$.
Using equation \eqref{E:COMMUTEDLAPSEPDERENORMALIZEDHIGHERDERIVATIVES}
and the elliptic estimate \eqref{E:HIGHBASICELLIPTICESTIMATE},
we deduce
\begin{align} \label{E:LAPSEHIGHFIRSTL2EST}
		&
		\sum_{L=0}^2
		\scale^{4/3 + (2/3)L}(t)
		\left\|
			\nabla^L \mathscr{Z}^{[1,M]} \newlapse
		\right\|_{L_{\newg}^2(\Sigma_t)}
			\\
	& \leq
		C
		\sum_{1 \leq |\vec{I}| \leq M}
				\left\lbrace
					2 \sqrt{\frac{2}{3}}
					\left\|
						\mathscr{Z}^{\vec{I}} \newtimescalar
					\right\|_{L_{\newg}^2(\Sigma_t)}
					+
					\left\|
						\CommutedLapseHighBorderInhom{\vec{I}}
					\right\|_{L_{\newg}^2(\Sigma_t)}
					+
					\scale^{4/3}(t)
					\left\|
						\CommutedLapseHighJunkInhom{\vec{I}}
					\right\|_{L_{\newg}^2(\Sigma_t)}
				\right\rbrace.
				\notag
\end{align}
From \eqref{E:LAPSEHIGHFIRSTL2EST},
\eqref{E:SFENERGYCOERCIVENESS},
\eqref{E:LAPSEHIGHBORDERINHOMPRELIMINARYL2},
\eqref{E:LAPSEHIGHJUNKINHOMPRELIMINARYL2},
\eqref{E:COMMUTEDMETRICL2BOUNDSINTERMSOFLAPSEANDENERGY},
and Cor.~\ref{C:SCALEFACTORTIMEINTEGRALS},
we deduce
\begin{align} \label{E:LAPSEHIGHFIRSTSOBOLEVGRONWALLREADY}
		&
		\sum_{L=0}^2
		\scale^{4/3 + (2/3)L}(t)
		\left\|
			\nabla^L \mathscr{Z}^{[1,M]} \newlapse
		\right\|_{L_{\newg}^2(\Sigma_t)}
			\\
	& \leq
		C \varepsilon^2 \scale^{4/3 - c \sqrt{\varepsilon}}(t) 
		+ 
		C 
		\SupTotalenergy{M}{\smallparameter_*}^{1/2}(t)
		+
		\underbrace{
		C \sqrt{\varepsilon} 
		\scale^{- c \varepsilon}(t) 
		\SupTotalenergy{M-1}{\smallparameter_*}^{1/2}(t)
		}_{\mbox{\upshape absent if $M=1$}}
		\notag		\\
	& \ \
		+
		C \sqrt{\varepsilon}
		\scale^{8/3 -c \sqrt{\varepsilon}}(t)
		\left\|
			\nabla^{\leq 1} \mathscr{Z}^{[1,M]} \newlapse
		\right\|_{L_{\newg}^2(\Sigma_t)}
		+
		\underbrace{
		C \sqrt{\varepsilon}
		\scale^{4/3 -c \sqrt{\varepsilon}}(t)
		\left\|
			\mathscr{Z}^{[1,M-1]} \newlapse
		\right\|_{L_{\newg}^2(\Sigma_t)}
		}_{\mbox{\upshape absent if $M=1$}}
		\notag
			\\
	& \ \
			+
			C \sqrt{\varepsilon}
			\int_{s=0}^t
				\scale^{5/3 -c \sqrt{\varepsilon}}(s)
				\left\|
					\mathscr{Z}^{[1,M]} \newlapse
				\right\|_{L_{\newg}^2(\Sigma_s)}
			\, ds.
			\notag
\end{align}
We now derive \eqref{E:L2HIGHORDERLAPSEINTERMSOFENERGIES}
using induction in $M$. 
Note that for $\varepsilon$ sufficiently small,
we can absorb the term
$
C \sqrt{\varepsilon}
		\scale^{8/3 -c \sqrt{\varepsilon}}(t)
		\left\|
			\nabla^{\leq 1} \mathscr{Z}^{[1,M]} \newlapse
		\right\|_{L_{\newg}^2(\Sigma_t)}
$
on RHS~\eqref{E:LAPSEHIGHFIRSTSOBOLEVGRONWALLREADY} back into the LHS.
Thus, in the base case $M=1$,
the estimate \eqref{E:L2HIGHORDERLAPSEINTERMSOFENERGIES}
follows from \eqref{E:LAPSEHIGHFIRSTSOBOLEVGRONWALLREADY},
Gronwall's inequality,
and Cor.~\ref{C:SCALEFACTORTIMEINTEGRALS}.
We now show how to obtain \eqref{E:L2HIGHORDERLAPSEINTERMSOFENERGIES}
in the case $M$ with the help of case $M-1$.
First, we use the induction hypothesis
to obtain the following bound for the next-to-last term on RHS~\eqref{E:LAPSEHIGHFIRSTSOBOLEVGRONWALLREADY}:
$
\sqrt{\varepsilon}
\scale^{4/3 -c \sqrt{\varepsilon}}(t)
		\left\|
			\mathscr{Z}^{[1,M-1]} \newlapse
		\right\|_{L_{\newg}^2(\Sigma_t)}
\lesssim 	
\varepsilon^2
\scale^{4/3 - c \sqrt{\varepsilon}}(t)
+	
\sqrt{\varepsilon} \scale^{- c \sqrt{\varepsilon}}
\SupTotalenergy{M-1}{\smallparameter_*}^{1/2}(t)
$.
We now use this bound to control the 
next-to-last-term on RHS~\eqref{E:LAPSEHIGHFIRSTSOBOLEVGRONWALLREADY}
and then apply Gronwall's inequality,
thereby deducing, with the help of Cor.~\ref{C:SCALEFACTORTIMEINTEGRALS},
\eqref{E:L2HIGHORDERLAPSEINTERMSOFENERGIES} in the case $M$.	
We have therefore closed the induction, which proves \eqref{E:L2HIGHORDERLAPSEINTERMSOFENERGIES}.
	
	We now prove \eqref{E:L2LOWESTORDERLAPSEINTERMSOFENERGIES}.
	We start by using 
	equation \eqref{E:COMMUTEDLAPSEPDERENORMALIZEDLOWERDERIVATIVES}
	and the elliptic estimate \eqref{E:LOWBASICELLIPTICESTIMATE} 
	to deduce
	\begin{align} \label{E:LAPSELOWELLIPTICESTIMATESTILLHAVETOBOUNDINHOMOGENEOUS}
		\sum_{L=0}^2
		\scale^{(2/3) L}(t)
		\left\|
			\nabla^L \mathscr{Z}^{[1,M]} \newlapse
		\right\|_{L_{\newg}^2(\Sigma_t)}
		& \leq 
			C
			\sum_{1 \leq |\vec{I}| \leq M}
				\left\lbrace
					\left\|
						\CommutedLapseLowBorderInhom{\vec{I}}
					\right\|_{L_{\newg}^2(\Sigma_t)}
					+
					\scale^{4/3}(t)
					\left\|
						\CommutedLapseLowJunkInhom{\vec{I}}
					\right\|_{L_{\newg}^2(\Sigma_t)}
				\right\rbrace.
	\end{align}
	We then insert the estimates	
	\eqref{E:LAPSELOWBORDERINHOMPRELIMINARYL2}
	and
	\eqref{E:LAPSELOWJUNKINHOMPRELIMINARYL2}
	into
	RHS~\eqref{E:LAPSELOWELLIPTICESTIMATESTILLHAVETOBOUNDINHOMOGENEOUS}
	and use
	\eqref{E:COMMUTEDMETRICL2BOUNDSINTERMSOFLAPSEANDENERGY}-\eqref{E:COMMUTEDSFSPATIALINTERMSOFLAPSEANDENERGY}
	and Cor.~\ref{C:SCALEFACTORTIMEINTEGRALS},
	as well as the bound
	$
	\left\|
		\nabla^{\leq 1} \mathscr{Z}^{[1,M]} \newlapse
	\right\|_{L_{\newg}^2(\Sigma_s)}
	\lesssim
	\scale^{- c \sqrt{\varepsilon}}(s)
	\left\|
		\mathscr{Z}^{[1,M+1]} \newlapse
	\right\|_{L_{\newg}^2(\Sigma_s)}
	$
	(which follows from \eqref{E:IMPROVEDGNORMSTMETRICCOMPARISON})
	to deduce that for $c > 0$ chosen sufficiently large, we have
	\begin{align} \label{E:LAPSELOWFIRSTSOBOLEVGRONWALLREADY}
		&
		\scale^{c \sqrt{\varepsilon}}(t)
		\sum_{L=0}^2
		\scale^{(2/3)L}(t)
		\left\|
			\nabla^L \mathscr{Z}^{[1,M]} \newlapse
		\right\|_{L_{\newg}^2(\Sigma_t)}
			\\
	& \leq
			C
		 \varepsilon^2 \scale^{-c \sqrt{\varepsilon}}(t)
		+ 
		C
		\frac{1}{\sqrt{\varepsilon}} \SupTotalenergy{M+2}{\smallparameter_*}^{1/2}(t)
		+
		C
		\varepsilon
		\scale^{4/3 -c \sqrt{\varepsilon}}(t)
		\left\|
			\nabla^{\leq 1} \mathscr{Z}^{[1,M]} \newlapse
		\right\|_{L_{\newg}^2(\Sigma_t)}
			\notag \\
	& \ \
		+
		C
		\int_{s=0}^t
		\scale^{1/3 - c \sqrt{\varepsilon}}(s)
		\left\|
			\mathscr{Z}^{[1,M+2]} \newlapse
		\right\|_{L_{\newg}^2(\Sigma_s)}
		\, ds.
		\notag
\end{align}
Note that for $\varepsilon$ sufficiently small,
we can absorb the term 
$
C
		\varepsilon
		\scale^{4/3 -c \sqrt{\varepsilon}}(t)
		\left\|
			\nabla^{\leq 1} \mathscr{Z}^{[1,M]} \newlapse
		\right\|_{L_{\newg}^2(\Sigma_t)}
$
on RHS~\eqref{E:LAPSELOWFIRSTSOBOLEVGRONWALLREADY}
into the terms on LHS~\eqref{E:LAPSELOWFIRSTSOBOLEVGRONWALLREADY}.
Next, we use the already proven estimate \eqref{E:L2HIGHORDERLAPSEINTERMSOFENERGIES}
and Cor.~\ref{C:SCALEFACTORTIMEINTEGRALS}
to bound the last integral on RHS~\eqref{E:LAPSELOWFIRSTSOBOLEVGRONWALLREADY} as follows:
\begin{align} \label{E:LAPSEESIMATEUSINGALREADYPROVENHIGHORDERBOUNDS}
\int_{s=0}^t
		\scale^{1/3 - c \sqrt{\varepsilon}}(s)
		\left\|
			\mathscr{Z}^{[1,M+2]} \newlapse
		\right\|_{L_{\newg}^2(\Sigma_s)}
		\, ds
& \lesssim 
\int_{s=0}^t
		\left\lbrace
			\varepsilon^2
		 	\scale^{1/3-c\sqrt{\varepsilon}}(s)
		 	+
			\scale^{-1 - c \sqrt{\varepsilon}}(s)
			\SupTotalenergy{M+2}{\smallparameter_*}^{1/2}(s)
		\right\rbrace
\, ds
	\\
& \lesssim
		\varepsilon^2 
		+
		\frac{1}{\sqrt{\varepsilon}} \scale^{-c\sqrt{\varepsilon}}(t) \SupTotalenergy{M+2}{\smallparameter_*}^{1/2}(t).
	\notag
\end{align}
Substituting RHS~\eqref{E:LAPSEESIMATEUSINGALREADYPROVENHIGHORDERBOUNDS}
for the last term on RHS~\eqref{E:LAPSELOWFIRSTSOBOLEVGRONWALLREADY},
we arrive at the desired estimate \eqref{E:L2LOWESTORDERLAPSEINTERMSOFENERGIES}.

	Next, we note that
	\eqref{E:COMMUTEDMETRICL2BOUNDSINTERMSOFENERGY} 
	follows from
	\eqref{E:COMMUTEDMETRICL2BOUNDSINTERMSOFLAPSEANDENERGY},
	\eqref{E:L2HIGHORDERLAPSEINTERMSOFENERGIES},
	\eqref{E:SCALEFACTORTIMEINTEGRALS}
	with $p = - 1$ and $p = - 1 - c \sqrt{\varepsilon}$,
	and \eqref{E:SCALEFACTORLOGPOWERBOUND}.
	Similarly,
	\eqref{E:BELOWTOPORDERDERIVATIVESOFSPATIALSACALARFIELD} 
	follows from
	\eqref{E:COMMUTEDSFSPATIALINTERMSOFLAPSEANDENERGY},
	\eqref{E:L2HIGHORDERLAPSEINTERMSOFENERGIES},
	\eqref{E:SCALEFACTORTIMEINTEGRALS}
	with $p = - 1$ and $p = - 1 - c \sqrt{\varepsilon}$,
	and \eqref{E:SCALEFACTORLOGPOWERBOUND}.
	
	We have therefore proved Prop.~\ref{P:BOUNDFORLAPSEANDBELOWTOPMETRICINTERMSOFENERGIES}.
	$\hfill \qed$

\subsection{Preliminary $L^2$ estimates for the non-differentiated solution variables}
\label{SS:L2ESTIMATESFORNONDIFFERNTIATED}
In this subsection, we derive preliminary $L^2$ estimates for the non-differentiated time-rescaled
solution variables, showing that they can be controlled by the energies. For convenience,
we allow these estimates to lose derivatives, which is permissible at the lowest order.

\begin{lemma}[\textbf{Preliminary $L^2$ estimates for the non-differentiated variables}]
\label{L:L2ESTIMATESFORNONDIFFERNTIATED}
	The following estimates hold:
	\begin{subequations}
		\begin{align}
		\left\|
			\newg - \StMet
		\right\|_{L_{\newg}^2(\Sigma_t)},
			\,
		\left\|
			\newg^{-1} - \StMet^{-1}
		\right\|_{L_{\newg}^2(\Sigma_t)}
		& \lesssim
			\varepsilon^2 \scale^{- c \sqrt{\varepsilon}}(t)
			+	
			\frac{1}{\sqrt{\varepsilon}}
			\scale^{-c \sqrt{\varepsilon}}(t)
			\SupTotalenergy{2}{\smallparameter_*}^{1/2}(t),
				\label{E:NONDIFFERENTIATEDMETRICL2} \\
		\left\|
			\FreeNewSec
		\right\|_{L_{\newg}^2(\Sigma_t)}
		& \lesssim
			\varepsilon^2 \scale^{- c \sqrt{\varepsilon}}(t)
			+
			\frac{1}{\sqrt{\varepsilon}}
			\scale^{- c \sqrt{\varepsilon}}(t) \SupTotalenergy{2}{\smallparameter_*}^{1/2}(t),
				\label{E:NONDIFFERENTIATEDSECONDFUNDL2}
					\\
		\left\|
			\newlapse
		\right\|_{L_{\newg}^2(\Sigma_t)}
		& \lesssim
			\varepsilon^2 \scale^{-c \sqrt{\varepsilon}}(t)
			+	
			\frac{1}{\sqrt{\varepsilon}}
			\scale^{-c \sqrt{\varepsilon}}(t)
			\SupTotalenergy{2}{\smallparameter_*}^{1/2}(t),
				\label{E:NONDIFFERENTIATEDLAPSEL2}
				\\
		\left\|
			\newtimescalar
		\right\|_{L_{\newg}^2(\Sigma_t)}
		& \lesssim
		\varepsilon^2 \scale^{- c \sqrt{\varepsilon}}(t)
		+
		\frac{1}{\sqrt{\varepsilon}}
		\scale^{-c \sqrt{\varepsilon}}(t)
		\SupTotalenergy{2}{\smallparameter_*}^{1/2}(t),
			\label{E:NONDIFFERENTIATEDTIMESCALARFIELDL2} \\
		\left\|
			\newspacescalar
		\right\|_{L_{\newg}^2(\Sigma_t)}
		& \lesssim
		\varepsilon^2 \scale^{-c \sqrt{\varepsilon}}(t)
		+
		\frac{1}{\sqrt{\varepsilon}}
		\scale^{-c \sqrt{\varepsilon}}(t)
		\SupTotalenergy{1}{\smallparameter_*}^{1/2}(t).
		\label{E:NONDIFFERENTIATEDSPACESCALARFIELDL2}
		\end{align}
	\end{subequations}
\end{lemma}

\begin{proof}
	The proof is a combination of elliptic estimates for the lapse and Gronwall estimates, carried out
	in an appropriate order. We start with the elliptic estimates.
	From \eqref{E:LOWESTORDERELLIPTICLAPSEINHOMOGENEOUSTERMPOINTWISE}
	and \eqref{E:COMMUTEDMETRICL2BOUNDSINTERMSOFENERGY},
	we deduce
	$
	\left\|
		\widetilde{\mathscr{L}} \newlapse
	\right\|_{L_{\newg}^2(\Sigma_t)}
	\leq	
		C
		\varepsilon^2
		\scale^{- c \sqrt{\varepsilon}}(t)
		+
		C
		\scale^{- c \sqrt{\varepsilon}}(t)
		\left\|
			\newg - \StMet
		\right\|_{L_{\newg}^2(\Sigma_t)}
		+
		C
		\scale^{- c \sqrt{\varepsilon}}(t)
		\left\|
			\newg^{-1} - \StMet^{-1}
		\right\|_{L_{\newg}^2(\Sigma_t)} 
		+
			C \sqrt{\varepsilon}
			\scale^{- c \sqrt{\varepsilon}}(t)
			\left\|
				\newspacescalar
			\right\|_{L_{\newg}^2(\Sigma_t)}
		+
		\frac{C}{\sqrt{\varepsilon}}
		\scale^{- c \sqrt{\varepsilon}}(t)
		\SupTotalenergy{2}{\smallparameter_*}^{1/2}(t)
	$.
	From this bound and \eqref{E:LOWBASICELLIPTICESTIMATE},
	we obtain
	\begin{align} \label{E:LOWESTORDERLAPSEL2PRELIM}
	\left\|
		\newlapse
	\right\|_{L_{\newg}^2(\Sigma_t)}
	& \leq
		C \varepsilon^2 \scale^{- c \sqrt{\varepsilon}}(t) 
		+ 
		C
		\scale^{- c \sqrt{\varepsilon}}(t)
		\left\|
			\newg - \StMet
		\right\|_{L_{\newg}^2(\Sigma_t)}
		+
		C
		\scale^{- c \sqrt{\varepsilon}}(t)
		\left\|
			\newg^{-1} - \StMet^{-1}
		\right\|_{L_{\newg}^2(\Sigma_t)}
			\\
	& \ \ 
		+
		C \sqrt{\varepsilon} \scale^{- c \sqrt{\varepsilon}}(t)
		\left\|
			\newspacescalar
		\right\|_{L_{\newg}^2(\Sigma_t)}
		+
		\frac{C}{\sqrt{\varepsilon}}
		\scale^{- c \sqrt{\varepsilon}}(t)
		\SupTotalenergy{2}{\smallparameter_*}^{1/2}(t).
		\notag
	\end{align}
	
	Next, we use
	Lemma~\ref{L:POINTWISEESTIMATESBASELEVEL},
	\eqref{E:BASICGRONWALLESTIMATEFORATENSORFIELD},
	Prop.~\ref{P:BOUNDFORLAPSEANDBELOWTOPMETRICINTERMSOFENERGIES},
	the small-data bound \eqref{E:SMALLDATA},
	and Lemma~\ref{L:L2NORMSCOMPARISONESTIMATES}
	to deduce
	\begin{align} \label{E:SPACESFL2LOWESTORDERFIRSTEST}
	\left\|
		\newspacescalar
	\right\|_{L_{\newg}^2(\Sigma_t)}
	& \leq
			C \varepsilon^2
			\scale^{- c \sqrt{\varepsilon}}(t)
			+
			C
			\scale^{- c \sqrt{\varepsilon}}(t)
			\int_{s=0}^t
				\scale^{-1 - c \sqrt{\varepsilon}} \SupTotalenergy{1}{\smallparameter_*}^{1/2}(s)
			\, ds.
\end{align}
From \eqref{E:SPACESFL2LOWESTORDERFIRSTEST}
and \eqref{E:SCALEFACTORTIMEINTEGRALS} with 
$p = - 1 - c \sqrt{\varepsilon}$,
we conclude the desired bound
\eqref{E:NONDIFFERENTIATEDSPACESCALARFIELDL2}.
		
	Next, we use
	Lemma~\ref{L:POINTWISEESTIMATESBASELEVEL},
	\eqref{E:BASICGRONWALLESTIMATEFORATENSORFIELD},
	Prop.~\ref{P:BOUNDFORLAPSEANDBELOWTOPMETRICINTERMSOFENERGIES},
	the small-data bound \eqref{E:SMALLDATA},
	Lemma~\ref{L:L2NORMSCOMPARISONESTIMATES}
	and the already proven bound \eqref{E:NONDIFFERENTIATEDSPACESCALARFIELDL2}
	to deduce
	\begin{align}
	\left\|
		\newg - \StMet
	\right\|_{L_{\newg}^2(\Sigma_t)},
		\,
	\left\|
		\newg^{-1} - \StMet^{-1}
	\right\|_{L_{\newg}^2(\Sigma_t)}
	& \leq
			C \varepsilon^2
			\scale^{- c \sqrt{\varepsilon}}(t)
			+
			C
			\scale^{- c \sqrt{\varepsilon}}(t)
			\int_{s=0}^t
			\scale^{-1}(s)
				\left\|
					\FreeNewSec
				\right\|_{L_{\newg}^2(\Sigma_s)}
			\, ds
				\label{E:METRICBASELEVELFIRSTL2BOUND} \\
		& \ \
			+
			C
			\scale^{- c \sqrt{\varepsilon}}(t)
			\int_{s=0}^t
			\scale^{1/3}(s)
				\left\|
					\newlapse
				\right\|_{L_{\newg}^2(\Sigma_s)}
			\, ds,
		   \notag \\
	\left\|
		\FreeNewSec
	\right\|_{L_{\newg}^2(\Sigma_t)}
	& \leq
			C \varepsilon^2
			\scale^{- c \sqrt{\varepsilon}}(t)
			+
			C
			\scale^{- c \sqrt{\varepsilon}}(t)
			\int_{s=0}^t
				\scale^{1/3 - c \sqrt{\varepsilon}}(s)
				\left\|
					\newg - \StMet
				\right\|_{L_{\newg}^2(\Sigma_s)}
			\, ds
				\label{E:SECONDFUNDBASELEVELFIRSTL2BOUND} \\
	& \ \
			+
			C
			\scale^{- c \sqrt{\varepsilon}}(t)
			\int_{s=0}^t
				\scale^{1/3 - c \sqrt{\varepsilon}}(s)
				\left\|
					\newg^{-1} - \StMet^{-1}
				\right\|_{L_{\newg}^2(\Sigma_s)}
			\, ds
				\notag \\
			& \ \
			+
			C
			\scale^{- c \sqrt{\varepsilon}}(t)
			\int_{s=0}^t
			\scale^{1/3}(s)
				\left\|
					\newlapse
				\right\|_{L_{\newg}^2(\Sigma_s)}
			\, ds
				\notag \\
		& \ \
			+
			C
			\scale^{- c \sqrt{\varepsilon}}(t)
			\int_{s=0}^t
				\scale^{1/3 - c \sqrt{\varepsilon}}(s)
				\SupTotalenergy{1}{\smallparameter_*}^{1/2}(s)
			\, ds,
			\notag
				\\
	\left\|
		\newtimescalar
	\right\|_{L_{\newg}^2(\Sigma_t)}
	& \leq
			C \varepsilon^2 \scale^{- c \sqrt{\varepsilon}}(t)
			+
			C
			\scale^{- c \sqrt{\varepsilon}}(t)
			\int_{s=0}^t
			\scale^{1/3}(s)
				\left\|
					\newlapse
				\right\|_{L_{\newg}^2(\Sigma_s)}
			\, ds
			\label{E:TIMESFL2LOWESTORDERFIRSTEST}		\\
	& \ \
	   +
			\frac{C}{\sqrt{\varepsilon}}
			\scale^{- c \sqrt{\varepsilon}}(t)
			\int_{s=0}^t
				\scale^{1/3 - c \sqrt{\varepsilon}}(s)
				\SupTotalenergy{2}{\smallparameter_*}^{1/2}(s)
			\, ds.
				\notag	
\end{align}

We now set
\begin{align} \label{E:BASELEVELNONMETRICQUANTITIES}
Q(t) 
:= 
\sup_{s \in [0,t]}
\left\lbrace
	\left\|
		\FreeNewSec
	\right\|_{L_{\newg}^2(\Sigma_s)}
	+
	\left\|
		\newtimescalar
	\right\|_{L_{\newg}^2(\Sigma_s)}
\right\rbrace.
\end{align}

From 
\eqref{E:LOWESTORDERLAPSEL2PRELIM},
\eqref{E:METRICBASELEVELFIRSTL2BOUND},
the already proven bound \eqref{E:NONDIFFERENTIATEDSPACESCALARFIELDL2},
definition \eqref{E:BASELEVELNONMETRICQUANTITIES}, 
Cor.~\ref{C:SCALEFACTORTIMEINTEGRALS},
we obtain
\begin{align} \label{E:METRICBASELEVELL2SECONDBOUND}
	&
	\left\|
		\newg - \StMet
	\right\|_{L_{\newg}^2(\Sigma_t)}
	+
	\left\|
		\newg^{-1} - \StMet^{-1}
	\right\|_{L_{\newg}^2(\Sigma_t)}
		\\
	& \leq
			C \varepsilon^2 \scale^{- c \sqrt{\varepsilon}}(t)
			+ 
			C
			\scale^{- c \sqrt{\varepsilon}}(t)
			\int_{s=0}^t
			\scale^{-1}(s)
				Q(s)
			\, ds
			\notag \\
	& \ \
			+
			\frac{C}{\sqrt{\varepsilon}}
			\scale^{- c \sqrt{\varepsilon}}(t)
			\int_{s=0}^t
				\scale^{1/3 - c \sqrt{\varepsilon}}(s)
				\SupTotalenergy{2}{\smallparameter_*}^{1/2}(s)
			\, ds
			\notag	\\
	& \ \
			+
			C
			\scale^{- c \sqrt{\varepsilon}}(t)
			\int_{s=0}^t
				\scale^{1/3 - c \sqrt{\varepsilon}}(s)
				\left\lbrace
					\left\|
						\newg - \StMet
					\right\|_{L_{\newg}^2(\Sigma_s)}
					+
					\left\|
						\newg^{-1} - \StMet^{-1}
					\right\|_{L_{\newg}^2(\Sigma_s)}
				\right\rbrace
			\, ds
			\notag	\\
		& \leq
			C \varepsilon^2 \scale^{- c \sqrt{\varepsilon}}(t)
			+ 
			C (1 + |\ln \scale(t)|)\scale^{- c \sqrt{\varepsilon}}(t) Q(t)
			+
			\frac{C}{\sqrt{\varepsilon}} \scale^{- c \sqrt{\varepsilon}}(t) \SupTotalenergy{2}{\smallparameter_*}^{1/2}(t)
				\notag \\
		& \ \
			+
			C
			\scale^{- c \sqrt{\varepsilon}}(t)
			\int_{s=0}^t
				\scale^{1/3 - c \sqrt{\varepsilon}}(s)
				\left\lbrace
					\left\|
						\newg - \StMet
					\right\|_{L_{\newg}^2(\Sigma_s)}
					+
					\left\|
						\newg^{-1} - \StMet^{-1}
					\right\|_{L_{\newg}^2(\Sigma_s)}
				\right\rbrace
			\, ds.
			\notag
	\end{align}
	From \eqref{E:METRICBASELEVELL2SECONDBOUND},
	Gronwall's inequality in the quantity
	$
	\scale^{c \sqrt{\varepsilon}}(t)
	\left\lbrace
		\left\|
		\newg - \StMet
	\right\|_{L_{\newg}^2(\Sigma_t)}
	+
	\left\|
		\newg^{-1} - \StMet^{-1}
	\right\|_{L_{\newg}^2(\Sigma_t)}
	\right\rbrace
	$,
	Lemma~\ref{L:ANALYSISOFFRIEDMANN},
	and Cor.~\ref{C:SCALEFACTORTIMEINTEGRALS}, we deduce
	\begin{align} \label{E:L2METRICLOWESTORDERGRONWALLED}
	\left\|
		\newg - \StMet
	\right\|_{L_{\newg}^2(\Sigma_t)}
	+
	\left\|
		\newg^{-1} - \StMet^{-1}
	\right\|_{L_{\newg}^2(\Sigma_t)}
	& \leq
			C \varepsilon^2 \scale^{- c \sqrt{\varepsilon}}(t)
			+ 
			C (1 + |\ln \scale(t)|)\scale^{- c \sqrt{\varepsilon}}(t) Q(t)
				\\
	& \ \
			+
			\frac{C}{\sqrt{\varepsilon}} \scale^{- c \sqrt{\varepsilon}}(t) \SupTotalenergy{2}{\smallparameter_*}^{1/2}(t).
			\notag
	\end{align}
	
	Next, from \eqref{E:SECONDFUNDBASELEVELFIRSTL2BOUND}, 
	\eqref{E:TIMESFL2LOWESTORDERFIRSTEST},
	definition \eqref{E:BASELEVELNONMETRICQUANTITIES},
	\eqref{E:LOWESTORDERLAPSEL2PRELIM},
	\eqref{E:L2METRICLOWESTORDERGRONWALLED},
	the already proven bound \eqref{E:NONDIFFERENTIATEDSPACESCALARFIELDL2},
	\eqref{E:SCALEFACTORLOGPOWERBOUND},
	and Cor.~\ref{C:SCALEFACTORTIMEINTEGRALS},
	we deduce
	\begin{align} \label{E:BASEQUANTITYGRONWALLREADY}
		Q(t)
		& \leq C \varepsilon^2 \scale^{- c \sqrt{\varepsilon}}(t)
			+
			\frac{1}{\sqrt{\varepsilon}}
			\scale^{- c \sqrt{\varepsilon}}(t)
			\SupTotalenergy{2}{\smallparameter_*}^{1/2}(t)
			+
			C \scale^{- c \sqrt{\varepsilon}}(t)
			\int_{s=0}^t
				\scale^{1/3 - c \sqrt{\varepsilon}}(s) (1 + |\ln \scale(s)|) Q(s)
			\, ds.
	\end{align}
	From Lemma~\ref{L:ANALYSISOFFRIEDMANN},
	we see that the following bound holds
	for the factor in the integral on RHS~\eqref{E:BASEQUANTITYGRONWALLREADY}
	for $s \in [0,\TCrunch)$:
	$\scale^{1/3 - c \sqrt{\varepsilon}}(s) (1 + |\ln \scale(s)|) \leq C$.
	Hence, from
	\eqref{E:BASEQUANTITYGRONWALLREADY}
	and Gronwall's inequality in the quantity $\scale^{c \sqrt{\varepsilon}}(t) Q(t)$, 
	we deduce
	\begin{align}  \label{E:BASEFACTORGRONWALLED}
		Q(t)
		& \leq C \varepsilon^2 \scale^{- c \sqrt{\varepsilon}}(t)
			+
			C
			\frac{1}{\sqrt{\varepsilon}}
			\scale^{- c \sqrt{\varepsilon}}(t)
			\SupTotalenergy{2}{\smallparameter_*}^{1/2}(t).
	\end{align}
	
	Next, from 
	\eqref{E:L2METRICLOWESTORDERGRONWALLED},
	\eqref{E:BASEFACTORGRONWALLED},
	and \eqref{E:SCALEFACTORLOGPOWERBOUND},
	we obtain
	\begin{align} \label{E:L2METRICLOWESTORDERFINALEST}
	\left\|
		\newg - \StMet
	\right\|_{L_{\newg}^2(\Sigma_t)}
	+
	\left\|
		\newg^{-1} - \StMet^{-1}
	\right\|_{L_{\newg}^2(\Sigma_t)}
	& \leq
			C \varepsilon^2 \scale^{- c \sqrt{\varepsilon}}(t)
			+
			C 
			\frac{1}{\sqrt{\varepsilon}}
			\scale^{- c \sqrt{\varepsilon}}(t)
			\SupTotalenergy{2}{\smallparameter_*}^{1/2}(t).
	\end{align}
	Next, from
	\eqref{E:LOWESTORDERLAPSEL2PRELIM},
	\eqref{E:BASEFACTORGRONWALLED},
	\eqref{E:L2METRICLOWESTORDERFINALEST},
	and the already proven bound \eqref{E:NONDIFFERENTIATEDSPACESCALARFIELDL2},
	we conclude
	\begin{align} \label{E:LOWESTORDERLAPSEL2FINALEST}
	\left\|
		\newlapse
	\right\|_{L_{\newg}^2(\Sigma_t)}
	& \lesssim
		\varepsilon^2 \scale^{- c \sqrt{\varepsilon}}(t)
		+
		\frac{1}{\sqrt{\varepsilon}}
		\scale^{- c \sqrt{\varepsilon}}(t)
		\SupTotalenergy{2}{\smallparameter_*}^{1/2}(t).
	\end{align}
	
	In view of definition \eqref{E:BASELEVELNONMETRICQUANTITIES},
	we see that
	\eqref{E:BASEFACTORGRONWALLED},
	\eqref{E:L2METRICLOWESTORDERFINALEST},
	and
	\eqref{E:LOWESTORDERLAPSEL2FINALEST}
	yield the remaining four desired estimates \eqref{E:NONDIFFERENTIATEDMETRICL2}-\eqref{E:NONDIFFERENTIATEDTIMESCALARFIELDL2},
	which completes the proof of the lemma.
	
\end{proof}

\section{$L^2$ estimates for the error terms in terms of the energies}
\label{S:L2BOUNDSFORTHEERRORTERMS}

In the next proposition, we control all of the error terms
in the $\Lie_{\mathscr{Z}}^{\vec{I}}$-commuted equations
in terms of the energies.

\begin{proposition}[$L^2$ \textbf{Estimates for the error terms in terms of the energies}]
\label{P:L2BOUNDSFORTHEERRORTERMS}
Let $\vec{I}$ be a $\mathscr{Z}$-multi-index with
$1 \leq |\vec{I}| \leq 16$
and let $\SupTotalenergy{M}{\smallparameter_*}(t)$
be the energy defined in \eqref{E:SUPTOTALENERGY},
where $\smallparameter_* > 0$ is the constant
from the statement of Prop.~\ref{P:FUNDAMENTALENERGYINTEGRALINEQUALITY}.
The inhomogeneous terms in the commuted equations
of Sect.\ \ref{S:COMMUTEDEQUATIONS} verify the following $L^2$ estimates:
\begingroup
\allowdisplaybreaks
\begin{subequations}
\begin{align}
	\left\| 
		\CommutedMomBorderInhomUp{\vec{I}}
	\right\|_{L_{\newg}^2(\Sigma_t)}^2,
	\left\| 
		\CommutedMomBorderInhomDown{\vec{I}}
	\right\|_{L_{\newg}^2(\Sigma_t)}^2
	& \lesssim
		\varepsilon^4 \scale^{-4/3 - c \sqrt{\varepsilon}}(t)
			\\
	& \ \
		+
		\varepsilon
		\scale^{-4/3}(t)
		\SupTotalenergy{|\vec{I}|}{\smallparameter_*}(t)
		+
		\scale^{- c \sqrt{\varepsilon}}(t)
		\SupTotalenergy{|\vec{I}|}{\smallparameter_*}(t),
			\notag \\
	\left\| 
		 \CommutedSecFunBorderInhom{\vec{I}} 
	\right\|_{L_{\newg}^2(\Sigma_t)}^2
	& \lesssim
		\varepsilon^4 \scale^{-8/3 - c \sqrt{\varepsilon}}(t)
			\\
	& \ \
		+
		\varepsilon^2
		\scale^{-8/3}(t)
		\SupTotalenergy{|\vec{I}|}{\smallparameter_*}(t)
		+
		\underbrace{
		\varepsilon^2
		\scale^{-8/3 - c \sqrt{\varepsilon}}(t)
		\SupTotalenergy{|\vec{I}|-1}{\smallparameter_*}(t)}_{\mbox{\upshape absent if $|\vec{I}|=1$}},
		\notag	\\
	\left\| 
		\CommutedGradMetBorderInhom{\vec{I}} 
	\right\|_{L_{\newg}^2(\Sigma_t)}^2
	& \lesssim
		\varepsilon^4 \scale^{- c \sqrt{\varepsilon}}(t)
			\\
	& \ \
		+
		\varepsilon
		\scale^{-4/3}(t)
		\SupTotalenergy{|\vec{I}|}{\smallparameter_*}(t)
		+
		\scale^{- c \sqrt{\varepsilon}}(t)
		\SupTotalenergy{|\vec{I}|}{\smallparameter_*}(t)
			\notag \\
	& \ \
		+
		\underbrace{
		\varepsilon
		\scale^{-4/3 - c \sqrt{\varepsilon}}(t)
		\SupTotalenergy{|\vec{I}|-1}{\smallparameter_*}(t)}_{\mbox{\upshape absent if $|\vec{I}|=1$}},
		\notag	\\
	\left\| 
		 \CommutedTimeSfBorderInhom{\vec{I}} 
	\right\|_{L_{\newg}^2(\Sigma_t)}^2
	& \lesssim
		\varepsilon^4 \scale^{-8/3 - c \sqrt{\varepsilon}}(t)
			\\
	& \ \
		+
		\varepsilon^2
		\scale^{-8/3}(t)
		\SupTotalenergy{|\vec{I}|}{\smallparameter_*}(t)
		+
		\underbrace{
		\varepsilon^2
		\scale^{-8/3 - c \sqrt{\varepsilon}}(t)
		\SupTotalenergy{|\vec{I}|-1}{\smallparameter_*}(t)}_{\mbox{\upshape absent if $|\vec{I}|=1$}},
		\notag	\\
	\left\| 
		\CommutedSpaceSfBorderInhom{\vec{I}}
	\right\|_{L_{\newg}^2(\Sigma_t)}^2
	& \lesssim
		\varepsilon^4 \scale^{- c \sqrt{\varepsilon}}(t)
		\\
	& \ \
		+
		\varepsilon
		\scale^{-4}(t)
		\SupTotalenergy{|\vec{I}|}{\smallparameter_*}(t)
		+
		\underbrace{
		\varepsilon
		\scale^{-4-c \sqrt{\varepsilon}}(t)
		\SupTotalenergy{|\vec{I}|-1}{\smallparameter_*}(t)}_{\mbox{\upshape absent if $|\vec{I}|=1$}},
			\notag \\
	\left\| 
		\CommutedLapseHighBorderInhom{\vec{I}}
	\right\|_{L_{\newg}^2(\Sigma_t)}^2
	& \lesssim
		\varepsilon^4 \scale^{- c \sqrt{\varepsilon}}(t)
			\\
	& \ \
		+
		\varepsilon
		\SupTotalenergy{|\vec{I}|}{\smallparameter_*}(t)
		+
		\underbrace{
		\varepsilon
		\scale^{-c \sqrt{\varepsilon}}(t)
		\SupTotalenergy{|\vec{I}|-1}{\smallparameter_*}(t)}_{\mbox{\upshape absent if $|\vec{I}|=1$}},
		\notag	\\
	\left\| 
		\CommutedLapseLowBorderInhom{\vec{I}}
	\right\|_{L_{\newg}^2(\Sigma_t)}^2
	& \lesssim
		\varepsilon^4 \scale^{- c \sqrt{\varepsilon}}(t)
		+
		\frac{1}{\varepsilon}
		\scale^{- c \sqrt{\varepsilon}}(t)
		\SupTotalenergy{|\vec{I}|+2}{\smallparameter_*}(t),
\end{align}
\end{subequations}
\endgroup

\begingroup
\allowdisplaybreaks
\begin{subequations}
\begin{align}
	\left\| 
		 \RicErrorInhom{\vec{I}} 
	\right\|_{L_{\newg}^2(\Sigma_t)}^2
	& \lesssim
		\varepsilon^4 \scale^{- c \sqrt{\varepsilon}}(t)
		+
		\varepsilon
		\scale^{-4/3 - c \sqrt{\varepsilon}}(t)
		\SupTotalenergy{|\vec{I}|}{\smallparameter_*}(t),
			\\
	\left\| 
		 \CommutedSecFunJunkInhom{\vec{I}} 
	\right\|_{L_{\newg}^2(\Sigma_t)}^2
	& \lesssim
		\varepsilon^4 \scale^{4/3 - c \sqrt{\varepsilon}}(t)
		+
		\scale^{-4/3 - c \sqrt{\varepsilon}}(t)
		\SupTotalenergy{|\vec{I}|}{\smallparameter_*}(t),
			\\
	\left\| 
		\CommutedGradMetJunkInhom{\vec{I}} 
	\right\|_{L_{\newg}^2(\Sigma_t)}^2
	& \lesssim
		\varepsilon^4 \scale^{- c \sqrt{\varepsilon}}(t)
		+
		\varepsilon
		\scale^{-8/3 - c \sqrt{\varepsilon}}(t)
		\SupTotalenergy{|\vec{I}|}{\smallparameter_*}(t),
			\\
	\left\| 
		 \CommutedTimeSfJunkInhom{\vec{I}} 
	\right\|_{L_{\newg}^2(\Sigma_t)}^2
	& \lesssim
		\varepsilon^4 \scale^{- c \sqrt{\varepsilon}}(t)
		+
		\varepsilon
		\scale^{-4/3 - c \sqrt{\varepsilon}}(t)
		\SupTotalenergy{|\vec{I}|}{\smallparameter_*}(t),	
		\\
	\left\| 
		\CommutedSpaceSfJunkInhom{\vec{I}}
	\right\|_{L_{\newg}^2(\Sigma_t)}^2
	& \lesssim
		\varepsilon^4 \scale^{- c \sqrt{\varepsilon}}(t)
		+
		\varepsilon
		\scale^{-4/3 - c \sqrt{\varepsilon}}(t)
		\SupTotalenergy{|\vec{I}|}{\smallparameter_*}(t),
		\\
	\left\| 
		\CommutedLapseHighJunkInhom{\vec{I}}
	\right\|_{L_{\newg}^2(\Sigma_t)}^2
	& \lesssim
		\varepsilon^4 \scale^{- c \sqrt{\varepsilon}}(t)
		+
		\scale^{-c \sqrt{\varepsilon}}(t)
		\SupTotalenergy{|\vec{I}|}{\smallparameter_*}(t),
		\\
	\left\| 
		\CommutedLapseLowJunkInhom{\vec{I}}
	\right\|_{L_{\newg}^2(\Sigma_t)}^2
	& \lesssim
		\varepsilon^4 \scale^{- c \sqrt{\varepsilon}}(t)
		+
		\scale^{-4/3 - c \sqrt{\varepsilon}}(t)
		\SupTotalenergy{|\vec{I}|+2}{\smallparameter_*}(t).
\end{align}
\end{subequations}
\endgroup
\end{proposition}

\begin{proof}
	We square the pointwise estimates of Prop.~\ref{P:POINTWISEESTIMATESFORERRORTERMS}
	and integrate the resulting inequalities over $\Sigma_t$ with respect to the volume form 
	$d \tvol$ of Def.\ \ref{D:VOLUMEFORM}.
	Using Lemma~\ref{L:ENERGYCOERCIVENESS}, we can directly bound all integrals
	by the energies $\SupTotalenergy{\cdot}{\smallparameter_*}(t)$
	of Def.\ \ref{D:ENERGIES} except for the integrals that depend on $\newlapse$,
	$
		\left|
			\SigmatLie_{\mathscr{Z}}^{[1,|\vec{I}|]} \newg
		\right|_{\newg}
	$,
	or
	$
		\left|
			\SigmatLie_{\mathscr{Z}}^{[1,|\vec{I}|]} \newg
		\right|_{\newg}
$.
	To bound these remaining integrals in terms of the energies,
	we use Prop.~\ref{P:BOUNDFORLAPSEANDBELOWTOPMETRICINTERMSOFENERGIES}.
\end{proof}

\section{Energy estimates and improvement of the bootstrap assumptions}
\label{S:ENERGYESTIMATES}
In this section, we derive the main estimates of the paper: a priori estimates for the energies.
The main result is Cor.~\ref{C:MAINAPRIORIENERGYESTIMATES}.
We start with the following simple lemma.

\begin{lemma}[\textbf{The energy is initially small}]
	\label{L:INITIALENERGYISSMALL}
	Let $\SupTotalenergy{16}{\smallparameter}$
	be the energy defined in \eqref{E:SUPTOTALENERGY}.
	There exists a constant $C > 0$ such that if
	$0 \leq \smallparameter \leq 1$, then the
	following bound holds:
	\begin{align} \label{E:INITIALENERGYISSMALL}
		\SupTotalenergy{16}{\smallparameter}(0)
		\leq C \varepsilon^4.
	\end{align}
\end{lemma}

\begin{proof}
	The lemma is a 
	straightforward consequence of Lemma~\ref{L:ENERGYCOERCIVENESS},
	Lemma~\ref{L:L2NORMSCOMPARISONESTIMATES} at $t=0$,
	and the assumption \eqref{E:SMALLDATA}.
\end{proof}

We now derive a hierarchy of integral inequalities verified by the energies.

\begin{proposition}[\textbf{Integral inequalities verified by the energies}]
\label{P:ENERGYINTEGRALINEQUALITIES}
	Let $1 \leq M \leq 16$ and let
	$\SupTotalenergy{\smallparameter_*}{M}(t)$
	be the energy defined by
	\eqref{E:SUPTOTALENERGY},
	where $\smallparameter_* > 0$ 
	is the small parameter from Prop.~\ref{P:FUNDAMENTALENERGYINTEGRALINEQUALITY}.
	Then there exist constants $C>0$ and $c>0$ such that
	the following system of inequalities holds on $[0,\Tboot)$:
	\begin{align} \label{E:ENERGYINTEGRALINEQUALITIES}
		\SupTotalenergy{\smallparameter_*}{M}(t)
		& \leq 
			C \varepsilon^4 \scale^{-c \sqrt{\varepsilon}}(t)
			\\
		& \ \
			+ C
				\int_{s=0}^t
					\scale^{-1/3 - c \sqrt{\varepsilon}}(s) \SupTotalenergy{\smallparameter_*}{M}(s)			
			\, ds
			+ c \varepsilon
				\int_{s=0}^t
					\scale^{-1}(s) \SupTotalenergy{\smallparameter_*}{M}(s)			
				\, ds 
				\notag \\
		& \ \ 
				+ 
				\underbrace{
				C \varepsilon
				\int_{s=0}^t
					\scale^{-1-c \sqrt{\varepsilon}}(s) \SupTotalenergy{\smallparameter_*}{M-1}(s)			
				\, ds
				}_{\mbox{absent if $M=1$}}.
				\notag
	\end{align}
\end{proposition}

\begin{proof}
	All of the key estimates have already been proved;
	here, we just assemble them.
	To prove 
	\eqref{E:ENERGYINTEGRALINEQUALITIES}, 
	we first show that the terms
	$\Energyborder{\vec{I}}(s)$
	and 
	$\Energyjunk{\vec{I}}(s)$
	on RHS~\eqref{E:FUNDAMENTALENERGYINTEGRALINEQUALITY}
	are $\leq \mbox{RHS~\eqref{E:ENERGYINTEGRALINEQUALITIES}}$.
	We start by bounding
	the terms $\Energyborder{\vec{I}}(s)$
	defined in \eqref{E:ENERGYBORDER}.
	All time integrals on RHS~\eqref{E:FUNDAMENTALENERGYINTEGRALINEQUALITY} that are generated by the terms
	on RHS~\eqref{E:ENERGYBORDER}
	are easily seen to be $\leq \mbox{RHS~\eqref{E:ENERGYINTEGRALINEQUALITIES}}$
	with the help of Prop.~\ref{P:L2BOUNDSFORTHEERRORTERMS} and Lemma~\ref{L:ENERGYCOERCIVENESS}.
	Similarly to bound
	the terms $\Energyjunk{\vec{I}}(s)$
	defined in \eqref{E:ENERGYJUNK}
	by $\leq \mbox{RHS~\eqref{E:ENERGYINTEGRALINEQUALITIES}}$,
	we use Prop.~\ref{P:L2BOUNDSFORTHEERRORTERMS} and \eqref{E:L2HIGHORDERLAPSEINTERMSOFENERGIES}.
	We now take the sup over time intervals
	of both sides of inequality \eqref{E:FUNDAMENTALENERGYINTEGRALINEQUALITY}
	and use Lemma~\ref{L:INITIALENERGYISSMALL} and the estimates noted above,
	thereby arriving (in view of definition \eqref{E:SUPTOTALENERGY})
	at the desired inequality \eqref{E:ENERGYINTEGRALINEQUALITIES}.
\end{proof}

With the help of Prop.~\ref{P:ENERGYINTEGRALINEQUALITIES}, we now derive 
the desired a priori energy estimates.

\begin{corollary}[\textbf{The main a priori energy estimates}]
	\label{C:MAINAPRIORIENERGYESTIMATES}
	Let $\SupTotalenergy{\smallparameter_*}{16}$ be the total
	energy defined by \eqref{E:SUPTOTALENERGY}
	and let $\highnorm{16}$ be the solution norm defined in
	\eqref{E:HIGHNORM}. Assume that
	$\highnorm{16}(0) := \varepsilon^2$,
	as in \eqref{E:SMALLDATA}.
	There exist constants $C > 0$ and $c > 0$ such that if
	$\varepsilon$ is sufficiently small, 
	then the following estimate holds for $t \in [0,\Tboot)$:
	\begin{align} \label{E:MAINAPRIORIENERGYESTIMATES}
		\SupTotalenergy{\smallparameter_*}{16}^{1/2}(t)
		& \leq 
			C \varepsilon^2 \scale^{-c \sqrt{\varepsilon}}(t).
	\end{align}
	
	Furthermore, the following estimate holds for $t \in [0,\Tboot)$:
	\begin{align} \label{E:MAINHIGHNORMENERGYESTIMATES}
		\highnorm{16}(t)
		& \leq C \varepsilon^{3/2} \scale^{-c \sqrt{\varepsilon}}(t).
	\end{align}
	In particular, if $\varepsilon$ is sufficiently small, then
	the estimate \eqref{E:MAINHIGHNORMENERGYESTIMATES}
	is a strict improvement of
	the bootstrap assumption \eqref{E:HIGHNORMBOOTSTRAP}.
\end{corollary}

\begin{proof}
	First, 
	for $1 \leq M \leq 16$,
	we will inductively derive the following bound:
	\begin{align} \label{E:PROOFMAINAPRIORIENERGYESTIMATES}
		\SupTotalenergy{\smallparameter_*}{M}(t)
		& \leq 
			C \varepsilon^4 \scale^{-c \sqrt{\varepsilon}}(t).
	\end{align}
	To prove \eqref{E:PROOFMAINAPRIORIENERGYESTIMATES} 
	in the base case $M=1$, we apply
	Gronwall's inequality to \eqref{E:ENERGYINTEGRALINEQUALITIES}
	and use \eqref{E:EXPONENTIATEDSCALEFACTORTIMEINTEGRALS}
	with $p=-1/3 - c \sqrt{\varepsilon}$ and $p=-1$,
	which yields the desired bound
	$
	\SupTotalenergy{\smallparameter_*}{1}(t)
	\leq 
	C \varepsilon^4 \scale^{-c \sqrt{\varepsilon}}(t)
	$.
	We now make the following induction hypothesis: the estimates have been proved in the case $M-1$.
	To derive the inequality in the case $M$, 
	we first insert the estimates obtained in the case $M-1$ 
	into the last integral on RHS~\eqref{E:ENERGYINTEGRALINEQUALITIES}.
	Also using \eqref{E:SCALEFACTORTIMEINTEGRALS} with 
	$p = - 1 - c \sqrt{\varepsilon}$, we bound the integral as follows:
	\begin{align} \label{E:INDUCTIVELYBOUNDEDINTEGRAL}
	C \varepsilon
		\int_{s=0}^t
			\scale^{-1-c \sqrt{\varepsilon}}(s) \SupTotalenergy{\smallparameter_*}{M-1}(s)			
		\, ds
	& \leq
	C \varepsilon^5
	\int_{s=0}^t
		\scale^{-1-c \sqrt{\varepsilon}}(s)		
	\, ds
		\\
	& 
	\leq
	C \varepsilon^4 \scale^{-c \sqrt{\varepsilon}}(t).
	\notag
\end{align}
With the estimate \eqref{E:INDUCTIVELYBOUNDEDINTEGRAL} in hand, we can then
apply Gronwall's inequality to \eqref{E:ENERGYINTEGRALINEQUALITIES},
again using \eqref{E:EXPONENTIATEDSCALEFACTORTIMEINTEGRALS}
with $p=-1/3$ and $p = - 1 - c \sqrt{\varepsilon}$,
which yields the desired bound
	$
	\SupTotalenergy{\smallparameter_*}{M}(t)
	\leq 
	C \varepsilon^4 \scale^{-c \sqrt{\varepsilon}}(t)
	$.
	We have therefore closed the induction and 
	shown that \eqref{E:PROOFMAINAPRIORIENERGYESTIMATES}
	holds for $M=1,\cdots,16$.
	In particular, we have proved \eqref{E:MAINAPRIORIENERGYESTIMATES}.
	
	We now prove \eqref{E:MAINHIGHNORMENERGYESTIMATES}.
	We first show that the terms on the first two lines of
	RHS~\eqref{E:HIGHNORM} are $\leq \mbox{RHS~\eqref{E:MAINHIGHNORMENERGYESTIMATES}}$.
	By Lemma~\ref{L:L2NORMSCOMPARISONESTIMATES}, it suffices to prove the same bound
	with the norms $\| \cdot \|_{L_{\StMet}(\Sigma_t)}$ on RHS~\eqref{E:HIGHNORM}
	replaced by the norms 	
	$\| \cdot \|_{L_{\newg}(\Sigma_t)}$.
	To this end, we first use Lemma~\ref{L:ENERGYCOERCIVENESS},
	\eqref{E:IMPROVEDOPERATORCOMPARISON},
	\eqref{E:PROOFMAINAPRIORIENERGYESTIMATES} in the case $M=16$,
	and the fact that $\SigmatLie_Z \StMet = 0$ for $Z \in \mathscr{Z}$
	to deduce
	\begin{align} \label{E:NORMBOUNDSTHATFOLLOWFROMENERGY}
			&
			\left\|
				\SigmatLie_{\mathscr{Z}}^{[1,16]} \FreeNewSec
			\right\|_{L_{\newg}^2(\Sigma_t)},
				\,
			\scale^{2/3}(t)
			\left\|
				\SigmatLie_{\mathscr{Z}}^{[1,17]} (\newg - \StMet)
			\right\|_{L_{\newg}^2(\Sigma_t)},
			\,
			\left\|
				\mathscr{Z}^{[1,17]} \newtimescalar
			\right\|_{L_{\newg}^2(\Sigma_t)},
				\,
			\scale^{2/3}(t)
			\left\|
				\SigmatLie_{\mathscr{Z}}^{[1,16]} \newspacescalar
			\right\|_{L_{\newg}^2(\Sigma_t)}
				\\
			& 
			\lesssim \varepsilon^2 \scale^{-c \sqrt{\varepsilon}}(t).
			\notag
	\end{align}
	Next, we use Prop.~\ref{P:BOUNDFORLAPSEANDBELOWTOPMETRICINTERMSOFENERGIES},
	\eqref{E:IMPROVEDOPERATORCOMPARISON},
	\eqref{E:PROOFMAINAPRIORIENERGYESTIMATES} in the case $M=16$,
	and the fact that $\SigmatLie_Z \StMet = \SigmatLie_Z \StMet^{-1} = 0$ for $Z \in \mathscr{Z}$
	to deduce
	\begin{align}
		&
		\left\|
			\mathscr{Z}^{[1,14]} \newlapse
		\right\|_{L_{\newg}^2(\Sigma_t)},
			\,
		\sum_{L=1}^4
		\scale^{(2/3)L}
		\left\|
			\newlapse
		\right\|_{\dot{H}_{\newg}^{14+L}(\Sigma_t)},
			\\
		&
		\left\|
			\SigmatLie_{\mathscr{Z}}^{[1,16]}(\newg - \StMet)
		\right\|_{L_{\newg}^2(\Sigma_t)},
		\,
		\left\|
			\SigmatLie_{\mathscr{Z}}^{[1,16]}(\newg^{-1} - \StMet^{-1})
		\right\|_{L_{\newg}^2(\Sigma_t)},
		\,
		\left\|
			\SigmatLie_{\mathscr{Z}}^{[1,15]} \newspacescalar
		\right\|_{L_{\newg}^2(\Sigma_t)}	
		\notag \\
	& \lesssim 
		\varepsilon^{3/2}
		\scale^{-c \sqrt{\varepsilon}}(t).
		\notag
	\end{align}
	Next, we use Lemma~\ref{L:L2ESTIMATESFORNONDIFFERNTIATED}
	and \eqref{E:PROOFMAINAPRIORIENERGYESTIMATES} in the case $M=2$
	to deduce
	\begin{align}
		\left\|
			\newg - \StMet
		\right\|_{L_{\newg}^2(\Sigma_t)},
			\,
		\left\|
			\newg^{-1} - \StMet^{-1}
		\right\|_{L_{\newg}^2(\Sigma_t)}
			\,
		\left\|
			\FreeNewSec
		\right\|_{L_{\newg}^2(\Sigma_t)},
			\,
		\left\|
			\newlapse
		\right\|_{L_{\newg}^2(\Sigma_t)},
			\,
		\left\|
			\newtimescalar
		\right\|_{L_{\newg}^2(\Sigma_t)},
			\,
		\left\|
			\newspacescalar
		\right\|_{L_{\newg}^2(\Sigma_t)}
			 & \lesssim
			\varepsilon^{3/2} \scale^{-c \sqrt{\varepsilon}}(t).
			\label{E:BASEVARIABLESL2ESTIMATEMAINTHM}
		\end{align}
	Combining \eqref{E:NORMBOUNDSTHATFOLLOWFROMENERGY}-\eqref{E:BASEVARIABLESL2ESTIMATEMAINTHM},
	we conclude that the terms on the first two lines of
	RHS~\eqref{E:HIGHNORM} are $\leq \mbox{RHS~\eqref{E:MAINHIGHNORMENERGYESTIMATES}}$,
	as desired.
	To bound the sup-norm terms
	on the last two lines of RHS~\eqref{E:HIGHNORM},
	we use the Sobolev
	embedding result \eqref{E:STSOBOLEV} 
	and the already obtained bounds for the terms on the first two lines
	of RHS~\eqref{E:HIGHNORM}. We have thus proved the desired bound 
	\eqref{E:MAINHIGHNORMENERGYESTIMATES}, which completes the proof of the corollary.

\end{proof}

\section{The main theorem}
\label{S:MAINTHM}
We now state and prove our main stable blowup result. 
We have already derived all of the difficult estimates.

\begin{theorem}[\textbf{Stable curvature blowup for solutions with near-FLRW data}]
	\label{T:MAINTHM}
	Consider geometric initial data (as described in Subsect.\ \ref{SS:IVP})
	for the Einstein-scalar field system \eqref{E:EINSTEINSF}-\eqref{E:WAVEMODEL}
	that verify the CMC condition \eqref{E:INITIALCMC}
	(see, however, Remark~\ref{R:NONEEDFORCMC}).
	Assume that the geometric data 
	induce data for the rescaled variables of Def.\ \ref{D:RESCALEDVARIABLES}
	such that\footnote{See Remark~\ref{R:NUMBEROFDERIVATIVES} regarding the number of derivatives that we use to close the proof.} 
	$\highnorm{16}(0) = \varepsilon^2$, 
	where the norm $\highnorm{16}(t)$ is defined in \eqref{E:HIGHNORM}.
	Note that when $\varepsilon = 0$, the solution
	is exactly the FLRW solution from Subsect.\ \ref{SS:FLRWANDSCALEFACTOR},
	which by Lemma~\ref{L:ANALYSISOFFRIEDMANN}
	exists on the time interval $(\TBang,\TCrunch)$
	and which exhibits curvature blowup at times $\TBang$ and $\TCrunch$
	(recall that $- \TBang = \TCrunch > 0$).
	Then if $\varepsilon > 0$ is sufficiently small,
	the corresponding perturbed solution 
	to the equations of Prop.~\ref{P:EINSTEININCMC}
	(that is, to the Einstein-scalar field equations 
	in CMC-transported spatial coordinates gauge)
	also exists on
	$(\TBang,\TCrunch) \times \mathbb{S}^3$,
	and the time-rescaled variables verify the norm estimate\footnote{In stating estimates, we have aimed for a clean presentation rather than
	for optimizing powers of $\varepsilon$. For this reason, some of the estimates stated in the theorem are non-optimal
	with respect to powers of $\varepsilon$.\label{FN:NONOPTIMALPOWERSOFEPSILON}} 
	\begin{align} \label{E:MAINTHMNORMEST}
		\highnorm{16}(t) \leq C \varepsilon^{3/2} \scale^{-c \sqrt{\varepsilon}}(t).
	\end{align}
	Moreover, the solution verifies the curvature estimate
	\eqref{E:SPACETIMERICCIINVARIANTBLOWUP},
	which, when combined with \eqref{E:PLIMITCLOSETOFLRW}, shows that
	$\Ricfour^{\alpha \beta} \Ricfour_{\alpha \beta}$
	blows up as $t \uparrow \TCrunch$, since $\scale(\TCrunch)=0$.
	A similar blowup result holds as $t \downarrow \TBang$.
	Thus, with $\gfour, \partial_t \phi, \nabla \phi$ denoting the non-time-rescaled solution
	variables (see Remark~\ref{R:WEDONOTESTIMATEPHI}), we have that
	$\left((-\TCrunch,\TCrunch) \times \mathbb{S}^3, \gfour, \partial_t \phi, \nabla \phi \right)$
	is the maximal globally hyperbolic development of the data.
	
	\medskip
	
	\noindent{\textbf{Convergence results} (Here we consider only the limit $t \uparrow \TCrunch$; analogous statements hold
	as $t \downarrow \TBang$ and we omit those details):}
	For integers $M \geq 0$, let $C_{\StMet}^M(\mathbb{S}^3)$
	denote the Banach space\footnote{Actually, we are loosely using the notation
	$C_{\StMet}^M(\mathbb{S}^3)$ to denote a family of Banach spaces depending on the order of the tensorfields.} 
	of $M$-times continuously differentiable tensorfields
	on $\mathbb{S}^3$ with square norm
	$
	\| \xi \|_{C_{\StMet}^M(\mathbb{S}^3)}^2
		:= \sum_{|\vec{I}| \leq M} \sup_{p \in \mathbb{S}^3} 
		\left| \SigmatLie_{\mathscr{Z}}^{\vec{I}} 
		\xi(p) \right|_{\StMet}^2
	$.
	Let $d \varpi_g$ denote the volume form of $g$.
	There exist a function $\Psi_{Crunch} \in C^{10}(\mathbb{S}^3)$,
	a type $\binom{1}{1}$ tensorfield $K_{Crunch} \in C_{\StMet}^{10}(\mathbb{S}^3)$,
	and a (type $\binom{0}{3}$) volume form $d \varpi_{Crunch} \in C_{\StMet}^8(\mathbb{S}^3)$,
	which we view as tensorfields on $(\TBang,\TCrunch) \times \mathbb{S}^3$
	that are independent of $t$,
	such that the following convergence estimates hold
	for $t \in [0,\TCrunch)$:
\begin{subequations}
\begin{align}
	\left\| n - 1  \right\|_{C^8(\Sigma_t)} 
	& \lesssim \varepsilon \scale^{4/3 - c \sqrt{\varepsilon}}(t),
		\label{E:LAPSELIMIT} \\
	\left \| \scale^{-1} d \varpi_g - d \varpi_{Crunch} \right \|_{C_{\StMet}^8(\Sigma_t)} 
	& \lesssim \varepsilon \scale^{4/3 - c \sqrt{\varepsilon}}(t),
		\label{E:VOLFORMLIMIT} \\
	\left \| \scale \SecondFund - K_{Crunch} \right \|_{C_{\StMet}^{10}(\Sigma_t)} 
	& \lesssim \varepsilon \scale^{4/3 - c \sqrt{\varepsilon}}(t),
		\label{E:KLIMIT} \\
	\left \| \scale \partial_t \phi - \Psi_{Crunch} \right \|_{C^{10}(\Sigma_t)} 
	& \lesssim \varepsilon \scale^{4/3 - c \sqrt{\varepsilon}}(t),
		\label{E:TIMESFLIMIT}  \\
	\left\| \phi -  \left(\int_{s=0}^t \scale^{-1}(s) \, ds \right) \Psi_{Crunch} \right\|_{\dot{C}^M(\Sigma_t)} 
	& \lesssim \varepsilon,
		\label{E:SPACESFLIMIT}
	&& (1 \leq M \leq 10),
\end{align}
\end{subequations}
where in \eqref{E:KLIMIT}, we are viewing $\SecondFund$ to be a type $\binom{1}{1}$ tensorfield.
Furthermore, the limiting fields are close to the corresponding time-rescaled FLRW fields in the following sense:
\begin{subequations}
\begin{align}
\left \| d \varpi_{Crunch} - d \Sttvol \right \|_{C_{\StMet}^8(\mathbb{S}^3)} & \lesssim \varepsilon, 
	\label{E:VOLFORMLIMITCLOSETOSTMETRICVOLUMEFORM} \\
\left \| K_{Crunch} - \frac{1}{3} \ID  \right\|_{C_{\StMet}^{10}(\mathbb{S}^3)} & \lesssim \varepsilon, 
	\label{E:KLIMITCLOSETOFLRW} \\
\left \| \Psi_{Crunch} - \sqrt{\frac{2}{3}} \right\|_{C^{10}(\mathbb{S}^3)} & \lesssim \varepsilon,
	\label{E:PLIMITCLOSETOFLRW}
\end{align}
\end{subequations}
where $\ID$ denotes the identity transformation.

In addition, the limiting fields verify the following relations:
\begin{subequations}
\begin{align} 
	(K_{Crunch})_{\ a}^a & = - 1, 
		\label{E:LIMITINGKTRACE} \\
	\Psi_{Crunch}^2 + K_{Crunch} \cdot K_{Crunch} & = 1.
		\label{E:LIMITINGFIELDCONSTRAINT}
\end{align}
\end{subequations}

In addition, there exists a type $\binom{0}{2}$ tensorfield $M^{Bang} \in C_{\StMet}^{10}(\mathbb{S}^3)$
such that
\begin{align} \label{E:METRICENDSTATE}
	\left \| M^{Bang} - \StMet \right \|_{C_{\StMet}^{10}(\mathbb{S}^3)} 
	& \lesssim \varepsilon 
\end{align}
and such that the following convergence estimates hold for $t \in [0,\TCrunch)$:
\begin{align} \label{E:LIMITINGMETRICBEHAVIOR}
	\left\| g \cdot \mbox{\upshape exp} \left\lbrace 2 \left(\int_{s=0}^t \scale^{-1}(s) \, ds \right) K_{Crunch} \right\rbrace 
		- 
		M^{Bang} \right\|_{C_{\StMet}^{10}(\mathbb{S}^3)} 
	& \lesssim \varepsilon t^{4/3 - c \sqrt{\varepsilon}},
\end{align}
where relative to arbitrary local coordinates on 
$\mathbb{S}^3$, 
$
(\mbox{\upshape exp} \left\lbrace 2 \left(\int_{s=0}^t \scale^{-1}(s) \, ds \right) K_{Crunch} \right\rbrace)_{\ j}^i
$
denotes the $(i,j)$ component of the type $\binom{1}{1}$ tensorfield
whose components are given by the exponential of the matrix of components 
of the type $\binom{1}{1}$ tensorfield
$
2 \left(\int_{s=0}^t \scale^{-1}(s) \, ds \right) K_{Crunch}
$.

\medskip

\noindent{\textbf{Quantities that blow up:}} The norm 
$|\SecondFund|_g = |\SecondFund_{\ b}^a \SecondFund_{\ a}^b|^{1/2}$
of the second fundamental form $\SecondFund$ of $\Sigma_t$ verifies the estimate
\begin{align} \label{E:SECFUNDFORMBLOWUP}
	\left \| \scale |\SecondFund_{\ b}^a \SecondFund_{\ a}^b|^{1/2} - |K_{Crunch} \cdot K_{Crunch}|^{1/2} \right \|_{C^0(\Sigma_t)} 
	& \lesssim \varepsilon \scale^{4/3 - c \sqrt{\epsilon}}(t),
\end{align}
which shows that $|\SecondFund|_g$ blows up like $\scale^{-1}(t)$ as $t \uparrow \TCrunch$. 

The spacetime Ricci curvature invariant $\Ricfour^{\alpha \beta} \Ricfour_{\alpha \beta}$
verifies the estimate
\begin{align} \label{E:SPACETIMERICCIINVARIANTBLOWUP}
	\left\| 
		\scale^4 \Ricfour^{\alpha \beta} \Ricfour_{\alpha \beta}
		- \Psi_{Crunch}^4 \right 
		\|_{C^0(\Sigma_t)} 
		& \lesssim C \varepsilon \scale^{4/3 - c \sqrt{\epsilon}}(t),
\end{align}
which shows that $\Ricfour^{\alpha \beta} \Ricfour_{\alpha \beta}$
blows up like $\scale^{-4}(t)$ as $t \uparrow \TCrunch$.

\ \\
	
\noindent{\textbf{Geodesic incompleteness:}}
Every future-directed causal geodesic $\pmb{\zeta}$ that emanates from 
$\Sigma_0$ crashes into the singular hypersurface $\Sigma_{\TCrunch}$ in finite affine parameter time $a(\TCrunch)$,
where
	\begin{align} \label{E:AFFINTEBLOWUPTIME}
	a(\TCrunch) 
	& \leq 
		a'(0)
		\int_{\tau=0}^{\TCrunch}
			\exp \left\lbrace \left(\frac{1}{3} + C \varepsilon \right) \int_{s=0}^{\tau} \scale^{-1}(s) \, ds \right\rbrace
		\, d \tau
		< \infty,
\end{align}	
and $a = a(t)$ is the affine parameter along $\pmb{\zeta}$ viewed as a function of $t$ along $\pmb{\zeta}$ 
(normalized by $a(0) = 0$).
The finiteness of the double time integral in \eqref{E:AFFINTEBLOWUPTIME}
follows from Cor.~\ref{C:SCALEFACTORTIMEINTEGRALS}.
Similarly, every past-directed causal geodesic $\pmb{\zeta}$ that emanates from 
$\Sigma_0$ crashes into the singular hypersurface $\Sigma_{\TBang}$ in finite affine parameter time.

\end{theorem}

\begin{remark}[\textbf{AVTD behavior}]
	\label{R:AVTD}
	The convergence estimates \eqref{E:LAPSELIMIT}-\eqref{E:SPACESFLIMIT}
	capture the AVTD behavior described in Theorem~\ref{T:VERYROUGH}.
	More precisely, consider the fields
	$(\widetilde{n}:= 1, 
	\widetilde{g} 
	:= M^{Bang} \cdot \mbox{\upshape exp} \left\lbrace - 2 \left(\int_{s=0}^t \scale^{-1}(s) \, ds \right) K_{Crunch} \right\rbrace,
	\widetilde{\SecondFund} := \scale^{-1} K_{Crunch},
	\partial_t \widetilde{\phi}:= \scale^{-1} \Psi_{Crunch})$
	that are formally obtained by setting RHSs~\eqref{E:LAPSELIMIT}-\eqref{E:SPACESFLIMIT}
	equal to $0$. It is easy to see that these fields are solutions
	to the ``VTD equations,'' 
	which by definition are obtained by setting
	the spatial derivative terms  
	in the equations of Prop.~\ref{P:EINSTEININCMC}
	equal to $0$. It is in this sense that
	the nonlinear solution converges towards a solution of the VTD equations.
\end{remark}

\begin{remark}
	Note that our smallness assumption on $\highnorm{16}(0)$ is not
	a smallness assumption on the geometric initial data
	since it also entails a smallness assumption on $n-1$ along the initial Cauchy hypersurface
	$\Sigma_0$ (the lapse $n$ is not one of the geometric data). However, we could have formulated a near-FLRW 
	assumption on the geometric data in such a way that the smallness of $\highnorm{16}(0)$ 
	would follow as a consequence; one could derive the desired initial smallness of $n-1$
	as a consequence of a near-FLRW geometric data assumption
	via the elliptic PDE \eqref{E:LAPSEPDERENORMALIZEDHIGHERDERIVATIVES}.
	For convenience, we have avoided doing this.
\end{remark}

\begin{remark}
	It is possible to derive additional information about the
	solution, for example that the Kretschmann scalar
	$\Riemfour^{\alpha \beta \gamma \delta} \Riemfour_{\alpha \beta \gamma \delta}$
	blows up like $\scale^{-4}$ and that
	product of $\scale^2$ and the Weyl curvature tensor of $\gfour$
	remains $\mathcal{O}(\varepsilon)$ throughout the evolution.
	Readers can consult \cite{iRjS2014b}*{Theorem~2} for
	information about the kinds of additional estimates that hold and for the main ideas behind
	how to prove them using the estimates we have already derived in this paper.
\end{remark}

\begin{proof}[\textbf{Proof of Theorem~\ref{T:MAINTHM}}]
	We will prove the results of the theorem only for the half-space $[0,\TCrunch) \times \mathbb{S}^3$ 
	since the complementary half-space $(\TBang,0] \times \mathbb{S}^3$ can be treated
	using the same ideas.
	To proceed, we note that the solution variables $(n,g,k,\partial_t \phi,\nabla \phi)$
	featured in the equations of Prop.~\ref{P:EINSTEININCMC}
	are uniquely determined by the rescaled variables of Def.\ \ref{D:RESCALEDVARIABLES}
	and vice versa for $t \in [0,\TCrunch)$,
	the reason being that $\scale(t)$ is positive on this time interval.
	Thus, for $t \in [0,\TCrunch)$,
	equations \eqref{E:HAMILTONIAN}-\eqref{E:LAPSE}
	are equivalent to
	the equations verified by the rescaled variables,
	that is, equations
	\eqref{E:RENORMALIZEDHAMILTONIAN}-\eqref{E:ALTERNATERENORMALIZEDMOMENTUM},
	\eqref{E:EVOLUTIONMETRICRENORMALIZED}-\eqref{E:EVOLUTIONSECONDFUNDRENORMALIZED},
	\eqref{E:WAVEEQUATIONRENORMALIZED}-\eqref{E:EVOLUTIONSPACESCALARRENORMALIZED},
	\eqref{E:LAPSEPDERENORMALIZEDHIGHERDERIVATIVES},
	and \eqref{E:LAPSEPDERENORMALIZEDLOWERDERIVATIVES}.
	Next, we note the following standard local well-posedness result (see \cite[Theorem 6.2]{aS2010}):
	since the data verify the CMC condition \eqref{E:INITIALCMC}
	and the constraints \eqref{E:GAUSSINTRO}-\eqref{E:CODAZZIINTRO},
	there exists a time $\Tboot$ with $0 < \Tboot < \TCrunch$
	such that if $\varepsilon$ is sufficiently small, then
	\textbf{i)} the rescaled variables are classical solutions to the equations mentioned above
	on $[0,\Tboot) \times \mathbb{S}^3$
	and \textbf{ii)} the bootstrap assumption \eqref{E:HIGHNORMBOOTSTRAP}
	holds on $[0,\Tboot)$
	with $\upsigma = \varepsilon^{1/4}$ (consistent with \eqref{E:PARAMETERCONSTRAINT}).
	Let $T_{Max}$ be the sup over all such times $\Tboot$.
	We will show that $T_{Max} = \TCrunch$.
	To this end, 
	we note the following standard continuation principle
	(see, for example, \cite{lAvM2003}):
	if $\Tboot < \TCrunch$, if the solution exists classically on $[0,\Tboot) \times \mathbb{S}^3$,
	and if
	$\sup_{t \in [0,\Tboot)} \highnorm{16}(t) < A$ for some real number $A < \infty$,
	then there exists a $\delta > 0$ such that 
	$\Tboot + \delta < \TCrunch$,
	such that the solution
	exists (classically) on $[0,\Tboot + \delta] \times \mathbb{S}^3$,
	and such that $\sup_{t \in [0,\Tboot + \delta]} \highnorm{16}(t) < A$.
	Thus, to show that $T_{Max} = T_{Crunch}$
	(and in particular that the solution exists classically on 
	$[0,\TCrunch) \times \mathbb{S}^3$), 
	we need only to have a priori estimates guaranteeing that
	the bootstrap assumption \eqref{E:HIGHNORMBOOTSTRAP}
	is never saturated on $[0,T_{Max})$.
	Such estimates follow from \eqref{E:MAINHIGHNORMENERGYESTIMATES}
	whenever $\varepsilon$ is sufficiently small.
	
	
	Based on the estimate
	\eqref{E:MAINHIGHNORMENERGYESTIMATES}
	and the strong sup-norm estimates provided by Prop.~\ref{P:STRONGSUPNORMESTIMATES},
	the remaining aspects of the theorem can be proved
	by using the same arguments given in the proof
	of \cite{iRjS2014b}*{Theorem~2} and thus we refer
	the reader there for details. Here we 
	prove only the estimates \eqref{E:KLIMIT} and \eqref{E:KLIMITCLOSETOFLRW}
	to give the reader a feel for the arguments.
	To proceed, we first revisit the proof of \eqref{E:NOLOSSTIMEDERIVATIVESKSTRONGSUPNROMSTMETRIC},
	using the estimate \eqref{E:MAINHIGHNORMENERGYESTIMATES} in place of the original bootstrap assumption
	\eqref{E:HIGHNORMBOOTSTRAP}, which yields the following pointwise for $t \in [0,\TCrunch)$:
	$
	\left\|
		\SigmatLie_{\partial_t} \FreeNewSec
	\right\|_{C_{\StMet}^{10}(\Sigma_t)}
	\lesssim \varepsilon^{3/2} \scale^{1/3 -c \sqrt{\varepsilon}}(t)
	$.
	Integrating in time and using this bound and Cor.~\ref{C:SCALEFACTORTIMEINTEGRALS},
	we deduce that the following bound holds
	for $0 < s \leq t < \TCrunch$:
	$
	\left\|
		\FreeNewSec(t,\cdot)
		-
		\FreeNewSec(s,\cdot)
	\right\|_{C_{\StMet}^{10}(\mathbb{S}^3)}
	\lesssim
	\varepsilon^{3/2}
	\scale^{4/3 -c \sqrt{\varepsilon}}(s)
	$.
	From this bound, 
	Lemma~\ref{L:ANALYSISOFFRIEDMANN},
	and the completeness of the space
	$C_{\StMet}^{10}(\mathbb{S}^3)$,
	it follows that
	$
	\lim_{t \uparrow \TCrunch}
	\FreeNewSec(t,\cdot)
	$
	exists as an element of
	$C_{\StMet}^{10}(\mathbb{S}^3)$.
	Since 
	$\scale \SecondFund 
	= \FreeNewSec
		-
		\frac{1}{3} \scale' \ID
	$, with $\ID$ the identity transformation,
	and since
	$
	\lim_{t \uparrow \TCrunch}
	\scale'(t)
	=
	-1
	$
	(see Lemma~\ref{L:ANALYSISOFFRIEDMANN}),
	it also follows that
	$
	\lim_{t \uparrow \TCrunch}
	\scale(t) \SecondFund(t,\cdot)
	$
	exists as an element of
	$C_{\StMet}^{10}(\mathbb{S}^3)$,
	and we denote the limit
	by $K_{Crunch}$.
	From the above facts 
	and the small-data bound
	$
	\left\|
		\FreeNewSec
	\right\|_{C^{10}(\Sigma_0)}
	\lesssim
	\varepsilon^2
	$
	(see \eqref{E:SMALLDATA}),
	we arrive at the desired estimates \eqref{E:KLIMIT} and \eqref{E:KLIMITCLOSETOFLRW}
	(which are non-optimal with respect to powers of $\varepsilon$).
	
\end{proof}

\bibliographystyle{amsalpha}
\bibliography{JBib}

\end{document}